\documentclass[10pt,reqno,oneside]{amsproc} \title[{On the pointwise Schauder estimates for elliptic equations}]{On the pointwise Schauder estimates for elliptic equations} \author[I.~Kukavica]{Igor Kukavica} \address{Department of Mathematics, University of Southern California, Los Angeles, CA 90089} \email{kukavica@usc.edu} \author[Q.~Le]{Quinn Le} \address{Department of Mathematics, University of Southern California, Los Angeles, CA 90089} \email{ntle@usc.edu}   \chardef\forshowkeys=0   \chardef\refcheck=0   \chardef\showllabel=0   \chardef\sketches=0 \usepackage{enumitem} \usepackage{fancyhdr} \usepackage{comment} \ifnum\forshowkeys=1      \usepackage[notref,notcite,color]{showkeys} \fi \usepackage[margin=1in]{geometry} \usepackage{amsmath, amsthm, amssymb} \usepackage{times} \usepackage{graphicx} \usepackage[usenames,dvipsnames,svgnames,table]{xcolor} \usepackage{marginnote} \usepackage[unicode,breaklinks=true,colorlinks=true,linkcolor=blue,urlcolor=blue,citecolor=blue]{hyperref} \usepackage{tikz}  \ifnum\refcheck=1   \usepackage{refcheck} \fi \begin{document} \baselineskip=4.5truemm \def\FF{\tilde F} \def\PP{\mathcal{P}} \def\dPP{\dot{\mathcal{P}}} \def\dQQ{\dot{\mathcal{Q}}} \def\LL{\mathcal L} \def\qd{q_{\textrm D}^{}} \def\AA{C_2} \def\BB{(C_1+C_2)} \def\pp{p} \def\qq{q} \def\MM{M} \def\KK{K} \def\CC{C} \def\uu{\tilde u } \def\vv{\tilde v } \def\ww{\tilde w } \def\ggg{u} \def\UU{\tilde U} \def\eps{\epsilon} \def\QQ{{\bar Q}} \def\cc{c} \def\RR{\mathbb R} \def\TT{\mathbb T} \def\ZZ{\mathbb Z} \def\erf{\mathrm{Erf}} \def\red#1{\textcolor{red}{#1}} \def\blue#1{\textcolor{blue}{#1}} \def\mgt#1{\textcolor{magenta}{#1}} \def\ppsi{(u - P)} \def\les{\lesssim} \def\ges{\gtrsim} \renewcommand*{\Re}{\ensuremath{\mathrm{{\mathbb R}e\,}}} \renewcommand*{\Im}{\ensuremath{\mathrm{{\mathbb I}m\,}}} \ifnum\showllabel=1  \def\llabel#1{\marginnote{\color{lightgray}\rm\small(#1)}[-0.0cm]\notag} \else  \def\llabel#1{\notag} \fi \newcommand{\norm}[1]{\left\|#1\right\|} \newcommand{\nnorm}[1]{\lVert #1\rVert} \newcommand{\abs}[1]{\left|#1\right|} \newcommand{\NORM}[1]{|\!|\!| #1|\!|\!|} \newtheorem{Theorem}{Theorem}[section] \newtheorem{Corollary}[Theorem]{Corollary} \newtheorem{Definition}[Theorem]{Definition} \newtheorem{Proposition}[Theorem]{Proposition} \newtheorem{Lemma}[Theorem]{Lemma} \newtheorem{Remark}[Theorem]{Remark} \def\theequation{\thesection.\arabic{equation}} \numberwithin{equation}{section} \definecolor{myblue}{rgb}{.8, .8, 1} \newlength\mytemplen \newsavebox\mytempbox \makeatletter \makeatother \def\biglinem{\vskip0.5truecm\par==========================\par\vskip0.5truecm} \def\inon#1{\hbox{\quad{} }\hbox{#1}}                 \def\and{\text{\indeq and\indeq}} \def\onon#1{\quad{}\text{on~$#1$}} \def\inin#1{\quad{}\text{in~$#1$}} \def\startnewsection#1#2{\section{#1}\label{#2}\setcounter{equation}{0}}    \def\sgn{\mathop{\rm sgn\,}\nolimits}     \def\Tr{\mathop{\rm Tr}\nolimits}     \def\div{\mathop{\rm div}\nolimits} \def\curl{\mathop{\rm curl}\nolimits} \def\dist{\mathop{\rm dist}\nolimits}   \def\sp{\mathop{\rm sp}\nolimits}   \def\supp{\mathop{\rm supp}\nolimits} \def\indeq{\quad{}}            \def\colr{\color{red}} \def\colb{\color{black}} \def\coly{\color{lightgray}} \definecolor{colorgggg}{rgb}{0.1,0.5,0.3} \definecolor{colorllll}{rgb}{0.0,0.7,0.0} \definecolor{colorhhhh}{rgb}{0.3,0.75,0.4} \definecolor{colorpppp}{rgb}{0.7,0.0,0.2} \definecolor{coloroooo}{rgb}{0.45,0.0,0.0} \definecolor{colorqqqq}{rgb}{0.1,0.7,0} \def\colg{\color{colorgggg}} \def\collg{\color{colorllll}} \def\cole{\color{coloroooo}} \def\coleo{\color{colorpppp}} \def\cole{} \def\colu{\color{blue}} \def\colc{\color{colorhhhh}} \def\colW{\colb}    \definecolor{coloraaaa}{rgb}{0.6,0.6,0.6} \def\colw{\color{coloraaaa}} \def\comma{ {\rm ,\qquad{}} }             \def\commaone{ {\rm ,\quad{}} }           \def\nts#1{{\color{red}\hbox{\bf ~#1~}}}  \def\ntsf#1{\footnote{\color{colorgggg}\hbox{#1}}}  \def\blackdot{{\color{red}{\hskip-.0truecm\rule[-1mm]{4mm}{4mm}\hskip.2truecm}}\hskip-.3truecm} \def\bluedot{{\color{blue}{\hskip-.0truecm\rule[-1mm]{4mm}{4mm}\hskip.2truecm}}\hskip-.3truecm} \def\purpledot{{\color{colorpppp}{\hskip-.0truecm\rule[-1mm]{4mm}{4mm}\hskip.2truecm}}\hskip-.3truecm} \def\greendot{{\color{colorgggg}{\hskip-.0truecm\rule[-1mm]{4mm}{4mm}\hskip.2truecm}}\hskip-.3truecm} \def\cyandot{{\color{cyan}{\hskip-.0truecm\rule[-1mm]{4mm}{4mm}\hskip.2truecm}}\hskip-.3truecm} \def\reddot{{\color{red}{\hskip-.0truecm\rule[-1mm]{4mm}{4mm}\hskip.2truecm}}\hskip-.3truecm} \def\tdot{{\color{green}{\hskip-.0truecm\rule[-.5mm]{2mm}{4mm}\hskip.2truecm}}\hskip-.2truecm}
\def\gdot{\greendot} \def\bdot{\bluedot} \def\ydot{\cyandot} \def\zxvczxbcvdfghasdfrtsdafasdfasdfdsfgsdgf{\int}\def\zxvczxbcvdfghasdfrtsdafasdfasdfdsfgsdgh{\Vert} \def\rdot{\cyandot} \def\fractext#1#2{{#1}/{#2}} \def\quinn#1{\text{{\textcolor{colorqqqq}{#1}}}} \def\igor#1{\footnote{\text{{\textcolor{colorqqqq}{#1}}}}} \begin{abstract} We consider the pointwise in space $L^p$-type regularity for elliptic and parabolic equations of order $m$ in $\mathbb{R}^{n}$. We provide pointwise Schauder estimates for the general range of $L^{p}$ exponents, extending previous results from $p>n/m$ to $1<p<n/m$. \end{abstract}
\maketitle \setcounter{tocdepth}{2}  \tableofcontents \colb \startnewsection{Introduction}{sec01} We address Schauder estimates for elliptic equations      \begin{equation}       Lu        \equiv \sum_{0\leq \nu\leq m}^m a_{\nu}(x) \partial^{\nu} u        = f        \comma       x \in B_1     \label{8ThswELzXU3X7Ebd1KdZ7v1rN3GiirRXGKWK099ovBM0FDJCvkopYNQ2aN94Z7k0UnUKamE3OjU8DFYFFokbSI2J9V9gVlM8ALWThDPnPu3EL7HPD2VDaZTggzcCCmbvc70qqPcC9mt60ogcrTiA3HEjwTK8ymKeuJMc4q6dVz200XnYUtLR9GYjPXvFOVr6W1zUK1WbPToaWJJuKnxBLnd0ftDEbMmj4loHYyhZyMjM91zQS4p7z8eKa9h0JrbacekcirexG0z4n3xz0QOWSvFj3jLhWXUIU21iIAwJtI3RbWa90I7rzAIqI3UElUJG7tLtUXzw4KQNETvXzqWaujEMenYlNIzLGxgB3AuJEQ01}     \end{equation} of an arbitrary order~$m$, where $f \in L^p(B_1)$ and all $p \in [ 1, \infty)$. In particular, we determine the range of $q$ such that we obtain the vanishing order of solution $u$ of \eqref{8ThswELzXU3X7Ebd1KdZ7v1rN3GiirRXGKWK099ovBM0FDJCvkopYNQ2aN94Z7k0UnUKamE3OjU8DFYFFokbSI2J9V9gVlM8ALWThDPnPu3EL7HPD2VDaZTggzcCCmbvc70qqPcC9mt60ogcrTiA3HEjwTK8ymKeuJMc4q6dVz200XnYUtLR9GYjPXvFOVr6W1zUK1WbPToaWJJuKnxBLnd0ftDEbMmj4loHYyhZyMjM91zQS4p7z8eKa9h0JrbacekcirexG0z4n3xz0QOWSvFj3jLhWXUIU21iIAwJtI3RbWa90I7rzAIqI3UElUJG7tLtUXzw4KQNETvXzqWaujEMenYlNIzLGxgB3AuJEQ01} in $L^q$ by approximating  with polynomials. Additionally, we derive the pointwise Schauder estimate of the form      \begin{equation}       \lbrack u \rbrack_{C^{d,\alpha}_{L^q}(0)}       \les \zxvczxbcvdfghasdfrtsdafasdfasdfdsfgsdgh u \zxvczxbcvdfghasdfrtsdafasdfasdfdsfgsdgh_{L^p(B_1)}              + \zxvczxbcvdfghasdfrtsdafasdfasdfdsfgsdgh f \zxvczxbcvdfghasdfrtsdafasdfasdfdsfgsdgh_{L^p(B_1)}              + \lbrack f \rbrack_{C^{d-m,\alpha}_{L^p}(0)}       ,    \llabel{8Th sw ELzX U3X7 Ebd1Kd Z7 v 1rN 3Gi irR XG KWK0 99ov BM0FDJ Cv k opY NQ2 aN9 4Z 7k0U nUKa mE3OjU 8D F YFF okb SI2 J9 V9gV lM8A LWThDP nP u 3EL 7HP D2V Da ZTgg zcCC mbvc70 qq P cC9 mt6 0og cr TiA3 HEjw TK8ymK eu J Mc4 q6d Vz2 00 XnYU tLR9 GYjPXv FO V r6W 1zU K1W bP ToaW JJuK nxBLnd 0f t DEb Mmj 4lo HY yhZy MjM9 1zQS4p 7z 8 eKa 9h0 Jrb ac ekci rexG 0z4n3x z0 Q OWS vFj 3jL hW XUIU 21iI AwJtI3 Rb W a90 I7r zAI qI 3UEl UJG7 tLtUXz w4 K QNE TvX zqW au jEMe nYlN IzLGxg B3 A uJ8 6VS 6Rc PJ 8OXW w8im tcKZEz Ho p 84G 1gS As0 PC owMI 2fLK TdD608ThswELzXU3X7Ebd1KdZ7v1rN3GiirRXGKWK099ovBM0FDJCvkopYNQ2aN94Z7k0UnUKamE3OjU8DFYFFokbSI2J9V9gVlM8ALWThDPnPu3EL7HPD2VDaZTggzcCCmbvc70qqPcC9mt60ogcrTiA3HEjwTK8ymKeuJMc4q6dVz200XnYUtLR9GYjPXvFOVr6W1zUK1WbPToaWJJuKnxBLnd0ftDEbMmj4loHYyhZyMjM91zQS4p7z8eKa9h0JrbacekcirexG0z4n3xz0QOWSvFj3jLhWXUIU21iIAwJtI3RbWa90I7rzAIqI3UElUJG7tLtUXzw4KQNETvXzqWaujEMenYlNIzLGxgB3AuJEQ02}     \end{equation} where      \begin{equation}       \lbrack u \rbrack_{C^{d,\alpha}_{L^p}(0)}       = \inf_{P \in \mathbb P_d} \sup_{0<r<1}          \frac{\zxvczxbcvdfghasdfrtsdafasdfasdfdsfgsdgh u-P \zxvczxbcvdfghasdfrtsdafasdfasdfdsfgsdgh_{L^p(B_r)}}{r^{d+\alpha+n/p}}    \llabel{ mt6 0og cr TiA3 HEjw TK8ymK eu J Mc4 q6d Vz2 00 XnYU tLR9 GYjPXv FO V r6W 1zU K1W bP ToaW JJuK nxBLnd 0f t DEb Mmj 4lo HY yhZy MjM9 1zQS4p 7z 8 eKa 9h0 Jrb ac ekci rexG 0z4n3x z0 Q OWS vFj 3jL hW XUIU 21iI AwJtI3 Rb W a90 I7r zAI qI 3UEl UJG7 tLtUXz w4 K QNE TvX zqW au jEMe nYlN IzLGxg B3 A uJ8 6VS 6Rc PJ 8OXW w8im tcKZEz Ho p 84G 1gS As0 PC owMI 2fLK TdD60y nH g 7lk NFj JLq Oo Qvfk fZBN G3o1Dg Cn 9 hyU h5V SP5 z6 1qvQ wceU dVJJsB vX D G4E LHQ HIa PT bMTr sLsm tXGyOB 7p 2 Os4 3US bq5 ik 4Lin 769O TkUxmp I8 u GYn fBK bYI 9A QzCF w3h0 8ThswELzXU3X7Ebd1KdZ7v1rN3GiirRXGKWK099ovBM0FDJCvkopYNQ2aN94Z7k0UnUKamE3OjU8DFYFFokbSI2J9V9gVlM8ALWThDPnPu3EL7HPD2VDaZTggzcCCmbvc70qqPcC9mt60ogcrTiA3HEjwTK8ymKeuJMc4q6dVz200XnYUtLR9GYjPXvFOVr6W1zUK1WbPToaWJJuKnxBLnd0ftDEbMmj4loHYyhZyMjM91zQS4p7z8eKa9h0JrbacekcirexG0z4n3xz0QOWSvFj3jLhWXUIU21iIAwJtI3RbWa90I7rzAIqI3UElUJG7tLtUXzw4KQNETvXzqWaujEMenYlNIzLGxgB3AuJEQ03}     \end{equation} denotes a pointwise continuity norm~\cite{H1}. The employed method also works for a parabolic equation with  maintaining the minimal requirements for $d$, $m$, $p$, and the coefficients $a_{\nu}(x)$ from the elliptic case, as shown in the last section. \par Schauder estimates are a cornerstone of the regularity theory for elliptic and parabolic equations. For the classical theory of pointwise Schauder estimates for the second order elliptic equations, see the classical textbook~\cite{GT}. In \cite{B}, Bers considered an elliptic equation of a general order and established the existence of an asymptotic polynomial at a point where the solution vanishes.  He namely determined the behavior of solution near $0$ with an estimate for error term, i.e, $u(x) = P_d(x) + O(|x|^{d+\eps})$ where $P_d$ being a polynomial of degree $d$ and $\eps$ being H\"older exponent of the leading coefficients of~$L$.  This in a sense generalizes the existence of a Taylor polynomial to solutions of elliptic equations with rough coefficients (the leading coefficient is only H\"older). An asymptotic polynomial for solutions of the second order parabolic equations was obtained by Alessandrini and Vessella in~\cite{AV}, where they proved that  a solution $u$ either has a zero of infinite order or $u$ can be approximated by some polynomial $P$ of order $d$ such that they can estimate $u-P$ in $L^{\infty}(B_r)$ for small~$r>0$.  The Schauder estimates in $L^{q}$ (rather then pointwise) were obtained by Han in  \cite{H1} in the elliptic and in \cite{H2} for the parabolic case. He namely  derived an estimation for $\zxvczxbcvdfghasdfrtsdafasdfasdfdsfgsdgh u - P \zxvczxbcvdfghasdfrtsdafasdfasdfdsfgsdgh_{L^{\infty}(B_r)}$ under the restriction $p > m/n$.  Note that this restriction is natural for the purpose of Schauder estimates due to the Sobolev embedding $W^{m,p}\hookrightarrow L^{\infty}$ under this condition. We refer the reader to \cite{Br,DK1,DK2,L,M,T} for some other works on the Schauder estimates for elliptic and to \cite{CK,D,S,W} for parabolic type equations. \par In this paper, we address the remaining range $p < m/n$ and  introduce an interval of $q$ where we can estimate $\zxvczxbcvdfghasdfrtsdafasdfasdfdsfgsdgh u - P\zxvczxbcvdfghasdfrtsdafasdfasdfdsfgsdgh_{L^q(B_r)}$  for a suitable asymptotic polynomial. Our results are in correspondence with Han's result when $p$ approaches $m/n$, in which the end point of the interval for $q$ goes to infinity.  We obtain explicit estimates for $\zxvczxbcvdfghasdfrtsdafasdfasdfdsfgsdgh u - P\zxvczxbcvdfghasdfrtsdafasdfasdfdsfgsdgh_{L^q(B_r)}$ and $\zxvczxbcvdfghasdfrtsdafasdfasdfdsfgsdgh D^i(u-P) \zxvczxbcvdfghasdfrtsdafasdfasdfdsfgsdgh_{L^q(B_r)}$ where $1 \leq q < np/(n-mp)$ and $P$ is a homogeneous polynomial. We also obtain an upper bound for $\lbrack u \rbrack_{C^{d+l,\alpha}_{L^q}(0)}$ where the Schauder pointwise norm      \begin{equation}       \lbrack u \rbrack_{C^{d,\alpha}_{L^p}(0)}        = \inf_{P \in \mathbb P_d} \sup_{0<r<1}          \frac{\zxvczxbcvdfghasdfrtsdafasdfasdfdsfgsdgh u-P \zxvczxbcvdfghasdfrtsdafasdfasdfdsfgsdgh_{L^p(B_r)}}{r^{d+\alpha+n/p}}    \llabel{Q OWS vFj 3jL hW XUIU 21iI AwJtI3 Rb W a90 I7r zAI qI 3UEl UJG7 tLtUXz w4 K QNE TvX zqW au jEMe nYlN IzLGxg B3 A uJ8 6VS 6Rc PJ 8OXW w8im tcKZEz Ho p 84G 1gS As0 PC owMI 2fLK TdD60y nH g 7lk NFj JLq Oo Qvfk fZBN G3o1Dg Cn 9 hyU h5V SP5 z6 1qvQ wceU dVJJsB vX D G4E LHQ HIa PT bMTr sLsm tXGyOB 7p 2 Os4 3US bq5 ik 4Lin 769O TkUxmp I8 u GYn fBK bYI 9A QzCF w3h0 geJftZ ZK U 74r Yle ajm km ZJdi TGHO OaSt1N nl B 7Y7 h0y oWJ ry rVrT zHO8 2S7oub QA W x9d z2X YWB e5 Kf3A LsUF vqgtM2 O2 I dim rjZ 7RN 28 4KGY trVa WW4nTZ XV b RVo Q77 hVL X6 K2kq 8ThswELzXU3X7Ebd1KdZ7v1rN3GiirRXGKWK099ovBM0FDJCvkopYNQ2aN94Z7k0UnUKamE3OjU8DFYFFokbSI2J9V9gVlM8ALWThDPnPu3EL7HPD2VDaZTggzcCCmbvc70qqPcC9mt60ogcrTiA3HEjwTK8ymKeuJMc4q6dVz200XnYUtLR9GYjPXvFOVr6W1zUK1WbPToaWJJuKnxBLnd0ftDEbMmj4loHYyhZyMjM91zQS4p7z8eKa9h0JrbacekcirexG0z4n3xz0QOWSvFj3jLhWXUIU21iIAwJtI3RbWa90I7rzAIqI3UElUJG7tLtUXzw4KQNETvXzqWaujEMenYlNIzLGxgB3AuJEQ04}     \end{equation} was introduced in \cite{H1}; note that, by translation, we may always restrict the analysis to the point~$0$ and thus all the statements apply to an arbitrary point in space. For elliptic case, we obtain      \begin{equation}       \lbrack u \rbrack_{C^{d+l,\alpha}_{L^q}(0)}             \les \zxvczxbcvdfghasdfrtsdafasdfasdfdsfgsdgh u \zxvczxbcvdfghasdfrtsdafasdfasdfdsfgsdgh_{L^p(B_1)}             + \lbrack f \rbrack_{C^{d-m}_{L^p}}(0)              + \lbrack f \rbrack_{C^{d-m+l,\alpha}_{L^p}(0)}             ,    \llabel{y nH g 7lk NFj JLq Oo Qvfk fZBN G3o1Dg Cn 9 hyU h5V SP5 z6 1qvQ wceU dVJJsB vX D G4E LHQ HIa PT bMTr sLsm tXGyOB 7p 2 Os4 3US bq5 ik 4Lin 769O TkUxmp I8 u GYn fBK bYI 9A QzCF w3h0 geJftZ ZK U 74r Yle ajm km ZJdi TGHO OaSt1N nl B 7Y7 h0y oWJ ry rVrT zHO8 2S7oub QA W x9d z2X YWB e5 Kf3A LsUF vqgtM2 O2 I dim rjZ 7RN 28 4KGY trVa WW4nTZ XV b RVo Q77 hVL X6 K2kq FWFm aZnsF9 Ch p 8Kx rsc SGP iS tVXB J3xZ cD5IP4 Fu 9 Lcd TR2 Vwb cL DlGK 1ro3 EEyqEA zw 6 sKe Eg2 sFf jz MtrZ 9kbd xNw66c xf t lzD GZh xQA WQ KkSX jqmm rEpNuG 6P y loq 8hH lSf Ma 8ThswELzXU3X7Ebd1KdZ7v1rN3GiirRXGKWK099ovBM0FDJCvkopYNQ2aN94Z7k0UnUKamE3OjU8DFYFFokbSI2J9V9gVlM8ALWThDPnPu3EL7HPD2VDaZTggzcCCmbvc70qqPcC9mt60ogcrTiA3HEjwTK8ymKeuJMc4q6dVz200XnYUtLR9GYjPXvFOVr6W1zUK1WbPToaWJJuKnxBLnd0ftDEbMmj4loHYyhZyMjM91zQS4p7z8eKa9h0JrbacekcirexG0z4n3xz0QOWSvFj3jLhWXUIU21iIAwJtI3RbWa90I7rzAIqI3UElUJG7tLtUXzw4KQNETvXzqWaujEMenYlNIzLGxgB3AuJEQ05}     \end{equation} while for the for parabolic case     \begin{equation}     \sum_{i=d}^{d+1} \lbrack u \rbrack_{C^i_{L^q}(0)}       + \lbrack u \rbrack_{C^{d+l,\alpha}_{L^q}(0)}       \les \zxvczxbcvdfghasdfrtsdafasdfasdfdsfgsdgh u \zxvczxbcvdfghasdfrtsdafasdfasdfdsfgsdgh_{L^p(Q_1)}             + \sum_{i=d-m}^{d-m+l} \lbrack f \rbrack_{C^i_{L^p}(0)}             + \lbrack f \rbrack_{C^{d-m+l+\alpha}_{L^p}(0)}     ,    \llabel{geJftZ ZK U 74r Yle ajm km ZJdi TGHO OaSt1N nl B 7Y7 h0y oWJ ry rVrT zHO8 2S7oub QA W x9d z2X YWB e5 Kf3A LsUF vqgtM2 O2 I dim rjZ 7RN 28 4KGY trVa WW4nTZ XV b RVo Q77 hVL X6 K2kq FWFm aZnsF9 Ch p 8Kx rsc SGP iS tVXB J3xZ cD5IP4 Fu 9 Lcd TR2 Vwb cL DlGK 1ro3 EEyqEA zw 6 sKe Eg2 sFf jz MtrZ 9kbd xNw66c xf t lzD GZh xQA WQ KkSX jqmm rEpNuG 6P y loq 8hH lSf Ma LXm5 RzEX W4Y1Bq ib 3 UOh Yw9 5h6 f6 o8kw 6frZ wg6fIy XP n ae1 TQJ Mt2 TT fWWf jJrX ilpYGr Ul Q 4uM 7Ds p0r Vg 3gIE mQOz TFh9LA KO 8 csQ u6m h25 r8 WqRI DZWg SYkWDu lL 8 Gpt ZW1 0G8ThswELzXU3X7Ebd1KdZ7v1rN3GiirRXGKWK099ovBM0FDJCvkopYNQ2aN94Z7k0UnUKamE3OjU8DFYFFokbSI2J9V9gVlM8ALWThDPnPu3EL7HPD2VDaZTggzcCCmbvc70qqPcC9mt60ogcrTiA3HEjwTK8ymKeuJMc4q6dVz200XnYUtLR9GYjPXvFOVr6W1zUK1WbPToaWJJuKnxBLnd0ftDEbMmj4loHYyhZyMjM91zQS4p7z8eKa9h0JrbacekcirexG0z4n3xz0QOWSvFj3jLhWXUIU21iIAwJtI3RbWa90I7rzAIqI3UElUJG7tLtUXzw4KQNETvXzqWaujEMenYlNIzLGxgB3AuJEQ06}     \end{equation} where the implicit constants does not depend on $u$ or~$f$. In our approach we use many ideas from the works \cite{H1,H2}, which in turn draw from \cite{B} and~\cite{AV}. Throughout this paper, we use several results on the bounds on the fundamental solutions. For elliptic equation, we use the upper bound established by Bers in~\cite{B} (for the second order  parabolic equation see~\cite{AV}) as well as results on homogeneity of polynomials.  \par The paper is structured as follows. In Section~\ref{sec02}, we provide the setup for the problem, introduce the terminology that we use throughout, and state two main theorems on Schauder estimates for elliptic equations. Section~\ref{sec03} contains the proofs of the elliptic theorems from Section~\ref{sec02}. In Section~\ref{sec04}, we provide an analogous results for parabolic equations and state the necessary modifications of the proofs from Section~\ref{sec03}. \par \startnewsection{The main results}{sec02} We derive the Schauder estimate for the general order elliptic equation    \begin{equation}    Lu = f    ,    \label{8ThswELzXU3X7Ebd1KdZ7v1rN3GiirRXGKWK099ovBM0FDJCvkopYNQ2aN94Z7k0UnUKamE3OjU8DFYFFokbSI2J9V9gVlM8ALWThDPnPu3EL7HPD2VDaZTggzcCCmbvc70qqPcC9mt60ogcrTiA3HEjwTK8ymKeuJMc4q6dVz200XnYUtLR9GYjPXvFOVr6W1zUK1WbPToaWJJuKnxBLnd0ftDEbMmj4loHYyhZyMjM91zQS4p7z8eKa9h0JrbacekcirexG0z4n3xz0QOWSvFj3jLhWXUIU21iIAwJtI3RbWa90I7rzAIqI3UElUJG7tLtUXzw4KQNETvXzqWaujEMenYlNIzLGxgB3AuJEQ07}   \end{equation}  where $L$ is an $m$-th order homogeneous elliptic linear operator in $B_1 \subseteq \mathbb R^n$ given by   \begin{equation}       L = \sum_{|\nu| \leq m} a_{\nu}(x) \partial^{\nu}       ,     \label{8ThswELzXU3X7Ebd1KdZ7v1rN3GiirRXGKWK099ovBM0FDJCvkopYNQ2aN94Z7k0UnUKamE3OjU8DFYFFokbSI2J9V9gVlM8ALWThDPnPu3EL7HPD2VDaZTggzcCCmbvc70qqPcC9mt60ogcrTiA3HEjwTK8ymKeuJMc4q6dVz200XnYUtLR9GYjPXvFOVr6W1zUK1WbPToaWJJuKnxBLnd0ftDEbMmj4loHYyhZyMjM91zQS4p7z8eKa9h0JrbacekcirexG0z4n3xz0QOWSvFj3jLhWXUIU21iIAwJtI3RbWa90I7rzAIqI3UElUJG7tLtUXzw4KQNETvXzqWaujEMenYlNIzLGxgB3AuJEQ08}   \end{equation} where we assume   $    m<n    .   $ (The case $m\geq n$ is covered by the results in~\cite{H1}.) Suppose that the coefficients of $L$ satisfy the ellipticity condition     \begin{equation}       \sum_{|\nu| = m} a_{\nu}(x) \xi^{\nu}        \geq \frac{1}{K}       \comma \xi \in \mathbb \partial B_1       \commaone x \in B_1       ,     \label{8ThswELzXU3X7Ebd1KdZ7v1rN3GiirRXGKWK099ovBM0FDJCvkopYNQ2aN94Z7k0UnUKamE3OjU8DFYFFokbSI2J9V9gVlM8ALWThDPnPu3EL7HPD2VDaZTggzcCCmbvc70qqPcC9mt60ogcrTiA3HEjwTK8ymKeuJMc4q6dVz200XnYUtLR9GYjPXvFOVr6W1zUK1WbPToaWJJuKnxBLnd0ftDEbMmj4loHYyhZyMjM91zQS4p7z8eKa9h0JrbacekcirexG0z4n3xz0QOWSvFj3jLhWXUIU21iIAwJtI3RbWa90I7rzAIqI3UElUJG7tLtUXzw4KQNETvXzqWaujEMenYlNIzLGxgB3AuJEQ09}     \end{equation} the boundedness     \begin{equation}       \sum_{|\nu| \leq m} |a_{\nu}(x)|        \leq K       \comma x \in B_1       ,     \label{8ThswELzXU3X7Ebd1KdZ7v1rN3GiirRXGKWK099ovBM0FDJCvkopYNQ2aN94Z7k0UnUKamE3OjU8DFYFFokbSI2J9V9gVlM8ALWThDPnPu3EL7HPD2VDaZTggzcCCmbvc70qqPcC9mt60ogcrTiA3HEjwTK8ymKeuJMc4q6dVz200XnYUtLR9GYjPXvFOVr6W1zUK1WbPToaWJJuKnxBLnd0ftDEbMmj4loHYyhZyMjM91zQS4p7z8eKa9h0JrbacekcirexG0z4n3xz0QOWSvFj3jLhWXUIU21iIAwJtI3RbWa90I7rzAIqI3UElUJG7tLtUXzw4KQNETvXzqWaujEMenYlNIzLGxgB3AuJEQ10}     \end{equation} and the H\"older continuity of the leading coefficients      \begin{equation}       \sum_{|\nu| = m} |a_{\nu}(x) - a_{\nu} (0)|        \leq K |x|^{\alpha}        \comma  x \in B_1       ,     \label{8ThswELzXU3X7Ebd1KdZ7v1rN3GiirRXGKWK099ovBM0FDJCvkopYNQ2aN94Z7k0UnUKamE3OjU8DFYFFokbSI2J9V9gVlM8ALWThDPnPu3EL7HPD2VDaZTggzcCCmbvc70qqPcC9mt60ogcrTiA3HEjwTK8ymKeuJMc4q6dVz200XnYUtLR9GYjPXvFOVr6W1zUK1WbPToaWJJuKnxBLnd0ftDEbMmj4loHYyhZyMjM91zQS4p7z8eKa9h0JrbacekcirexG0z4n3xz0QOWSvFj3jLhWXUIU21iIAwJtI3RbWa90I7rzAIqI3UElUJG7tLtUXzw4KQNETvXzqWaujEMenYlNIzLGxgB3AuJEQ11}     \end{equation} for some positive constants $K$ and $\alpha \in (0,1)$.
\par In the sequel, $C$ denotes a positive constant, which may change from line to line and is allowed to depend on the space dimension~$n$ and the order of the differential operator~$m$. We also allow all implicit constants to depend on $K$ and $\alpha$ without mention. We write $a\les b$ when $a\leq C b$ for some constant~$C$. Denote by $\mathcal{P}_d$ the set of polynomials in $n$ variables of degree at most $d$ and $\dot{\mathcal{P}}_d$ the set of homogeneous polynomials of degree exactly $d$, with an addition of the zero polynomial.  \par \begin{Definition} \label{D01} Let $u\in L^p(B_1)$ for $p \in \lbrack 1, \infty \rbrack$. For an integer $d \geq 1$, we say that $u \in C^d_{L^p}(0)$ if there exists $P\in \PP_{d-1}$ satisfying     \begin{equation}       \sup_{r \leq 1}\frac{\zxvczxbcvdfghasdfrtsdafasdfasdfdsfgsdgh u - P \zxvczxbcvdfghasdfrtsdafasdfasdfdsfgsdgh_{L^p(B_r)}}{r^{d+n/p}}       < \infty       .    \llabel{FWFm aZnsF9 Ch p 8Kx rsc SGP iS tVXB J3xZ cD5IP4 Fu 9 Lcd TR2 Vwb cL DlGK 1ro3 EEyqEA zw 6 sKe Eg2 sFf jz MtrZ 9kbd xNw66c xf t lzD GZh xQA WQ KkSX jqmm rEpNuG 6P y loq 8hH lSf Ma LXm5 RzEX W4Y1Bq ib 3 UOh Yw9 5h6 f6 o8kw 6frZ wg6fIy XP n ae1 TQJ Mt2 TT fWWf jJrX ilpYGr Ul Q 4uM 7Ds p0r Vg 3gIE mQOz TFh9LA KO 8 csQ u6m h25 r8 WqRI DZWg SYkWDu lL 8 Gpt ZW1 0Gd SY FUXL zyQZ hVZMn9 am P 9aE Wzk au0 6d ZghM ym3R jfdePG ln 8 s7x HYC IV9 Hw Ka6v EjH5 J8Ipr7 Nk C xWR 84T Wnq s0 fsiP qGgs Id1fs5 3A T 71q RIc zPX 77 Si23 GirL 9MQZ4F pi g dru N8ThswELzXU3X7Ebd1KdZ7v1rN3GiirRXGKWK099ovBM0FDJCvkopYNQ2aN94Z7k0UnUKamE3OjU8DFYFFokbSI2J9V9gVlM8ALWThDPnPu3EL7HPD2VDaZTggzcCCmbvc70qqPcC9mt60ogcrTiA3HEjwTK8ymKeuJMc4q6dVz200XnYUtLR9GYjPXvFOVr6W1zUK1WbPToaWJJuKnxBLnd0ftDEbMmj4loHYyhZyMjM91zQS4p7z8eKa9h0JrbacekcirexG0z4n3xz0QOWSvFj3jLhWXUIU21iIAwJtI3RbWa90I7rzAIqI3UElUJG7tLtUXzw4KQNETvXzqWaujEMenYlNIzLGxgB3AuJEQ19}     \end{equation} For $\alpha \in (0,1)$, we say that $u \in C^{d,\alpha}_{L^p}(0)$ if there exists $P\in \mathcal{P}_{d}$ such that     \begin{equation}       \sup_{r \leq 1}\frac{\zxvczxbcvdfghasdfrtsdafasdfasdfdsfgsdgh u - P \zxvczxbcvdfghasdfrtsdafasdfasdfdsfgsdgh_{L^p(B_r)}}{r^{d+\alpha+n/p}}       < \infty       .    \llabel{LXm5 RzEX W4Y1Bq ib 3 UOh Yw9 5h6 f6 o8kw 6frZ wg6fIy XP n ae1 TQJ Mt2 TT fWWf jJrX ilpYGr Ul Q 4uM 7Ds p0r Vg 3gIE mQOz TFh9LA KO 8 csQ u6m h25 r8 WqRI DZWg SYkWDu lL 8 Gpt ZW1 0Gd SY FUXL zyQZ hVZMn9 am P 9aE Wzk au0 6d ZghM ym3R jfdePG ln 8 s7x HYC IV9 Hw Ka6v EjH5 J8Ipr7 Nk C xWR 84T Wnq s0 fsiP qGgs Id1fs5 3A T 71q RIc zPX 77 Si23 GirL 9MQZ4F pi g dru NYt h1K 4M Zilv rRk6 B4W5B8 Id 3 Xq9 nhx EN4 P6 ipZl a2UQ Qx8mda g7 r VD3 zdD rhB vk LDJo tKyV 5IrmyJ R5 e txS 1cv EsY xG zj2T rfSR myZo4L m5 D mqN iZd acg GQ 0KRw QKGX g9o8v8 wm B 8ThswELzXU3X7Ebd1KdZ7v1rN3GiirRXGKWK099ovBM0FDJCvkopYNQ2aN94Z7k0UnUKamE3OjU8DFYFFokbSI2J9V9gVlM8ALWThDPnPu3EL7HPD2VDaZTggzcCCmbvc70qqPcC9mt60ogcrTiA3HEjwTK8ymKeuJMc4q6dVz200XnYUtLR9GYjPXvFOVr6W1zUK1WbPToaWJJuKnxBLnd0ftDEbMmj4loHYyhZyMjM91zQS4p7z8eKa9h0JrbacekcirexG0z4n3xz0QOWSvFj3jLhWXUIU21iIAwJtI3RbWa90I7rzAIqI3UElUJG7tLtUXzw4KQNETvXzqWaujEMenYlNIzLGxgB3AuJEQ20}     \end{equation} We also introduce the corresponding semi-norms as     \begin{equation}       [ u ]_{C^d_{L^p}(0)}       = \inf_{P\in\PP_{d-1}} \sup_{0<r<1} \frac{\zxvczxbcvdfghasdfrtsdafasdfasdfdsfgsdgh u - P \zxvczxbcvdfghasdfrtsdafasdfasdfdsfgsdgh_{L^p(B_r)}}{r^{d+n/p}}    \llabel{d SY FUXL zyQZ hVZMn9 am P 9aE Wzk au0 6d ZghM ym3R jfdePG ln 8 s7x HYC IV9 Hw Ka6v EjH5 J8Ipr7 Nk C xWR 84T Wnq s0 fsiP qGgs Id1fs5 3A T 71q RIc zPX 77 Si23 GirL 9MQZ4F pi g dru NYt h1K 4M Zilv rRk6 B4W5B8 Id 3 Xq9 nhx EN4 P6 ipZl a2UQ Qx8mda g7 r VD3 zdD rhB vk LDJo tKyV 5IrmyJ R5 e txS 1cv EsY xG zj2T rfSR myZo4L m5 D mqN iZd acg GQ 0KRw QKGX g9o8v8 wm B fUu tCO cKc zz kx4U fhuA a8pYzW Vq 9 Sp6 CmA cZL Mx ceBX Dwug sjWuii Gl v JDb 08h BOV C1 pni6 4TTq Opzezq ZB J y5o KS8 BhH sd nKkH gnZl UCm7j0 Iv Y jQE 7JN 9fd ED ddys 3y1x 52pbiG 8ThswELzXU3X7Ebd1KdZ7v1rN3GiirRXGKWK099ovBM0FDJCvkopYNQ2aN94Z7k0UnUKamE3OjU8DFYFFokbSI2J9V9gVlM8ALWThDPnPu3EL7HPD2VDaZTggzcCCmbvc70qqPcC9mt60ogcrTiA3HEjwTK8ymKeuJMc4q6dVz200XnYUtLR9GYjPXvFOVr6W1zUK1WbPToaWJJuKnxBLnd0ftDEbMmj4loHYyhZyMjM91zQS4p7z8eKa9h0JrbacekcirexG0z4n3xz0QOWSvFj3jLhWXUIU21iIAwJtI3RbWa90I7rzAIqI3UElUJG7tLtUXzw4KQNETvXzqWaujEMenYlNIzLGxgB3AuJEQ21}     \end{equation} and    \begin{equation}       \lbrack u \rbrack_{C^{d,\alpha}_{L^p}(0) }      = \inf_{P\in\PP_{d}} \sup_{0<r<1} \frac{\zxvczxbcvdfghasdfrtsdafasdfasdfdsfgsdgh u - P \zxvczxbcvdfghasdfrtsdafasdfasdfdsfgsdgh_{L^p(B_r)}}{r^{d+\alpha +n/p}}      .    \llabel{Yt h1K 4M Zilv rRk6 B4W5B8 Id 3 Xq9 nhx EN4 P6 ipZl a2UQ Qx8mda g7 r VD3 zdD rhB vk LDJo tKyV 5IrmyJ R5 e txS 1cv EsY xG zj2T rfSR myZo4L m5 D mqN iZd acg GQ 0KRw QKGX g9o8v8 wm B fUu tCO cKc zz kx4U fhuA a8pYzW Vq 9 Sp6 CmA cZL Mx ceBX Dwug sjWuii Gl v JDb 08h BOV C1 pni6 4TTq Opzezq ZB J y5o KS8 BhH sd nKkH gnZl UCm7j0 Iv Y jQE 7JN 9fd ED ddys 3y1x 52pbiG Lc a 71j G3e uli Ce uzv2 R40Q 50JZUB uK d U3m May 0uo S7 ulWD h7qG 2FKw2T JX z BES 2Jk Q4U Dy 4aJ2 IXs4 RNH41s py T GNh hk0 w5Z C8 B3nU Bp9p 8eLKh8 UO 4 fMq Y6w lcA GM xCHt vlOx Mq8ThswELzXU3X7Ebd1KdZ7v1rN3GiirRXGKWK099ovBM0FDJCvkopYNQ2aN94Z7k0UnUKamE3OjU8DFYFFokbSI2J9V9gVlM8ALWThDPnPu3EL7HPD2VDaZTggzcCCmbvc70qqPcC9mt60ogcrTiA3HEjwTK8ymKeuJMc4q6dVz200XnYUtLR9GYjPXvFOVr6W1zUK1WbPToaWJJuKnxBLnd0ftDEbMmj4loHYyhZyMjM91zQS4p7z8eKa9h0JrbacekcirexG0z4n3xz0QOWSvFj3jLhWXUIU21iIAwJtI3RbWa90I7rzAIqI3UElUJG7tLtUXzw4KQNETvXzqWaujEMenYlNIzLGxgB3AuJEQ22}   \end{equation} \end{Definition} \par The following is the main result of this paper; it provides a Schauder estimate for the elliptic equation~\eqref{8ThswELzXU3X7Ebd1KdZ7v1rN3GiirRXGKWK099ovBM0FDJCvkopYNQ2aN94Z7k0UnUKamE3OjU8DFYFFokbSI2J9V9gVlM8ALWThDPnPu3EL7HPD2VDaZTggzcCCmbvc70qqPcC9mt60ogcrTiA3HEjwTK8ymKeuJMc4q6dVz200XnYUtLR9GYjPXvFOVr6W1zUK1WbPToaWJJuKnxBLnd0ftDEbMmj4loHYyhZyMjM91zQS4p7z8eKa9h0JrbacekcirexG0z4n3xz0QOWSvFj3jLhWXUIU21iIAwJtI3RbWa90I7rzAIqI3UElUJG7tLtUXzw4KQNETvXzqWaujEMenYlNIzLGxgB3AuJEQ07}. See Section~\ref{sec04} for the analogous results for parabolic type equations. \par \cole \begin{Theorem} \label{T02} Let $u \in W^{m,p}$, where   \begin{equation}    1 < p < \frac{n}{m}    ,    \label{8ThswELzXU3X7Ebd1KdZ7v1rN3GiirRXGKWK099ovBM0FDJCvkopYNQ2aN94Z7k0UnUKamE3OjU8DFYFFokbSI2J9V9gVlM8ALWThDPnPu3EL7HPD2VDaZTggzcCCmbvc70qqPcC9mt60ogcrTiA3HEjwTK8ymKeuJMc4q6dVz200XnYUtLR9GYjPXvFOVr6W1zUK1WbPToaWJJuKnxBLnd0ftDEbMmj4loHYyhZyMjM91zQS4p7z8eKa9h0JrbacekcirexG0z4n3xz0QOWSvFj3jLhWXUIU21iIAwJtI3RbWa90I7rzAIqI3UElUJG7tLtUXzw4KQNETvXzqWaujEMenYlNIzLGxgB3AuJEQ13}   \end{equation} be a solution of $Lu=f$ in $B_1$ for $f \in L^p(B_1)$ and $L$ an elliptic operator \eqref{8ThswELzXU3X7Ebd1KdZ7v1rN3GiirRXGKWK099ovBM0FDJCvkopYNQ2aN94Z7k0UnUKamE3OjU8DFYFFokbSI2J9V9gVlM8ALWThDPnPu3EL7HPD2VDaZTggzcCCmbvc70qqPcC9mt60ogcrTiA3HEjwTK8ymKeuJMc4q6dVz200XnYUtLR9GYjPXvFOVr6W1zUK1WbPToaWJJuKnxBLnd0ftDEbMmj4loHYyhZyMjM91zQS4p7z8eKa9h0JrbacekcirexG0z4n3xz0QOWSvFj3jLhWXUIU21iIAwJtI3RbWa90I7rzAIqI3UElUJG7tLtUXzw4KQNETvXzqWaujEMenYlNIzLGxgB3AuJEQ08}  satisfying \eqref{8ThswELzXU3X7Ebd1KdZ7v1rN3GiirRXGKWK099ovBM0FDJCvkopYNQ2aN94Z7k0UnUKamE3OjU8DFYFFokbSI2J9V9gVlM8ALWThDPnPu3EL7HPD2VDaZTggzcCCmbvc70qqPcC9mt60ogcrTiA3HEjwTK8ymKeuJMc4q6dVz200XnYUtLR9GYjPXvFOVr6W1zUK1WbPToaWJJuKnxBLnd0ftDEbMmj4loHYyhZyMjM91zQS4p7z8eKa9h0JrbacekcirexG0z4n3xz0QOWSvFj3jLhWXUIU21iIAwJtI3RbWa90I7rzAIqI3UElUJG7tLtUXzw4KQNETvXzqWaujEMenYlNIzLGxgB3AuJEQ09}--\eqref{8ThswELzXU3X7Ebd1KdZ7v1rN3GiirRXGKWK099ovBM0FDJCvkopYNQ2aN94Z7k0UnUKamE3OjU8DFYFFokbSI2J9V9gVlM8ALWThDPnPu3EL7HPD2VDaZTggzcCCmbvc70qqPcC9mt60ogcrTiA3HEjwTK8ymKeuJMc4q6dVz200XnYUtLR9GYjPXvFOVr6W1zUK1WbPToaWJJuKnxBLnd0ftDEbMmj4loHYyhZyMjM91zQS4p7z8eKa9h0JrbacekcirexG0z4n3xz0QOWSvFj3jLhWXUIU21iIAwJtI3RbWa90I7rzAIqI3UElUJG7tLtUXzw4KQNETvXzqWaujEMenYlNIzLGxgB3AuJEQ11} in $B_1$. Suppose that      \begin{equation}       c       = \sup_{r \leq 1}\frac{\zxvczxbcvdfghasdfrtsdafasdfasdfdsfgsdgh u \zxvczxbcvdfghasdfrtsdafasdfasdfdsfgsdgh_{L^p(B_r)}}{r^{d+n/p}}       < \infty     \label{8ThswELzXU3X7Ebd1KdZ7v1rN3GiirRXGKWK099ovBM0FDJCvkopYNQ2aN94Z7k0UnUKamE3OjU8DFYFFokbSI2J9V9gVlM8ALWThDPnPu3EL7HPD2VDaZTggzcCCmbvc70qqPcC9mt60ogcrTiA3HEjwTK8ymKeuJMc4q6dVz200XnYUtLR9GYjPXvFOVr6W1zUK1WbPToaWJJuKnxBLnd0ftDEbMmj4loHYyhZyMjM91zQS4p7z8eKa9h0JrbacekcirexG0z4n3xz0QOWSvFj3jLhWXUIU21iIAwJtI3RbWa90I7rzAIqI3UElUJG7tLtUXzw4KQNETvXzqWaujEMenYlNIzLGxgB3AuJEQ23}     \end{equation} and      \begin{equation}      c_f      = \sup_{r \leq 1}\frac{\zxvczxbcvdfghasdfrtsdafasdfasdfdsfgsdgh f \zxvczxbcvdfghasdfrtsdafasdfasdfdsfgsdgh_{L^p(B_r)}}{r^{d-m+n/p}}       < \infty       ,     \llabel{fUu tCO cKc zz kx4U fhuA a8pYzW Vq 9 Sp6 CmA cZL Mx ceBX Dwug sjWuii Gl v JDb 08h BOV C1 pni6 4TTq Opzezq ZB J y5o KS8 BhH sd nKkH gnZl UCm7j0 Iv Y jQE 7JN 9fd ED ddys 3y1x 52pbiG Lc a 71j G3e uli Ce uzv2 R40Q 50JZUB uK d U3m May 0uo S7 ulWD h7qG 2FKw2T JX z BES 2Jk Q4U Dy 4aJ2 IXs4 RNH41s py T GNh hk0 w5Z C8 B3nU Bp9p 8eLKh8 UO 4 fMq Y6w lcA GM xCHt vlOx MqAJoQ QU 1 e8a 2aX 9Y6 2r lIS6 dejK Y3KCUm 25 7 oCl VeE e8p 1z UJSv bmLd Fy7ObQ FN l J6F RdF kEm qM N0Fd NZJ0 8DYuq2 pL X JNz 4rO ZkZ X2 IjTD 1fVt z4BmFI Pi 0 GKD R2W PhO zH zTLP lb8ThswELzXU3X7Ebd1KdZ7v1rN3GiirRXGKWK099ovBM0FDJCvkopYNQ2aN94Z7k0UnUKamE3OjU8DFYFFokbSI2J9V9gVlM8ALWThDPnPu3EL7HPD2VDaZTggzcCCmbvc70qqPcC9mt60ogcrTiA3HEjwTK8ymKeuJMc4q6dVz200XnYUtLR9GYjPXvFOVr6W1zUK1WbPToaWJJuKnxBLnd0ftDEbMmj4loHYyhZyMjM91zQS4p7z8eKa9h0JrbacekcirexG0z4n3xz0QOWSvFj3jLhWXUIU21iIAwJtI3RbWa90I7rzAIqI3UElUJG7tLtUXzw4KQNETvXzqWaujEMenYlNIzLGxgB3AuJEQ24}     \end{equation} for some integer~$d\geq m$. Assume additionally that $f\in C^{d-m+k,\alpha}_{L^p}(0)$ and $a_{\nu} \in C^{k-m+|\nu|,\alpha}_{L^p}(0)$ for some $\alpha \in (0,1)$ and $k \in \mathbb N$, for any $|\nu|\geq \max \{ m-k, 0 \}$. Then for all $q\in [1,np/(n-mp))$, we have $u \in C^{d+k, \alpha}_{L^q}(0)$, and     \begin{equation}       \lbrack u \rbrack_{C^{d+k,\alpha}_{L^q}(0)}       \les \zxvczxbcvdfghasdfrtsdafasdfasdfdsfgsdgh u \zxvczxbcvdfghasdfrtsdafasdfasdfdsfgsdgh_{L^p(B_1)}             + \lbrack f \rbrack_{C^{d-m}_{L^p}(0)}             + \lbrack f \rbrack_{C^{d-m+k,\alpha}_{L^p}(0)}             + 1             ,     \label{8ThswELzXU3X7Ebd1KdZ7v1rN3GiirRXGKWK099ovBM0FDJCvkopYNQ2aN94Z7k0UnUKamE3OjU8DFYFFokbSI2J9V9gVlM8ALWThDPnPu3EL7HPD2VDaZTggzcCCmbvc70qqPcC9mt60ogcrTiA3HEjwTK8ymKeuJMc4q6dVz200XnYUtLR9GYjPXvFOVr6W1zUK1WbPToaWJJuKnxBLnd0ftDEbMmj4loHYyhZyMjM91zQS4p7z8eKa9h0JrbacekcirexG0z4n3xz0QOWSvFj3jLhWXUIU21iIAwJtI3RbWa90I7rzAIqI3UElUJG7tLtUXzw4KQNETvXzqWaujEMenYlNIzLGxgB3AuJEQ25}     \end{equation}  where the implicit constant in \eqref{8ThswELzXU3X7Ebd1KdZ7v1rN3GiirRXGKWK099ovBM0FDJCvkopYNQ2aN94Z7k0UnUKamE3OjU8DFYFFokbSI2J9V9gVlM8ALWThDPnPu3EL7HPD2VDaZTggzcCCmbvc70qqPcC9mt60ogcrTiA3HEjwTK8ymKeuJMc4q6dVz200XnYUtLR9GYjPXvFOVr6W1zUK1WbPToaWJJuKnxBLnd0ftDEbMmj4loHYyhZyMjM91zQS4p7z8eKa9h0JrbacekcirexG0z4n3xz0QOWSvFj3jLhWXUIU21iIAwJtI3RbWa90I7rzAIqI3UElUJG7tLtUXzw4KQNETvXzqWaujEMenYlNIzLGxgB3AuJEQ25} depends only on $p$, $d$, $k$, $c$, $c_{f}$, and $\lbrack a_{\nu} \rbrack_{C^{l-m+|\nu|,\alpha}_{L^p}(0)}$ for $|\nu| \geq m-k$. \end{Theorem} \colb \par The theorem extends the main result in \cite{H1} to the case~$p<m/n$. It also asserts the continuity for the range $q\in [1,np/(n-mp))$. \par Observe that if a function $u$ satisfies $|u|\les |x|^{d}$ on $B_1$, where $d\geq 0$ and $n\geq2$, then $\zxvczxbcvdfghasdfrtsdafasdfasdfdsfgsdgh u\zxvczxbcvdfghasdfrtsdafasdfasdfdsfgsdgh_{L^{p}(B_r)}\les r^{d+\fractext{n}{p}}$. Thus the assumption \eqref{8ThswELzXU3X7Ebd1KdZ7v1rN3GiirRXGKWK099ovBM0FDJCvkopYNQ2aN94Z7k0UnUKamE3OjU8DFYFFokbSI2J9V9gVlM8ALWThDPnPu3EL7HPD2VDaZTggzcCCmbvc70qqPcC9mt60ogcrTiA3HEjwTK8ymKeuJMc4q6dVz200XnYUtLR9GYjPXvFOVr6W1zUK1WbPToaWJJuKnxBLnd0ftDEbMmj4loHYyhZyMjM91zQS4p7z8eKa9h0JrbacekcirexG0z4n3xz0QOWSvFj3jLhWXUIU21iIAwJtI3RbWa90I7rzAIqI3UElUJG7tLtUXzw4KQNETvXzqWaujEMenYlNIzLGxgB3AuJEQ23} may be interpreted as $u$ vanishes at~$0$ with of order at least $d$ in the $L^{p}$ sense. \par In the following auxiliary statement, we estimate the high $L^{p}$-H\"older type norms around the points where $u$ vanishes of order $d$ in the $L^{p}$ sense and $f$ vanishes of order $d-m$ in the $L^{p}$ sense. The result we obtain applies to the $L^q$-norm where $q$ belongs to $[1,np/(n-mp))$, which includes $p$ itself. \par \cole \begin{Theorem} \label{T01} Let $u \in W^{m,p}(B_1)$, with $p$ as in \eqref{8ThswELzXU3X7Ebd1KdZ7v1rN3GiirRXGKWK099ovBM0FDJCvkopYNQ2aN94Z7k0UnUKamE3OjU8DFYFFokbSI2J9V9gVlM8ALWThDPnPu3EL7HPD2VDaZTggzcCCmbvc70qqPcC9mt60ogcrTiA3HEjwTK8ymKeuJMc4q6dVz200XnYUtLR9GYjPXvFOVr6W1zUK1WbPToaWJJuKnxBLnd0ftDEbMmj4loHYyhZyMjM91zQS4p7z8eKa9h0JrbacekcirexG0z4n3xz0QOWSvFj3jLhWXUIU21iIAwJtI3RbWa90I7rzAIqI3UElUJG7tLtUXzw4KQNETvXzqWaujEMenYlNIzLGxgB3AuJEQ13}, be a solution of    $     Lu = f $, where $L$ is the elliptic operator \eqref{8ThswELzXU3X7Ebd1KdZ7v1rN3GiirRXGKWK099ovBM0FDJCvkopYNQ2aN94Z7k0UnUKamE3OjU8DFYFFokbSI2J9V9gVlM8ALWThDPnPu3EL7HPD2VDaZTggzcCCmbvc70qqPcC9mt60ogcrTiA3HEjwTK8ymKeuJMc4q6dVz200XnYUtLR9GYjPXvFOVr6W1zUK1WbPToaWJJuKnxBLnd0ftDEbMmj4loHYyhZyMjM91zQS4p7z8eKa9h0JrbacekcirexG0z4n3xz0QOWSvFj3jLhWXUIU21iIAwJtI3RbWa90I7rzAIqI3UElUJG7tLtUXzw4KQNETvXzqWaujEMenYlNIzLGxgB3AuJEQ08}  satisfying \eqref{8ThswELzXU3X7Ebd1KdZ7v1rN3GiirRXGKWK099ovBM0FDJCvkopYNQ2aN94Z7k0UnUKamE3OjU8DFYFFokbSI2J9V9gVlM8ALWThDPnPu3EL7HPD2VDaZTggzcCCmbvc70qqPcC9mt60ogcrTiA3HEjwTK8ymKeuJMc4q6dVz200XnYUtLR9GYjPXvFOVr6W1zUK1WbPToaWJJuKnxBLnd0ftDEbMmj4loHYyhZyMjM91zQS4p7z8eKa9h0JrbacekcirexG0z4n3xz0QOWSvFj3jLhWXUIU21iIAwJtI3RbWa90I7rzAIqI3UElUJG7tLtUXzw4KQNETvXzqWaujEMenYlNIzLGxgB3AuJEQ09}--\eqref{8ThswELzXU3X7Ebd1KdZ7v1rN3GiirRXGKWK099ovBM0FDJCvkopYNQ2aN94Z7k0UnUKamE3OjU8DFYFFokbSI2J9V9gVlM8ALWThDPnPu3EL7HPD2VDaZTggzcCCmbvc70qqPcC9mt60ogcrTiA3HEjwTK8ymKeuJMc4q6dVz200XnYUtLR9GYjPXvFOVr6W1zUK1WbPToaWJJuKnxBLnd0ftDEbMmj4loHYyhZyMjM91zQS4p7z8eKa9h0JrbacekcirexG0z4n3xz0QOWSvFj3jLhWXUIU21iIAwJtI3RbWa90I7rzAIqI3UElUJG7tLtUXzw4KQNETvXzqWaujEMenYlNIzLGxgB3AuJEQ11} and $f \in L^p(B_1)$. Suppose that, with an integer $d \geq m$, we have    \begin{equation}       \zxvczxbcvdfghasdfrtsdafasdfasdfdsfgsdgh f - Q \zxvczxbcvdfghasdfrtsdafasdfasdfdsfgsdgh_{L^p(B_r)}        \leq \MM r^{d-m+\alpha+n/p}     \comma r\in(0,1]    ,     \label{8ThswELzXU3X7Ebd1KdZ7v1rN3GiirRXGKWK099ovBM0FDJCvkopYNQ2aN94Z7k0UnUKamE3OjU8DFYFFokbSI2J9V9gVlM8ALWThDPnPu3EL7HPD2VDaZTggzcCCmbvc70qqPcC9mt60ogcrTiA3HEjwTK8ymKeuJMc4q6dVz200XnYUtLR9GYjPXvFOVr6W1zUK1WbPToaWJJuKnxBLnd0ftDEbMmj4loHYyhZyMjM91zQS4p7z8eKa9h0JrbacekcirexG0z4n3xz0QOWSvFj3jLhWXUIU21iIAwJtI3RbWa90I7rzAIqI3UElUJG7tLtUXzw4KQNETvXzqWaujEMenYlNIzLGxgB3AuJEQ12}   \end{equation} for some $Q\in\dPP_{d-m}$ and~$\MM >0$. Then if     \begin{equation}       c_0        = \sup_{r \leq 1}\frac{\zxvczxbcvdfghasdfrtsdafasdfasdfdsfgsdgh u \zxvczxbcvdfghasdfrtsdafasdfasdfdsfgsdgh_{L^p(B_r)}}{r^{d-1+\eta+n/p}}       < \infty     ,     \label{8ThswELzXU3X7Ebd1KdZ7v1rN3GiirRXGKWK099ovBM0FDJCvkopYNQ2aN94Z7k0UnUKamE3OjU8DFYFFokbSI2J9V9gVlM8ALWThDPnPu3EL7HPD2VDaZTggzcCCmbvc70qqPcC9mt60ogcrTiA3HEjwTK8ymKeuJMc4q6dVz200XnYUtLR9GYjPXvFOVr6W1zUK1WbPToaWJJuKnxBLnd0ftDEbMmj4loHYyhZyMjM91zQS4p7z8eKa9h0JrbacekcirexG0z4n3xz0QOWSvFj3jLhWXUIU21iIAwJtI3RbWa90I7rzAIqI3UElUJG7tLtUXzw4KQNETvXzqWaujEMenYlNIzLGxgB3AuJEQ14}     \end{equation} for some $\eta \in (0,1] $,  there exists $P\in\dPP_{d}$ solving     \begin{equation}       \sum_{|\nu|=m} a_{\nu}(0) \partial^{\nu}P = Q      \llabel{Lc a 71j G3e uli Ce uzv2 R40Q 50JZUB uK d U3m May 0uo S7 ulWD h7qG 2FKw2T JX z BES 2Jk Q4U Dy 4aJ2 IXs4 RNH41s py T GNh hk0 w5Z C8 B3nU Bp9p 8eLKh8 UO 4 fMq Y6w lcA GM xCHt vlOx MqAJoQ QU 1 e8a 2aX 9Y6 2r lIS6 dejK Y3KCUm 25 7 oCl VeE e8p 1z UJSv bmLd Fy7ObQ FN l J6F RdF kEm qM N0Fd NZJ0 8DYuq2 pL X JNz 4rO ZkZ X2 IjTD 1fVt z4BmFI Pi 0 GKD R2W PhO zH zTLP lbAE OT9XW0 gb T Lb3 XRQ qGG 8o 4TPE 6WRc uMqMXh s6 x Ofv 8st jDi u8 rtJt TKSK jlGkGw t8 n FDx jA9 fCm iu FqMW jeox 5Akw3w Sd 8 1vK 8c4 C0O dj CHIs eHUO hyqGx3 Kw O lDq l1Y 4NY 4I vI8ThswELzXU3X7Ebd1KdZ7v1rN3GiirRXGKWK099ovBM0FDJCvkopYNQ2aN94Z7k0UnUKamE3OjU8DFYFFokbSI2J9V9gVlM8ALWThDPnPu3EL7HPD2VDaZTggzcCCmbvc70qqPcC9mt60ogcrTiA3HEjwTK8ymKeuJMc4q6dVz200XnYUtLR9GYjPXvFOVr6W1zUK1WbPToaWJJuKnxBLnd0ftDEbMmj4loHYyhZyMjM91zQS4p7z8eKa9h0JrbacekcirexG0z4n3xz0QOWSvFj3jLhWXUIU21iIAwJtI3RbWa90I7rzAIqI3UElUJG7tLtUXzw4KQNETvXzqWaujEMenYlNIzLGxgB3AuJEQ15}     \end{equation} in $\mathbb R^n$ such that     \begin{equation}       \zxvczxbcvdfghasdfrtsdafasdfasdfdsfgsdgh P \zxvczxbcvdfghasdfrtsdafasdfasdfdsfgsdgh_{L^p(B_r)}       \les (\MM                 + c_0              + \zxvczxbcvdfghasdfrtsdafasdfasdfdsfgsdgh Q \zxvczxbcvdfghasdfrtsdafasdfasdfdsfgsdgh_{L^p(B_1)}              ) r^{d+n/p}    \comma r>0     \label{8ThswELzXU3X7Ebd1KdZ7v1rN3GiirRXGKWK099ovBM0FDJCvkopYNQ2aN94Z7k0UnUKamE3OjU8DFYFFokbSI2J9V9gVlM8ALWThDPnPu3EL7HPD2VDaZTggzcCCmbvc70qqPcC9mt60ogcrTiA3HEjwTK8ymKeuJMc4q6dVz200XnYUtLR9GYjPXvFOVr6W1zUK1WbPToaWJJuKnxBLnd0ftDEbMmj4loHYyhZyMjM91zQS4p7z8eKa9h0JrbacekcirexG0z4n3xz0QOWSvFj3jLhWXUIU21iIAwJtI3RbWa90I7rzAIqI3UElUJG7tLtUXzw4KQNETvXzqWaujEMenYlNIzLGxgB3AuJEQ16}     \end{equation} and      \begin{equation}        \sum_{|\nu|\leq m}          r^{|\nu|} \zxvczxbcvdfghasdfrtsdafasdfasdfdsfgsdgh \partial^{\nu} (u-P) \zxvczxbcvdfghasdfrtsdafasdfasdfdsfgsdgh_{L^q(B_r)}        \les  \big(\MM                 + c_0               + \zxvczxbcvdfghasdfrtsdafasdfasdfdsfgsdgh Q \zxvczxbcvdfghasdfrtsdafasdfasdfdsfgsdgh_{L^p(B_1)}              \big) r^{d+\alpha+n/q}              \comma              r\in(0,1/2]              ,      \label{8ThswELzXU3X7Ebd1KdZ7v1rN3GiirRXGKWK099ovBM0FDJCvkopYNQ2aN94Z7k0UnUKamE3OjU8DFYFFokbSI2J9V9gVlM8ALWThDPnPu3EL7HPD2VDaZTggzcCCmbvc70qqPcC9mt60ogcrTiA3HEjwTK8ymKeuJMc4q6dVz200XnYUtLR9GYjPXvFOVr6W1zUK1WbPToaWJJuKnxBLnd0ftDEbMmj4loHYyhZyMjM91zQS4p7z8eKa9h0JrbacekcirexG0z4n3xz0QOWSvFj3jLhWXUIU21iIAwJtI3RbWa90I7rzAIqI3UElUJG7tLtUXzw4KQNETvXzqWaujEMenYlNIzLGxgB3AuJEQ17}      \end{equation} for all $q\in [1,np/(n-mp))$. Moreover,           \begin{equation}      \zxvczxbcvdfghasdfrtsdafasdfasdfdsfgsdgh u \zxvczxbcvdfghasdfrtsdafasdfasdfdsfgsdgh_{L^p(B_r)}      \les  \big(\MM                        +c_0                + \zxvczxbcvdfghasdfrtsdafasdfasdfdsfgsdgh Q \zxvczxbcvdfghasdfrtsdafasdfasdfdsfgsdgh_{L^p(B_1)}                \big) r^{d + n/p}
             \comma              r\in(0,1]              .              \colb      \label{8ThswELzXU3X7Ebd1KdZ7v1rN3GiirRXGKWK099ovBM0FDJCvkopYNQ2aN94Z7k0UnUKamE3OjU8DFYFFokbSI2J9V9gVlM8ALWThDPnPu3EL7HPD2VDaZTggzcCCmbvc70qqPcC9mt60ogcrTiA3HEjwTK8ymKeuJMc4q6dVz200XnYUtLR9GYjPXvFOVr6W1zUK1WbPToaWJJuKnxBLnd0ftDEbMmj4loHYyhZyMjM91zQS4p7z8eKa9h0JrbacekcirexG0z4n3xz0QOWSvFj3jLhWXUIU21iIAwJtI3RbWa90I7rzAIqI3UElUJG7tLtUXzw4KQNETvXzqWaujEMenYlNIzLGxgB3AuJEQ18}      \end{equation} The implicit constants depend only on $p$, $q$, and~$d$. \end{Theorem}  \colb \par Theorem~\ref{T01} provides control on the leading polynomial and the error terms, given control on the source term $f$ and the assumption that $u$ vanishes of order slightly less than $d$ at 0. \par It is illustrative to consider the case when $f\equiv Q\equiv 0$, in which case we are looking at a solution of $L u=0$. Then the assumption \eqref{8ThswELzXU3X7Ebd1KdZ7v1rN3GiirRXGKWK099ovBM0FDJCvkopYNQ2aN94Z7k0UnUKamE3OjU8DFYFFokbSI2J9V9gVlM8ALWThDPnPu3EL7HPD2VDaZTggzcCCmbvc70qqPcC9mt60ogcrTiA3HEjwTK8ymKeuJMc4q6dVz200XnYUtLR9GYjPXvFOVr6W1zUK1WbPToaWJJuKnxBLnd0ftDEbMmj4loHYyhZyMjM91zQS4p7z8eKa9h0JrbacekcirexG0z4n3xz0QOWSvFj3jLhWXUIU21iIAwJtI3RbWa90I7rzAIqI3UElUJG7tLtUXzw4KQNETvXzqWaujEMenYlNIzLGxgB3AuJEQ14} requires that $u$ vanishes of a fractional order strictly greater than $d-1$ in the $L^{p}$ sense and \eqref{8ThswELzXU3X7Ebd1KdZ7v1rN3GiirRXGKWK099ovBM0FDJCvkopYNQ2aN94Z7k0UnUKamE3OjU8DFYFFokbSI2J9V9gVlM8ALWThDPnPu3EL7HPD2VDaZTggzcCCmbvc70qqPcC9mt60ogcrTiA3HEjwTK8ymKeuJMc4q6dVz200XnYUtLR9GYjPXvFOVr6W1zUK1WbPToaWJJuKnxBLnd0ftDEbMmj4loHYyhZyMjM91zQS4p7z8eKa9h0JrbacekcirexG0z4n3xz0QOWSvFj3jLhWXUIU21iIAwJtI3RbWa90I7rzAIqI3UElUJG7tLtUXzw4KQNETvXzqWaujEMenYlNIzLGxgB3AuJEQ18} asserts that $u$ vanishes of order $d$ (in the $L^{p}$ sense). \par \startnewsection{Proofs of the main results}{sec03} The proofs of Theorems~\ref{T02} and~\ref{T01} rely on the following interior existence theorem and a bootstrapping argument in Lemma~\ref{L04} below. \par \subsection{Interior $W^{m,p}$ existence lemma} The following statement asserts the existence of a solution in a small neighborhood which vanishes of high degree provided the same holds for the source term. \par \cole \begin{Lemma} \label{L01} Assume that $L = \sum_{|\nu|\leq m} a_{\nu}(x) \partial^{\nu}$, defined in $B_1$, satisfies the conditions~\eqref{8ThswELzXU3X7Ebd1KdZ7v1rN3GiirRXGKWK099ovBM0FDJCvkopYNQ2aN94Z7k0UnUKamE3OjU8DFYFFokbSI2J9V9gVlM8ALWThDPnPu3EL7HPD2VDaZTggzcCCmbvc70qqPcC9mt60ogcrTiA3HEjwTK8ymKeuJMc4q6dVz200XnYUtLR9GYjPXvFOVr6W1zUK1WbPToaWJJuKnxBLnd0ftDEbMmj4loHYyhZyMjM91zQS4p7z8eKa9h0JrbacekcirexG0z4n3xz0QOWSvFj3jLhWXUIU21iIAwJtI3RbWa90I7rzAIqI3UElUJG7tLtUXzw4KQNETvXzqWaujEMenYlNIzLGxgB3AuJEQ09}--\eqref{8ThswELzXU3X7Ebd1KdZ7v1rN3GiirRXGKWK099ovBM0FDJCvkopYNQ2aN94Z7k0UnUKamE3OjU8DFYFFokbSI2J9V9gVlM8ALWThDPnPu3EL7HPD2VDaZTggzcCCmbvc70qqPcC9mt60ogcrTiA3HEjwTK8ymKeuJMc4q6dVz200XnYUtLR9GYjPXvFOVr6W1zUK1WbPToaWJJuKnxBLnd0ftDEbMmj4loHYyhZyMjM91zQS4p7z8eKa9h0JrbacekcirexG0z4n3xz0QOWSvFj3jLhWXUIU21iIAwJtI3RbWa90I7rzAIqI3UElUJG7tLtUXzw4KQNETvXzqWaujEMenYlNIzLGxgB3AuJEQ11}. Suppose that $f \in L^p(B_1)$, where \eqref{8ThswELzXU3X7Ebd1KdZ7v1rN3GiirRXGKWK099ovBM0FDJCvkopYNQ2aN94Z7k0UnUKamE3OjU8DFYFFokbSI2J9V9gVlM8ALWThDPnPu3EL7HPD2VDaZTggzcCCmbvc70qqPcC9mt60ogcrTiA3HEjwTK8ymKeuJMc4q6dVz200XnYUtLR9GYjPXvFOVr6W1zUK1WbPToaWJJuKnxBLnd0ftDEbMmj4loHYyhZyMjM91zQS4p7z8eKa9h0JrbacekcirexG0z4n3xz0QOWSvFj3jLhWXUIU21iIAwJtI3RbWa90I7rzAIqI3UElUJG7tLtUXzw4KQNETvXzqWaujEMenYlNIzLGxgB3AuJEQ13} holds, is such that  $\zxvczxbcvdfghasdfrtsdafasdfasdfdsfgsdgh f \zxvczxbcvdfghasdfrtsdafasdfasdfdsfgsdgh_{L^p(B_r)} \leq \MM r^{d-m+\gamma+n/p}$, for $ r\in(0,1]$, where $\gamma \in (0,1)$ and~$d \geq m$. Then there exist $R > 0$, depending on $\alpha$,  and $u \in W^{m,p}(B_R)$ solving  $L u = f$  in $B_R$ such that     \begin{equation}       \sum_{|\nu|\leq m} r^{|\nu|} \zxvczxbcvdfghasdfrtsdafasdfasdfdsfgsdgh \partial^\nu u \zxvczxbcvdfghasdfrtsdafasdfasdfdsfgsdgh_{L^p(B_r)}       \les \MM r^{d+\gamma +n/p}       \comma r \in (0,R]       .     \label{8ThswELzXU3X7Ebd1KdZ7v1rN3GiirRXGKWK099ovBM0FDJCvkopYNQ2aN94Z7k0UnUKamE3OjU8DFYFFokbSI2J9V9gVlM8ALWThDPnPu3EL7HPD2VDaZTggzcCCmbvc70qqPcC9mt60ogcrTiA3HEjwTK8ymKeuJMc4q6dVz200XnYUtLR9GYjPXvFOVr6W1zUK1WbPToaWJJuKnxBLnd0ftDEbMmj4loHYyhZyMjM91zQS4p7z8eKa9h0JrbacekcirexG0z4n3xz0QOWSvFj3jLhWXUIU21iIAwJtI3RbWa90I7rzAIqI3UElUJG7tLtUXzw4KQNETvXzqWaujEMenYlNIzLGxgB3AuJEQ26}     \end{equation} \end{Lemma} \colb \par In particular, if the source term vanishes of order at least $d-m+\gamma$, then there exists a local solution $u$ which vanishes of order~$d+\gamma$. \par Before the proof, we need  to estimate $\zxvczxbcvdfghasdfrtsdafasdfasdfdsfgsdgf (|f(y)|/|y|^{b})\,dy$ over $B_r$ and~$B_1\backslash B_r$. Denote by $p'$ the H\"older conjugate of~$p$. \par \cole \begin{Lemma} \label{L02} Assume that $f\in L^p(B_1)$, where $p\in [1,\infty]$, satisfies   \begin{equation}    \zxvczxbcvdfghasdfrtsdafasdfasdfdsfgsdgh f\zxvczxbcvdfghasdfrtsdafasdfasdfdsfgsdgh_{L^{p}(B_r)}    \leq M r^{a}    \comma r\in(0,1]    ,    \llabel{AJoQ QU 1 e8a 2aX 9Y6 2r lIS6 dejK Y3KCUm 25 7 oCl VeE e8p 1z UJSv bmLd Fy7ObQ FN l J6F RdF kEm qM N0Fd NZJ0 8DYuq2 pL X JNz 4rO ZkZ X2 IjTD 1fVt z4BmFI Pi 0 GKD R2W PhO zH zTLP lbAE OT9XW0 gb T Lb3 XRQ qGG 8o 4TPE 6WRc uMqMXh s6 x Ofv 8st jDi u8 rtJt TKSK jlGkGw t8 n FDx jA9 fCm iu FqMW jeox 5Akw3w Sd 8 1vK 8c4 C0O dj CHIs eHUO hyqGx3 Kw O lDq l1Y 4NY 4I vI7X DE4c FeXdFV bC F HaJ sb4 OC0 hu Mj65 J4fa vgGo7q Y5 X tLy izY DvH TR zd9x SRVg 0Pl6Z8 9X z fLh GlH IYB x9 OELo 5loZ x4wag4 cn F aCE KfA 0uz fw HMUV M9Qy eARFe3 Py 6 kQG GFx rPf 8ThswELzXU3X7Ebd1KdZ7v1rN3GiirRXGKWK099ovBM0FDJCvkopYNQ2aN94Z7k0UnUKamE3OjU8DFYFFokbSI2J9V9gVlM8ALWThDPnPu3EL7HPD2VDaZTggzcCCmbvc70qqPcC9mt60ogcrTiA3HEjwTK8ymKeuJMc4q6dVz200XnYUtLR9GYjPXvFOVr6W1zUK1WbPToaWJJuKnxBLnd0ftDEbMmj4loHYyhZyMjM91zQS4p7z8eKa9h0JrbacekcirexG0z4n3xz0QOWSvFj3jLhWXUIU21iIAwJtI3RbWa90I7rzAIqI3UElUJG7tLtUXzw4KQNETvXzqWaujEMenYlNIzLGxgB3AuJEQ27}   \end{equation} where $a\in \mathbb{R}$.\\ (i) If  $b< a+\fractext{n}{p'}$, we have   \begin{equation}    \zxvczxbcvdfghasdfrtsdafasdfasdfdsfgsdgf_{B_r}    \frac{|f(y)|\,dy}{|y|^{b}}    \les    M r^{a-b+n/p'}    \comma r\in (0,1]    ,    \llabel{AE OT9XW0 gb T Lb3 XRQ qGG 8o 4TPE 6WRc uMqMXh s6 x Ofv 8st jDi u8 rtJt TKSK jlGkGw t8 n FDx jA9 fCm iu FqMW jeox 5Akw3w Sd 8 1vK 8c4 C0O dj CHIs eHUO hyqGx3 Kw O lDq l1Y 4NY 4I vI7X DE4c FeXdFV bC F HaJ sb4 OC0 hu Mj65 J4fa vgGo7q Y5 X tLy izY DvH TR zd9x SRVg 0Pl6Z8 9X z fLh GlH IYB x9 OELo 5loZ x4wag4 cn F aCE KfA 0uz fw HMUV M9Qy eARFe3 Py 6 kQG GFx rPf 6T ZBQR la1a 6Aeker Xg k blz nSm mhY jc z3io WYjz h33sxR JM k Dos EAA hUO Oz aQfK Z0cn 5kqYPn W7 1 vCT 69a EC9 LD EQ5S BK4J fVFLAo Qp N dzZ HAl JaL Mn vRqH 7pBB qOr7fv oa e BSA 8TE8ThswELzXU3X7Ebd1KdZ7v1rN3GiirRXGKWK099ovBM0FDJCvkopYNQ2aN94Z7k0UnUKamE3OjU8DFYFFokbSI2J9V9gVlM8ALWThDPnPu3EL7HPD2VDaZTggzcCCmbvc70qqPcC9mt60ogcrTiA3HEjwTK8ymKeuJMc4q6dVz200XnYUtLR9GYjPXvFOVr6W1zUK1WbPToaWJJuKnxBLnd0ftDEbMmj4loHYyhZyMjM91zQS4p7z8eKa9h0JrbacekcirexG0z4n3xz0QOWSvFj3jLhWXUIU21iIAwJtI3RbWa90I7rzAIqI3UElUJG7tLtUXzw4KQNETvXzqWaujEMenYlNIzLGxgB3AuJEQ29}   \end{equation} where the implicit constant depends on $a$, $b$, and~$p$.\\ (ii) If   \begin{equation}       b > a       + \frac{n}{p'}    ,    \label{8ThswELzXU3X7Ebd1KdZ7v1rN3GiirRXGKWK099ovBM0FDJCvkopYNQ2aN94Z7k0UnUKamE3OjU8DFYFFokbSI2J9V9gVlM8ALWThDPnPu3EL7HPD2VDaZTggzcCCmbvc70qqPcC9mt60ogcrTiA3HEjwTK8ymKeuJMc4q6dVz200XnYUtLR9GYjPXvFOVr6W1zUK1WbPToaWJJuKnxBLnd0ftDEbMmj4loHYyhZyMjM91zQS4p7z8eKa9h0JrbacekcirexG0z4n3xz0QOWSvFj3jLhWXUIU21iIAwJtI3RbWa90I7rzAIqI3UElUJG7tLtUXzw4KQNETvXzqWaujEMenYlNIzLGxgB3AuJEQ30}   \end{equation} we have   \begin{equation}    \zxvczxbcvdfghasdfrtsdafasdfasdfdsfgsdgf_{B_1\backslash B_r}    \frac{|f(y)|\,dy}{|y|^{b}}    \les    M r^{a-b+n/p'}    \comma r\in (0,1]    ,    \llabel{7X DE4c FeXdFV bC F HaJ sb4 OC0 hu Mj65 J4fa vgGo7q Y5 X tLy izY DvH TR zd9x SRVg 0Pl6Z8 9X z fLh GlH IYB x9 OELo 5loZ x4wag4 cn F aCE KfA 0uz fw HMUV M9Qy eARFe3 Py 6 kQG GFx rPf 6T ZBQR la1a 6Aeker Xg k blz nSm mhY jc z3io WYjz h33sxR JM k Dos EAA hUO Oz aQfK Z0cn 5kqYPn W7 1 vCT 69a EC9 LD EQ5S BK4J fVFLAo Qp N dzZ HAl JaL Mn vRqH 7pBB qOr7fv oa e BSA 8TE btx y3 jwK3 v244 dlfwRL Dc g X14 vTp Wd8 zy YWjw eQmF yD5y5l DN l ZbA Jac cld kx Yn3V QYIV v6fwmH z1 9 w3y D4Y ezR M9 BduE L7D9 2wTHHc Do g ZxZ WRW Jxi pv fz48 ZVB7 FZtgK0 Y1 w oC8ThswELzXU3X7Ebd1KdZ7v1rN3GiirRXGKWK099ovBM0FDJCvkopYNQ2aN94Z7k0UnUKamE3OjU8DFYFFokbSI2J9V9gVlM8ALWThDPnPu3EL7HPD2VDaZTggzcCCmbvc70qqPcC9mt60ogcrTiA3HEjwTK8ymKeuJMc4q6dVz200XnYUtLR9GYjPXvFOVr6W1zUK1WbPToaWJJuKnxBLnd0ftDEbMmj4loHYyhZyMjM91zQS4p7z8eKa9h0JrbacekcirexG0z4n3xz0QOWSvFj3jLhWXUIU21iIAwJtI3RbWa90I7rzAIqI3UElUJG7tLtUXzw4KQNETvXzqWaujEMenYlNIzLGxgB3AuJEQ31}   \end{equation} where the implicit constant depends on $a$, $b$, and~$p$. \end{Lemma} \colb \par \begin{proof}[Proof of Lemma~\ref{L02}] (i) Dividing $B_r$ into dyadic shells, we have   \begin{align}    \begin{split}    \zxvczxbcvdfghasdfrtsdafasdfasdfdsfgsdgf_{B_r}    \frac{|f(y)|\,dy}{|y|^{b}}    &    =    \sum_{m=0}^{\infty}    \zxvczxbcvdfghasdfrtsdafasdfasdfdsfgsdgf_{B_{2^{-m} r}\backslash B_{2^{-m-1}r}}    \frac{|f(y)|\,dy}{|y|^{b}}    \les    \sum_{m=0}^{\infty}    \zxvczxbcvdfghasdfrtsdafasdfasdfdsfgsdgh f\zxvczxbcvdfghasdfrtsdafasdfasdfdsfgsdgh_{L^p(B_{2^{-m}r})}    \zxvczxbcvdfghasdfrtsdafasdfasdfdsfgsdgh |y|^{-b}\zxvczxbcvdfghasdfrtsdafasdfasdfdsfgsdgh_{L^{p'}(B_{2^{-m} r}\backslash B_{2^{-m-1}r})}    \\&    \les    \sum_{m=0}^{\infty}    M (2^{-m}r)^{a-b+n/p'}    \les    M r^{a-b+n/p'}    ,    \end{split}    \label{8ThswELzXU3X7Ebd1KdZ7v1rN3GiirRXGKWK099ovBM0FDJCvkopYNQ2aN94Z7k0UnUKamE3OjU8DFYFFokbSI2J9V9gVlM8ALWThDPnPu3EL7HPD2VDaZTggzcCCmbvc70qqPcC9mt60ogcrTiA3HEjwTK8ymKeuJMc4q6dVz200XnYUtLR9GYjPXvFOVr6W1zUK1WbPToaWJJuKnxBLnd0ftDEbMmj4loHYyhZyMjM91zQS4p7z8eKa9h0JrbacekcirexG0z4n3xz0QOWSvFj3jLhWXUIU21iIAwJtI3RbWa90I7rzAIqI3UElUJG7tLtUXzw4KQNETvXzqWaujEMenYlNIzLGxgB3AuJEQ32}   \end{align} under the condition $b< a+\fractext{n}{p'}$. To obtain the second inequality in \eqref{8ThswELzXU3X7Ebd1KdZ7v1rN3GiirRXGKWK099ovBM0FDJCvkopYNQ2aN94Z7k0UnUKamE3OjU8DFYFFokbSI2J9V9gVlM8ALWThDPnPu3EL7HPD2VDaZTggzcCCmbvc70qqPcC9mt60ogcrTiA3HEjwTK8ymKeuJMc4q6dVz200XnYUtLR9GYjPXvFOVr6W1zUK1WbPToaWJJuKnxBLnd0ftDEbMmj4loHYyhZyMjM91zQS4p7z8eKa9h0JrbacekcirexG0z4n3xz0QOWSvFj3jLhWXUIU21iIAwJtI3RbWa90I7rzAIqI3UElUJG7tLtUXzw4KQNETvXzqWaujEMenYlNIzLGxgB3AuJEQ32}, we need to separate the cases $  b p' = n$ and $bp'\neq n$; however, it is easy to check that in both cases we obtain the same result. \par (ii) Choose $m_0  \in \mathbb{N}_0$ such that $2^{-m_0}<r\leq 2^{-m_0+1}$. Then we have   \begin{align}    \begin{split}    \zxvczxbcvdfghasdfrtsdafasdfasdfdsfgsdgf_{B_1\backslash B_r}    \frac{|f(y)|\,dy}{|y|^{b}}    &    \leq    \sum_{m=0}^{m_0}    \zxvczxbcvdfghasdfrtsdafasdfasdfdsfgsdgf_{B_{2^{-m}}\backslash B_{2^{-m-1}}}    \frac{|f(y)|\,dy}{|y|^{b}}    \les    \sum_{m=0}^{m_0}    \zxvczxbcvdfghasdfrtsdafasdfasdfdsfgsdgh f\zxvczxbcvdfghasdfrtsdafasdfasdfdsfgsdgh_{L^p(B_{2^{-m}})}    \zxvczxbcvdfghasdfrtsdafasdfasdfdsfgsdgh |y|^{-b}\zxvczxbcvdfghasdfrtsdafasdfasdfdsfgsdgh_{L^{p'}(B_{2^{-m} }\backslash B_{2^{-m-1}})}    \\&    \les    \sum_{m=0}^{m_0}    (2^{-m})^{a-b+n/p'}    =    \sum_{m=0}^{m_0}    (2^{m})^{-a+b-n/p'}    \les    (2^{m_0})^{-a+b-n/p'}    \les    r^{a-b+n/p'}    ,    \end{split}    \llabel{6T ZBQR la1a 6Aeker Xg k blz nSm mhY jc z3io WYjz h33sxR JM k Dos EAA hUO Oz aQfK Z0cn 5kqYPn W7 1 vCT 69a EC9 LD EQ5S BK4J fVFLAo Qp N dzZ HAl JaL Mn vRqH 7pBB qOr7fv oa e BSA 8TE btx y3 jwK3 v244 dlfwRL Dc g X14 vTp Wd8 zy YWjw eQmF yD5y5l DN l ZbA Jac cld kx Yn3V QYIV v6fwmH z1 9 w3y D4Y ezR M9 BduE L7D9 2wTHHc Do g ZxZ WRW Jxi pv fz48 ZVB7 FZtgK0 Y1 w oCo hLA i70 NO Ta06 u2sY GlmspV l2 x y0X B37 x43 k5 kaoZ deyE sDglRF Xi 9 6b6 w9B dId Ko gSUM NLLb CRzeQL UZ m i9O 2qv VzD hz v1r6 spSl jwNhG6 s6 i SdX hob hbp 2u sEdl 95LP AtrBBi bP8ThswELzXU3X7Ebd1KdZ7v1rN3GiirRXGKWK099ovBM0FDJCvkopYNQ2aN94Z7k0UnUKamE3OjU8DFYFFokbSI2J9V9gVlM8ALWThDPnPu3EL7HPD2VDaZTggzcCCmbvc70qqPcC9mt60ogcrTiA3HEjwTK8ymKeuJMc4q6dVz200XnYUtLR9GYjPXvFOVr6W1zUK1WbPToaWJJuKnxBLnd0ftDEbMmj4loHYyhZyMjM91zQS4p7z8eKa9h0JrbacekcirexG0z4n3xz0QOWSvFj3jLhWXUIU21iIAwJtI3RbWa90I7rzAIqI3UElUJG7tLtUXzw4KQNETvXzqWaujEMenYlNIzLGxgB3AuJEQ33}   \end{align} again checking separately the case~$bp'=n$. \end{proof} \par Throughout this section, we use the notation   \begin{equation}   L(0) = \sum_{|\nu|=m} a_{\nu}(0) \partial^{\nu}
  \label{8ThswELzXU3X7Ebd1KdZ7v1rN3GiirRXGKWK099ovBM0FDJCvkopYNQ2aN94Z7k0UnUKamE3OjU8DFYFFokbSI2J9V9gVlM8ALWThDPnPu3EL7HPD2VDaZTggzcCCmbvc70qqPcC9mt60ogcrTiA3HEjwTK8ymKeuJMc4q6dVz200XnYUtLR9GYjPXvFOVr6W1zUK1WbPToaWJJuKnxBLnd0ftDEbMmj4loHYyhZyMjM91zQS4p7z8eKa9h0JrbacekcirexG0z4n3xz0QOWSvFj3jLhWXUIU21iIAwJtI3RbWa90I7rzAIqI3UElUJG7tLtUXzw4KQNETvXzqWaujEMenYlNIzLGxgB3AuJEQL}   ,   \end{equation} while $\Gamma$ denotes the fundamental solution of $L(0)$; in particular, $L(0) \Gamma = \delta_0$.  Since $m<n$, it satisfies the estimate     \begin{equation}     \label{8ThswELzXU3X7Ebd1KdZ7v1rN3GiirRXGKWK099ovBM0FDJCvkopYNQ2aN94Z7k0UnUKamE3OjU8DFYFFokbSI2J9V9gVlM8ALWThDPnPu3EL7HPD2VDaZTggzcCCmbvc70qqPcC9mt60ogcrTiA3HEjwTK8ymKeuJMc4q6dVz200XnYUtLR9GYjPXvFOVr6W1zUK1WbPToaWJJuKnxBLnd0ftDEbMmj4loHYyhZyMjM91zQS4p7z8eKa9h0JrbacekcirexG0z4n3xz0QOWSvFj3jLhWXUIU21iIAwJtI3RbWa90I7rzAIqI3UElUJG7tLtUXzw4KQNETvXzqWaujEMenYlNIzLGxgB3AuJEQG}       |\partial_x^{\beta} \Gamma(x)|       \les \frac{1}{|x|^{n+|\beta|-m}}       \comma \beta\in \mathbb{N}_0^{d}       ,     \end{equation} where the constant depends on $|\beta|$; see~\cite{B}. \par \begin{proof}[Proof of Lemma~\ref{L01}] We start with the case  when the coefficients are constant and the lower order terms are not present, i.e., $a_{\nu} \equiv a_{\nu}(0)$ for $|\nu| = m$ and $a_{\nu} \equiv 0$ for~$|\nu| < m$.  The convolution     \begin{equation}     w(x) = \zxvczxbcvdfghasdfrtsdafasdfasdfdsfgsdgf_{|y|<1} \Gamma (x-y) f(y) \,dy     \llabel{ btx y3 jwK3 v244 dlfwRL Dc g X14 vTp Wd8 zy YWjw eQmF yD5y5l DN l ZbA Jac cld kx Yn3V QYIV v6fwmH z1 9 w3y D4Y ezR M9 BduE L7D9 2wTHHc Do g ZxZ WRW Jxi pv fz48 ZVB7 FZtgK0 Y1 w oCo hLA i70 NO Ta06 u2sY GlmspV l2 x y0X B37 x43 k5 kaoZ deyE sDglRF Xi 9 6b6 w9B dId Ko gSUM NLLb CRzeQL UZ m i9O 2qv VzD hz v1r6 spSl jwNhG6 s6 i SdX hob hbp 2u sEdl 95LP AtrBBi bP C wSh pFC CUa yz xYS5 78ro f3UwDP sC I pES HB1 qFP SW 5tt0 I7oz jXun6c z4 c QLB J4M NmI 6F 08S2 Il8C 0JQYiU lI 1 YkK oiu bVt fG uOeg Sllv b4HGn3 bS Z LlX efa eN6 v1 B6m3 Ek3J SXUI8ThswELzXU3X7Ebd1KdZ7v1rN3GiirRXGKWK099ovBM0FDJCvkopYNQ2aN94Z7k0UnUKamE3OjU8DFYFFokbSI2J9V9gVlM8ALWThDPnPu3EL7HPD2VDaZTggzcCCmbvc70qqPcC9mt60ogcrTiA3HEjwTK8ymKeuJMc4q6dVz200XnYUtLR9GYjPXvFOVr6W1zUK1WbPToaWJJuKnxBLnd0ftDEbMmj4loHYyhZyMjM91zQS4p7z8eKa9h0JrbacekcirexG0z4n3xz0QOWSvFj3jLhWXUIU21iIAwJtI3RbWa90I7rzAIqI3UElUJG7tLtUXzw4KQNETvXzqWaujEMenYlNIzLGxgB3AuJEQ35}     \end{equation}  satisfies $L(0) w = f$ in~$B_1$. Also, define     \begin{equation}     P_w(x)       =  \sum_{|\beta|\leq d}      \frac{x^{\beta}}{\beta!}      \zxvczxbcvdfghasdfrtsdafasdfasdfdsfgsdgf_{|y|<1}      \partial^{\beta}_x \Gamma(-y)      f(y) \,dy     .     \llabel{o hLA i70 NO Ta06 u2sY GlmspV l2 x y0X B37 x43 k5 kaoZ deyE sDglRF Xi 9 6b6 w9B dId Ko gSUM NLLb CRzeQL UZ m i9O 2qv VzD hz v1r6 spSl jwNhG6 s6 i SdX hob hbp 2u sEdl 95LP AtrBBi bP C wSh pFC CUa yz xYS5 78ro f3UwDP sC I pES HB1 qFP SW 5tt0 I7oz jXun6c z4 c QLB J4M NmI 6F 08S2 Il8C 0JQYiU lI 1 YkK oiu bVt fG uOeg Sllv b4HGn3 bS Z LlX efa eN6 v1 B6m3 Ek3J SXUIjX 8P d NKI UFN JvP Ha Vr4T eARP dXEV7B xM 0 A7w 7je p8M 4Q ahOi hEVo Pxbi1V uG e tOt HbP tsO 5r 363R ez9n A5EJ55 pc L lQQ Hg6 X1J EW K8Cf 9kZm 14A5li rN 7 kKZ rY0 K10 It eJd3 kMGw8ThswELzXU3X7Ebd1KdZ7v1rN3GiirRXGKWK099ovBM0FDJCvkopYNQ2aN94Z7k0UnUKamE3OjU8DFYFFokbSI2J9V9gVlM8ALWThDPnPu3EL7HPD2VDaZTggzcCCmbvc70qqPcC9mt60ogcrTiA3HEjwTK8ymKeuJMc4q6dVz200XnYUtLR9GYjPXvFOVr6W1zUK1WbPToaWJJuKnxBLnd0ftDEbMmj4loHYyhZyMjM91zQS4p7z8eKa9h0JrbacekcirexG0z4n3xz0QOWSvFj3jLhWXUIU21iIAwJtI3RbWa90I7rzAIqI3UElUJG7tLtUXzw4KQNETvXzqWaujEMenYlNIzLGxgB3AuJEQ36}     \end{equation}  To obtain the necessary integrability, observe that   \begin{align}    \begin{split}    | \partial_{x}^{\beta}\Gamma(-y) f(y)    |    \les     \frac{|f(y)|}{|y|^{n+|\beta|-m}}     ,    \end{split}    \label{8ThswELzXU3X7Ebd1KdZ7v1rN3GiirRXGKWK099ovBM0FDJCvkopYNQ2aN94Z7k0UnUKamE3OjU8DFYFFokbSI2J9V9gVlM8ALWThDPnPu3EL7HPD2VDaZTggzcCCmbvc70qqPcC9mt60ogcrTiA3HEjwTK8ymKeuJMc4q6dVz200XnYUtLR9GYjPXvFOVr6W1zUK1WbPToaWJJuKnxBLnd0ftDEbMmj4loHYyhZyMjM91zQS4p7z8eKa9h0JrbacekcirexG0z4n3xz0QOWSvFj3jLhWXUIU21iIAwJtI3RbWa90I7rzAIqI3UElUJG7tLtUXzw4KQNETvXzqWaujEMenYlNIzLGxgB3AuJEQ37}   \end{align} by \eqref{8ThswELzXU3X7Ebd1KdZ7v1rN3GiirRXGKWK099ovBM0FDJCvkopYNQ2aN94Z7k0UnUKamE3OjU8DFYFFokbSI2J9V9gVlM8ALWThDPnPu3EL7HPD2VDaZTggzcCCmbvc70qqPcC9mt60ogcrTiA3HEjwTK8ymKeuJMc4q6dVz200XnYUtLR9GYjPXvFOVr6W1zUK1WbPToaWJJuKnxBLnd0ftDEbMmj4loHYyhZyMjM91zQS4p7z8eKa9h0JrbacekcirexG0z4n3xz0QOWSvFj3jLhWXUIU21iIAwJtI3RbWa90I7rzAIqI3UElUJG7tLtUXzw4KQNETvXzqWaujEMenYlNIzLGxgB3AuJEQG}, and apply Lemma~\ref{L02}~(i); note that the condition $b< a+\fractext{n}{p'}$ becomes $|\beta|<d+\gamma$, which indeed holds.  Thus,  $P_w$ is a polynomial of degree less than or equal to $d$ such that $L(0)P_w = 0$.  Now, we prove that the function     \begin{equation}       u(x)        = w(x) - P_w(x)        = \zxvczxbcvdfghasdfrtsdafasdfasdfdsfgsdgf_{|y|<1} \biggl( \Gamma(x-y) -        \sum_{0\leq |\beta|\leq d} \partial^{\beta}_x \Gamma(-y) \frac{x^{\beta}}{\beta!}                      \biggr) f(y) \,dy       ,     \llabel{ C wSh pFC CUa yz xYS5 78ro f3UwDP sC I pES HB1 qFP SW 5tt0 I7oz jXun6c z4 c QLB J4M NmI 6F 08S2 Il8C 0JQYiU lI 1 YkK oiu bVt fG uOeg Sllv b4HGn3 bS Z LlX efa eN6 v1 B6m3 Ek3J SXUIjX 8P d NKI UFN JvP Ha Vr4T eARP dXEV7B xM 0 A7w 7je p8M 4Q ahOi hEVo Pxbi1V uG e tOt HbP tsO 5r 363R ez9n A5EJ55 pc L lQQ Hg6 X1J EW K8Cf 9kZm 14A5li rN 7 kKZ rY0 K10 It eJd3 kMGw opVnfY EG 2 orG fj0 TTA Xt ecJK eTM0 x1N9f0 lR p QkP M37 3r0 iA 6EFs 1F6f 4mjOB5 zu 5 GGT Ncl Bmk b5 jOOK 4yny My04oz 6m 6 Akz NnP JXh Bn PHRu N5Ly qSguz5 Nn W 2lU Yx3 fX4 hu LieH8ThswELzXU3X7Ebd1KdZ7v1rN3GiirRXGKWK099ovBM0FDJCvkopYNQ2aN94Z7k0UnUKamE3OjU8DFYFFokbSI2J9V9gVlM8ALWThDPnPu3EL7HPD2VDaZTggzcCCmbvc70qqPcC9mt60ogcrTiA3HEjwTK8ymKeuJMc4q6dVz200XnYUtLR9GYjPXvFOVr6W1zUK1WbPToaWJJuKnxBLnd0ftDEbMmj4loHYyhZyMjM91zQS4p7z8eKa9h0JrbacekcirexG0z4n3xz0QOWSvFj3jLhWXUIU21iIAwJtI3RbWa90I7rzAIqI3UElUJG7tLtUXzw4KQNETvXzqWaujEMenYlNIzLGxgB3AuJEQ38}     \end{equation} which satisfies $L(0) u = f$,  verifies~\eqref{8ThswELzXU3X7Ebd1KdZ7v1rN3GiirRXGKWK099ovBM0FDJCvkopYNQ2aN94Z7k0UnUKamE3OjU8DFYFFokbSI2J9V9gVlM8ALWThDPnPu3EL7HPD2VDaZTggzcCCmbvc70qqPcC9mt60ogcrTiA3HEjwTK8ymKeuJMc4q6dVz200XnYUtLR9GYjPXvFOVr6W1zUK1WbPToaWJJuKnxBLnd0ftDEbMmj4loHYyhZyMjM91zQS4p7z8eKa9h0JrbacekcirexG0z4n3xz0QOWSvFj3jLhWXUIU21iIAwJtI3RbWa90I7rzAIqI3UElUJG7tLtUXzw4KQNETvXzqWaujEMenYlNIzLGxgB3AuJEQ26}. For any $r \in(0,1/2]$, we have     \begin{align}     \begin{split}       \zxvczxbcvdfghasdfrtsdafasdfasdfdsfgsdgh u \zxvczxbcvdfghasdfrtsdafasdfasdfdsfgsdgh_{L^p(B_r)}        &       \leq      \biggl\zxvczxbcvdfghasdfrtsdafasdfasdfdsfgsdgh \zxvczxbcvdfghasdfrtsdafasdfasdfdsfgsdgf_{|y|\leq 2r} \Gamma(x-y) f(y) \,dy \biggr\zxvczxbcvdfghasdfrtsdafasdfasdfdsfgsdgh_{L^p(B_r)}      +      \biggl\zxvczxbcvdfghasdfrtsdafasdfasdfdsfgsdgh \zxvczxbcvdfghasdfrtsdafasdfasdfdsfgsdgf_{|y|\leq 2r} \sum_{0\leq |\beta|\leq d} \partial^{\beta}_x \Gamma(-y) \frac{x^{\beta}}{\beta!}    f(y) \,dy \biggr\zxvczxbcvdfghasdfrtsdafasdfasdfdsfgsdgh_{L^p(B_r)}      \\&\indeq      +      \biggl\zxvczxbcvdfghasdfrtsdafasdfasdfdsfgsdgh \zxvczxbcvdfghasdfrtsdafasdfasdfdsfgsdgf_{2r\leq|y|<1}              \biggl(                 \Gamma(x-y) - \sum_{0\leq |\beta|\leq d} \partial^{\beta}_x \Gamma(-y) \frac{x^{\beta}}{\beta!}                 \biggr) f(y) \,dy       \biggr\zxvczxbcvdfghasdfrtsdafasdfasdfdsfgsdgh_{L^p(B_r)}     \\&     = I_1 + I_2 + I_3     .     \end{split}     \label{8ThswELzXU3X7Ebd1KdZ7v1rN3GiirRXGKWK099ovBM0FDJCvkopYNQ2aN94Z7k0UnUKamE3OjU8DFYFFokbSI2J9V9gVlM8ALWThDPnPu3EL7HPD2VDaZTggzcCCmbvc70qqPcC9mt60ogcrTiA3HEjwTK8ymKeuJMc4q6dVz200XnYUtLR9GYjPXvFOVr6W1zUK1WbPToaWJJuKnxBLnd0ftDEbMmj4loHYyhZyMjM91zQS4p7z8eKa9h0JrbacekcirexG0z4n3xz0QOWSvFj3jLhWXUIU21iIAwJtI3RbWa90I7rzAIqI3UElUJG7tLtUXzw4KQNETvXzqWaujEMenYlNIzLGxgB3AuJEQ39}     \end{align} For the first term, we use Young's inequality to bound     \begin{align}     \begin{split}       I_1        &=        \biggl\zxvczxbcvdfghasdfrtsdafasdfasdfdsfgsdgh \zxvczxbcvdfghasdfrtsdafasdfasdfdsfgsdgf \Gamma(x-y) f(y)\chi_{B_{2r}}(y)  \,dy \biggr\zxvczxbcvdfghasdfrtsdafasdfasdfdsfgsdgh_{L^p(B_r)}       \leq \zxvczxbcvdfghasdfrtsdafasdfasdfdsfgsdgh \Gamma \zxvczxbcvdfghasdfrtsdafasdfasdfdsfgsdgh_{L^1(\mathbb R^n)} \zxvczxbcvdfghasdfrtsdafasdfasdfdsfgsdgh f \zxvczxbcvdfghasdfrtsdafasdfasdfdsfgsdgh_{L^p(B_r)}       \leq \zxvczxbcvdfghasdfrtsdafasdfasdfdsfgsdgh \Gamma \zxvczxbcvdfghasdfrtsdafasdfasdfdsfgsdgh_{L^1(B_{3r})} \zxvczxbcvdfghasdfrtsdafasdfasdfdsfgsdgh f \zxvczxbcvdfghasdfrtsdafasdfasdfdsfgsdgh_{L^p(B_{2r})}       \\&       \les       \MM r^{d-m+\gamma+n/p}\zxvczxbcvdfghasdfrtsdafasdfasdfdsfgsdgf_0^{3r} \frac{s^{n-1}}{s^{n-m}} \,ds       \les \MM r^{d+\gamma+n/p}       ,     \end{split}     \llabel{jX 8P d NKI UFN JvP Ha Vr4T eARP dXEV7B xM 0 A7w 7je p8M 4Q ahOi hEVo Pxbi1V uG e tOt HbP tsO 5r 363R ez9n A5EJ55 pc L lQQ Hg6 X1J EW K8Cf 9kZm 14A5li rN 7 kKZ rY0 K10 It eJd3 kMGw opVnfY EG 2 orG fj0 TTA Xt ecJK eTM0 x1N9f0 lR p QkP M37 3r0 iA 6EFs 1F6f 4mjOB5 zu 5 GGT Ncl Bmk b5 jOOK 4yny My04oz 6m 6 Akz NnP JXh Bn PHRu N5Ly qSguz5 Nn W 2lU Yx3 fX4 hu LieH L30w g93Xwc gj 1 I9d O9b EPC R0 vc6A 005Q VFy1ly K7 o VRV pbJ zZn xY dcld XgQa DXY3gz x3 6 8OR JFK 9Uh XT e3xY bVHG oYqdHg Vy f 5kK Qzm mK4 9x xiAp jVkw gzJOdE 4v g hAv 9bV IHe wc8ThswELzXU3X7Ebd1KdZ7v1rN3GiirRXGKWK099ovBM0FDJCvkopYNQ2aN94Z7k0UnUKamE3OjU8DFYFFokbSI2J9V9gVlM8ALWThDPnPu3EL7HPD2VDaZTggzcCCmbvc70qqPcC9mt60ogcrTiA3HEjwTK8ymKeuJMc4q6dVz200XnYUtLR9GYjPXvFOVr6W1zUK1WbPToaWJJuKnxBLnd0ftDEbMmj4loHYyhZyMjM91zQS4p7z8eKa9h0JrbacekcirexG0z4n3xz0QOWSvFj3jLhWXUIU21iIAwJtI3RbWa90I7rzAIqI3UElUJG7tLtUXzw4KQNETvXzqWaujEMenYlNIzLGxgB3AuJEQ40}     \end{align} where we used \eqref{8ThswELzXU3X7Ebd1KdZ7v1rN3GiirRXGKWK099ovBM0FDJCvkopYNQ2aN94Z7k0UnUKamE3OjU8DFYFFokbSI2J9V9gVlM8ALWThDPnPu3EL7HPD2VDaZTggzcCCmbvc70qqPcC9mt60ogcrTiA3HEjwTK8ymKeuJMc4q6dVz200XnYUtLR9GYjPXvFOVr6W1zUK1WbPToaWJJuKnxBLnd0ftDEbMmj4loHYyhZyMjM91zQS4p7z8eKa9h0JrbacekcirexG0z4n3xz0QOWSvFj3jLhWXUIU21iIAwJtI3RbWa90I7rzAIqI3UElUJG7tLtUXzw4KQNETvXzqWaujEMenYlNIzLGxgB3AuJEQG} in the fourth step; note that the $L^{p}(B_r)$ norm is taken in the $x$ variable. To estimate $I_2$, we write    \begin{align}    \begin{split}    &\biggl| \zxvczxbcvdfghasdfrtsdafasdfasdfdsfgsdgf_{|y|\leq2r} \partial^{\beta}_x \Gamma(-y) \frac{x^{\beta}}{\beta!}    f(y) \,dy \biggr|    \les     |x|^{|\beta|}        \zxvczxbcvdfghasdfrtsdafasdfasdfdsfgsdgf_{|y|\leq 2r} \frac{|f(y)|}{|y|^{n+|\beta|-m}} \,dy      \les       r^{|\beta|}      \zxvczxbcvdfghasdfrtsdafasdfasdfdsfgsdgf_{|y|\leq2r} \frac{|f(y)|}{|y|^{n+|\beta|-m}} \,dy     ,    \end{split}    \llabel{ opVnfY EG 2 orG fj0 TTA Xt ecJK eTM0 x1N9f0 lR p QkP M37 3r0 iA 6EFs 1F6f 4mjOB5 zu 5 GGT Ncl Bmk b5 jOOK 4yny My04oz 6m 6 Akz NnP JXh Bn PHRu N5Ly qSguz5 Nn W 2lU Yx3 fX4 hu LieH L30w g93Xwc gj 1 I9d O9b EPC R0 vc6A 005Q VFy1ly K7 o VRV pbJ zZn xY dcld XgQa DXY3gz x3 6 8OR JFK 9Uh XT e3xY bVHG oYqdHg Vy f 5kK Qzm mK4 9x xiAp jVkw gzJOdE 4v g hAv 9bV IHe wc Vqcb SUcF 1pHzol Nj T l1B urc Sam IP zkUS 8wwS a7wVWR 4D L VGf 1RF r59 9H tyGq hDT0 TDlooa mg j 9am png aWe nG XU2T zXLh IYOW5v 2d A rCG sLk s53 pW AuAy DQlF 6spKyd HT 9 Z1X n2s U8ThswELzXU3X7Ebd1KdZ7v1rN3GiirRXGKWK099ovBM0FDJCvkopYNQ2aN94Z7k0UnUKamE3OjU8DFYFFokbSI2J9V9gVlM8ALWThDPnPu3EL7HPD2VDaZTggzcCCmbvc70qqPcC9mt60ogcrTiA3HEjwTK8ymKeuJMc4q6dVz200XnYUtLR9GYjPXvFOVr6W1zUK1WbPToaWJJuKnxBLnd0ftDEbMmj4loHYyhZyMjM91zQS4p7z8eKa9h0JrbacekcirexG0z4n3xz0QOWSvFj3jLhWXUIU21iIAwJtI3RbWa90I7rzAIqI3UElUJG7tLtUXzw4KQNETvXzqWaujEMenYlNIzLGxgB3AuJEQ41}   \end{align} for any $\beta$ such that $|\beta|\leq d$. Applying Lemma~\ref{L02}~(i) with $a = d - m + \gamma + n/p$ and $b = n + |\beta| -m$, we have      \begin{equation}       \biggl| \zxvczxbcvdfghasdfrtsdafasdfasdfdsfgsdgf_{|y|\leq2r} \partial^{\beta}_x \Gamma(-y) \frac{x^{\beta}}{\beta!}    f(y) \,dy \biggr|       \les \MM r^{|\beta|} r^{d + \gamma - |\beta|}       \les \MM r^{d+\gamma}       .     \label{8ThswELzXU3X7Ebd1KdZ7v1rN3GiirRXGKWK099ovBM0FDJCvkopYNQ2aN94Z7k0UnUKamE3OjU8DFYFFokbSI2J9V9gVlM8ALWThDPnPu3EL7HPD2VDaZTggzcCCmbvc70qqPcC9mt60ogcrTiA3HEjwTK8ymKeuJMc4q6dVz200XnYUtLR9GYjPXvFOVr6W1zUK1WbPToaWJJuKnxBLnd0ftDEbMmj4loHYyhZyMjM91zQS4p7z8eKa9h0JrbacekcirexG0z4n3xz0QOWSvFj3jLhWXUIU21iIAwJtI3RbWa90I7rzAIqI3UElUJG7tLtUXzw4KQNETvXzqWaujEMenYlNIzLGxgB3AuJEQ42}     \end{equation} Note that the condition $b<a+n/p'$ becomes~$\gamma>0$. Therefore, it follows that $I_2 \les M r^{d+\gamma + n/p}$ as desired. To estimate $I_3$, we first claim that      \begin{equation}       I(x)       \les \MM r^{d+\gamma}       \comma        |x| \leq r         ,     \llabel{ L30w g93Xwc gj 1 I9d O9b EPC R0 vc6A 005Q VFy1ly K7 o VRV pbJ zZn xY dcld XgQa DXY3gz x3 6 8OR JFK 9Uh XT e3xY bVHG oYqdHg Vy f 5kK Qzm mK4 9x xiAp jVkw gzJOdE 4v g hAv 9bV IHe wc Vqcb SUcF 1pHzol Nj T l1B urc Sam IP zkUS 8wwS a7wVWR 4D L VGf 1RF r59 9H tyGq hDT0 TDlooa mg j 9am png aWe nG XU2T zXLh IYOW5v 2d A rCG sLk s53 pW AuAy DQlF 6spKyd HT 9 Z1X n2s U1g 0D Llao YuLP PB6YKo D1 M 0fi qHU l4A Ia joiV Q6af VT6wvY Md 0 pCY BZp 7RX Hd xTb0 sjJ0 Beqpkc 8b N OgZ 0Tr 0wq h1 C2Hn YQXM 8nJ0Pf uG J Be2 vuq Duk LV AJwv 2tYc JOM1uK h7 p cgo 8ThswELzXU3X7Ebd1KdZ7v1rN3GiirRXGKWK099ovBM0FDJCvkopYNQ2aN94Z7k0UnUKamE3OjU8DFYFFokbSI2J9V9gVlM8ALWThDPnPu3EL7HPD2VDaZTggzcCCmbvc70qqPcC9mt60ogcrTiA3HEjwTK8ymKeuJMc4q6dVz200XnYUtLR9GYjPXvFOVr6W1zUK1WbPToaWJJuKnxBLnd0ftDEbMmj4loHYyhZyMjM91zQS4p7z8eKa9h0JrbacekcirexG0z4n3xz0QOWSvFj3jLhWXUIU21iIAwJtI3RbWa90I7rzAIqI3UElUJG7tLtUXzw4KQNETvXzqWaujEMenYlNIzLGxgB3AuJEQ43}     \end{equation} where      \begin{equation}       I(x)        = \biggl| \zxvczxbcvdfghasdfrtsdafasdfasdfdsfgsdgf_{2r\leq|y|<1}              \biggl( \Gamma(x-y)                      - \sum_{|\beta| \leq d} \partial_x^{\beta} \Gamma (-y) \frac{x^{\beta}}{\beta !}             \biggr)              f(y) \,dy           \biggr|       .     \label{8ThswELzXU3X7Ebd1KdZ7v1rN3GiirRXGKWK099ovBM0FDJCvkopYNQ2aN94Z7k0UnUKamE3OjU8DFYFFokbSI2J9V9gVlM8ALWThDPnPu3EL7HPD2VDaZTggzcCCmbvc70qqPcC9mt60ogcrTiA3HEjwTK8ymKeuJMc4q6dVz200XnYUtLR9GYjPXvFOVr6W1zUK1WbPToaWJJuKnxBLnd0ftDEbMmj4loHYyhZyMjM91zQS4p7z8eKa9h0JrbacekcirexG0z4n3xz0QOWSvFj3jLhWXUIU21iIAwJtI3RbWa90I7rzAIqI3UElUJG7tLtUXzw4KQNETvXzqWaujEMenYlNIzLGxgB3AuJEQ44}     \end{equation} By Taylor's theorem      \begin{equation}       \biggl| \Gamma(x-y)                      - \sum_{|\beta| \leq d} \partial_x^{\beta} \Gamma (-y) \frac{x^{\beta}}{\beta !}       \biggr|        \les \frac{|x|^{d+1}}{|y-\theta x |^{n+d+1-m}}       \les \frac{|x|^{d+1}}{|y|^{n+d+1-m}}       ,     \label{8ThswELzXU3X7Ebd1KdZ7v1rN3GiirRXGKWK099ovBM0FDJCvkopYNQ2aN94Z7k0UnUKamE3OjU8DFYFFokbSI2J9V9gVlM8ALWThDPnPu3EL7HPD2VDaZTggzcCCmbvc70qqPcC9mt60ogcrTiA3HEjwTK8ymKeuJMc4q6dVz200XnYUtLR9GYjPXvFOVr6W1zUK1WbPToaWJJuKnxBLnd0ftDEbMmj4loHYyhZyMjM91zQS4p7z8eKa9h0JrbacekcirexG0z4n3xz0QOWSvFj3jLhWXUIU21iIAwJtI3RbWa90I7rzAIqI3UElUJG7tLtUXzw4KQNETvXzqWaujEMenYlNIzLGxgB3AuJEQ45}     \end{equation}
for some $\theta\in(0,1)$, where we used \eqref{8ThswELzXU3X7Ebd1KdZ7v1rN3GiirRXGKWK099ovBM0FDJCvkopYNQ2aN94Z7k0UnUKamE3OjU8DFYFFokbSI2J9V9gVlM8ALWThDPnPu3EL7HPD2VDaZTggzcCCmbvc70qqPcC9mt60ogcrTiA3HEjwTK8ymKeuJMc4q6dVz200XnYUtLR9GYjPXvFOVr6W1zUK1WbPToaWJJuKnxBLnd0ftDEbMmj4loHYyhZyMjM91zQS4p7z8eKa9h0JrbacekcirexG0z4n3xz0QOWSvFj3jLhWXUIU21iIAwJtI3RbWa90I7rzAIqI3UElUJG7tLtUXzw4KQNETvXzqWaujEMenYlNIzLGxgB3AuJEQ37} in the first inequality and $|x|\leq r\leq |y|/2$ in the second. Hence, we obtain   \begin{equation*}    I(x)    \les    |x|^{d+1}    \zxvczxbcvdfghasdfrtsdafasdfasdfdsfgsdgf_{2r\leq |y|<1}    \frac{     |f(y)|        }{     |y|^{n+d+1-m}    }\,dy   .   \end{equation*} We apply Lemma~\ref{L02}~(ii) to get     \begin{equation}       I(x)       \les \MM r^{d+1} r^{\gamma - 1}       \les \MM r^{d+\gamma}       ;     \label{8ThswELzXU3X7Ebd1KdZ7v1rN3GiirRXGKWK099ovBM0FDJCvkopYNQ2aN94Z7k0UnUKamE3OjU8DFYFFokbSI2J9V9gVlM8ALWThDPnPu3EL7HPD2VDaZTggzcCCmbvc70qqPcC9mt60ogcrTiA3HEjwTK8ymKeuJMc4q6dVz200XnYUtLR9GYjPXvFOVr6W1zUK1WbPToaWJJuKnxBLnd0ftDEbMmj4loHYyhZyMjM91zQS4p7z8eKa9h0JrbacekcirexG0z4n3xz0QOWSvFj3jLhWXUIU21iIAwJtI3RbWa90I7rzAIqI3UElUJG7tLtUXzw4KQNETvXzqWaujEMenYlNIzLGxgB3AuJEQ46}     \end{equation} note that the condition \eqref{8ThswELzXU3X7Ebd1KdZ7v1rN3GiirRXGKWK099ovBM0FDJCvkopYNQ2aN94Z7k0UnUKamE3OjU8DFYFFokbSI2J9V9gVlM8ALWThDPnPu3EL7HPD2VDaZTggzcCCmbvc70qqPcC9mt60ogcrTiA3HEjwTK8ymKeuJMc4q6dVz200XnYUtLR9GYjPXvFOVr6W1zUK1WbPToaWJJuKnxBLnd0ftDEbMmj4loHYyhZyMjM91zQS4p7z8eKa9h0JrbacekcirexG0z4n3xz0QOWSvFj3jLhWXUIU21iIAwJtI3RbWa90I7rzAIqI3UElUJG7tLtUXzw4KQNETvXzqWaujEMenYlNIzLGxgB3AuJEQ30} becomes $\gamma<1$, which indeed holds. Then, by \eqref{8ThswELzXU3X7Ebd1KdZ7v1rN3GiirRXGKWK099ovBM0FDJCvkopYNQ2aN94Z7k0UnUKamE3OjU8DFYFFokbSI2J9V9gVlM8ALWThDPnPu3EL7HPD2VDaZTggzcCCmbvc70qqPcC9mt60ogcrTiA3HEjwTK8ymKeuJMc4q6dVz200XnYUtLR9GYjPXvFOVr6W1zUK1WbPToaWJJuKnxBLnd0ftDEbMmj4loHYyhZyMjM91zQS4p7z8eKa9h0JrbacekcirexG0z4n3xz0QOWSvFj3jLhWXUIU21iIAwJtI3RbWa90I7rzAIqI3UElUJG7tLtUXzw4KQNETvXzqWaujEMenYlNIzLGxgB3AuJEQ46}, $I_3 \les \MM r^{d+\gamma+n/p}$ as desired. Therefore,     \begin{equation}       \zxvczxbcvdfghasdfrtsdafasdfasdfdsfgsdgh u \zxvczxbcvdfghasdfrtsdafasdfasdfdsfgsdgh_{L^p(B_r)}        \les \MM r^{d+\gamma+n/p}       \comma r\leq\frac{1}{2}      .     \llabel{ Vqcb SUcF 1pHzol Nj T l1B urc Sam IP zkUS 8wwS a7wVWR 4D L VGf 1RF r59 9H tyGq hDT0 TDlooa mg j 9am png aWe nG XU2T zXLh IYOW5v 2d A rCG sLk s53 pW AuAy DQlF 6spKyd HT 9 Z1X n2s U1g 0D Llao YuLP PB6YKo D1 M 0fi qHU l4A Ia joiV Q6af VT6wvY Md 0 pCY BZp 7RX Hd xTb0 sjJ0 Beqpkc 8b N OgZ 0Tr 0wq h1 C2Hn YQXM 8nJ0Pf uG J Be2 vuq Duk LV AJwv 2tYc JOM1uK h7 p cgo iiK t0b 3e URec DVM7 ivRMh1 T6 p AWl upj kEj UL R3xN VAu5 kEbnrV HE 1 OrJ 2bx dUP yD vyVi x6sC BpGDSx jB C n9P Fiu xkF vw 0QPo fRjy 2OFItV eD B tDz lc9 xVy A0 de9Y 5h8c 7dYCFk Fl v8ThswELzXU3X7Ebd1KdZ7v1rN3GiirRXGKWK099ovBM0FDJCvkopYNQ2aN94Z7k0UnUKamE3OjU8DFYFFokbSI2J9V9gVlM8ALWThDPnPu3EL7HPD2VDaZTggzcCCmbvc70qqPcC9mt60ogcrTiA3HEjwTK8ymKeuJMc4q6dVz200XnYUtLR9GYjPXvFOVr6W1zUK1WbPToaWJJuKnxBLnd0ftDEbMmj4loHYyhZyMjM91zQS4p7z8eKa9h0JrbacekcirexG0z4n3xz0QOWSvFj3jLhWXUIU21iIAwJtI3RbWa90I7rzAIqI3UElUJG7tLtUXzw4KQNETvXzqWaujEMenYlNIzLGxgB3AuJEQ47}     \end{equation} The inequality \eqref{8ThswELzXU3X7Ebd1KdZ7v1rN3GiirRXGKWK099ovBM0FDJCvkopYNQ2aN94Z7k0UnUKamE3OjU8DFYFFokbSI2J9V9gVlM8ALWThDPnPu3EL7HPD2VDaZTggzcCCmbvc70qqPcC9mt60ogcrTiA3HEjwTK8ymKeuJMc4q6dVz200XnYUtLR9GYjPXvFOVr6W1zUK1WbPToaWJJuKnxBLnd0ftDEbMmj4loHYyhZyMjM91zQS4p7z8eKa9h0JrbacekcirexG0z4n3xz0QOWSvFj3jLhWXUIU21iIAwJtI3RbWa90I7rzAIqI3UElUJG7tLtUXzw4KQNETvXzqWaujEMenYlNIzLGxgB3AuJEQ26} then holds by applying the elliptic regularity to $Lu=f$, leading to       $r^{|\nu|}\zxvczxbcvdfghasdfrtsdafasdfasdfdsfgsdgh \partial^{\nu} u \zxvczxbcvdfghasdfrtsdafasdfasdfdsfgsdgh_{L^p(B_r)}        \les  \zxvczxbcvdfghasdfrtsdafasdfasdfdsfgsdgh u \zxvczxbcvdfghasdfrtsdafasdfasdfdsfgsdgh_{L^p(B_{2r})} + r^{m}\zxvczxbcvdfghasdfrtsdafasdfasdfdsfgsdgh f \zxvczxbcvdfghasdfrtsdafasdfasdfdsfgsdgh_{L^p(B_{2r})}       \les \MM r^{d+\gamma+n/p}$ for any $r\leq 1/4$ and $|\nu| \leq m$. \par For the variable coefficient case, we follow the approach in  \cite[pp.~489--490]{B} and \cite[pp.~464--466]{H1}. Assume first  that the coefficients of $L$ are uniformly close to the case considered first, i.e., we have     \begin{equation}       \sum_{|\nu|=m} |a_{\nu}(x) - a_{\nu}(0)|        + \sum_{|\nu|<m} |a_{\nu}(x)|        \leq \eps       ,     \label{8ThswELzXU3X7Ebd1KdZ7v1rN3GiirRXGKWK099ovBM0FDJCvkopYNQ2aN94Z7k0UnUKamE3OjU8DFYFFokbSI2J9V9gVlM8ALWThDPnPu3EL7HPD2VDaZTggzcCCmbvc70qqPcC9mt60ogcrTiA3HEjwTK8ymKeuJMc4q6dVz200XnYUtLR9GYjPXvFOVr6W1zUK1WbPToaWJJuKnxBLnd0ftDEbMmj4loHYyhZyMjM91zQS4p7z8eKa9h0JrbacekcirexG0z4n3xz0QOWSvFj3jLhWXUIU21iIAwJtI3RbWa90I7rzAIqI3UElUJG7tLtUXzw4KQNETvXzqWaujEMenYlNIzLGxgB3AuJEQ48}     \end{equation} for some small positive $\eps$ depending on $p$, $d$, and $\gamma$ only.  We rewrite $L u = f$ as     \begin{equation}       L(0)u = f+(L(0)-L)u       ,     \label{8ThswELzXU3X7Ebd1KdZ7v1rN3GiirRXGKWK099ovBM0FDJCvkopYNQ2aN94Z7k0UnUKamE3OjU8DFYFFokbSI2J9V9gVlM8ALWThDPnPu3EL7HPD2VDaZTggzcCCmbvc70qqPcC9mt60ogcrTiA3HEjwTK8ymKeuJMc4q6dVz200XnYUtLR9GYjPXvFOVr6W1zUK1WbPToaWJJuKnxBLnd0ftDEbMmj4loHYyhZyMjM91zQS4p7z8eKa9h0JrbacekcirexG0z4n3xz0QOWSvFj3jLhWXUIU21iIAwJtI3RbWa90I7rzAIqI3UElUJG7tLtUXzw4KQNETvXzqWaujEMenYlNIzLGxgB3AuJEQ49}     \end{equation} where $L(0)$ is introduced in \eqref{8ThswELzXU3X7Ebd1KdZ7v1rN3GiirRXGKWK099ovBM0FDJCvkopYNQ2aN94Z7k0UnUKamE3OjU8DFYFFokbSI2J9V9gVlM8ALWThDPnPu3EL7HPD2VDaZTggzcCCmbvc70qqPcC9mt60ogcrTiA3HEjwTK8ymKeuJMc4q6dVz200XnYUtLR9GYjPXvFOVr6W1zUK1WbPToaWJJuKnxBLnd0ftDEbMmj4loHYyhZyMjM91zQS4p7z8eKa9h0JrbacekcirexG0z4n3xz0QOWSvFj3jLhWXUIU21iIAwJtI3RbWa90I7rzAIqI3UElUJG7tLtUXzw4KQNETvXzqWaujEMenYlNIzLGxgB3AuJEQL}. Consider the convex set   \begin{equation}    S = \biggl\{ u \in W^{m,p} (B_1):              \zxvczxbcvdfghasdfrtsdafasdfasdfdsfgsdgh u \zxvczxbcvdfghasdfrtsdafasdfasdfdsfgsdgh_{W^{m,p}(B_1)} \leq M_1,            \sum_{|\nu|\leq m} r^{|\nu|} \zxvczxbcvdfghasdfrtsdafasdfasdfdsfgsdgh \partial^{\nu} u 	   \zxvczxbcvdfghasdfrtsdafasdfasdfdsfgsdgh_{L^p(B_r)} \leq M_2 r^{d+\gamma+n/p}, \forall r<1        \biggr\}    ,       \llabel{1g 0D Llao YuLP PB6YKo D1 M 0fi qHU l4A Ia joiV Q6af VT6wvY Md 0 pCY BZp 7RX Hd xTb0 sjJ0 Beqpkc 8b N OgZ 0Tr 0wq h1 C2Hn YQXM 8nJ0Pf uG J Be2 vuq Duk LV AJwv 2tYc JOM1uK h7 p cgo iiK t0b 3e URec DVM7 ivRMh1 T6 p AWl upj kEj UL R3xN VAu5 kEbnrV HE 1 OrJ 2bx dUP yD vyVi x6sC BpGDSx jB C n9P Fiu xkF vw 0QPo fRjy 2OFItV eD B tDz lc9 xVy A0 de9Y 5h8c 7dYCFk Fl v WPD SuN VI6 MZ 72u9 MBtK 9BGLNs Yp l X2y b5U HgH AD bW8X Rzkv UJZShW QH G oKX yVA rsH TQ 1Vbd dK2M IxmTf6 wE T 9cX Fbu uVx Cb SBBp 0v2J MQ5Z8z 3p M EGp TU6 KCc YN 2BlW dp2t mliPDH8ThswELzXU3X7Ebd1KdZ7v1rN3GiirRXGKWK099ovBM0FDJCvkopYNQ2aN94Z7k0UnUKamE3OjU8DFYFFokbSI2J9V9gVlM8ALWThDPnPu3EL7HPD2VDaZTggzcCCmbvc70qqPcC9mt60ogcrTiA3HEjwTK8ymKeuJMc4q6dVz200XnYUtLR9GYjPXvFOVr6W1zUK1WbPToaWJJuKnxBLnd0ftDEbMmj4loHYyhZyMjM91zQS4p7z8eKa9h0JrbacekcirexG0z4n3xz0QOWSvFj3jLhWXUIU21iIAwJtI3RbWa90I7rzAIqI3UElUJG7tLtUXzw4KQNETvXzqWaujEMenYlNIzLGxgB3AuJEQ28}   \end{equation} and define $T\colon S \to W^{m,p}(B_1)$ as     \begin{equation}       T(u)        = \zxvczxbcvdfghasdfrtsdafasdfasdfdsfgsdgf_{|y|<1} \Gamma (x-y) \big( f(y) + (L(0)-L) u(y) \big) \,dy      .     \llabel{iiK t0b 3e URec DVM7 ivRMh1 T6 p AWl upj kEj UL R3xN VAu5 kEbnrV HE 1 OrJ 2bx dUP yD vyVi x6sC BpGDSx jB C n9P Fiu xkF vw 0QPo fRjy 2OFItV eD B tDz lc9 xVy A0 de9Y 5h8c 7dYCFk Fl v WPD SuN VI6 MZ 72u9 MBtK 9BGLNs Yp l X2y b5U HgH AD bW8X Rzkv UJZShW QH G oKX yVA rsH TQ 1Vbd dK2M IxmTf6 wE T 9cX Fbu uVx Cb SBBp 0v2J MQ5Z8z 3p M EGp TU6 KCc YN 2BlW dp2t mliPDH JQ W jIR Rgq i5l AP gikl c8ru HnvYFM AI r Ih7 Ths 9tE hA AYgS swZZ fws19P 5w e JvM imb sFH Th CnSZ HORm yt98w3 U3 z ant zAy Twq 0C jgDI Etkb h98V4u o5 2 jjA Zz1 kLo C8 oHGv Z5Ru G8ThswELzXU3X7Ebd1KdZ7v1rN3GiirRXGKWK099ovBM0FDJCvkopYNQ2aN94Z7k0UnUKamE3OjU8DFYFFokbSI2J9V9gVlM8ALWThDPnPu3EL7HPD2VDaZTggzcCCmbvc70qqPcC9mt60ogcrTiA3HEjwTK8ymKeuJMc4q6dVz200XnYUtLR9GYjPXvFOVr6W1zUK1WbPToaWJJuKnxBLnd0ftDEbMmj4loHYyhZyMjM91zQS4p7z8eKa9h0JrbacekcirexG0z4n3xz0QOWSvFj3jLhWXUIU21iIAwJtI3RbWa90I7rzAIqI3UElUJG7tLtUXzw4KQNETvXzqWaujEMenYlNIzLGxgB3AuJEQ50}     \end{equation} We claim $T(S)\subseteq S$ and that $T$ has a fixed point is $S$; the fixed point solves \eqref{8ThswELzXU3X7Ebd1KdZ7v1rN3GiirRXGKWK099ovBM0FDJCvkopYNQ2aN94Z7k0UnUKamE3OjU8DFYFFokbSI2J9V9gVlM8ALWThDPnPu3EL7HPD2VDaZTggzcCCmbvc70qqPcC9mt60ogcrTiA3HEjwTK8ymKeuJMc4q6dVz200XnYUtLR9GYjPXvFOVr6W1zUK1WbPToaWJJuKnxBLnd0ftDEbMmj4loHYyhZyMjM91zQS4p7z8eKa9h0JrbacekcirexG0z4n3xz0QOWSvFj3jLhWXUIU21iIAwJtI3RbWa90I7rzAIqI3UElUJG7tLtUXzw4KQNETvXzqWaujEMenYlNIzLGxgB3AuJEQ49} and satisfies the condition~\eqref{8ThswELzXU3X7Ebd1KdZ7v1rN3GiirRXGKWK099ovBM0FDJCvkopYNQ2aN94Z7k0UnUKamE3OjU8DFYFFokbSI2J9V9gVlM8ALWThDPnPu3EL7HPD2VDaZTggzcCCmbvc70qqPcC9mt60ogcrTiA3HEjwTK8ymKeuJMc4q6dVz200XnYUtLR9GYjPXvFOVr6W1zUK1WbPToaWJJuKnxBLnd0ftDEbMmj4loHYyhZyMjM91zQS4p7z8eKa9h0JrbacekcirexG0z4n3xz0QOWSvFj3jLhWXUIU21iIAwJtI3RbWa90I7rzAIqI3UElUJG7tLtUXzw4KQNETvXzqWaujEMenYlNIzLGxgB3AuJEQ26}.  \par We first claim $T(S)\subseteq S$ for appropriate $M_1$ and~$M_2$. For any $u\in S$, there exists $T(u)$ in $W^{m,p}$ such that     \begin{equation}       L(0) (T(u)) = f+(L(0)-L)u       ,     \llabel{ WPD SuN VI6 MZ 72u9 MBtK 9BGLNs Yp l X2y b5U HgH AD bW8X Rzkv UJZShW QH G oKX yVA rsH TQ 1Vbd dK2M IxmTf6 wE T 9cX Fbu uVx Cb SBBp 0v2J MQ5Z8z 3p M EGp TU6 KCc YN 2BlW dp2t mliPDH JQ W jIR Rgq i5l AP gikl c8ru HnvYFM AI r Ih7 Ths 9tE hA AYgS swZZ fws19P 5w e JvM imb sFH Th CnSZ HORm yt98w3 U3 z ant zAy Twq 0C jgDI Etkb h98V4u o5 2 jjA Zz1 kLo C8 oHGv Z5Ru Gwv3kK 4W B 50T oMt q7Q WG 9mtb SIlc 87ruZf Kw Z Ph3 1ZA Osq 8l jVQJ LTXC gyQn0v KE S iSq Bpa wtH xc IJe4 SiE1 izzxim ke P Y3s 7SX 5DA SG XHqC r38V YP3Hxv OI R ZtM fqN oLF oU 7vNd t8ThswELzXU3X7Ebd1KdZ7v1rN3GiirRXGKWK099ovBM0FDJCvkopYNQ2aN94Z7k0UnUKamE3OjU8DFYFFokbSI2J9V9gVlM8ALWThDPnPu3EL7HPD2VDaZTggzcCCmbvc70qqPcC9mt60ogcrTiA3HEjwTK8ymKeuJMc4q6dVz200XnYUtLR9GYjPXvFOVr6W1zUK1WbPToaWJJuKnxBLnd0ftDEbMmj4loHYyhZyMjM91zQS4p7z8eKa9h0JrbacekcirexG0z4n3xz0QOWSvFj3jLhWXUIU21iIAwJtI3RbWa90I7rzAIqI3UElUJG7tLtUXzw4KQNETvXzqWaujEMenYlNIzLGxgB3AuJEQ51}     \end{equation} by the previous case. Moreover,     \begin{equation}       \zxvczxbcvdfghasdfrtsdafasdfasdfdsfgsdgh T(u) \zxvczxbcvdfghasdfrtsdafasdfasdfdsfgsdgh_{W^{m,p}(B_1)}        \les \zxvczxbcvdfghasdfrtsdafasdfasdfdsfgsdgh f \zxvczxbcvdfghasdfrtsdafasdfasdfdsfgsdgh_{L^p(B_1)}             + \zxvczxbcvdfghasdfrtsdafasdfasdfdsfgsdgh (L(0)-L)u \zxvczxbcvdfghasdfrtsdafasdfasdfdsfgsdgh_{L^p(B_1)}       \les \zxvczxbcvdfghasdfrtsdafasdfasdfdsfgsdgh f \zxvczxbcvdfghasdfrtsdafasdfasdfdsfgsdgh_{L^p(B_1)}             + \eps \zxvczxbcvdfghasdfrtsdafasdfasdfdsfgsdgh u \zxvczxbcvdfghasdfrtsdafasdfasdfdsfgsdgh_{W^{m,p}(B_1)}        ,     \label{8ThswELzXU3X7Ebd1KdZ7v1rN3GiirRXGKWK099ovBM0FDJCvkopYNQ2aN94Z7k0UnUKamE3OjU8DFYFFokbSI2J9V9gVlM8ALWThDPnPu3EL7HPD2VDaZTggzcCCmbvc70qqPcC9mt60ogcrTiA3HEjwTK8ymKeuJMc4q6dVz200XnYUtLR9GYjPXvFOVr6W1zUK1WbPToaWJJuKnxBLnd0ftDEbMmj4loHYyhZyMjM91zQS4p7z8eKa9h0JrbacekcirexG0z4n3xz0QOWSvFj3jLhWXUIU21iIAwJtI3RbWa90I7rzAIqI3UElUJG7tLtUXzw4KQNETvXzqWaujEMenYlNIzLGxgB3AuJEQ52}     \end{equation} where we used \eqref{8ThswELzXU3X7Ebd1KdZ7v1rN3GiirRXGKWK099ovBM0FDJCvkopYNQ2aN94Z7k0UnUKamE3OjU8DFYFFokbSI2J9V9gVlM8ALWThDPnPu3EL7HPD2VDaZTggzcCCmbvc70qqPcC9mt60ogcrTiA3HEjwTK8ymKeuJMc4q6dVz200XnYUtLR9GYjPXvFOVr6W1zUK1WbPToaWJJuKnxBLnd0ftDEbMmj4loHYyhZyMjM91zQS4p7z8eKa9h0JrbacekcirexG0z4n3xz0QOWSvFj3jLhWXUIU21iIAwJtI3RbWa90I7rzAIqI3UElUJG7tLtUXzw4KQNETvXzqWaujEMenYlNIzLGxgB3AuJEQ48} in the last inequality. Letting $C_0$ be the implicit constant in \eqref{8ThswELzXU3X7Ebd1KdZ7v1rN3GiirRXGKWK099ovBM0FDJCvkopYNQ2aN94Z7k0UnUKamE3OjU8DFYFFokbSI2J9V9gVlM8ALWThDPnPu3EL7HPD2VDaZTggzcCCmbvc70qqPcC9mt60ogcrTiA3HEjwTK8ymKeuJMc4q6dVz200XnYUtLR9GYjPXvFOVr6W1zUK1WbPToaWJJuKnxBLnd0ftDEbMmj4loHYyhZyMjM91zQS4p7z8eKa9h0JrbacekcirexG0z4n3xz0QOWSvFj3jLhWXUIU21iIAwJtI3RbWa90I7rzAIqI3UElUJG7tLtUXzw4KQNETvXzqWaujEMenYlNIzLGxgB3AuJEQ52}, we choose $F$, $M_1$, and $\eps$ such that $C_0F = M_1/2$ and $C_0\eps \leq 1/2$. By \eqref{8ThswELzXU3X7Ebd1KdZ7v1rN3GiirRXGKWK099ovBM0FDJCvkopYNQ2aN94Z7k0UnUKamE3OjU8DFYFFokbSI2J9V9gVlM8ALWThDPnPu3EL7HPD2VDaZTggzcCCmbvc70qqPcC9mt60ogcrTiA3HEjwTK8ymKeuJMc4q6dVz200XnYUtLR9GYjPXvFOVr6W1zUK1WbPToaWJJuKnxBLnd0ftDEbMmj4loHYyhZyMjM91zQS4p7z8eKa9h0JrbacekcirexG0z4n3xz0QOWSvFj3jLhWXUIU21iIAwJtI3RbWa90I7rzAIqI3UElUJG7tLtUXzw4KQNETvXzqWaujEMenYlNIzLGxgB3AuJEQ52}, we have $\zxvczxbcvdfghasdfrtsdafasdfasdfdsfgsdgh T(u) \zxvczxbcvdfghasdfrtsdafasdfasdfdsfgsdgh_{W^{m,p}(B_1)} \leq M_1$ for $\zxvczxbcvdfghasdfrtsdafasdfasdfdsfgsdgh u \zxvczxbcvdfghasdfrtsdafasdfasdfdsfgsdgh_{L^p(B_1)} \leq M_1$ and $\zxvczxbcvdfghasdfrtsdafasdfasdfdsfgsdgh f \zxvczxbcvdfghasdfrtsdafasdfasdfdsfgsdgh_{L^p(B_1)} \leq F$. Also,  choosing $M_2 = 2C_0\MM $, for any $r\leq 1$ we may write     \begin{equation}       \sum_{|\nu|\leq m} r^{|\nu|} \zxvczxbcvdfghasdfrtsdafasdfasdfdsfgsdgh \partial^{\nu} T(u) \zxvczxbcvdfghasdfrtsdafasdfasdfdsfgsdgh_{L^p(B_r)}        \leq C_0(\MM + \eps M_2) r^{d+\gamma +n/p}       \leq N r^{d+\gamma + n/p}     \llabel{ JQ W jIR Rgq i5l AP gikl c8ru HnvYFM AI r Ih7 Ths 9tE hA AYgS swZZ fws19P 5w e JvM imb sFH Th CnSZ HORm yt98w3 U3 z ant zAy Twq 0C jgDI Etkb h98V4u o5 2 jjA Zz1 kLo C8 oHGv Z5Ru Gwv3kK 4W B 50T oMt q7Q WG 9mtb SIlc 87ruZf Kw Z Ph3 1ZA Osq 8l jVQJ LTXC gyQn0v KE S iSq Bpa wtH xc IJe4 SiE1 izzxim ke P Y3s 7SX 5DA SG XHqC r38V YP3Hxv OI R ZtM fqN oLF oU 7vNd txzw UkX32t 94 n Fdq qTR QOv Yq Ebig jrSZ kTN7Xw tP F gNs O7M 1mb DA btVB 3LGC pgE9hV FK Y LcS GmF 863 7a ZDiz 4CuJ bLnpE7 yl 8 5jg Many Thanks, POL OG EPOe Mru1 v25XLJ Fz h wgE lnu8ThswELzXU3X7Ebd1KdZ7v1rN3GiirRXGKWK099ovBM0FDJCvkopYNQ2aN94Z7k0UnUKamE3OjU8DFYFFokbSI2J9V9gVlM8ALWThDPnPu3EL7HPD2VDaZTggzcCCmbvc70qqPcC9mt60ogcrTiA3HEjwTK8ymKeuJMc4q6dVz200XnYUtLR9GYjPXvFOVr6W1zUK1WbPToaWJJuKnxBLnd0ftDEbMmj4loHYyhZyMjM91zQS4p7z8eKa9h0JrbacekcirexG0z4n3xz0QOWSvFj3jLhWXUIU21iIAwJtI3RbWa90I7rzAIqI3UElUJG7tLtUXzw4KQNETvXzqWaujEMenYlNIzLGxgB3AuJEQ53}     \end{equation} if $\sum_{|\nu|\leq m} r^{|\nu|} \zxvczxbcvdfghasdfrtsdafasdfasdfdsfgsdgh \partial^{\nu} u \zxvczxbcvdfghasdfrtsdafasdfasdfdsfgsdgh_{L^p(B_r)} \leq M_2 r^{d+\gamma+n/p}$ and $\zxvczxbcvdfghasdfrtsdafasdfasdfdsfgsdgh f \zxvczxbcvdfghasdfrtsdafasdfasdfdsfgsdgh_{L^p(B_r)} \leq \MM r^{d-m+\gamma+n/p}$. Hence, $T(S)\subseteq S$ for $M_1$, $M_2$, and~$\eps$ chosen above.  \par We now claim that $T$ is a contraction.  For any $u$ and $v$ in $S$, we have     \begin{equation}       T(u) - T(v)         = \zxvczxbcvdfghasdfrtsdafasdfasdfdsfgsdgf_{|y|<1} \Gamma(x-y) (L(0)-L)(u-v)(y) \,dy       .     \llabel{wv3kK 4W B 50T oMt q7Q WG 9mtb SIlc 87ruZf Kw Z Ph3 1ZA Osq 8l jVQJ LTXC gyQn0v KE S iSq Bpa wtH xc IJe4 SiE1 izzxim ke P Y3s 7SX 5DA SG XHqC r38V YP3Hxv OI R ZtM fqN oLF oU 7vNd txzw UkX32t 94 n Fdq qTR QOv Yq Ebig jrSZ kTN7Xw tP F gNs O7M 1mb DA btVB 3LGC pgE9hV FK Y LcS GmF 863 7a ZDiz 4CuJ bLnpE7 yl 8 5jg Many Thanks, POL OG EPOe Mru1 v25XLJ Fz h wgE lnu Ymq rX 1YKV Kvgm MK7gI4 6h 5 kZB OoJ tfC 5g VvA1 kNJr 2o7om1 XN p Uwt CWX fFT SW DjsI wuxO JxLU1S xA 5 ObG 3IO UdL qJ cCAr gzKM 08DvX2 mu i 13T t71 Iwq oF UI0E Ef5S V2vxcy SY I QG8ThswELzXU3X7Ebd1KdZ7v1rN3GiirRXGKWK099ovBM0FDJCvkopYNQ2aN94Z7k0UnUKamE3OjU8DFYFFokbSI2J9V9gVlM8ALWThDPnPu3EL7HPD2VDaZTggzcCCmbvc70qqPcC9mt60ogcrTiA3HEjwTK8ymKeuJMc4q6dVz200XnYUtLR9GYjPXvFOVr6W1zUK1WbPToaWJJuKnxBLnd0ftDEbMmj4loHYyhZyMjM91zQS4p7z8eKa9h0JrbacekcirexG0z4n3xz0QOWSvFj3jLhWXUIU21iIAwJtI3RbWa90I7rzAIqI3UElUJG7tLtUXzw4KQNETvXzqWaujEMenYlNIzLGxgB3AuJEQ54}     \end{equation} Therefore, by the global $W^{m,p}$ estimates and expression of $T(u)-T(v)$, we obtain     \begin{equation}       \zxvczxbcvdfghasdfrtsdafasdfasdfdsfgsdgh T(u) - T(v) \zxvczxbcvdfghasdfrtsdafasdfasdfdsfgsdgh_{W^{m,p} (B_1)}        \leq C_0 \eps \zxvczxbcvdfghasdfrtsdafasdfasdfdsfgsdgh u - v \zxvczxbcvdfghasdfrtsdafasdfasdfdsfgsdgh_{W^{m,p}(B_1)}       \leq \frac{1}{2} \zxvczxbcvdfghasdfrtsdafasdfasdfdsfgsdgh u - v \zxvczxbcvdfghasdfrtsdafasdfasdfdsfgsdgh_{W^{m,p}(B_1)}               .          \llabel{xzw UkX32t 94 n Fdq qTR QOv Yq Ebig jrSZ kTN7Xw tP F gNs O7M 1mb DA btVB 3LGC pgE9hV FK Y LcS GmF 863 7a ZDiz 4CuJ bLnpE7 yl 8 5jg Many Thanks, POL OG EPOe Mru1 v25XLJ Fz h wgE lnu Ymq rX 1YKV Kvgm MK7gI4 6h 5 kZB OoJ tfC 5g VvA1 kNJr 2o7om1 XN p Uwt CWX fFT SW DjsI wuxO JxLU1S xA 5 ObG 3IO UdL qJ cCAr gzKM 08DvX2 mu i 13T t71 Iwq oF UI0E Ef5S V2vxcy SY I QGr qrB HID TJ v1OB 1CzD IDdW4E 4j J mv6 Ktx oBO s9 ADWB q218 BJJzRy UQ i 2Gp weE T8L aO 4ho9 5g4v WQmoiq jS w MA9 Cvn Gqx l1 LrYu MjGb oUpuvY Q2 C dBl AB9 7ew jc 5RJE SFGs ORedoM 0b8ThswELzXU3X7Ebd1KdZ7v1rN3GiirRXGKWK099ovBM0FDJCvkopYNQ2aN94Z7k0UnUKamE3OjU8DFYFFokbSI2J9V9gVlM8ALWThDPnPu3EL7HPD2VDaZTggzcCCmbvc70qqPcC9mt60ogcrTiA3HEjwTK8ymKeuJMc4q6dVz200XnYUtLR9GYjPXvFOVr6W1zUK1WbPToaWJJuKnxBLnd0ftDEbMmj4loHYyhZyMjM91zQS4p7z8eKa9h0JrbacekcirexG0z4n3xz0QOWSvFj3jLhWXUIU21iIAwJtI3RbWa90I7rzAIqI3UElUJG7tLtUXzw4KQNETvXzqWaujEMenYlNIzLGxgB3AuJEQ55}     \end{equation} Thus, $T$ is a contraction from $S$ to~$S$. Let $u$ be a fixed point of~$T$. Then $u$ solves \eqref{8ThswELzXU3X7Ebd1KdZ7v1rN3GiirRXGKWK099ovBM0FDJCvkopYNQ2aN94Z7k0UnUKamE3OjU8DFYFFokbSI2J9V9gVlM8ALWThDPnPu3EL7HPD2VDaZTggzcCCmbvc70qqPcC9mt60ogcrTiA3HEjwTK8ymKeuJMc4q6dVz200XnYUtLR9GYjPXvFOVr6W1zUK1WbPToaWJJuKnxBLnd0ftDEbMmj4loHYyhZyMjM91zQS4p7z8eKa9h0JrbacekcirexG0z4n3xz0QOWSvFj3jLhWXUIU21iIAwJtI3RbWa90I7rzAIqI3UElUJG7tLtUXzw4KQNETvXzqWaujEMenYlNIzLGxgB3AuJEQ49} and satisfies \eqref{8ThswELzXU3X7Ebd1KdZ7v1rN3GiirRXGKWK099ovBM0FDJCvkopYNQ2aN94Z7k0UnUKamE3OjU8DFYFFokbSI2J9V9gVlM8ALWThDPnPu3EL7HPD2VDaZTggzcCCmbvc70qqPcC9mt60ogcrTiA3HEjwTK8ymKeuJMc4q6dVz200XnYUtLR9GYjPXvFOVr6W1zUK1WbPToaWJJuKnxBLnd0ftDEbMmj4loHYyhZyMjM91zQS4p7z8eKa9h0JrbacekcirexG0z4n3xz0QOWSvFj3jLhWXUIU21iIAwJtI3RbWa90I7rzAIqI3UElUJG7tLtUXzw4KQNETvXzqWaujEMenYlNIzLGxgB3AuJEQ26}, as desired. \par For the general case, when \eqref{8ThswELzXU3X7Ebd1KdZ7v1rN3GiirRXGKWK099ovBM0FDJCvkopYNQ2aN94Z7k0UnUKamE3OjU8DFYFFokbSI2J9V9gVlM8ALWThDPnPu3EL7HPD2VDaZTggzcCCmbvc70qqPcC9mt60ogcrTiA3HEjwTK8ymKeuJMc4q6dVz200XnYUtLR9GYjPXvFOVr6W1zUK1WbPToaWJJuKnxBLnd0ftDEbMmj4loHYyhZyMjM91zQS4p7z8eKa9h0JrbacekcirexG0z4n3xz0QOWSvFj3jLhWXUIU21iIAwJtI3RbWa90I7rzAIqI3UElUJG7tLtUXzw4KQNETvXzqWaujEMenYlNIzLGxgB3AuJEQ48} is not required,  we consider the transformation~$x \to Rx$. For $R$ sufficiently small (depending on $K$ and~$\alpha$, considered fixed), the condition~\eqref{8ThswELzXU3X7Ebd1KdZ7v1rN3GiirRXGKWK099ovBM0FDJCvkopYNQ2aN94Z7k0UnUKamE3OjU8DFYFFokbSI2J9V9gVlM8ALWThDPnPu3EL7HPD2VDaZTggzcCCmbvc70qqPcC9mt60ogcrTiA3HEjwTK8ymKeuJMc4q6dVz200XnYUtLR9GYjPXvFOVr6W1zUK1WbPToaWJJuKnxBLnd0ftDEbMmj4loHYyhZyMjM91zQS4p7z8eKa9h0JrbacekcirexG0z4n3xz0QOWSvFj3jLhWXUIU21iIAwJtI3RbWa90I7rzAIqI3UElUJG7tLtUXzw4KQNETvXzqWaujEMenYlNIzLGxgB3AuJEQ48} is satisfied so the result of the previous case can be applied.  \end{proof} \par Note that in the proof of Lemma~\ref{L01}, we have also obtained the following statement. \par \cole \begin{Lemma} \label{L10} Assume that $w(x) = \zxvczxbcvdfghasdfrtsdafasdfasdfdsfgsdgf_{|y|<1} \Gamma(x-y) f(y) \, dy$, where \begin{equation} \zxvczxbcvdfghasdfrtsdafasdfasdfdsfgsdgh f \zxvczxbcvdfghasdfrtsdafasdfasdfdsfgsdgh_{L^p(B_r)} \leq M r^{d-m+\gamma+n/p} ,    \llabel{ Ymq rX 1YKV Kvgm MK7gI4 6h 5 kZB OoJ tfC 5g VvA1 kNJr 2o7om1 XN p Uwt CWX fFT SW DjsI wuxO JxLU1S xA 5 ObG 3IO UdL qJ cCAr gzKM 08DvX2 mu i 13T t71 Iwq oF UI0E Ef5S V2vxcy SY I QGr qrB HID TJ v1OB 1CzD IDdW4E 4j J mv6 Ktx oBO s9 ADWB q218 BJJzRy UQ i 2Gp weE T8L aO 4ho9 5g4v WQmoiq jS w MA9 Cvn Gqx l1 LrYu MjGb oUpuvY Q2 C dBl AB9 7ew jc 5RJE SFGs ORedoM 0b B k25 VEK B8V A9 ytAE Oyof G8QIj2 7a I 3jy Rmz yET Kx pgUq 4Bvb cD1b1g KB y oE3 azg elV Nu 8iZ1 w1tq twKx8C LN 2 8yn jdo jUW vN H9qy HaXZ GhjUgm uL I 87i Y7Q 9MQ Wa iFFS Gzt8 4mSQ8ThswELzXU3X7Ebd1KdZ7v1rN3GiirRXGKWK099ovBM0FDJCvkopYNQ2aN94Z7k0UnUKamE3OjU8DFYFFokbSI2J9V9gVlM8ALWThDPnPu3EL7HPD2VDaZTggzcCCmbvc70qqPcC9mt60ogcrTiA3HEjwTK8ymKeuJMc4q6dVz200XnYUtLR9GYjPXvFOVr6W1zUK1WbPToaWJJuKnxBLnd0ftDEbMmj4loHYyhZyMjM91zQS4p7z8eKa9h0JrbacekcirexG0z4n3xz0QOWSvFj3jLhWXUIU21iIAwJtI3RbWa90I7rzAIqI3UElUJG7tLtUXzw4KQNETvXzqWaujEMenYlNIzLGxgB3AuJEQ87} \end{equation}  for some $\gamma \in (0,1)$ and for $p$ satisfying \eqref{8ThswELzXU3X7Ebd1KdZ7v1rN3GiirRXGKWK099ovBM0FDJCvkopYNQ2aN94Z7k0UnUKamE3OjU8DFYFFokbSI2J9V9gVlM8ALWThDPnPu3EL7HPD2VDaZTggzcCCmbvc70qqPcC9mt60ogcrTiA3HEjwTK8ymKeuJMc4q6dVz200XnYUtLR9GYjPXvFOVr6W1zUK1WbPToaWJJuKnxBLnd0ftDEbMmj4loHYyhZyMjM91zQS4p7z8eKa9h0JrbacekcirexG0z4n3xz0QOWSvFj3jLhWXUIU21iIAwJtI3RbWa90I7rzAIqI3UElUJG7tLtUXzw4KQNETvXzqWaujEMenYlNIzLGxgB3AuJEQ13}, and let $P_w(x)$ be its Taylor polynomial of degree~$d$, i.e., \begin{equation} P_w(x) = \sum_{|\beta| \leq d} \frac{x^{\beta}}{\beta!} \zxvczxbcvdfghasdfrtsdafasdfasdfdsfgsdgf_{|y|<1} \partial_x^{\beta} \Gamma(-y) f(y) \, dy .    \llabel{r qrB HID TJ v1OB 1CzD IDdW4E 4j J mv6 Ktx oBO s9 ADWB q218 BJJzRy UQ i 2Gp weE T8L aO 4ho9 5g4v WQmoiq jS w MA9 Cvn Gqx l1 LrYu MjGb oUpuvY Q2 C dBl AB9 7ew jc 5RJE SFGs ORedoM 0b B k25 VEK B8V A9 ytAE Oyof G8QIj2 7a I 3jy Rmz yET Kx pgUq 4Bvb cD1b1g KB y oE3 azg elV Nu 8iZ1 w1tq twKx8C LN 2 8yn jdo jUW vN H9qy HaXZ GhjUgm uL I 87i Y7Q 9MQ Wa iFFS Gzt8 4mSQq2 5O N ltT gbl 8YD QS AzXq pJEK 7bGL1U Jn 0 f59 vPr wdt d6 sDLj Loo1 8tQXf5 5u p mTa dJD sEL pH 2vqY uTAm YzDg95 1P K FP6 pEi zIJ Qd 8Ngn HTND 6z6ExR XV 0 ouU jWT kAK AB eAC9 Rfja8ThswELzXU3X7Ebd1KdZ7v1rN3GiirRXGKWK099ovBM0FDJCvkopYNQ2aN94Z7k0UnUKamE3OjU8DFYFFokbSI2J9V9gVlM8ALWThDPnPu3EL7HPD2VDaZTggzcCCmbvc70qqPcC9mt60ogcrTiA3HEjwTK8ymKeuJMc4q6dVz200XnYUtLR9GYjPXvFOVr6W1zUK1WbPToaWJJuKnxBLnd0ftDEbMmj4loHYyhZyMjM91zQS4p7z8eKa9h0JrbacekcirexG0z4n3xz0QOWSvFj3jLhWXUIU21iIAwJtI3RbWa90I7rzAIqI3UElUJG7tLtUXzw4KQNETvXzqWaujEMenYlNIzLGxgB3AuJEQ103} \end{equation} Then  \begin{equation} \zxvczxbcvdfghasdfrtsdafasdfasdfdsfgsdgh w - P_w \zxvczxbcvdfghasdfrtsdafasdfasdfdsfgsdgh_{L^p(B_r)} \leq Mr^{d+\gamma + n/p} \comma r \in (0, 1/2] .    \llabel{ B k25 VEK B8V A9 ytAE Oyof G8QIj2 7a I 3jy Rmz yET Kx pgUq 4Bvb cD1b1g KB y oE3 azg elV Nu 8iZ1 w1tq twKx8C LN 2 8yn jdo jUW vN H9qy HaXZ GhjUgm uL I 87i Y7Q 9MQ Wa iFFS Gzt8 4mSQq2 5O N ltT gbl 8YD QS AzXq pJEK 7bGL1U Jn 0 f59 vPr wdt d6 sDLj Loo1 8tQXf5 5u p mTa dJD sEL pH 2vqY uTAm YzDg95 1P K FP6 pEi zIJ Qd 8Ngn HTND 6z6ExR XV 0 ouU jWT kAK AB eAC9 Rfja c43Ajk Xn H dgS y3v 5cB et s3VX qfpP BqiGf9 0a w g4d W9U kvR iJ y46G bH3U cJ86hW Va C Mje dsU cqD SZ 1DlP 2mfB hzu5dv u1 i 6eW 2YN LhM 3f WOdz KS6Q ov14wx YY d 8sa S38 hIl cP tS4l8ThswELzXU3X7Ebd1KdZ7v1rN3GiirRXGKWK099ovBM0FDJCvkopYNQ2aN94Z7k0UnUKamE3OjU8DFYFFokbSI2J9V9gVlM8ALWThDPnPu3EL7HPD2VDaZTggzcCCmbvc70qqPcC9mt60ogcrTiA3HEjwTK8ymKeuJMc4q6dVz200XnYUtLR9GYjPXvFOVr6W1zUK1WbPToaWJJuKnxBLnd0ftDEbMmj4loHYyhZyMjM91zQS4p7z8eKa9h0JrbacekcirexG0z4n3xz0QOWSvFj3jLhWXUIU21iIAwJtI3RbWa90I7rzAIqI3UElUJG7tLtUXzw4KQNETvXzqWaujEMenYlNIzLGxgB3AuJEQ104} \end{equation} \end{Lemma} \colb \par \subsection{Extended range for the interior existence lemma} The interior $W^{m,p}$ estimate may be extended to the $L^q$ norm for $1 \leq q < np/(n-mp)$ as follows. \par \cole \begin{Lemma} \label{L03} Assume that $L$ in \eqref{8ThswELzXU3X7Ebd1KdZ7v1rN3GiirRXGKWK099ovBM0FDJCvkopYNQ2aN94Z7k0UnUKamE3OjU8DFYFFokbSI2J9V9gVlM8ALWThDPnPu3EL7HPD2VDaZTggzcCCmbvc70qqPcC9mt60ogcrTiA3HEjwTK8ymKeuJMc4q6dVz200XnYUtLR9GYjPXvFOVr6W1zUK1WbPToaWJJuKnxBLnd0ftDEbMmj4loHYyhZyMjM91zQS4p7z8eKa9h0JrbacekcirexG0z4n3xz0QOWSvFj3jLhWXUIU21iIAwJtI3RbWa90I7rzAIqI3UElUJG7tLtUXzw4KQNETvXzqWaujEMenYlNIzLGxgB3AuJEQ08} satisfies the conditions~\eqref{8ThswELzXU3X7Ebd1KdZ7v1rN3GiirRXGKWK099ovBM0FDJCvkopYNQ2aN94Z7k0UnUKamE3OjU8DFYFFokbSI2J9V9gVlM8ALWThDPnPu3EL7HPD2VDaZTggzcCCmbvc70qqPcC9mt60ogcrTiA3HEjwTK8ymKeuJMc4q6dVz200XnYUtLR9GYjPXvFOVr6W1zUK1WbPToaWJJuKnxBLnd0ftDEbMmj4loHYyhZyMjM91zQS4p7z8eKa9h0JrbacekcirexG0z4n3xz0QOWSvFj3jLhWXUIU21iIAwJtI3RbWa90I7rzAIqI3UElUJG7tLtUXzw4KQNETvXzqWaujEMenYlNIzLGxgB3AuJEQ09}--\eqref{8ThswELzXU3X7Ebd1KdZ7v1rN3GiirRXGKWK099ovBM0FDJCvkopYNQ2aN94Z7k0UnUKamE3OjU8DFYFFokbSI2J9V9gVlM8ALWThDPnPu3EL7HPD2VDaZTggzcCCmbvc70qqPcC9mt60ogcrTiA3HEjwTK8ymKeuJMc4q6dVz200XnYUtLR9GYjPXvFOVr6W1zUK1WbPToaWJJuKnxBLnd0ftDEbMmj4loHYyhZyMjM91zQS4p7z8eKa9h0JrbacekcirexG0z4n3xz0QOWSvFj3jLhWXUIU21iIAwJtI3RbWa90I7rzAIqI3UElUJG7tLtUXzw4KQNETvXzqWaujEMenYlNIzLGxgB3AuJEQ11}.  Suppose that $f \in L^p(B_1)$, where $p$ satisfies \eqref{8ThswELzXU3X7Ebd1KdZ7v1rN3GiirRXGKWK099ovBM0FDJCvkopYNQ2aN94Z7k0UnUKamE3OjU8DFYFFokbSI2J9V9gVlM8ALWThDPnPu3EL7HPD2VDaZTggzcCCmbvc70qqPcC9mt60ogcrTiA3HEjwTK8ymKeuJMc4q6dVz200XnYUtLR9GYjPXvFOVr6W1zUK1WbPToaWJJuKnxBLnd0ftDEbMmj4loHYyhZyMjM91zQS4p7z8eKa9h0JrbacekcirexG0z4n3xz0QOWSvFj3jLhWXUIU21iIAwJtI3RbWa90I7rzAIqI3UElUJG7tLtUXzw4KQNETvXzqWaujEMenYlNIzLGxgB3AuJEQ13}, and that there exist $\gamma \in (0,1)$ and $d \geq m$ such that $\zxvczxbcvdfghasdfrtsdafasdfasdfdsfgsdgh f \zxvczxbcvdfghasdfrtsdafasdfasdfdsfgsdgh_{L^p(B_r)} \leq \MM r^{d-m+\gamma+n/p}$, for all~$r\leq 1$. Then there exists a positive constant $R$, depending on $\alpha$, and a solution $u \in W^{m,q}(B_R)$ of     \begin{equation}
      Lu = f       \inin{B_R}     \label{8ThswELzXU3X7Ebd1KdZ7v1rN3GiirRXGKWK099ovBM0FDJCvkopYNQ2aN94Z7k0UnUKamE3OjU8DFYFFokbSI2J9V9gVlM8ALWThDPnPu3EL7HPD2VDaZTggzcCCmbvc70qqPcC9mt60ogcrTiA3HEjwTK8ymKeuJMc4q6dVz200XnYUtLR9GYjPXvFOVr6W1zUK1WbPToaWJJuKnxBLnd0ftDEbMmj4loHYyhZyMjM91zQS4p7z8eKa9h0JrbacekcirexG0z4n3xz0QOWSvFj3jLhWXUIU21iIAwJtI3RbWa90I7rzAIqI3UElUJG7tLtUXzw4KQNETvXzqWaujEMenYlNIzLGxgB3AuJEQ56}     \end{equation} such that      \begin{equation}       \sum_{|\nu|\leq m} r^{|\nu|} \zxvczxbcvdfghasdfrtsdafasdfasdfdsfgsdgh \partial^{\nu} u \zxvczxbcvdfghasdfrtsdafasdfasdfdsfgsdgh_{L^q(B_r)}       \les \MM r^{d+\gamma +n/q}       \comma r \leq R       ,     \label{8ThswELzXU3X7Ebd1KdZ7v1rN3GiirRXGKWK099ovBM0FDJCvkopYNQ2aN94Z7k0UnUKamE3OjU8DFYFFokbSI2J9V9gVlM8ALWThDPnPu3EL7HPD2VDaZTggzcCCmbvc70qqPcC9mt60ogcrTiA3HEjwTK8ymKeuJMc4q6dVz200XnYUtLR9GYjPXvFOVr6W1zUK1WbPToaWJJuKnxBLnd0ftDEbMmj4loHYyhZyMjM91zQS4p7z8eKa9h0JrbacekcirexG0z4n3xz0QOWSvFj3jLhWXUIU21iIAwJtI3RbWa90I7rzAIqI3UElUJG7tLtUXzw4KQNETvXzqWaujEMenYlNIzLGxgB3AuJEQ57}     \end{equation} for     \begin{equation}     1\leq q < \frac{np}{n-mp}     ,     \label{8ThswELzXU3X7Ebd1KdZ7v1rN3GiirRXGKWK099ovBM0FDJCvkopYNQ2aN94Z7k0UnUKamE3OjU8DFYFFokbSI2J9V9gVlM8ALWThDPnPu3EL7HPD2VDaZTggzcCCmbvc70qqPcC9mt60ogcrTiA3HEjwTK8ymKeuJMc4q6dVz200XnYUtLR9GYjPXvFOVr6W1zUK1WbPToaWJJuKnxBLnd0ftDEbMmj4loHYyhZyMjM91zQS4p7z8eKa9h0JrbacekcirexG0z4n3xz0QOWSvFj3jLhWXUIU21iIAwJtI3RbWa90I7rzAIqI3UElUJG7tLtUXzw4KQNETvXzqWaujEMenYlNIzLGxgB3AuJEQ58}     \end{equation} where the constant in \eqref{8ThswELzXU3X7Ebd1KdZ7v1rN3GiirRXGKWK099ovBM0FDJCvkopYNQ2aN94Z7k0UnUKamE3OjU8DFYFFokbSI2J9V9gVlM8ALWThDPnPu3EL7HPD2VDaZTggzcCCmbvc70qqPcC9mt60ogcrTiA3HEjwTK8ymKeuJMc4q6dVz200XnYUtLR9GYjPXvFOVr6W1zUK1WbPToaWJJuKnxBLnd0ftDEbMmj4loHYyhZyMjM91zQS4p7z8eKa9h0JrbacekcirexG0z4n3xz0QOWSvFj3jLhWXUIU21iIAwJtI3RbWa90I7rzAIqI3UElUJG7tLtUXzw4KQNETvXzqWaujEMenYlNIzLGxgB3AuJEQ57} depends on $p$, $q$, $d$, and~$\alpha$. \end{Lemma} \colb \par \begin{proof}[Proof of Lemma~\ref{L03}] The proof is similar to that of Lemma~\ref{L01}, except for the estimates of $I_1$, $I_2$, and~$I_3$, defined in~\eqref{8ThswELzXU3X7Ebd1KdZ7v1rN3GiirRXGKWK099ovBM0FDJCvkopYNQ2aN94Z7k0UnUKamE3OjU8DFYFFokbSI2J9V9gVlM8ALWThDPnPu3EL7HPD2VDaZTggzcCCmbvc70qqPcC9mt60ogcrTiA3HEjwTK8ymKeuJMc4q6dVz200XnYUtLR9GYjPXvFOVr6W1zUK1WbPToaWJJuKnxBLnd0ftDEbMmj4loHYyhZyMjM91zQS4p7z8eKa9h0JrbacekcirexG0z4n3xz0QOWSvFj3jLhWXUIU21iIAwJtI3RbWa90I7rzAIqI3UElUJG7tLtUXzw4KQNETvXzqWaujEMenYlNIzLGxgB3AuJEQ39}, with $L^{p}(B_{r})$ replaced by~$L^{q}(B_{r})$. For $I_1$, we have, assuming $r\leq 1/2$,     \begin{equation}       I_1        = \biggl \zxvczxbcvdfghasdfrtsdafasdfasdfdsfgsdgh \zxvczxbcvdfghasdfrtsdafasdfasdfdsfgsdgf_{|y|<2r} \Gamma (x-y) f(y) \,dy \biggr \zxvczxbcvdfghasdfrtsdafasdfasdfdsfgsdgh_{L^q(B_r)}       \les \zxvczxbcvdfghasdfrtsdafasdfasdfdsfgsdgh \Gamma \zxvczxbcvdfghasdfrtsdafasdfasdfdsfgsdgh_{L^{\tilde p} (B_{3r})}             \zxvczxbcvdfghasdfrtsdafasdfasdfdsfgsdgh f \zxvczxbcvdfghasdfrtsdafasdfasdfdsfgsdgh_{L^p(B_{2r})}       \les \biggl( \zxvczxbcvdfghasdfrtsdafasdfasdfdsfgsdgf_0^{3r} s^{(m-n)\tilde p+n-1} \,ds \biggr)^{1/\tilde p}              \zxvczxbcvdfghasdfrtsdafasdfasdfdsfgsdgh f \zxvczxbcvdfghasdfrtsdafasdfasdfdsfgsdgh_{L^p(B_{2r})}       ,     \llabel{q2 5O N ltT gbl 8YD QS AzXq pJEK 7bGL1U Jn 0 f59 vPr wdt d6 sDLj Loo1 8tQXf5 5u p mTa dJD sEL pH 2vqY uTAm YzDg95 1P K FP6 pEi zIJ Qd 8Ngn HTND 6z6ExR XV 0 ouU jWT kAK AB eAC9 Rfja c43Ajk Xn H dgS y3v 5cB et s3VX qfpP BqiGf9 0a w g4d W9U kvR iJ y46G bH3U cJ86hW Va C Mje dsU cqD SZ 1DlP 2mfB hzu5dv u1 i 6eW 2YN LhM 3f WOdz KS6Q ov14wx YY d 8sa S38 hIl cP tS4l 9B7h FC3JXJ Gp s tll 7a7 WNr VM wunm nmDc 5duVpZ xT C l8F I01 jhn 5B l4Jz aEV7 CKMThL ji 1 gyZ uXc Iv4 03 3NqZ LITG Ux3ClP CB K O3v RUi mJq l5 blI9 GrWy irWHof lH 7 3ZT eZX kop eq8ThswELzXU3X7Ebd1KdZ7v1rN3GiirRXGKWK099ovBM0FDJCvkopYNQ2aN94Z7k0UnUKamE3OjU8DFYFFokbSI2J9V9gVlM8ALWThDPnPu3EL7HPD2VDaZTggzcCCmbvc70qqPcC9mt60ogcrTiA3HEjwTK8ymKeuJMc4q6dVz200XnYUtLR9GYjPXvFOVr6W1zUK1WbPToaWJJuKnxBLnd0ftDEbMmj4loHYyhZyMjM91zQS4p7z8eKa9h0JrbacekcirexG0z4n3xz0QOWSvFj3jLhWXUIU21iIAwJtI3RbWa90I7rzAIqI3UElUJG7tLtUXzw4KQNETvXzqWaujEMenYlNIzLGxgB3AuJEQ59}     \end{equation}  where $1/\tilde p+1/p = 1+1/q$. Since $m-n+n/\tilde p > 0$ by \eqref{8ThswELzXU3X7Ebd1KdZ7v1rN3GiirRXGKWK099ovBM0FDJCvkopYNQ2aN94Z7k0UnUKamE3OjU8DFYFFokbSI2J9V9gVlM8ALWThDPnPu3EL7HPD2VDaZTggzcCCmbvc70qqPcC9mt60ogcrTiA3HEjwTK8ymKeuJMc4q6dVz200XnYUtLR9GYjPXvFOVr6W1zUK1WbPToaWJJuKnxBLnd0ftDEbMmj4loHYyhZyMjM91zQS4p7z8eKa9h0JrbacekcirexG0z4n3xz0QOWSvFj3jLhWXUIU21iIAwJtI3RbWa90I7rzAIqI3UElUJG7tLtUXzw4KQNETvXzqWaujEMenYlNIzLGxgB3AuJEQ58}, we obtain     \begin{equation}       I_1       \les \MM r^{m-n+n/\tilde p} r^{d-m+\gamma+n/p}       \les \MM r^{d+\gamma+n/q}       ,     \llabel{ c43Ajk Xn H dgS y3v 5cB et s3VX qfpP BqiGf9 0a w g4d W9U kvR iJ y46G bH3U cJ86hW Va C Mje dsU cqD SZ 1DlP 2mfB hzu5dv u1 i 6eW 2YN LhM 3f WOdz KS6Q ov14wx YY d 8sa S38 hIl cP tS4l 9B7h FC3JXJ Gp s tll 7a7 WNr VM wunm nmDc 5duVpZ xT C l8F I01 jhn 5B l4Jz aEV7 CKMThL ji 1 gyZ uXc Iv4 03 3NqZ LITG Ux3ClP CB K O3v RUi mJq l5 blI9 GrWy irWHof lH 7 3ZT eZX kop eq 8XL1 RQ3a Uj6Ess nj 2 0MA 3As rSV ft 3F9w zB1q DQVOnH Cm m P3d WSb jst oj 3oGj advz qcMB6Y 6k D 9sZ 0bd Mjt UT hULG TWU9 Nmr3E4 CN b zUO vTh hqL 1p xAxT ezrH dVMgLY TT r Sfx LUX C8ThswELzXU3X7Ebd1KdZ7v1rN3GiirRXGKWK099ovBM0FDJCvkopYNQ2aN94Z7k0UnUKamE3OjU8DFYFFokbSI2J9V9gVlM8ALWThDPnPu3EL7HPD2VDaZTggzcCCmbvc70qqPcC9mt60ogcrTiA3HEjwTK8ymKeuJMc4q6dVz200XnYUtLR9GYjPXvFOVr6W1zUK1WbPToaWJJuKnxBLnd0ftDEbMmj4loHYyhZyMjM91zQS4p7z8eKa9h0JrbacekcirexG0z4n3xz0QOWSvFj3jLhWXUIU21iIAwJtI3RbWa90I7rzAIqI3UElUJG7tLtUXzw4KQNETvXzqWaujEMenYlNIzLGxgB3AuJEQ60}     \end{equation} where we used $1/\tilde p+1/p=1+1/q$ in the last inequality.  Since $\bigl| \zxvczxbcvdfghasdfrtsdafasdfasdfdsfgsdgf_{|y|<2r} \sum_{k=0}^d  \partial^k_x \Gamma(-y) f(y) \,dy \bigr| \les \MM r^{d+\gamma}$ for all $r\leq1/2$, as shown \eqref{8ThswELzXU3X7Ebd1KdZ7v1rN3GiirRXGKWK099ovBM0FDJCvkopYNQ2aN94Z7k0UnUKamE3OjU8DFYFFokbSI2J9V9gVlM8ALWThDPnPu3EL7HPD2VDaZTggzcCCmbvc70qqPcC9mt60ogcrTiA3HEjwTK8ymKeuJMc4q6dVz200XnYUtLR9GYjPXvFOVr6W1zUK1WbPToaWJJuKnxBLnd0ftDEbMmj4loHYyhZyMjM91zQS4p7z8eKa9h0JrbacekcirexG0z4n3xz0QOWSvFj3jLhWXUIU21iIAwJtI3RbWa90I7rzAIqI3UElUJG7tLtUXzw4KQNETvXzqWaujEMenYlNIzLGxgB3AuJEQ42},  we have     \begin{equation}       I_2        = \biggl( \zxvczxbcvdfghasdfrtsdafasdfasdfdsfgsdgf_{|x|<r}                    \biggl( \zxvczxbcvdfghasdfrtsdafasdfasdfdsfgsdgf_{|y|<2r} \sum_{k=0}^d                           \partial^k_x \Gamma(-y) f(y) \,dy                    \biggr)^q \,dx           \biggr)^{1/q}       \les \biggl(  \zxvczxbcvdfghasdfrtsdafasdfasdfdsfgsdgf_{|x|<r}                         ( \MM r^{d+\gamma})^q \,dx               \biggr)^{1/q}       \les \MM r^{d+\gamma+n/q}       .    \llabel{ 9B7h FC3JXJ Gp s tll 7a7 WNr VM wunm nmDc 5duVpZ xT C l8F I01 jhn 5B l4Jz aEV7 CKMThL ji 1 gyZ uXc Iv4 03 3NqZ LITG Ux3ClP CB K O3v RUi mJq l5 blI9 GrWy irWHof lH 7 3ZT eZX kop eq 8XL1 RQ3a Uj6Ess nj 2 0MA 3As rSV ft 3F9w zB1q DQVOnH Cm m P3d WSb jst oj 3oGj advz qcMB6Y 6k D 9sZ 0bd Mjt UT hULG TWU9 Nmr3E4 CN b zUO vTh hqL 1p xAxT ezrH dVMgLY TT r Sfx LUX CMr WA bE69 K6XH i5re1f x4 G DKk iB7 f2D Xz Xez2 k2Yc Yc4QjU yM Y R1o DeY NWf 74 hByF dsWk 4cUbCR DX a q4e DWd 7qb Ot 7GOu oklg jJ00J9 Il O Jxn tzF VBC Ft pABp VLEE 2y5Qcg b3 5 DU4 8ThswELzXU3X7Ebd1KdZ7v1rN3GiirRXGKWK099ovBM0FDJCvkopYNQ2aN94Z7k0UnUKamE3OjU8DFYFFokbSI2J9V9gVlM8ALWThDPnPu3EL7HPD2VDaZTggzcCCmbvc70qqPcC9mt60ogcrTiA3HEjwTK8ymKeuJMc4q6dVz200XnYUtLR9GYjPXvFOVr6W1zUK1WbPToaWJJuKnxBLnd0ftDEbMmj4loHYyhZyMjM91zQS4p7z8eKa9h0JrbacekcirexG0z4n3xz0QOWSvFj3jLhWXUIU21iIAwJtI3RbWa90I7rzAIqI3UElUJG7tLtUXzw4KQNETvXzqWaujEMenYlNIzLGxgB3AuJEQ61}     \end{equation} Similarly, with $I$ defined in \eqref{8ThswELzXU3X7Ebd1KdZ7v1rN3GiirRXGKWK099ovBM0FDJCvkopYNQ2aN94Z7k0UnUKamE3OjU8DFYFFokbSI2J9V9gVlM8ALWThDPnPu3EL7HPD2VDaZTggzcCCmbvc70qqPcC9mt60ogcrTiA3HEjwTK8ymKeuJMc4q6dVz200XnYUtLR9GYjPXvFOVr6W1zUK1WbPToaWJJuKnxBLnd0ftDEbMmj4loHYyhZyMjM91zQS4p7z8eKa9h0JrbacekcirexG0z4n3xz0QOWSvFj3jLhWXUIU21iIAwJtI3RbWa90I7rzAIqI3UElUJG7tLtUXzw4KQNETvXzqWaujEMenYlNIzLGxgB3AuJEQ44},     \begin{equation}       I_3        = \biggl( \zxvczxbcvdfghasdfrtsdafasdfasdfdsfgsdgf_{|x|<r} I(x)^q \,dx \biggr)^{1/q}       \les \MM r^{d+\gamma+n/q}    \llabel{ 8XL1 RQ3a Uj6Ess nj 2 0MA 3As rSV ft 3F9w zB1q DQVOnH Cm m P3d WSb jst oj 3oGj advz qcMB6Y 6k D 9sZ 0bd Mjt UT hULG TWU9 Nmr3E4 CN b zUO vTh hqL 1p xAxT ezrH dVMgLY TT r Sfx LUX CMr WA bE69 K6XH i5re1f x4 G DKk iB7 f2D Xz Xez2 k2Yc Yc4QjU yM Y R1o DeY NWf 74 hByF dsWk 4cUbCR DX a q4e DWd 7qb Ot 7GOu oklg jJ00J9 Il O Jxn tzF VBC Ft pABp VLEE 2y5Qcg b3 5 DU4 igj 4dz zW soNF wvqj bNFma0 am F Kiv Aap pzM zr VqYf OulM HafaBk 6J r eOQ BaT EsJ BB tHXj n2EU CNleWp cv W JIg gWX Ksn B3 wvmo WK49 Nl492o gR 6 fvc 8ff jJm sW Jr0j zI9p CBsIUV of D8ThswELzXU3X7Ebd1KdZ7v1rN3GiirRXGKWK099ovBM0FDJCvkopYNQ2aN94Z7k0UnUKamE3OjU8DFYFFokbSI2J9V9gVlM8ALWThDPnPu3EL7HPD2VDaZTggzcCCmbvc70qqPcC9mt60ogcrTiA3HEjwTK8ymKeuJMc4q6dVz200XnYUtLR9GYjPXvFOVr6W1zUK1WbPToaWJJuKnxBLnd0ftDEbMmj4loHYyhZyMjM91zQS4p7z8eKa9h0JrbacekcirexG0z4n3xz0QOWSvFj3jLhWXUIU21iIAwJtI3RbWa90I7rzAIqI3UElUJG7tLtUXzw4KQNETvXzqWaujEMenYlNIzLGxgB3AuJEQ62}     \end{equation} since $I(x) \les \MM r^{d+\gamma}$, as proven in~\eqref{8ThswELzXU3X7Ebd1KdZ7v1rN3GiirRXGKWK099ovBM0FDJCvkopYNQ2aN94Z7k0UnUKamE3OjU8DFYFFokbSI2J9V9gVlM8ALWThDPnPu3EL7HPD2VDaZTggzcCCmbvc70qqPcC9mt60ogcrTiA3HEjwTK8ymKeuJMc4q6dVz200XnYUtLR9GYjPXvFOVr6W1zUK1WbPToaWJJuKnxBLnd0ftDEbMmj4loHYyhZyMjM91zQS4p7z8eKa9h0JrbacekcirexG0z4n3xz0QOWSvFj3jLhWXUIU21iIAwJtI3RbWa90I7rzAIqI3UElUJG7tLtUXzw4KQNETvXzqWaujEMenYlNIzLGxgB3AuJEQ45}. \end{proof} \par \subsection{The bootstrapping argument} Before proving Theorem~\ref{T01}, we demonstrate the following bootstrapping argument, which works under the assumptions of Theorem~\ref{T01}. \par \cole \begin{Lemma} \label{L04} Under the assumptions of Theorem~\ref{T01},  suppose that there exists $k\in\mathbb{N}_0$  such that $(k+1) \alpha + \eta < 1$. If $u$ vanishes of $L^{p}$-order $d-1+\eta+k\alpha$, i.e.,     \begin{equation}           c_k        = \sup_{r \leq 1}           \frac{\zxvczxbcvdfghasdfrtsdafasdfasdfdsfgsdgh u \zxvczxbcvdfghasdfrtsdafasdfasdfdsfgsdgh_{L^p(B_r)}}                 {r^{d-1+\eta+k\alpha+n/p}}       <\infty       ,     \label{8ThswELzXU3X7Ebd1KdZ7v1rN3GiirRXGKWK099ovBM0FDJCvkopYNQ2aN94Z7k0UnUKamE3OjU8DFYFFokbSI2J9V9gVlM8ALWThDPnPu3EL7HPD2VDaZTggzcCCmbvc70qqPcC9mt60ogcrTiA3HEjwTK8ymKeuJMc4q6dVz200XnYUtLR9GYjPXvFOVr6W1zUK1WbPToaWJJuKnxBLnd0ftDEbMmj4loHYyhZyMjM91zQS4p7z8eKa9h0JrbacekcirexG0z4n3xz0QOWSvFj3jLhWXUIU21iIAwJtI3RbWa90I7rzAIqI3UElUJG7tLtUXzw4KQNETvXzqWaujEMenYlNIzLGxgB3AuJEQ65}     \end{equation} then $u$ vanishes of $L^{p}$-order $d-1+\eta+(k+1)\alpha$, and moreover     \begin{equation}      c_{k+1}=      \sup_{r\leq1} \frac{\zxvczxbcvdfghasdfrtsdafasdfasdfdsfgsdgh u \zxvczxbcvdfghasdfrtsdafasdfasdfdsfgsdgh_{L^p(B_r)}}{r^{d-1+\eta+(k+1)\alpha+n/p}}      \les c_k + M   + \zxvczxbcvdfghasdfrtsdafasdfasdfdsfgsdgh Q \zxvczxbcvdfghasdfrtsdafasdfasdfdsfgsdgh_{L^p(B_1)}      < \infty     \label{8ThswELzXU3X7Ebd1KdZ7v1rN3GiirRXGKWK099ovBM0FDJCvkopYNQ2aN94Z7k0UnUKamE3OjU8DFYFFokbSI2J9V9gVlM8ALWThDPnPu3EL7HPD2VDaZTggzcCCmbvc70qqPcC9mt60ogcrTiA3HEjwTK8ymKeuJMc4q6dVz200XnYUtLR9GYjPXvFOVr6W1zUK1WbPToaWJJuKnxBLnd0ftDEbMmj4loHYyhZyMjM91zQS4p7z8eKa9h0JrbacekcirexG0z4n3xz0QOWSvFj3jLhWXUIU21iIAwJtI3RbWa90I7rzAIqI3UElUJG7tLtUXzw4KQNETvXzqWaujEMenYlNIzLGxgB3AuJEQ64}     \end{equation} holds, where $M$ and $Q\in\dPP_{d-m}$ are defined in Theorem~\ref{T01}, and the constant depends on $d$, $p$, and~$q$. \end{Lemma} \colb \par Recall that all constants are allowed to depend  on $m$, $n$, $\alpha$, and $K$ (which are considered fixed). \par \begin{proof}[Proof of Lemma~\ref{L04}] Assuming that \eqref{8ThswELzXU3X7Ebd1KdZ7v1rN3GiirRXGKWK099ovBM0FDJCvkopYNQ2aN94Z7k0UnUKamE3OjU8DFYFFokbSI2J9V9gVlM8ALWThDPnPu3EL7HPD2VDaZTggzcCCmbvc70qqPcC9mt60ogcrTiA3HEjwTK8ymKeuJMc4q6dVz200XnYUtLR9GYjPXvFOVr6W1zUK1WbPToaWJJuKnxBLnd0ftDEbMmj4loHYyhZyMjM91zQS4p7z8eKa9h0JrbacekcirexG0z4n3xz0QOWSvFj3jLhWXUIU21iIAwJtI3RbWa90I7rzAIqI3UElUJG7tLtUXzw4KQNETvXzqWaujEMenYlNIzLGxgB3AuJEQ65} holds, we first show that     \begin{equation}       c_{k+1}        = \sup_{r \leq 1}           \frac{\zxvczxbcvdfghasdfrtsdafasdfasdfdsfgsdgh u \zxvczxbcvdfghasdfrtsdafasdfasdfdsfgsdgh_{L^p(B_r)}}                 {r^{d-1+\eta+(k+1)\alpha+n/p}}       < \infty       .     \llabel{Mr WA bE69 K6XH i5re1f x4 G DKk iB7 f2D Xz Xez2 k2Yc Yc4QjU yM Y R1o DeY NWf 74 hByF dsWk 4cUbCR DX a q4e DWd 7qb Ot 7GOu oklg jJ00J9 Il O Jxn tzF VBC Ft pABp VLEE 2y5Qcg b3 5 DU4 igj 4dz zW soNF wvqj bNFma0 am F Kiv Aap pzM zr VqYf OulM HafaBk 6J r eOQ BaT EsJ BB tHXj n2EU CNleWp cv W JIg gWX Ksn B3 wvmo WK49 Nl492o gR 6 fvc 8ff jJm sW Jr0j zI9p CBsIUV of D kKH Ub7 vxp uQ UXA6 hMUr yvxEpc Tq l Tkz z0q HbX pO 8jFu h6nw zVPPzp A8 9 61V 78c O2W aw 0yGn CHVq BVjTUH lk p 6dG HOd voE E8 cw7Q DL1o 1qg5TX qo V 720 hhQ TyF tp TJDg 9E8D nsp1Qi8ThswELzXU3X7Ebd1KdZ7v1rN3GiirRXGKWK099ovBM0FDJCvkopYNQ2aN94Z7k0UnUKamE3OjU8DFYFFokbSI2J9V9gVlM8ALWThDPnPu3EL7HPD2VDaZTggzcCCmbvc70qqPcC9mt60ogcrTiA3HEjwTK8ymKeuJMc4q6dVz200XnYUtLR9GYjPXvFOVr6W1zUK1WbPToaWJJuKnxBLnd0ftDEbMmj4loHYyhZyMjM91zQS4p7z8eKa9h0JrbacekcirexG0z4n3xz0QOWSvFj3jLhWXUIU21iIAwJtI3RbWa90I7rzAIqI3UElUJG7tLtUXzw4KQNETvXzqWaujEMenYlNIzLGxgB3AuJEQ66}     \end{equation} Applying the classical elliptic interior estimate  to \eqref{8ThswELzXU3X7Ebd1KdZ7v1rN3GiirRXGKWK099ovBM0FDJCvkopYNQ2aN94Z7k0UnUKamE3OjU8DFYFFokbSI2J9V9gVlM8ALWThDPnPu3EL7HPD2VDaZTggzcCCmbvc70qqPcC9mt60ogcrTiA3HEjwTK8ymKeuJMc4q6dVz200XnYUtLR9GYjPXvFOVr6W1zUK1WbPToaWJJuKnxBLnd0ftDEbMmj4loHYyhZyMjM91zQS4p7z8eKa9h0JrbacekcirexG0z4n3xz0QOWSvFj3jLhWXUIU21iIAwJtI3RbWa90I7rzAIqI3UElUJG7tLtUXzw4KQNETvXzqWaujEMenYlNIzLGxgB3AuJEQ56}, we have     \begin{align}     \begin{split}       \sum_{|\nu|\leq m} r^{|\nu|} \zxvczxbcvdfghasdfrtsdafasdfasdfdsfgsdgh \partial^\nu u \zxvczxbcvdfghasdfrtsdafasdfasdfdsfgsdgh_{L^p(B_r)}       &\les \zxvczxbcvdfghasdfrtsdafasdfasdfdsfgsdgh u \zxvczxbcvdfghasdfrtsdafasdfasdfdsfgsdgh_{L^p(B_{2r})}              + r^m \zxvczxbcvdfghasdfrtsdafasdfasdfdsfgsdgh f \zxvczxbcvdfghasdfrtsdafasdfasdfdsfgsdgh_{L^p(B_{2r})}       \les \zxvczxbcvdfghasdfrtsdafasdfasdfdsfgsdgh u \zxvczxbcvdfghasdfrtsdafasdfasdfdsfgsdgh_{L^p(B_{2r})}              + r^m \zxvczxbcvdfghasdfrtsdafasdfasdfdsfgsdgh f - Q \zxvczxbcvdfghasdfrtsdafasdfasdfdsfgsdgh_{L^p(B_{2r})}             + r^m \zxvczxbcvdfghasdfrtsdafasdfasdfdsfgsdgh Q \zxvczxbcvdfghasdfrtsdafasdfasdfdsfgsdgh_{L^p(B_{2r})}       \\&       \les \big(c_k + \MM   +  \zxvczxbcvdfghasdfrtsdafasdfasdfdsfgsdgh Q \zxvczxbcvdfghasdfrtsdafasdfasdfdsfgsdgh_{L^p(B_1)} \big)r^{d-1+\eta+k\alpha+n/p}       \comma       r \in (0, 1/2 ]       ,     \end{split}     \label{8ThswELzXU3X7Ebd1KdZ7v1rN3GiirRXGKWK099ovBM0FDJCvkopYNQ2aN94Z7k0UnUKamE3OjU8DFYFFokbSI2J9V9gVlM8ALWThDPnPu3EL7HPD2VDaZTggzcCCmbvc70qqPcC9mt60ogcrTiA3HEjwTK8ymKeuJMc4q6dVz200XnYUtLR9GYjPXvFOVr6W1zUK1WbPToaWJJuKnxBLnd0ftDEbMmj4loHYyhZyMjM91zQS4p7z8eKa9h0JrbacekcirexG0z4n3xz0QOWSvFj3jLhWXUIU21iIAwJtI3RbWa90I7rzAIqI3UElUJG7tLtUXzw4KQNETvXzqWaujEMenYlNIzLGxgB3AuJEQ68}     \end{align} where we used \eqref{8ThswELzXU3X7Ebd1KdZ7v1rN3GiirRXGKWK099ovBM0FDJCvkopYNQ2aN94Z7k0UnUKamE3OjU8DFYFFokbSI2J9V9gVlM8ALWThDPnPu3EL7HPD2VDaZTggzcCCmbvc70qqPcC9mt60ogcrTiA3HEjwTK8ymKeuJMc4q6dVz200XnYUtLR9GYjPXvFOVr6W1zUK1WbPToaWJJuKnxBLnd0ftDEbMmj4loHYyhZyMjM91zQS4p7z8eKa9h0JrbacekcirexG0z4n3xz0QOWSvFj3jLhWXUIU21iIAwJtI3RbWa90I7rzAIqI3UElUJG7tLtUXzw4KQNETvXzqWaujEMenYlNIzLGxgB3AuJEQ12}, \eqref{8ThswELzXU3X7Ebd1KdZ7v1rN3GiirRXGKWK099ovBM0FDJCvkopYNQ2aN94Z7k0UnUKamE3OjU8DFYFFokbSI2J9V9gVlM8ALWThDPnPu3EL7HPD2VDaZTggzcCCmbvc70qqPcC9mt60ogcrTiA3HEjwTK8ymKeuJMc4q6dVz200XnYUtLR9GYjPXvFOVr6W1zUK1WbPToaWJJuKnxBLnd0ftDEbMmj4loHYyhZyMjM91zQS4p7z8eKa9h0JrbacekcirexG0z4n3xz0QOWSvFj3jLhWXUIU21iIAwJtI3RbWa90I7rzAIqI3UElUJG7tLtUXzw4KQNETvXzqWaujEMenYlNIzLGxgB3AuJEQ65}, and that $Q$ is a homogeneous polynomial of degree $d-m$ in the last inequality. By $Lu=f$, we have     \begin{equation}       L(0) u -Q       = \sum_{|\nu|=m} \big( a_{\nu}(0) - a_{\nu} \big) \partial^{\nu} u       - \sum_{|\nu|<m} a_{\nu} \partial^{\nu} u       + (f-Q)       .     \label{8ThswELzXU3X7Ebd1KdZ7v1rN3GiirRXGKWK099ovBM0FDJCvkopYNQ2aN94Z7k0UnUKamE3OjU8DFYFFokbSI2J9V9gVlM8ALWThDPnPu3EL7HPD2VDaZTggzcCCmbvc70qqPcC9mt60ogcrTiA3HEjwTK8ymKeuJMc4q6dVz200XnYUtLR9GYjPXvFOVr6W1zUK1WbPToaWJJuKnxBLnd0ftDEbMmj4loHYyhZyMjM91zQS4p7z8eKa9h0JrbacekcirexG0z4n3xz0QOWSvFj3jLhWXUIU21iIAwJtI3RbWa90I7rzAIqI3UElUJG7tLtUXzw4KQNETvXzqWaujEMenYlNIzLGxgB3AuJEQ69}     \end{equation} Then taking the $L^p$-norm of \eqref{8ThswELzXU3X7Ebd1KdZ7v1rN3GiirRXGKWK099ovBM0FDJCvkopYNQ2aN94Z7k0UnUKamE3OjU8DFYFFokbSI2J9V9gVlM8ALWThDPnPu3EL7HPD2VDaZTggzcCCmbvc70qqPcC9mt60ogcrTiA3HEjwTK8ymKeuJMc4q6dVz200XnYUtLR9GYjPXvFOVr6W1zUK1WbPToaWJJuKnxBLnd0ftDEbMmj4loHYyhZyMjM91zQS4p7z8eKa9h0JrbacekcirexG0z4n3xz0QOWSvFj3jLhWXUIU21iIAwJtI3RbWa90I7rzAIqI3UElUJG7tLtUXzw4KQNETvXzqWaujEMenYlNIzLGxgB3AuJEQ69}, we obtain     \begin{align}     \begin{split}       \zxvczxbcvdfghasdfrtsdafasdfasdfdsfgsdgh L(0) u - Q \zxvczxbcvdfghasdfrtsdafasdfasdfdsfgsdgh_{L^p(B_r)}       &\les \sum_{|\nu|=m} \zxvczxbcvdfghasdfrtsdafasdfasdfdsfgsdgh a_{\nu}(0) - a_{\nu}(x)\zxvczxbcvdfghasdfrtsdafasdfasdfdsfgsdgh_{L^{\infty}(B_r)}  \zxvczxbcvdfghasdfrtsdafasdfasdfdsfgsdgh \partial^{\nu} u \zxvczxbcvdfghasdfrtsdafasdfasdfdsfgsdgh_{L^p(B_r)}       \\&\indeq              + \sum_{|\nu|<m} |a_{\nu}(x)| \zxvczxbcvdfghasdfrtsdafasdfasdfdsfgsdgh \partial^{\nu} u \zxvczxbcvdfghasdfrtsdafasdfasdfdsfgsdgh_{L^p(B_r)}              + \zxvczxbcvdfghasdfrtsdafasdfasdfdsfgsdgh f-Q \zxvczxbcvdfghasdfrtsdafasdfasdfdsfgsdgh_{L^p(B_{r})}       \\&       \les \big( c_k + \MM   + \zxvczxbcvdfghasdfrtsdafasdfasdfdsfgsdgh Q \zxvczxbcvdfghasdfrtsdafasdfasdfdsfgsdgh_{L^p(B_1)} \big)               r^{d-1+\eta+(k+1)\alpha-m+n/p}       \comma       r \in (0, 1/2 ]       ,     \end{split}     \label{8ThswELzXU3X7Ebd1KdZ7v1rN3GiirRXGKWK099ovBM0FDJCvkopYNQ2aN94Z7k0UnUKamE3OjU8DFYFFokbSI2J9V9gVlM8ALWThDPnPu3EL7HPD2VDaZTggzcCCmbvc70qqPcC9mt60ogcrTiA3HEjwTK8ymKeuJMc4q6dVz200XnYUtLR9GYjPXvFOVr6W1zUK1WbPToaWJJuKnxBLnd0ftDEbMmj4loHYyhZyMjM91zQS4p7z8eKa9h0JrbacekcirexG0z4n3xz0QOWSvFj3jLhWXUIU21iIAwJtI3RbWa90I7rzAIqI3UElUJG7tLtUXzw4KQNETvXzqWaujEMenYlNIzLGxgB3AuJEQ71}     \end{align} where we used \eqref{8ThswELzXU3X7Ebd1KdZ7v1rN3GiirRXGKWK099ovBM0FDJCvkopYNQ2aN94Z7k0UnUKamE3OjU8DFYFFokbSI2J9V9gVlM8ALWThDPnPu3EL7HPD2VDaZTggzcCCmbvc70qqPcC9mt60ogcrTiA3HEjwTK8ymKeuJMc4q6dVz200XnYUtLR9GYjPXvFOVr6W1zUK1WbPToaWJJuKnxBLnd0ftDEbMmj4loHYyhZyMjM91zQS4p7z8eKa9h0JrbacekcirexG0z4n3xz0QOWSvFj3jLhWXUIU21iIAwJtI3RbWa90I7rzAIqI3UElUJG7tLtUXzw4KQNETvXzqWaujEMenYlNIzLGxgB3AuJEQ10}, \eqref{8ThswELzXU3X7Ebd1KdZ7v1rN3GiirRXGKWK099ovBM0FDJCvkopYNQ2aN94Z7k0UnUKamE3OjU8DFYFFokbSI2J9V9gVlM8ALWThDPnPu3EL7HPD2VDaZTggzcCCmbvc70qqPcC9mt60ogcrTiA3HEjwTK8ymKeuJMc4q6dVz200XnYUtLR9GYjPXvFOVr6W1zUK1WbPToaWJJuKnxBLnd0ftDEbMmj4loHYyhZyMjM91zQS4p7z8eKa9h0JrbacekcirexG0z4n3xz0QOWSvFj3jLhWXUIU21iIAwJtI3RbWa90I7rzAIqI3UElUJG7tLtUXzw4KQNETvXzqWaujEMenYlNIzLGxgB3AuJEQ11}, and \eqref{8ThswELzXU3X7Ebd1KdZ7v1rN3GiirRXGKWK099ovBM0FDJCvkopYNQ2aN94Z7k0UnUKamE3OjU8DFYFFokbSI2J9V9gVlM8ALWThDPnPu3EL7HPD2VDaZTggzcCCmbvc70qqPcC9mt60ogcrTiA3HEjwTK8ymKeuJMc4q6dVz200XnYUtLR9GYjPXvFOVr6W1zUK1WbPToaWJJuKnxBLnd0ftDEbMmj4loHYyhZyMjM91zQS4p7z8eKa9h0JrbacekcirexG0z4n3xz0QOWSvFj3jLhWXUIU21iIAwJtI3RbWa90I7rzAIqI3UElUJG7tLtUXzw4KQNETvXzqWaujEMenYlNIzLGxgB3AuJEQ68} in the last inequality.  Similarly to Lemma~\ref{L02}, we write $u = v + w$, where     \begin{equation}       w(x)       = \zxvczxbcvdfghasdfrtsdafasdfasdfdsfgsdgf_{|y|\leq 1/2} \Gamma(x-y) \big( L(0) u(y) - Q(y) \big) \, dy       ,    \label{8ThswELzXU3X7Ebd1KdZ7v1rN3GiirRXGKWK099ovBM0FDJCvkopYNQ2aN94Z7k0UnUKamE3OjU8DFYFFokbSI2J9V9gVlM8ALWThDPnPu3EL7HPD2VDaZTggzcCCmbvc70qqPcC9mt60ogcrTiA3HEjwTK8ymKeuJMc4q6dVz200XnYUtLR9GYjPXvFOVr6W1zUK1WbPToaWJJuKnxBLnd0ftDEbMmj4loHYyhZyMjM91zQS4p7z8eKa9h0JrbacekcirexG0z4n3xz0QOWSvFj3jLhWXUIU21iIAwJtI3RbWa90I7rzAIqI3UElUJG7tLtUXzw4KQNETvXzqWaujEMenYlNIzLGxgB3AuJEQ72}     \end{equation} and approximate  $v$ and $w$ by $P_v$ and $P_w$, respectively (with $P_w$ defined in \eqref{8ThswELzXU3X7Ebd1KdZ7v1rN3GiirRXGKWK099ovBM0FDJCvkopYNQ2aN94Z7k0UnUKamE3OjU8DFYFFokbSI2J9V9gVlM8ALWThDPnPu3EL7HPD2VDaZTggzcCCmbvc70qqPcC9mt60ogcrTiA3HEjwTK8ymKeuJMc4q6dVz200XnYUtLR9GYjPXvFOVr6W1zUK1WbPToaWJJuKnxBLnd0ftDEbMmj4loHYyhZyMjM91zQS4p7z8eKa9h0JrbacekcirexG0z4n3xz0QOWSvFj3jLhWXUIU21iIAwJtI3RbWa90I7rzAIqI3UElUJG7tLtUXzw4KQNETvXzqWaujEMenYlNIzLGxgB3AuJEQ75} below and $P_v$ after \eqref{8ThswELzXU3X7Ebd1KdZ7v1rN3GiirRXGKWK099ovBM0FDJCvkopYNQ2aN94Z7k0UnUKamE3OjU8DFYFFokbSI2J9V9gVlM8ALWThDPnPu3EL7HPD2VDaZTggzcCCmbvc70qqPcC9mt60ogcrTiA3HEjwTK8ymKeuJMc4q6dVz200XnYUtLR9GYjPXvFOVr6W1zUK1WbPToaWJJuKnxBLnd0ftDEbMmj4loHYyhZyMjM91zQS4p7z8eKa9h0JrbacekcirexG0z4n3xz0QOWSvFj3jLhWXUIU21iIAwJtI3RbWa90I7rzAIqI3UElUJG7tLtUXzw4KQNETvXzqWaujEMenYlNIzLGxgB3AuJEQ70}). Note that $w$ solves the equation      \begin{equation}       L(0) w = L(0) u - Q       .     \llabel{igj 4dz zW soNF wvqj bNFma0 am F Kiv Aap pzM zr VqYf OulM HafaBk 6J r eOQ BaT EsJ BB tHXj n2EU CNleWp cv W JIg gWX Ksn B3 wvmo WK49 Nl492o gR 6 fvc 8ff jJm sW Jr0j zI9p CBsIUV of D kKH Ub7 vxp uQ UXA6 hMUr yvxEpc Tq l Tkz z0q HbX pO 8jFu h6nw zVPPzp A8 9 61V 78c O2W aw 0yGn CHVq BVjTUH lk p 6dG HOd voE E8 cw7Q DL1o 1qg5TX qo V 720 hhQ TyF tp TJDg 9E8D nsp1Qi X9 8 ZVQ N3s duZ qc n9IX ozWh Fd16IB 0K 9 JeB Hvi 364 kQ lFMM JOn0 OUBrnv pY y jUB Ofs Pzx l4 zcMn JHdq OjSi6N Mn 8 bR6 kPe klT Fd VlwD SrhT 8Qr0sC hN h 88j 8ZA vvW VD 03wt ETKK N8ThswELzXU3X7Ebd1KdZ7v1rN3GiirRXGKWK099ovBM0FDJCvkopYNQ2aN94Z7k0UnUKamE3OjU8DFYFFokbSI2J9V9gVlM8ALWThDPnPu3EL7HPD2VDaZTggzcCCmbvc70qqPcC9mt60ogcrTiA3HEjwTK8ymKeuJMc4q6dVz200XnYUtLR9GYjPXvFOVr6W1zUK1WbPToaWJJuKnxBLnd0ftDEbMmj4loHYyhZyMjM91zQS4p7z8eKa9h0JrbacekcirexG0z4n3xz0QOWSvFj3jLhWXUIU21iIAwJtI3RbWa90I7rzAIqI3UElUJG7tLtUXzw4KQNETvXzqWaujEMenYlNIzLGxgB3AuJEQ73}     \end{equation} Since $\eta + (k+1)\alpha \in (0,1)$, we apply Lemma~\ref{L10} to \eqref{8ThswELzXU3X7Ebd1KdZ7v1rN3GiirRXGKWK099ovBM0FDJCvkopYNQ2aN94Z7k0UnUKamE3OjU8DFYFFokbSI2J9V9gVlM8ALWThDPnPu3EL7HPD2VDaZTggzcCCmbvc70qqPcC9mt60ogcrTiA3HEjwTK8ymKeuJMc4q6dVz200XnYUtLR9GYjPXvFOVr6W1zUK1WbPToaWJJuKnxBLnd0ftDEbMmj4loHYyhZyMjM91zQS4p7z8eKa9h0JrbacekcirexG0z4n3xz0QOWSvFj3jLhWXUIU21iIAwJtI3RbWa90I7rzAIqI3UElUJG7tLtUXzw4KQNETvXzqWaujEMenYlNIzLGxgB3AuJEQ72} and use \eqref{8ThswELzXU3X7Ebd1KdZ7v1rN3GiirRXGKWK099ovBM0FDJCvkopYNQ2aN94Z7k0UnUKamE3OjU8DFYFFokbSI2J9V9gVlM8ALWThDPnPu3EL7HPD2VDaZTggzcCCmbvc70qqPcC9mt60ogcrTiA3HEjwTK8ymKeuJMc4q6dVz200XnYUtLR9GYjPXvFOVr6W1zUK1WbPToaWJJuKnxBLnd0ftDEbMmj4loHYyhZyMjM91zQS4p7z8eKa9h0JrbacekcirexG0z4n3xz0QOWSvFj3jLhWXUIU21iIAwJtI3RbWa90I7rzAIqI3UElUJG7tLtUXzw4KQNETvXzqWaujEMenYlNIzLGxgB3AuJEQ71} to get     \begin{equation}       \zxvczxbcvdfghasdfrtsdafasdfasdfdsfgsdgh w - P_w \zxvczxbcvdfghasdfrtsdafasdfasdfdsfgsdgh_{L^q(B_r)}       \les \big( c_k + \MM   + \zxvczxbcvdfghasdfrtsdafasdfasdfdsfgsdgh Q \zxvczxbcvdfghasdfrtsdafasdfasdfdsfgsdgh_{L^p(B_1)} \big) r^{d-1+\eta+(k+1)\alpha+n/q}       \comma       r \in (0, 1/2 ] \colb       ,     \llabel{ kKH Ub7 vxp uQ UXA6 hMUr yvxEpc Tq l Tkz z0q HbX pO 8jFu h6nw zVPPzp A8 9 61V 78c O2W aw 0yGn CHVq BVjTUH lk p 6dG HOd voE E8 cw7Q DL1o 1qg5TX qo V 720 hhQ TyF tp TJDg 9E8D nsp1Qi X9 8 ZVQ N3s duZ qc n9IX ozWh Fd16IB 0K 9 JeB Hvi 364 kQ lFMM JOn0 OUBrnv pY y jUB Ofs Pzx l4 zcMn JHdq OjSi6N Mn 8 bR6 kPe klT Fd VlwD SrhT 8Qr0sC hN h 88j 8ZA vvW VD 03wt ETKK NUdr7W EK 1 jKS IHF Kh2 sr 1RRV Ra8J mBtkWI 1u k uZT F2B 4p8 E7 Y3p0 DX20 JM3XzQ tZ 3 bMC vM4 DEA wB Fp8q YKpL So1a5s dR P fTg 5R6 7v1 T4 eCJ1 qg14 CTK7u7 ag j Q0A tZ1 Nh6 hk Sys5 C8ThswELzXU3X7Ebd1KdZ7v1rN3GiirRXGKWK099ovBM0FDJCvkopYNQ2aN94Z7k0UnUKamE3OjU8DFYFFokbSI2J9V9gVlM8ALWThDPnPu3EL7HPD2VDaZTggzcCCmbvc70qqPcC9mt60ogcrTiA3HEjwTK8ymKeuJMc4q6dVz200XnYUtLR9GYjPXvFOVr6W1zUK1WbPToaWJJuKnxBLnd0ftDEbMmj4loHYyhZyMjM91zQS4p7z8eKa9h0JrbacekcirexG0z4n3xz0QOWSvFj3jLhWXUIU21iIAwJtI3RbWa90I7rzAIqI3UElUJG7tLtUXzw4KQNETvXzqWaujEMenYlNIzLGxgB3AuJEQ74}     \end{equation} where      \begin{equation}       P_w(x)        = \sum_{|\beta| \leq d-1} \frac{x^{\beta}}{\beta!}\zxvczxbcvdfghasdfrtsdafasdfasdfdsfgsdgf_{|y|\leq1/2} \partial^{\beta}_x \Gamma(-y)           \big( L(0) u(y) - Q(y) \big) \,dy        .     \label{8ThswELzXU3X7Ebd1KdZ7v1rN3GiirRXGKWK099ovBM0FDJCvkopYNQ2aN94Z7k0UnUKamE3OjU8DFYFFokbSI2J9V9gVlM8ALWThDPnPu3EL7HPD2VDaZTggzcCCmbvc70qqPcC9mt60ogcrTiA3HEjwTK8ymKeuJMc4q6dVz200XnYUtLR9GYjPXvFOVr6W1zUK1WbPToaWJJuKnxBLnd0ftDEbMmj4loHYyhZyMjM91zQS4p7z8eKa9h0JrbacekcirexG0z4n3xz0QOWSvFj3jLhWXUIU21iIAwJtI3RbWa90I7rzAIqI3UElUJG7tLtUXzw4KQNETvXzqWaujEMenYlNIzLGxgB3AuJEQ75}     \end{equation} Moreover, by \eqref{8ThswELzXU3X7Ebd1KdZ7v1rN3GiirRXGKWK099ovBM0FDJCvkopYNQ2aN94Z7k0UnUKamE3OjU8DFYFFokbSI2J9V9gVlM8ALWThDPnPu3EL7HPD2VDaZTggzcCCmbvc70qqPcC9mt60ogcrTiA3HEjwTK8ymKeuJMc4q6dVz200XnYUtLR9GYjPXvFOVr6W1zUK1WbPToaWJJuKnxBLnd0ftDEbMmj4loHYyhZyMjM91zQS4p7z8eKa9h0JrbacekcirexG0z4n3xz0QOWSvFj3jLhWXUIU21iIAwJtI3RbWa90I7rzAIqI3UElUJG7tLtUXzw4KQNETvXzqWaujEMenYlNIzLGxgB3AuJEQ72},     \begin{equation}       \zxvczxbcvdfghasdfrtsdafasdfasdfdsfgsdgh w \zxvczxbcvdfghasdfrtsdafasdfasdfdsfgsdgh_{L^p(B_{1/2})}       \les \zxvczxbcvdfghasdfrtsdafasdfasdfdsfgsdgh u \zxvczxbcvdfghasdfrtsdafasdfasdfdsfgsdgh_{W^{m,p}(B_{1/2})}              + \zxvczxbcvdfghasdfrtsdafasdfasdfdsfgsdgh Q \zxvczxbcvdfghasdfrtsdafasdfasdfdsfgsdgh_{L^p(B_1)}       .     \llabel{ X9 8 ZVQ N3s duZ qc n9IX ozWh Fd16IB 0K 9 JeB Hvi 364 kQ lFMM JOn0 OUBrnv pY y jUB Ofs Pzx l4 zcMn JHdq OjSi6N Mn 8 bR6 kPe klT Fd VlwD SrhT 8Qr0sC hN h 88j 8ZA vvW VD 03wt ETKK NUdr7W EK 1 jKS IHF Kh2 sr 1RRV Ra8J mBtkWI 1u k uZT F2B 4p8 E7 Y3p0 DX20 JM3XzQ tZ 3 bMC vM4 DEA wB Fp8q YKpL So1a5s dR P fTg 5R6 7v1 T4 eCJ1 qg14 CTK7u7 ag j Q0A tZ1 Nh6 hk Sys5 CWon IOqgCL 3u 7 feR BHz odS Jp 7JH8 u6Rw sYE0mc P4 r LaW Atl yRw kH F3ei UyhI iA19ZB u8 m ywf 42n uyX 0e ljCt 3Lkd 1eUQEZ oO Z rA2 Oqf oQ5 Ca hrBy KzFg DOseim 0j Y BmX csL Ayc cC J8ThswELzXU3X7Ebd1KdZ7v1rN3GiirRXGKWK099ovBM0FDJCvkopYNQ2aN94Z7k0UnUKamE3OjU8DFYFFokbSI2J9V9gVlM8ALWThDPnPu3EL7HPD2VDaZTggzcCCmbvc70qqPcC9mt60ogcrTiA3HEjwTK8ymKeuJMc4q6dVz200XnYUtLR9GYjPXvFOVr6W1zUK1WbPToaWJJuKnxBLnd0ftDEbMmj4loHYyhZyMjM91zQS4p7z8eKa9h0JrbacekcirexG0z4n3xz0QOWSvFj3jLhWXUIU21iIAwJtI3RbWa90I7rzAIqI3UElUJG7tLtUXzw4KQNETvXzqWaujEMenYlNIzLGxgB3AuJEQ76}     \end{equation} Thus, the function $v = u-w$ satisfies $L(0)v = Q$, and  we have     \begin{equation}       \zxvczxbcvdfghasdfrtsdafasdfasdfdsfgsdgh v \zxvczxbcvdfghasdfrtsdafasdfasdfdsfgsdgh_{L^p(B_{1/2})}        \les \zxvczxbcvdfghasdfrtsdafasdfasdfdsfgsdgh u \zxvczxbcvdfghasdfrtsdafasdfasdfdsfgsdgh_{L^p(B_{1/2})}              + \zxvczxbcvdfghasdfrtsdafasdfasdfdsfgsdgh w \zxvczxbcvdfghasdfrtsdafasdfasdfdsfgsdgh_{L^p(B_{1/2})}       \les \zxvczxbcvdfghasdfrtsdafasdfasdfdsfgsdgh u \zxvczxbcvdfghasdfrtsdafasdfasdfdsfgsdgh_{W^{m,p}(B_{1/2})}             + \zxvczxbcvdfghasdfrtsdafasdfasdfdsfgsdgh Q \zxvczxbcvdfghasdfrtsdafasdfasdfdsfgsdgh_{L^p(B_1)}       .     \llabel{Udr7W EK 1 jKS IHF Kh2 sr 1RRV Ra8J mBtkWI 1u k uZT F2B 4p8 E7 Y3p0 DX20 JM3XzQ tZ 3 bMC vM4 DEA wB Fp8q YKpL So1a5s dR P fTg 5R6 7v1 T4 eCJ1 qg14 CTK7u7 ag j Q0A tZ1 Nh6 hk Sys5 CWon IOqgCL 3u 7 feR BHz odS Jp 7JH8 u6Rw sYE0mc P4 r LaW Atl yRw kH F3ei UyhI iA19ZB u8 m ywf 42n uyX 0e ljCt 3Lkd 1eUQEZ oO Z rA2 Oqf oQ5 Ca hrBy KzFg DOseim 0j Y BmX csL Ayc cC JBTZ PEjy zPb5hZ KW O xT6 dyt u82 Ia htpD m75Y DktQvd Nj W jIQ H1B Ace SZ KVVP 136v L8XhMm 1O H Kn2 gUy kFU wN 8JML Bqmn vGuwGR oW U oNZ Y2P nmS 5g QMcR YHxL yHuDo8 ba w aqM NYt onW8ThswELzXU3X7Ebd1KdZ7v1rN3GiirRXGKWK099ovBM0FDJCvkopYNQ2aN94Z7k0UnUKamE3OjU8DFYFFokbSI2J9V9gVlM8ALWThDPnPu3EL7HPD2VDaZTggzcCCmbvc70qqPcC9mt60ogcrTiA3HEjwTK8ymKeuJMc4q6dVz200XnYUtLR9GYjPXvFOVr6W1zUK1WbPToaWJJuKnxBLnd0ftDEbMmj4loHYyhZyMjM91zQS4p7z8eKa9h0JrbacekcirexG0z4n3xz0QOWSvFj3jLhWXUIU21iIAwJtI3RbWa90I7rzAIqI3UElUJG7tLtUXzw4KQNETvXzqWaujEMenYlNIzLGxgB3AuJEQ77}     \end{equation} Note that from \eqref{8ThswELzXU3X7Ebd1KdZ7v1rN3GiirRXGKWK099ovBM0FDJCvkopYNQ2aN94Z7k0UnUKamE3OjU8DFYFFokbSI2J9V9gVlM8ALWThDPnPu3EL7HPD2VDaZTggzcCCmbvc70qqPcC9mt60ogcrTiA3HEjwTK8ymKeuJMc4q6dVz200XnYUtLR9GYjPXvFOVr6W1zUK1WbPToaWJJuKnxBLnd0ftDEbMmj4loHYyhZyMjM91zQS4p7z8eKa9h0JrbacekcirexG0z4n3xz0QOWSvFj3jLhWXUIU21iIAwJtI3RbWa90I7rzAIqI3UElUJG7tLtUXzw4KQNETvXzqWaujEMenYlNIzLGxgB3AuJEQ68}, we have $\zxvczxbcvdfghasdfrtsdafasdfasdfdsfgsdgh \partial^{\nu} u \zxvczxbcvdfghasdfrtsdafasdfasdfdsfgsdgh_{L^p(B_{1/2})} \les c_k + M + \zxvczxbcvdfghasdfrtsdafasdfasdfdsfgsdgh Q \zxvczxbcvdfghasdfrtsdafasdfasdfdsfgsdgh_{L^p(B_1)}$ for all $|\nu|\leq m$, which implies  $ \zxvczxbcvdfghasdfrtsdafasdfasdfdsfgsdgh u \zxvczxbcvdfghasdfrtsdafasdfasdfdsfgsdgh_{W^{m,p}(B_{1/2})} \les  c_k + \MM   + \zxvczxbcvdfghasdfrtsdafasdfasdfdsfgsdgh Q \zxvczxbcvdfghasdfrtsdafasdfasdfdsfgsdgh_{L^p(B_1)} . $ Therefore,   \begin{equation}
   \zxvczxbcvdfghasdfrtsdafasdfasdfdsfgsdgh v \zxvczxbcvdfghasdfrtsdafasdfasdfdsfgsdgh_{L^p(B_{1/2})}     \les  c_k + \MM   + \zxvczxbcvdfghasdfrtsdafasdfasdfdsfgsdgh Q \zxvczxbcvdfghasdfrtsdafasdfasdfdsfgsdgh_{L^p(B_1)}    .    \label{8ThswELzXU3X7Ebd1KdZ7v1rN3GiirRXGKWK099ovBM0FDJCvkopYNQ2aN94Z7k0UnUKamE3OjU8DFYFFokbSI2J9V9gVlM8ALWThDPnPu3EL7HPD2VDaZTggzcCCmbvc70qqPcC9mt60ogcrTiA3HEjwTK8ymKeuJMc4q6dVz200XnYUtLR9GYjPXvFOVr6W1zUK1WbPToaWJJuKnxBLnd0ftDEbMmj4loHYyhZyMjM91zQS4p7z8eKa9h0JrbacekcirexG0z4n3xz0QOWSvFj3jLhWXUIU21iIAwJtI3RbWa90I7rzAIqI3UElUJG7tLtUXzw4KQNETvXzqWaujEMenYlNIzLGxgB3AuJEQ70}   \end{equation} Let $P_v$ be the Taylor polynomial of $v$ with the degree $d-1$,  then $L(0) P_v=Q$, and for $|x|=r\leq 1/4$, we may bound     \begin{align}     \begin{split}       | v(x) - P_v(x) |       &\les \zxvczxbcvdfghasdfrtsdafasdfasdfdsfgsdgh \partial^{d} v \zxvczxbcvdfghasdfrtsdafasdfasdfdsfgsdgh_{L^{\infty} (B_r)} r^d        \les \zxvczxbcvdfghasdfrtsdafasdfasdfdsfgsdgh \partial^{d} v \zxvczxbcvdfghasdfrtsdafasdfasdfdsfgsdgh_{L^{\infty} (B_{1/4})} r^d       \\&       \les \big( \zxvczxbcvdfghasdfrtsdafasdfasdfdsfgsdgh Q \zxvczxbcvdfghasdfrtsdafasdfasdfdsfgsdgh_{L^p(B_{1/2})}                       + \zxvczxbcvdfghasdfrtsdafasdfasdfdsfgsdgh v \zxvczxbcvdfghasdfrtsdafasdfasdfdsfgsdgh_{L^p(B_{1/2})}              \big)               r^d       \les \big( c_k + \MM   + \zxvczxbcvdfghasdfrtsdafasdfasdfdsfgsdgh Q \zxvczxbcvdfghasdfrtsdafasdfasdfdsfgsdgh_{L^p(B_1)} \big)               r^d       ,     \llabel{Won IOqgCL 3u 7 feR BHz odS Jp 7JH8 u6Rw sYE0mc P4 r LaW Atl yRw kH F3ei UyhI iA19ZB u8 m ywf 42n uyX 0e ljCt 3Lkd 1eUQEZ oO Z rA2 Oqf oQ5 Ca hrBy KzFg DOseim 0j Y BmX csL Ayc cC JBTZ PEjy zPb5hZ KW O xT6 dyt u82 Ia htpD m75Y DktQvd Nj W jIQ H1B Ace SZ KVVP 136v L8XhMm 1O H Kn2 gUy kFU wN 8JML Bqmn vGuwGR oW U oNZ Y2P nmS 5g QMcR YHxL yHuDo8 ba w aqM NYt onW u2 YIOz eB6R wHuGcn fi o 47U PM5 tOj sz QBNq 7mco fCNjou 83 e mcY 81s vsI 2Y DS3S yloB Nx5FBV Bc 9 6HZ EOX UO3 W1 fIF5 jtEM W6KW7D 63 t H0F CVT Zup Pl A9aI oN2s f1Bw31 gg L FoD O08ThswELzXU3X7Ebd1KdZ7v1rN3GiirRXGKWK099ovBM0FDJCvkopYNQ2aN94Z7k0UnUKamE3OjU8DFYFFokbSI2J9V9gVlM8ALWThDPnPu3EL7HPD2VDaZTggzcCCmbvc70qqPcC9mt60ogcrTiA3HEjwTK8ymKeuJMc4q6dVz200XnYUtLR9GYjPXvFOVr6W1zUK1WbPToaWJJuKnxBLnd0ftDEbMmj4loHYyhZyMjM91zQS4p7z8eKa9h0JrbacekcirexG0z4n3xz0QOWSvFj3jLhWXUIU21iIAwJtI3RbWa90I7rzAIqI3UElUJG7tLtUXzw4KQNETvXzqWaujEMenYlNIzLGxgB3AuJEQ78}     \end{split}     \end{align} where we used the Taylor theorem in the first step and the elliptic regularity for $L(0)v=Q$ in the second. Therefore,     \begin{equation}       \zxvczxbcvdfghasdfrtsdafasdfasdfdsfgsdgh v - P_v \zxvczxbcvdfghasdfrtsdafasdfasdfdsfgsdgh_{L^p(B_r)}        \les \big( c_k + \MM   + \zxvczxbcvdfghasdfrtsdafasdfasdfdsfgsdgh Q \zxvczxbcvdfghasdfrtsdafasdfasdfdsfgsdgh_{L^p(B_1)} \big)               r^{d+n/p}       \comma        r \in (0, 1/4 ]       .      \llabel{BTZ PEjy zPb5hZ KW O xT6 dyt u82 Ia htpD m75Y DktQvd Nj W jIQ H1B Ace SZ KVVP 136v L8XhMm 1O H Kn2 gUy kFU wN 8JML Bqmn vGuwGR oW U oNZ Y2P nmS 5g QMcR YHxL yHuDo8 ba w aqM NYt onW u2 YIOz eB6R wHuGcn fi o 47U PM5 tOj sz QBNq 7mco fCNjou 83 e mcY 81s vsI 2Y DS3S yloB Nx5FBV Bc 9 6HZ EOX UO3 W1 fIF5 jtEM W6KW7D 63 t H0F CVT Zup Pl A9aI oN2s f1Bw31 gg L FoD O0M x18 oo heEd KgZB Cqdqpa sa H Fhx BrE aRg Au I5dq mWWB MuHfv9 0y S PtG hFF dYJ JL f3Ap k5Ck Szr0Kb Vd i sQk uSA JEn DT YkjP AEMu a0VCtC Ff z 9R6 Vht 8Ua cB e7op AnGa 7AbLWj Hc s n8ThswELzXU3X7Ebd1KdZ7v1rN3GiirRXGKWK099ovBM0FDJCvkopYNQ2aN94Z7k0UnUKamE3OjU8DFYFFokbSI2J9V9gVlM8ALWThDPnPu3EL7HPD2VDaZTggzcCCmbvc70qqPcC9mt60ogcrTiA3HEjwTK8ymKeuJMc4q6dVz200XnYUtLR9GYjPXvFOVr6W1zUK1WbPToaWJJuKnxBLnd0ftDEbMmj4loHYyhZyMjM91zQS4p7z8eKa9h0JrbacekcirexG0z4n3xz0QOWSvFj3jLhWXUIU21iIAwJtI3RbWa90I7rzAIqI3UElUJG7tLtUXzw4KQNETvXzqWaujEMenYlNIzLGxgB3AuJEQ79}     \end{equation} Letting $P = P_v+P_w$, we have $u-P = (v-P_v) + (w - P_w)$. Thus, for all $r \leq 1/4$, we may bound     \begin{align}     \begin{split}       \zxvczxbcvdfghasdfrtsdafasdfasdfdsfgsdgh u - P \zxvczxbcvdfghasdfrtsdafasdfasdfdsfgsdgh_{L^p(B_r)}       &\les \big( c_k + \MM   + \zxvczxbcvdfghasdfrtsdafasdfasdfdsfgsdgh Q \zxvczxbcvdfghasdfrtsdafasdfasdfdsfgsdgh_{L^p(B_1)} \big)               r^{d+\min((k+1)\alpha+\eta-1,0)+n/p}       \\&       = \big( c_k + \MM   + \zxvczxbcvdfghasdfrtsdafasdfasdfdsfgsdgh Q \zxvczxbcvdfghasdfrtsdafasdfasdfdsfgsdgh_{L^p(B_1)} \big)               r^{d+(k+1)\alpha+\eta-1+n/p}              ,     \end{split}     \label{8ThswELzXU3X7Ebd1KdZ7v1rN3GiirRXGKWK099ovBM0FDJCvkopYNQ2aN94Z7k0UnUKamE3OjU8DFYFFokbSI2J9V9gVlM8ALWThDPnPu3EL7HPD2VDaZTggzcCCmbvc70qqPcC9mt60ogcrTiA3HEjwTK8ymKeuJMc4q6dVz200XnYUtLR9GYjPXvFOVr6W1zUK1WbPToaWJJuKnxBLnd0ftDEbMmj4loHYyhZyMjM91zQS4p7z8eKa9h0JrbacekcirexG0z4n3xz0QOWSvFj3jLhWXUIU21iIAwJtI3RbWa90I7rzAIqI3UElUJG7tLtUXzw4KQNETvXzqWaujEMenYlNIzLGxgB3AuJEQ80}     \end{align}  where we used $(k+1)\alpha + \eta < 1$ in the last equality.   We now claim that~$P \equiv 0$. If, contrary to the assertion, we have $P \not\equiv 0$, then using \eqref{8ThswELzXU3X7Ebd1KdZ7v1rN3GiirRXGKWK099ovBM0FDJCvkopYNQ2aN94Z7k0UnUKamE3OjU8DFYFFokbSI2J9V9gVlM8ALWThDPnPu3EL7HPD2VDaZTggzcCCmbvc70qqPcC9mt60ogcrTiA3HEjwTK8ymKeuJMc4q6dVz200XnYUtLR9GYjPXvFOVr6W1zUK1WbPToaWJJuKnxBLnd0ftDEbMmj4loHYyhZyMjM91zQS4p7z8eKa9h0JrbacekcirexG0z4n3xz0QOWSvFj3jLhWXUIU21iIAwJtI3RbWa90I7rzAIqI3UElUJG7tLtUXzw4KQNETvXzqWaujEMenYlNIzLGxgB3AuJEQ80} we obtain     \begin{equation}       \zxvczxbcvdfghasdfrtsdafasdfasdfdsfgsdgh P \zxvczxbcvdfghasdfrtsdafasdfasdfdsfgsdgh_{L^p(B_r)}        \leq \zxvczxbcvdfghasdfrtsdafasdfasdfdsfgsdgh u - P \zxvczxbcvdfghasdfrtsdafasdfasdfdsfgsdgh_{L^p(B_r)}               + \zxvczxbcvdfghasdfrtsdafasdfasdfdsfgsdgh u \zxvczxbcvdfghasdfrtsdafasdfasdfdsfgsdgh_{L^p(B_r)}        \les r^{d-1+\eta +(k+1)\alpha+ n/p}              + r^{d-1+\eta+k\alpha+n/p}       \les r^{d-1+\eta+k\alpha+n/p}       .     \label{8ThswELzXU3X7Ebd1KdZ7v1rN3GiirRXGKWK099ovBM0FDJCvkopYNQ2aN94Z7k0UnUKamE3OjU8DFYFFokbSI2J9V9gVlM8ALWThDPnPu3EL7HPD2VDaZTggzcCCmbvc70qqPcC9mt60ogcrTiA3HEjwTK8ymKeuJMc4q6dVz200XnYUtLR9GYjPXvFOVr6W1zUK1WbPToaWJJuKnxBLnd0ftDEbMmj4loHYyhZyMjM91zQS4p7z8eKa9h0JrbacekcirexG0z4n3xz0QOWSvFj3jLhWXUIU21iIAwJtI3RbWa90I7rzAIqI3UElUJG7tLtUXzw4KQNETvXzqWaujEMenYlNIzLGxgB3AuJEQ81}     \end{equation}     On the other hand, since $P$ is a nontrivial polynomial of degree less than or equal to $d-1$, we have $       \zxvczxbcvdfghasdfrtsdafasdfasdfdsfgsdgh P \zxvczxbcvdfghasdfrtsdafasdfasdfdsfgsdgh_{L^p(B_r)}       \ges r^{d-1+n/p} $, for all~$r>0$. Combining with \eqref{8ThswELzXU3X7Ebd1KdZ7v1rN3GiirRXGKWK099ovBM0FDJCvkopYNQ2aN94Z7k0UnUKamE3OjU8DFYFFokbSI2J9V9gVlM8ALWThDPnPu3EL7HPD2VDaZTggzcCCmbvc70qqPcC9mt60ogcrTiA3HEjwTK8ymKeuJMc4q6dVz200XnYUtLR9GYjPXvFOVr6W1zUK1WbPToaWJJuKnxBLnd0ftDEbMmj4loHYyhZyMjM91zQS4p7z8eKa9h0JrbacekcirexG0z4n3xz0QOWSvFj3jLhWXUIU21iIAwJtI3RbWa90I7rzAIqI3UElUJG7tLtUXzw4KQNETvXzqWaujEMenYlNIzLGxgB3AuJEQ81}, we get a contradiction when letting $r \to 0$ since~$\alpha,\eta>0$. This proves that~$P \equiv 0$.  Substituting $P\equiv0$ in \eqref{8ThswELzXU3X7Ebd1KdZ7v1rN3GiirRXGKWK099ovBM0FDJCvkopYNQ2aN94Z7k0UnUKamE3OjU8DFYFFokbSI2J9V9gVlM8ALWThDPnPu3EL7HPD2VDaZTggzcCCmbvc70qqPcC9mt60ogcrTiA3HEjwTK8ymKeuJMc4q6dVz200XnYUtLR9GYjPXvFOVr6W1zUK1WbPToaWJJuKnxBLnd0ftDEbMmj4loHYyhZyMjM91zQS4p7z8eKa9h0JrbacekcirexG0z4n3xz0QOWSvFj3jLhWXUIU21iIAwJtI3RbWa90I7rzAIqI3UElUJG7tLtUXzw4KQNETvXzqWaujEMenYlNIzLGxgB3AuJEQ80}, we get     \begin{equation}       \zxvczxbcvdfghasdfrtsdafasdfasdfdsfgsdgh u \zxvczxbcvdfghasdfrtsdafasdfasdfdsfgsdgh_{L^p(B_r)}       \les \big( c_k + \MM   + \zxvczxbcvdfghasdfrtsdafasdfasdfdsfgsdgh Q \zxvczxbcvdfghasdfrtsdafasdfasdfdsfgsdgh_{L^p(B_1)} \big)              r^{d+\min ( (k+1)\alpha+ \eta-1,0) + n/p}        \comma         r \in (0, 1/4 ]        .        \colb     \llabel{ u2 YIOz eB6R wHuGcn fi o 47U PM5 tOj sz QBNq 7mco fCNjou 83 e mcY 81s vsI 2Y DS3S yloB Nx5FBV Bc 9 6HZ EOX UO3 W1 fIF5 jtEM W6KW7D 63 t H0F CVT Zup Pl A9aI oN2s f1Bw31 gg L FoD O0M x18 oo heEd KgZB Cqdqpa sa H Fhx BrE aRg Au I5dq mWWB MuHfv9 0y S PtG hFF dYJ JL f3Ap k5Ck Szr0Kb Vd i sQk uSA JEn DT YkjP AEMu a0VCtC Ff z 9R6 Vht 8Ua cB e7op AnGa 7AbLWj Hc s nAR GMb n7a 9n paMf lftM 7jvb20 0T W xUC 4lt e92 9j oZrA IuIa o1Zqdr oC L 55L T4Q 8kN yv sIzP x4i5 9lKTq2 JB B sZb QCE Ctw ar VBMT H1QR 6v5srW hR r D4r wf8 ik7 KH Egee rFVT ErONml Q8ThswELzXU3X7Ebd1KdZ7v1rN3GiirRXGKWK099ovBM0FDJCvkopYNQ2aN94Z7k0UnUKamE3OjU8DFYFFokbSI2J9V9gVlM8ALWThDPnPu3EL7HPD2VDaZTggzcCCmbvc70qqPcC9mt60ogcrTiA3HEjwTK8ymKeuJMc4q6dVz200XnYUtLR9GYjPXvFOVr6W1zUK1WbPToaWJJuKnxBLnd0ftDEbMmj4loHYyhZyMjM91zQS4p7z8eKa9h0JrbacekcirexG0z4n3xz0QOWSvFj3jLhWXUIU21iIAwJtI3RbWa90I7rzAIqI3UElUJG7tLtUXzw4KQNETvXzqWaujEMenYlNIzLGxgB3AuJEQ83}     \end{equation} Since $(k+1)\alpha + \eta < 1$, we have proven     \begin{equation}       \zxvczxbcvdfghasdfrtsdafasdfasdfdsfgsdgh u \zxvczxbcvdfghasdfrtsdafasdfasdfdsfgsdgh_{L^p(B_r)}       \les \big( c_k + \MM   + \zxvczxbcvdfghasdfrtsdafasdfasdfdsfgsdgh Q \zxvczxbcvdfghasdfrtsdafasdfasdfdsfgsdgh_{L^p(B_1)} \big)              r^{d+ \eta-1+ n/p + (k+1)\alpha}       \comma       r \in (0, 1/4]       ,     \llabel{M x18 oo heEd KgZB Cqdqpa sa H Fhx BrE aRg Au I5dq mWWB MuHfv9 0y S PtG hFF dYJ JL f3Ap k5Ck Szr0Kb Vd i sQk uSA JEn DT YkjP AEMu a0VCtC Ff z 9R6 Vht 8Ua cB e7op AnGa 7AbLWj Hc s nAR GMb n7a 9n paMf lftM 7jvb20 0T W xUC 4lt e92 9j oZrA IuIa o1Zqdr oC L 55L T4Q 8kN yv sIzP x4i5 9lKTq2 JB B sZb QCE Ctw ar VBMT H1QR 6v5srW hR r D4r wf8 ik7 KH Egee rFVT ErONml Q5 L R8v XNZ LB3 9U DzRH ZbH9 fTBhRw kA 2 n3p g4I grH xd fEFu z6RE tDqPdw N7 H TVt cE1 8hW 6y n4Gn nCE3 MEQ51i Ps G Z2G Lbt CSt hu zvPF eE28 MM23ug TC d j7z 7Av TLa 1A GLiJ 5JwW CiD8ThswELzXU3X7Ebd1KdZ7v1rN3GiirRXGKWK099ovBM0FDJCvkopYNQ2aN94Z7k0UnUKamE3OjU8DFYFFokbSI2J9V9gVlM8ALWThDPnPu3EL7HPD2VDaZTggzcCCmbvc70qqPcC9mt60ogcrTiA3HEjwTK8ymKeuJMc4q6dVz200XnYUtLR9GYjPXvFOVr6W1zUK1WbPToaWJJuKnxBLnd0ftDEbMmj4loHYyhZyMjM91zQS4p7z8eKa9h0JrbacekcirexG0z4n3xz0QOWSvFj3jLhWXUIU21iIAwJtI3RbWa90I7rzAIqI3UElUJG7tLtUXzw4KQNETvXzqWaujEMenYlNIzLGxgB3AuJEQ85}     \end{equation} so \eqref{8ThswELzXU3X7Ebd1KdZ7v1rN3GiirRXGKWK099ovBM0FDJCvkopYNQ2aN94Z7k0UnUKamE3OjU8DFYFFokbSI2J9V9gVlM8ALWThDPnPu3EL7HPD2VDaZTggzcCCmbvc70qqPcC9mt60ogcrTiA3HEjwTK8ymKeuJMc4q6dVz200XnYUtLR9GYjPXvFOVr6W1zUK1WbPToaWJJuKnxBLnd0ftDEbMmj4loHYyhZyMjM91zQS4p7z8eKa9h0JrbacekcirexG0z4n3xz0QOWSvFj3jLhWXUIU21iIAwJtI3RbWa90I7rzAIqI3UElUJG7tLtUXzw4KQNETvXzqWaujEMenYlNIzLGxgB3AuJEQ64} holds. \end{proof} \par \subsection{The Schauder estimates for elliptic equations} \begin{proof}[Proof of Theorem~\ref{T01}] We first reduce $\eta$ slightly so that $\eta/k_0\notin \mathbb{Q}$. Let $k_0\in\mathbb{N}_0$ be such that $k_0\alpha + \eta < 1 \leq (k_0+1)\alpha+\eta$.  Since $\eta/k_0\notin\mathbb{Q}$, the second equality is strict. We first show that      \begin{equation}       \zxvczxbcvdfghasdfrtsdafasdfasdfdsfgsdgh u \zxvczxbcvdfghasdfrtsdafasdfasdfdsfgsdgh_{L^p(B_r)}       \les (c_0 + M   + \zxvczxbcvdfghasdfrtsdafasdfasdfdsfgsdgh Q \zxvczxbcvdfghasdfrtsdafasdfasdfdsfgsdgh_{L^p(B_1)})       r^{d+n/p}       .    \llabel{AR GMb n7a 9n paMf lftM 7jvb20 0T W xUC 4lt e92 9j oZrA IuIa o1Zqdr oC L 55L T4Q 8kN yv sIzP x4i5 9lKTq2 JB B sZb QCE Ctw ar VBMT H1QR 6v5srW hR r D4r wf8 ik7 KH Egee rFVT ErONml Q5 L R8v XNZ LB3 9U DzRH ZbH9 fTBhRw kA 2 n3p g4I grH xd fEFu z6RE tDqPdw N7 H TVt cE1 8hW 6y n4Gn nCE3 MEQ51i Ps G Z2G Lbt CSt hu zvPF eE28 MM23ug TC d j7z 7Av TLa 1A GLiJ 5JwW CiDPyM qa 8 tAK QZ9 cfP 42 kuUz V3h6 GsGFoW m9 h cfj 51d GtW yZ zC5D aVt2 Wi5IIs gD B 0cX LM1 FtE xE RIZI Z0Rt QUtWcU Cm F mSj xvW pZc gl dopk 0D7a EouRku Id O ZdW FOR uqb PY 6HkW OVi8ThswELzXU3X7Ebd1KdZ7v1rN3GiirRXGKWK099ovBM0FDJCvkopYNQ2aN94Z7k0UnUKamE3OjU8DFYFFokbSI2J9V9gVlM8ALWThDPnPu3EL7HPD2VDaZTggzcCCmbvc70qqPcC9mt60ogcrTiA3HEjwTK8ymKeuJMc4q6dVz200XnYUtLR9GYjPXvFOVr6W1zUK1WbPToaWJJuKnxBLnd0ftDEbMmj4loHYyhZyMjM91zQS4p7z8eKa9h0JrbacekcirexG0z4n3xz0QOWSvFj3jLhWXUIU21iIAwJtI3RbWa90I7rzAIqI3UElUJG7tLtUXzw4KQNETvXzqWaujEMenYlNIzLGxgB3AuJEQ86}     \end{equation}   By \eqref{8ThswELzXU3X7Ebd1KdZ7v1rN3GiirRXGKWK099ovBM0FDJCvkopYNQ2aN94Z7k0UnUKamE3OjU8DFYFFokbSI2J9V9gVlM8ALWThDPnPu3EL7HPD2VDaZTggzcCCmbvc70qqPcC9mt60ogcrTiA3HEjwTK8ymKeuJMc4q6dVz200XnYUtLR9GYjPXvFOVr6W1zUK1WbPToaWJJuKnxBLnd0ftDEbMmj4loHYyhZyMjM91zQS4p7z8eKa9h0JrbacekcirexG0z4n3xz0QOWSvFj3jLhWXUIU21iIAwJtI3RbWa90I7rzAIqI3UElUJG7tLtUXzw4KQNETvXzqWaujEMenYlNIzLGxgB3AuJEQ14}, we have $       c_0        = \sup_{r\leq1}          \fractext{\zxvczxbcvdfghasdfrtsdafasdfasdfdsfgsdgh u \zxvczxbcvdfghasdfrtsdafasdfasdfdsfgsdgh_{L^p(B_r)}}{r^{d-1+\eta+n/p}}       < \infty $. Applying Lemma~\ref{L04} $k_0$ times, we obtain     \begin{equation}       c_{k_0}       = \sup_{r\leq1}          \frac{\zxvczxbcvdfghasdfrtsdafasdfasdfdsfgsdgh u \zxvczxbcvdfghasdfrtsdafasdfasdfdsfgsdgh_{L^p(B_r)}}{r^{d-1+\eta+k_0\alpha+n/p}}       \les c_0 + M   + \zxvczxbcvdfghasdfrtsdafasdfasdfdsfgsdgh Q \zxvczxbcvdfghasdfrtsdafasdfasdfdsfgsdgh_{L^p(B_1)}       .     \llabel{5 L R8v XNZ LB3 9U DzRH ZbH9 fTBhRw kA 2 n3p g4I grH xd fEFu z6RE tDqPdw N7 H TVt cE1 8hW 6y n4Gn nCE3 MEQ51i Ps G Z2G Lbt CSt hu zvPF eE28 MM23ug TC d j7z 7Av TLa 1A GLiJ 5JwW CiDPyM qa 8 tAK QZ9 cfP 42 kuUz V3h6 GsGFoW m9 h cfj 51d GtW yZ zC5D aVt2 Wi5IIs gD B 0cX LM1 FtE xE RIZI Z0Rt QUtWcU Cm F mSj xvW pZc gl dopk 0D7a EouRku Id O ZdW FOR uqb PY 6HkW OVi7 FuVMLW nx p SaN omk rC5 uI ZK9C jpJy UIeO6k gb 7 tr2 SCY x5F 11 S6Xq OImr s7vv0u vA g rb9 hGP Fnk RM j92H gczJ 660kHb BB l QSI OY7 FcX 0c uyDl LjbU 3F6vZk Gb a KaM ufj uxp n4 Mi48ThswELzXU3X7Ebd1KdZ7v1rN3GiirRXGKWK099ovBM0FDJCvkopYNQ2aN94Z7k0UnUKamE3OjU8DFYFFokbSI2J9V9gVlM8ALWThDPnPu3EL7HPD2VDaZTggzcCCmbvc70qqPcC9mt60ogcrTiA3HEjwTK8ymKeuJMc4q6dVz200XnYUtLR9GYjPXvFOVr6W1zUK1WbPToaWJJuKnxBLnd0ftDEbMmj4loHYyhZyMjM91zQS4p7z8eKa9h0JrbacekcirexG0z4n3xz0QOWSvFj3jLhWXUIU21iIAwJtI3RbWa90I7rzAIqI3UElUJG7tLtUXzw4KQNETvXzqWaujEMenYlNIzLGxgB3AuJEQ88}     \end{equation} As in \eqref{8ThswELzXU3X7Ebd1KdZ7v1rN3GiirRXGKWK099ovBM0FDJCvkopYNQ2aN94Z7k0UnUKamE3OjU8DFYFFokbSI2J9V9gVlM8ALWThDPnPu3EL7HPD2VDaZTggzcCCmbvc70qqPcC9mt60ogcrTiA3HEjwTK8ymKeuJMc4q6dVz200XnYUtLR9GYjPXvFOVr6W1zUK1WbPToaWJJuKnxBLnd0ftDEbMmj4loHYyhZyMjM91zQS4p7z8eKa9h0JrbacekcirexG0z4n3xz0QOWSvFj3jLhWXUIU21iIAwJtI3RbWa90I7rzAIqI3UElUJG7tLtUXzw4KQNETvXzqWaujEMenYlNIzLGxgB3AuJEQ68}, we have     \begin{align}     \begin{split}       \sum_{|\nu|\leq m} r^{|\nu|} \zxvczxbcvdfghasdfrtsdafasdfasdfdsfgsdgh \partial^{\nu} u \zxvczxbcvdfghasdfrtsdafasdfasdfdsfgsdgh_{L^p(B_r)}       &\les \big(c_0 + \MM   +  \zxvczxbcvdfghasdfrtsdafasdfasdfdsfgsdgh Q \zxvczxbcvdfghasdfrtsdafasdfasdfdsfgsdgh_{L^p(B_1)} \big)r^{d-1+\eta+k_0\alpha+n/p}       \comma       r \leq \frac{1}{2}       ,     \end{split}     \label{8ThswELzXU3X7Ebd1KdZ7v1rN3GiirRXGKWK099ovBM0FDJCvkopYNQ2aN94Z7k0UnUKamE3OjU8DFYFFokbSI2J9V9gVlM8ALWThDPnPu3EL7HPD2VDaZTggzcCCmbvc70qqPcC9mt60ogcrTiA3HEjwTK8ymKeuJMc4q6dVz200XnYUtLR9GYjPXvFOVr6W1zUK1WbPToaWJJuKnxBLnd0ftDEbMmj4loHYyhZyMjM91zQS4p7z8eKa9h0JrbacekcirexG0z4n3xz0QOWSvFj3jLhWXUIU21iIAwJtI3RbWa90I7rzAIqI3UElUJG7tLtUXzw4KQNETvXzqWaujEMenYlNIzLGxgB3AuJEQ89}     \end{align} while as in \eqref{8ThswELzXU3X7Ebd1KdZ7v1rN3GiirRXGKWK099ovBM0FDJCvkopYNQ2aN94Z7k0UnUKamE3OjU8DFYFFokbSI2J9V9gVlM8ALWThDPnPu3EL7HPD2VDaZTggzcCCmbvc70qqPcC9mt60ogcrTiA3HEjwTK8ymKeuJMc4q6dVz200XnYUtLR9GYjPXvFOVr6W1zUK1WbPToaWJJuKnxBLnd0ftDEbMmj4loHYyhZyMjM91zQS4p7z8eKa9h0JrbacekcirexG0z4n3xz0QOWSvFj3jLhWXUIU21iIAwJtI3RbWa90I7rzAIqI3UElUJG7tLtUXzw4KQNETvXzqWaujEMenYlNIzLGxgB3AuJEQ71} we may bound     \begin{align}     \begin{split}       \zxvczxbcvdfghasdfrtsdafasdfasdfdsfgsdgh L(0) u - Q \zxvczxbcvdfghasdfrtsdafasdfasdfdsfgsdgh_{L^p(B_r)}       &       \les \big( c_0 + \MM   + \zxvczxbcvdfghasdfrtsdafasdfasdfdsfgsdgh Q \zxvczxbcvdfghasdfrtsdafasdfasdfdsfgsdgh_{L^p(B_1)} \big)               r^{d-1+\eta+(k_0+1)\alpha-m+n/p}       \comma       r \leq \frac{1}{2}       .     \end{split}     \label{8ThswELzXU3X7Ebd1KdZ7v1rN3GiirRXGKWK099ovBM0FDJCvkopYNQ2aN94Z7k0UnUKamE3OjU8DFYFFokbSI2J9V9gVlM8ALWThDPnPu3EL7HPD2VDaZTggzcCCmbvc70qqPcC9mt60ogcrTiA3HEjwTK8ymKeuJMc4q6dVz200XnYUtLR9GYjPXvFOVr6W1zUK1WbPToaWJJuKnxBLnd0ftDEbMmj4loHYyhZyMjM91zQS4p7z8eKa9h0JrbacekcirexG0z4n3xz0QOWSvFj3jLhWXUIU21iIAwJtI3RbWa90I7rzAIqI3UElUJG7tLtUXzw4KQNETvXzqWaujEMenYlNIzLGxgB3AuJEQ90}     \end{align} Letting   \begin{equation}    w(x) = \zxvczxbcvdfghasdfrtsdafasdfasdfdsfgsdgf_{|y|\leq 1/2} \Gamma(x-y) \big( L(0) u(y) - Q(y)\big) \,dy    ,    \label{8ThswELzXU3X7Ebd1KdZ7v1rN3GiirRXGKWK099ovBM0FDJCvkopYNQ2aN94Z7k0UnUKamE3OjU8DFYFFokbSI2J9V9gVlM8ALWThDPnPu3EL7HPD2VDaZTggzcCCmbvc70qqPcC9mt60ogcrTiA3HEjwTK8ymKeuJMc4q6dVz200XnYUtLR9GYjPXvFOVr6W1zUK1WbPToaWJJuKnxBLnd0ftDEbMmj4loHYyhZyMjM91zQS4p7z8eKa9h0JrbacekcirexG0z4n3xz0QOWSvFj3jLhWXUIU21iIAwJtI3RbWa90I7rzAIqI3UElUJG7tLtUXzw4KQNETvXzqWaujEMenYlNIzLGxgB3AuJEQ222}   \end{equation} it is clear that     \begin{equation}       L(0)w       = L(0)u-Q       .     \llabel{PyM qa 8 tAK QZ9 cfP 42 kuUz V3h6 GsGFoW m9 h cfj 51d GtW yZ zC5D aVt2 Wi5IIs gD B 0cX LM1 FtE xE RIZI Z0Rt QUtWcU Cm F mSj xvW pZc gl dopk 0D7a EouRku Id O ZdW FOR uqb PY 6HkW OVi7 FuVMLW nx p SaN omk rC5 uI ZK9C jpJy UIeO6k gb 7 tr2 SCY x5F 11 S6Xq OImr s7vv0u vA g rb9 hGP Fnk RM j92H gczJ 660kHb BB l QSI OY7 FcX 0c uyDl LjbU 3F6vZk Gb a KaM ufj uxp n4 Mi45 7MoL NW3eIm cj 6 OOS e59 afA hg lt9S BOiF cYQipj 5u N 19N KZ5 Czc 23 1wxG x1ut gJB4ue Mx x 5lr s8g VbZ s1 NEfI 02Rb pkfEOZ E4 e seo 9te NRU Ai nujf eJYa Ehns0Y 6X R UF1 PCf 5eE A8ThswELzXU3X7Ebd1KdZ7v1rN3GiirRXGKWK099ovBM0FDJCvkopYNQ2aN94Z7k0UnUKamE3OjU8DFYFFokbSI2J9V9gVlM8ALWThDPnPu3EL7HPD2VDaZTggzcCCmbvc70qqPcC9mt60ogcrTiA3HEjwTK8ymKeuJMc4q6dVz200XnYUtLR9GYjPXvFOVr6W1zUK1WbPToaWJJuKnxBLnd0ftDEbMmj4loHYyhZyMjM91zQS4p7z8eKa9h0JrbacekcirexG0z4n3xz0QOWSvFj3jLhWXUIU21iIAwJtI3RbWa90I7rzAIqI3UElUJG7tLtUXzw4KQNETvXzqWaujEMenYlNIzLGxgB3AuJEQ91}     \end{equation} Since $\eta + (k_0+1)\alpha - 1 \in (0,1)$, we apply Lemma~\ref{L10} to \eqref{8ThswELzXU3X7Ebd1KdZ7v1rN3GiirRXGKWK099ovBM0FDJCvkopYNQ2aN94Z7k0UnUKamE3OjU8DFYFFokbSI2J9V9gVlM8ALWThDPnPu3EL7HPD2VDaZTggzcCCmbvc70qqPcC9mt60ogcrTiA3HEjwTK8ymKeuJMc4q6dVz200XnYUtLR9GYjPXvFOVr6W1zUK1WbPToaWJJuKnxBLnd0ftDEbMmj4loHYyhZyMjM91zQS4p7z8eKa9h0JrbacekcirexG0z4n3xz0QOWSvFj3jLhWXUIU21iIAwJtI3RbWa90I7rzAIqI3UElUJG7tLtUXzw4KQNETvXzqWaujEMenYlNIzLGxgB3AuJEQ222} and use \eqref{8ThswELzXU3X7Ebd1KdZ7v1rN3GiirRXGKWK099ovBM0FDJCvkopYNQ2aN94Z7k0UnUKamE3OjU8DFYFFokbSI2J9V9gVlM8ALWThDPnPu3EL7HPD2VDaZTggzcCCmbvc70qqPcC9mt60ogcrTiA3HEjwTK8ymKeuJMc4q6dVz200XnYUtLR9GYjPXvFOVr6W1zUK1WbPToaWJJuKnxBLnd0ftDEbMmj4loHYyhZyMjM91zQS4p7z8eKa9h0JrbacekcirexG0z4n3xz0QOWSvFj3jLhWXUIU21iIAwJtI3RbWa90I7rzAIqI3UElUJG7tLtUXzw4KQNETvXzqWaujEMenYlNIzLGxgB3AuJEQ90} to get     \begin{equation}       \zxvczxbcvdfghasdfrtsdafasdfasdfdsfgsdgh w - P_w \zxvczxbcvdfghasdfrtsdafasdfasdfdsfgsdgh_{L^p(B_r)}       \les \big( c_0 + \MM   + \zxvczxbcvdfghasdfrtsdafasdfasdfdsfgsdgh Q \zxvczxbcvdfghasdfrtsdafasdfasdfdsfgsdgh_{L^p(B_1)} \big) r^{d-1+\eta+(k_0+1)\alpha+n/p}       ,     \label{8ThswELzXU3X7Ebd1KdZ7v1rN3GiirRXGKWK099ovBM0FDJCvkopYNQ2aN94Z7k0UnUKamE3OjU8DFYFFokbSI2J9V9gVlM8ALWThDPnPu3EL7HPD2VDaZTggzcCCmbvc70qqPcC9mt60ogcrTiA3HEjwTK8ymKeuJMc4q6dVz200XnYUtLR9GYjPXvFOVr6W1zUK1WbPToaWJJuKnxBLnd0ftDEbMmj4loHYyhZyMjM91zQS4p7z8eKa9h0JrbacekcirexG0z4n3xz0QOWSvFj3jLhWXUIU21iIAwJtI3RbWa90I7rzAIqI3UElUJG7tLtUXzw4KQNETvXzqWaujEMenYlNIzLGxgB3AuJEQ92}     \end{equation} where      \begin{equation}       P_w(x)        = \sum_{|\beta| \leq d} \frac{x^{\beta}}{\beta!}\zxvczxbcvdfghasdfrtsdafasdfasdfdsfgsdgf_{|y|\leq1/2} \partial^{\beta}_x \Gamma(-y)           \big( L(0) u(y) - Q(y) \big) \,dy        .     \label{8ThswELzXU3X7Ebd1KdZ7v1rN3GiirRXGKWK099ovBM0FDJCvkopYNQ2aN94Z7k0UnUKamE3OjU8DFYFFokbSI2J9V9gVlM8ALWThDPnPu3EL7HPD2VDaZTggzcCCmbvc70qqPcC9mt60ogcrTiA3HEjwTK8ymKeuJMc4q6dVz200XnYUtLR9GYjPXvFOVr6W1zUK1WbPToaWJJuKnxBLnd0ftDEbMmj4loHYyhZyMjM91zQS4p7z8eKa9h0JrbacekcirexG0z4n3xz0QOWSvFj3jLhWXUIU21iIAwJtI3RbWa90I7rzAIqI3UElUJG7tLtUXzw4KQNETvXzqWaujEMenYlNIzLGxgB3AuJEQ93}     \end{equation} Moreover, by \eqref{8ThswELzXU3X7Ebd1KdZ7v1rN3GiirRXGKWK099ovBM0FDJCvkopYNQ2aN94Z7k0UnUKamE3OjU8DFYFFokbSI2J9V9gVlM8ALWThDPnPu3EL7HPD2VDaZTggzcCCmbvc70qqPcC9mt60ogcrTiA3HEjwTK8ymKeuJMc4q6dVz200XnYUtLR9GYjPXvFOVr6W1zUK1WbPToaWJJuKnxBLnd0ftDEbMmj4loHYyhZyMjM91zQS4p7z8eKa9h0JrbacekcirexG0z4n3xz0QOWSvFj3jLhWXUIU21iIAwJtI3RbWa90I7rzAIqI3UElUJG7tLtUXzw4KQNETvXzqWaujEMenYlNIzLGxgB3AuJEQ222},     \begin{equation}       \zxvczxbcvdfghasdfrtsdafasdfasdfdsfgsdgh w \zxvczxbcvdfghasdfrtsdafasdfasdfdsfgsdgh_{L^p(B_{1/2})}       \les \zxvczxbcvdfghasdfrtsdafasdfasdfdsfgsdgh u \zxvczxbcvdfghasdfrtsdafasdfasdfdsfgsdgh_{W^{m,p}(B_{1/2})}              + \zxvczxbcvdfghasdfrtsdafasdfasdfdsfgsdgh Q \zxvczxbcvdfghasdfrtsdafasdfasdfdsfgsdgh_{L^p(B_1)}       .     \llabel{7 FuVMLW nx p SaN omk rC5 uI ZK9C jpJy UIeO6k gb 7 tr2 SCY x5F 11 S6Xq OImr s7vv0u vA g rb9 hGP Fnk RM j92H gczJ 660kHb BB l QSI OY7 FcX 0c uyDl LjbU 3F6vZk Gb a KaM ufj uxp n4 Mi45 7MoL NW3eIm cj 6 OOS e59 afA hg lt9S BOiF cYQipj 5u N 19N KZ5 Czc 23 1wxG x1ut gJB4ue Mx x 5lr s8g VbZ s1 NEfI 02Rb pkfEOZ E4 e seo 9te NRU Ai nujf eJYa Ehns0Y 6X R UF1 PCf 5eE AL 9DL6 a2vm BAU5Au DD t yQN 5YL LWw PW GjMt 4hu4 FIoLCZ Lx e BVY 5lZ DCD 5Y yBwO IJeH VQsKob Yd q fCX 1to mCb Ej 5m1p Nx9p nLn5A3 g7 U v77 7YU gBR lN rTyj shaq BZXeAF tj y FlW jfc 8ThswELzXU3X7Ebd1KdZ7v1rN3GiirRXGKWK099ovBM0FDJCvkopYNQ2aN94Z7k0UnUKamE3OjU8DFYFFokbSI2J9V9gVlM8ALWThDPnPu3EL7HPD2VDaZTggzcCCmbvc70qqPcC9mt60ogcrTiA3HEjwTK8ymKeuJMc4q6dVz200XnYUtLR9GYjPXvFOVr6W1zUK1WbPToaWJJuKnxBLnd0ftDEbMmj4loHYyhZyMjM91zQS4p7z8eKa9h0JrbacekcirexG0z4n3xz0QOWSvFj3jLhWXUIU21iIAwJtI3RbWa90I7rzAIqI3UElUJG7tLtUXzw4KQNETvXzqWaujEMenYlNIzLGxgB3AuJEQ94}     \end{equation} Thus the function $v = u-w$ satisfies $L(0)v = Q$ and      \begin{equation}       \zxvczxbcvdfghasdfrtsdafasdfasdfdsfgsdgh v \zxvczxbcvdfghasdfrtsdafasdfasdfdsfgsdgh_{L^p(B_{1/2})}        \les \zxvczxbcvdfghasdfrtsdafasdfasdfdsfgsdgh u \zxvczxbcvdfghasdfrtsdafasdfasdfdsfgsdgh_{L^p(B_{1/2})}              + \zxvczxbcvdfghasdfrtsdafasdfasdfdsfgsdgh w \zxvczxbcvdfghasdfrtsdafasdfasdfdsfgsdgh_{L^p(B_{1/2})}       \les \zxvczxbcvdfghasdfrtsdafasdfasdfdsfgsdgh u \zxvczxbcvdfghasdfrtsdafasdfasdfdsfgsdgh_{W^{m,p}(B_{1/2})}             + \zxvczxbcvdfghasdfrtsdafasdfasdfdsfgsdgh Q \zxvczxbcvdfghasdfrtsdafasdfasdfdsfgsdgh_{L^p(B_1)}       \les  c_0 + \MM   + \zxvczxbcvdfghasdfrtsdafasdfasdfdsfgsdgh Q \zxvczxbcvdfghasdfrtsdafasdfasdfdsfgsdgh_{L^p(B_1)}       ,     \llabel{5 7MoL NW3eIm cj 6 OOS e59 afA hg lt9S BOiF cYQipj 5u N 19N KZ5 Czc 23 1wxG x1ut gJB4ue Mx x 5lr s8g VbZ s1 NEfI 02Rb pkfEOZ E4 e seo 9te NRU Ai nujf eJYa Ehns0Y 6X R UF1 PCf 5eE AL 9DL6 a2vm BAU5Au DD t yQN 5YL LWw PW GjMt 4hu4 FIoLCZ Lx e BVY 5lZ DCD 5Y yBwO IJeH VQsKob Yd q fCX 1to mCb Ej 5m1p Nx9p nLn5A3 g7 U v77 7YU gBR lN rTyj shaq BZXeAF tj y FlW jfc 57t 2f abx5 Ns4d clCMJc Tl q kfq uFD iSd DP eX6m YLQz JzUmH0 43 M lgF edN mXQ Pj Aoba 07MY wBaC4C nj I 4dw KCZ PO9 wx 3en8 AoqX 7JjN8K lq j Q5c bMS dhR Fs tQ8Q r2ve 2HT0uO 5W j TAi8ThswELzXU3X7Ebd1KdZ7v1rN3GiirRXGKWK099ovBM0FDJCvkopYNQ2aN94Z7k0UnUKamE3OjU8DFYFFokbSI2J9V9gVlM8ALWThDPnPu3EL7HPD2VDaZTggzcCCmbvc70qqPcC9mt60ogcrTiA3HEjwTK8ymKeuJMc4q6dVz200XnYUtLR9GYjPXvFOVr6W1zUK1WbPToaWJJuKnxBLnd0ftDEbMmj4loHYyhZyMjM91zQS4p7z8eKa9h0JrbacekcirexG0z4n3xz0QOWSvFj3jLhWXUIU21iIAwJtI3RbWa90I7rzAIqI3UElUJG7tLtUXzw4KQNETvXzqWaujEMenYlNIzLGxgB3AuJEQ95}     \end{equation} where we used \eqref{8ThswELzXU3X7Ebd1KdZ7v1rN3GiirRXGKWK099ovBM0FDJCvkopYNQ2aN94Z7k0UnUKamE3OjU8DFYFFokbSI2J9V9gVlM8ALWThDPnPu3EL7HPD2VDaZTggzcCCmbvc70qqPcC9mt60ogcrTiA3HEjwTK8ymKeuJMc4q6dVz200XnYUtLR9GYjPXvFOVr6W1zUK1WbPToaWJJuKnxBLnd0ftDEbMmj4loHYyhZyMjM91zQS4p7z8eKa9h0JrbacekcirexG0z4n3xz0QOWSvFj3jLhWXUIU21iIAwJtI3RbWa90I7rzAIqI3UElUJG7tLtUXzw4KQNETvXzqWaujEMenYlNIzLGxgB3AuJEQ89} in the last inequality. Let $P_v$ be the Taylor polynomial of $v$ of degree~$d$. Then  $L(0) P_v=Q$ and for $|x|=r\leq 1/4$, we have     \begin{align}     \begin{split}       | v(x) - P_v(x) |       &\les \zxvczxbcvdfghasdfrtsdafasdfasdfdsfgsdgh \partial^{d+1} v \zxvczxbcvdfghasdfrtsdafasdfasdfdsfgsdgh_{L^{\infty} (B_r)} r^{d+1}       \les \zxvczxbcvdfghasdfrtsdafasdfasdfdsfgsdgh \partial^{d+1} v \zxvczxbcvdfghasdfrtsdafasdfasdfdsfgsdgh_{L^{\infty} (B_{1/4})} r^{d+1}       \\&       \les \big( \zxvczxbcvdfghasdfrtsdafasdfasdfdsfgsdgh Q \zxvczxbcvdfghasdfrtsdafasdfasdfdsfgsdgh_{L^p(B_{1/2})}                       + \zxvczxbcvdfghasdfrtsdafasdfasdfdsfgsdgh v \zxvczxbcvdfghasdfrtsdafasdfasdfdsfgsdgh_{L^p(B_{1/2})}
             \big)               r^{d+1}       \les \big( c_0 + \MM   + \zxvczxbcvdfghasdfrtsdafasdfasdfdsfgsdgh Q \zxvczxbcvdfghasdfrtsdafasdfasdfdsfgsdgh_{L^p(B_1)} \big)               r^{d+1}       ,     \end{split}     \llabel{L 9DL6 a2vm BAU5Au DD t yQN 5YL LWw PW GjMt 4hu4 FIoLCZ Lx e BVY 5lZ DCD 5Y yBwO IJeH VQsKob Yd q fCX 1to mCb Ej 5m1p Nx9p nLn5A3 g7 U v77 7YU gBR lN rTyj shaq BZXeAF tj y FlW jfc 57t 2f abx5 Ns4d clCMJc Tl q kfq uFD iSd DP eX6m YLQz JzUmH0 43 M lgF edN mXQ Pj Aoba 07MY wBaC4C nj I 4dw KCZ PO9 wx 3en8 AoqX 7JjN8K lq j Q5c bMS dhR Fs tQ8Q r2ve 2HT0uO 5W j TAi iIW n1C Wr U1BH BMvJ 3ywmAd qN D LY8 lbx XMx 0D Dvco 3RL9 Qz5eqy wV Y qEN nO8 MH0 PY zeVN i3yb 2msNYY Wz G 2DC PoG 1Vb Bx e9oZ GcTU 3AZuEK bk p 6rN eTX 0DS Mc zd91 nbSV DKEkVa zI 8ThswELzXU3X7Ebd1KdZ7v1rN3GiirRXGKWK099ovBM0FDJCvkopYNQ2aN94Z7k0UnUKamE3OjU8DFYFFokbSI2J9V9gVlM8ALWThDPnPu3EL7HPD2VDaZTggzcCCmbvc70qqPcC9mt60ogcrTiA3HEjwTK8ymKeuJMc4q6dVz200XnYUtLR9GYjPXvFOVr6W1zUK1WbPToaWJJuKnxBLnd0ftDEbMmj4loHYyhZyMjM91zQS4p7z8eKa9h0JrbacekcirexG0z4n3xz0QOWSvFj3jLhWXUIU21iIAwJtI3RbWa90I7rzAIqI3UElUJG7tLtUXzw4KQNETvXzqWaujEMenYlNIzLGxgB3AuJEQ96}     \end{align} where we used the Taylor theorem in the first step. Therefore,       \begin{equation}       \zxvczxbcvdfghasdfrtsdafasdfasdfdsfgsdgh v - P_v \zxvczxbcvdfghasdfrtsdafasdfasdfdsfgsdgh_{L^p(B_r)}        \les \big( c_0 + \MM   + \zxvczxbcvdfghasdfrtsdafasdfasdfdsfgsdgh Q \zxvczxbcvdfghasdfrtsdafasdfasdfdsfgsdgh_{L^p(B_1)} \big)               r^{d+1+n/p}       .     \label{8ThswELzXU3X7Ebd1KdZ7v1rN3GiirRXGKWK099ovBM0FDJCvkopYNQ2aN94Z7k0UnUKamE3OjU8DFYFFokbSI2J9V9gVlM8ALWThDPnPu3EL7HPD2VDaZTggzcCCmbvc70qqPcC9mt60ogcrTiA3HEjwTK8ymKeuJMc4q6dVz200XnYUtLR9GYjPXvFOVr6W1zUK1WbPToaWJJuKnxBLnd0ftDEbMmj4loHYyhZyMjM91zQS4p7z8eKa9h0JrbacekcirexG0z4n3xz0QOWSvFj3jLhWXUIU21iIAwJtI3RbWa90I7rzAIqI3UElUJG7tLtUXzw4KQNETvXzqWaujEMenYlNIzLGxgB3AuJEQ97}     \end{equation} Letting $\tilde P = P_v+P_w$, we have $u-\tilde P = (v-P_v) + (w - P_w)$. Thus, by \eqref{8ThswELzXU3X7Ebd1KdZ7v1rN3GiirRXGKWK099ovBM0FDJCvkopYNQ2aN94Z7k0UnUKamE3OjU8DFYFFokbSI2J9V9gVlM8ALWThDPnPu3EL7HPD2VDaZTggzcCCmbvc70qqPcC9mt60ogcrTiA3HEjwTK8ymKeuJMc4q6dVz200XnYUtLR9GYjPXvFOVr6W1zUK1WbPToaWJJuKnxBLnd0ftDEbMmj4loHYyhZyMjM91zQS4p7z8eKa9h0JrbacekcirexG0z4n3xz0QOWSvFj3jLhWXUIU21iIAwJtI3RbWa90I7rzAIqI3UElUJG7tLtUXzw4KQNETvXzqWaujEMenYlNIzLGxgB3AuJEQ92} and \eqref{8ThswELzXU3X7Ebd1KdZ7v1rN3GiirRXGKWK099ovBM0FDJCvkopYNQ2aN94Z7k0UnUKamE3OjU8DFYFFokbSI2J9V9gVlM8ALWThDPnPu3EL7HPD2VDaZTggzcCCmbvc70qqPcC9mt60ogcrTiA3HEjwTK8ymKeuJMc4q6dVz200XnYUtLR9GYjPXvFOVr6W1zUK1WbPToaWJJuKnxBLnd0ftDEbMmj4loHYyhZyMjM91zQS4p7z8eKa9h0JrbacekcirexG0z4n3xz0QOWSvFj3jLhWXUIU21iIAwJtI3RbWa90I7rzAIqI3UElUJG7tLtUXzw4KQNETvXzqWaujEMenYlNIzLGxgB3AuJEQ97},     \begin{align}     \begin{split}       \zxvczxbcvdfghasdfrtsdafasdfasdfdsfgsdgh u - \tilde P \zxvczxbcvdfghasdfrtsdafasdfasdfdsfgsdgh_{L^p(B_r)}       &\les \big( c_0 + \MM   + \zxvczxbcvdfghasdfrtsdafasdfasdfdsfgsdgh Q \zxvczxbcvdfghasdfrtsdafasdfasdfdsfgsdgh_{L^p(B_1)} \big)              r^{d+\min ( (k_0+1)\alpha+ \eta-1,1) + n/p}       \\& \les \big( c_0 + \MM   + \zxvczxbcvdfghasdfrtsdafasdfasdfdsfgsdgh Q \zxvczxbcvdfghasdfrtsdafasdfasdfdsfgsdgh_{L^p(B_1)} \big)              r^{d + (k_0+1)\alpha+\eta-1+n/p}        ,     \end{split}     \label{8ThswELzXU3X7Ebd1KdZ7v1rN3GiirRXGKWK099ovBM0FDJCvkopYNQ2aN94Z7k0UnUKamE3OjU8DFYFFokbSI2J9V9gVlM8ALWThDPnPu3EL7HPD2VDaZTggzcCCmbvc70qqPcC9mt60ogcrTiA3HEjwTK8ymKeuJMc4q6dVz200XnYUtLR9GYjPXvFOVr6W1zUK1WbPToaWJJuKnxBLnd0ftDEbMmj4loHYyhZyMjM91zQS4p7z8eKa9h0JrbacekcirexG0z4n3xz0QOWSvFj3jLhWXUIU21iIAwJtI3RbWa90I7rzAIqI3UElUJG7tLtUXzw4KQNETvXzqWaujEMenYlNIzLGxgB3AuJEQ99}     \end{align} where we used $(k_0+1)\alpha + \eta -1 < 1$. Since the degree of $\tilde P$ is less than or equal to $d$, we may combine \eqref{8ThswELzXU3X7Ebd1KdZ7v1rN3GiirRXGKWK099ovBM0FDJCvkopYNQ2aN94Z7k0UnUKamE3OjU8DFYFFokbSI2J9V9gVlM8ALWThDPnPu3EL7HPD2VDaZTggzcCCmbvc70qqPcC9mt60ogcrTiA3HEjwTK8ymKeuJMc4q6dVz200XnYUtLR9GYjPXvFOVr6W1zUK1WbPToaWJJuKnxBLnd0ftDEbMmj4loHYyhZyMjM91zQS4p7z8eKa9h0JrbacekcirexG0z4n3xz0QOWSvFj3jLhWXUIU21iIAwJtI3RbWa90I7rzAIqI3UElUJG7tLtUXzw4KQNETvXzqWaujEMenYlNIzLGxgB3AuJEQ14} with \eqref{8ThswELzXU3X7Ebd1KdZ7v1rN3GiirRXGKWK099ovBM0FDJCvkopYNQ2aN94Z7k0UnUKamE3OjU8DFYFFokbSI2J9V9gVlM8ALWThDPnPu3EL7HPD2VDaZTggzcCCmbvc70qqPcC9mt60ogcrTiA3HEjwTK8ymKeuJMc4q6dVz200XnYUtLR9GYjPXvFOVr6W1zUK1WbPToaWJJuKnxBLnd0ftDEbMmj4loHYyhZyMjM91zQS4p7z8eKa9h0JrbacekcirexG0z4n3xz0QOWSvFj3jLhWXUIU21iIAwJtI3RbWa90I7rzAIqI3UElUJG7tLtUXzw4KQNETvXzqWaujEMenYlNIzLGxgB3AuJEQ99}  and write     \begin{equation}     \zxvczxbcvdfghasdfrtsdafasdfasdfdsfgsdgh \tilde P \zxvczxbcvdfghasdfrtsdafasdfasdfdsfgsdgh_{L^p(B_r)}     \leq \zxvczxbcvdfghasdfrtsdafasdfasdfdsfgsdgh u - \tilde P \zxvczxbcvdfghasdfrtsdafasdfasdfdsfgsdgh_{L^p(B_r)}             + \zxvczxbcvdfghasdfrtsdafasdfasdfdsfgsdgh u \zxvczxbcvdfghasdfrtsdafasdfasdfdsfgsdgh_{L^p(B_r)}     \les r^{d+(k_0+1)\alpha + \eta -1 + n/p}     ,    \llabel{57t 2f abx5 Ns4d clCMJc Tl q kfq uFD iSd DP eX6m YLQz JzUmH0 43 M lgF edN mXQ Pj Aoba 07MY wBaC4C nj I 4dw KCZ PO9 wx 3en8 AoqX 7JjN8K lq j Q5c bMS dhR Fs tQ8Q r2ve 2HT0uO 5W j TAi iIW n1C Wr U1BH BMvJ 3ywmAd qN D LY8 lbx XMx 0D Dvco 3RL9 Qz5eqy wV Y qEN nO8 MH0 PY zeVN i3yb 2msNYY Wz G 2DC PoG 1Vb Bx e9oZ GcTU 3AZuEK bk p 6rN eTX 0DS Mc zd91 nbSV DKEkVa zI q NKU Qap NBP 5B 32Ey prwP FLvuPi wR P l1G TdQ BZE Aw 3d90 v8P5 CPAnX4 Yo 2 q7s yr5 BW8 Hc T7tM ioha BW9U4q rb u mEQ 6Xz MKR 2B REFX k3ZO MVMYSw 9S F 5ek q0m yNK Gn H0qi vlRA 18CbE8ThswELzXU3X7Ebd1KdZ7v1rN3GiirRXGKWK099ovBM0FDJCvkopYNQ2aN94Z7k0UnUKamE3OjU8DFYFFokbSI2J9V9gVlM8ALWThDPnPu3EL7HPD2VDaZTggzcCCmbvc70qqPcC9mt60ogcrTiA3HEjwTK8ymKeuJMc4q6dVz200XnYUtLR9GYjPXvFOVr6W1zUK1WbPToaWJJuKnxBLnd0ftDEbMmj4loHYyhZyMjM91zQS4p7z8eKa9h0JrbacekcirexG0z4n3xz0QOWSvFj3jLhWXUIU21iIAwJtI3RbWa90I7rzAIqI3UElUJG7tLtUXzw4KQNETvXzqWaujEMenYlNIzLGxgB3AuJEQ100}     \end{equation} implying~$\tilde P \equiv 0$. Therefore, from \eqref{8ThswELzXU3X7Ebd1KdZ7v1rN3GiirRXGKWK099ovBM0FDJCvkopYNQ2aN94Z7k0UnUKamE3OjU8DFYFFokbSI2J9V9gVlM8ALWThDPnPu3EL7HPD2VDaZTggzcCCmbvc70qqPcC9mt60ogcrTiA3HEjwTK8ymKeuJMc4q6dVz200XnYUtLR9GYjPXvFOVr6W1zUK1WbPToaWJJuKnxBLnd0ftDEbMmj4loHYyhZyMjM91zQS4p7z8eKa9h0JrbacekcirexG0z4n3xz0QOWSvFj3jLhWXUIU21iIAwJtI3RbWa90I7rzAIqI3UElUJG7tLtUXzw4KQNETvXzqWaujEMenYlNIzLGxgB3AuJEQ99},   \begin{equation}       \zxvczxbcvdfghasdfrtsdafasdfasdfdsfgsdgh u \zxvczxbcvdfghasdfrtsdafasdfasdfdsfgsdgh_{L^p(B_r)}       \les \big( c_0 + \MM   + \zxvczxbcvdfghasdfrtsdafasdfasdfdsfgsdgh Q \zxvczxbcvdfghasdfrtsdafasdfasdfdsfgsdgh_{L^p(B_1)} \big)              r^{d + (k_0+1)\alpha+\eta-1+n/p}       \les \big( c_0 + \MM   + \zxvczxbcvdfghasdfrtsdafasdfasdfdsfgsdgh Q \zxvczxbcvdfghasdfrtsdafasdfasdfdsfgsdgh_{L^p(B_1)} \big)              r^{d+n/p}              .     \label{8ThswELzXU3X7Ebd1KdZ7v1rN3GiirRXGKWK099ovBM0FDJCvkopYNQ2aN94Z7k0UnUKamE3OjU8DFYFFokbSI2J9V9gVlM8ALWThDPnPu3EL7HPD2VDaZTggzcCCmbvc70qqPcC9mt60ogcrTiA3HEjwTK8ymKeuJMc4q6dVz200XnYUtLR9GYjPXvFOVr6W1zUK1WbPToaWJJuKnxBLnd0ftDEbMmj4loHYyhZyMjM91zQS4p7z8eKa9h0JrbacekcirexG0z4n3xz0QOWSvFj3jLhWXUIU21iIAwJtI3RbWa90I7rzAIqI3UElUJG7tLtUXzw4KQNETvXzqWaujEMenYlNIzLGxgB3AuJEQ101}   \end{equation} \par We now construct the homogeneous polynomial~$P$.  Applying the classical elliptic interior estimate  to \eqref{8ThswELzXU3X7Ebd1KdZ7v1rN3GiirRXGKWK099ovBM0FDJCvkopYNQ2aN94Z7k0UnUKamE3OjU8DFYFFokbSI2J9V9gVlM8ALWThDPnPu3EL7HPD2VDaZTggzcCCmbvc70qqPcC9mt60ogcrTiA3HEjwTK8ymKeuJMc4q6dVz200XnYUtLR9GYjPXvFOVr6W1zUK1WbPToaWJJuKnxBLnd0ftDEbMmj4loHYyhZyMjM91zQS4p7z8eKa9h0JrbacekcirexG0z4n3xz0QOWSvFj3jLhWXUIU21iIAwJtI3RbWa90I7rzAIqI3UElUJG7tLtUXzw4KQNETvXzqWaujEMenYlNIzLGxgB3AuJEQ56} and using \eqref{8ThswELzXU3X7Ebd1KdZ7v1rN3GiirRXGKWK099ovBM0FDJCvkopYNQ2aN94Z7k0UnUKamE3OjU8DFYFFokbSI2J9V9gVlM8ALWThDPnPu3EL7HPD2VDaZTggzcCCmbvc70qqPcC9mt60ogcrTiA3HEjwTK8ymKeuJMc4q6dVz200XnYUtLR9GYjPXvFOVr6W1zUK1WbPToaWJJuKnxBLnd0ftDEbMmj4loHYyhZyMjM91zQS4p7z8eKa9h0JrbacekcirexG0z4n3xz0QOWSvFj3jLhWXUIU21iIAwJtI3RbWa90I7rzAIqI3UElUJG7tLtUXzw4KQNETvXzqWaujEMenYlNIzLGxgB3AuJEQ101}, we have     \begin{align}     \begin{split}       \sum_{|\nu|\leq m} r^{|\nu|} \zxvczxbcvdfghasdfrtsdafasdfasdfdsfgsdgh \partial^{\nu} u \zxvczxbcvdfghasdfrtsdafasdfasdfdsfgsdgh_{L^p(B_r)}       &\les \zxvczxbcvdfghasdfrtsdafasdfasdfdsfgsdgh u \zxvczxbcvdfghasdfrtsdafasdfasdfdsfgsdgh_{L^p(B_{2r})}              + r^m \zxvczxbcvdfghasdfrtsdafasdfasdfdsfgsdgh f \zxvczxbcvdfghasdfrtsdafasdfasdfdsfgsdgh_{L^p(B_{2r})}       \les \zxvczxbcvdfghasdfrtsdafasdfasdfdsfgsdgh u \zxvczxbcvdfghasdfrtsdafasdfasdfdsfgsdgh_{L^p(B_{2r})}             + r^m \zxvczxbcvdfghasdfrtsdafasdfasdfdsfgsdgh f - Q \zxvczxbcvdfghasdfrtsdafasdfasdfdsfgsdgh_{L^p(B_{2r})}             + r^m \zxvczxbcvdfghasdfrtsdafasdfasdfdsfgsdgh Q \zxvczxbcvdfghasdfrtsdafasdfasdfdsfgsdgh_{L^p(B_{2r})}       \\&       \les \big(c_0 + \MM   +  \zxvczxbcvdfghasdfrtsdafasdfasdfdsfgsdgh Q \zxvczxbcvdfghasdfrtsdafasdfasdfdsfgsdgh_{L^p(B_1)} \big)r^{d + n/p}       \comma       r \leq \frac{1}{2}       ,     \end{split}     \label{8ThswELzXU3X7Ebd1KdZ7v1rN3GiirRXGKWK099ovBM0FDJCvkopYNQ2aN94Z7k0UnUKamE3OjU8DFYFFokbSI2J9V9gVlM8ALWThDPnPu3EL7HPD2VDaZTggzcCCmbvc70qqPcC9mt60ogcrTiA3HEjwTK8ymKeuJMc4q6dVz200XnYUtLR9GYjPXvFOVr6W1zUK1WbPToaWJJuKnxBLnd0ftDEbMmj4loHYyhZyMjM91zQS4p7z8eKa9h0JrbacekcirexG0z4n3xz0QOWSvFj3jLhWXUIU21iIAwJtI3RbWa90I7rzAIqI3UElUJG7tLtUXzw4KQNETvXzqWaujEMenYlNIzLGxgB3AuJEQ102}     \end{align} where we used \eqref{8ThswELzXU3X7Ebd1KdZ7v1rN3GiirRXGKWK099ovBM0FDJCvkopYNQ2aN94Z7k0UnUKamE3OjU8DFYFFokbSI2J9V9gVlM8ALWThDPnPu3EL7HPD2VDaZTggzcCCmbvc70qqPcC9mt60ogcrTiA3HEjwTK8ymKeuJMc4q6dVz200XnYUtLR9GYjPXvFOVr6W1zUK1WbPToaWJJuKnxBLnd0ftDEbMmj4loHYyhZyMjM91zQS4p7z8eKa9h0JrbacekcirexG0z4n3xz0QOWSvFj3jLhWXUIU21iIAwJtI3RbWa90I7rzAIqI3UElUJG7tLtUXzw4KQNETvXzqWaujEMenYlNIzLGxgB3AuJEQ12}, \eqref{8ThswELzXU3X7Ebd1KdZ7v1rN3GiirRXGKWK099ovBM0FDJCvkopYNQ2aN94Z7k0UnUKamE3OjU8DFYFFokbSI2J9V9gVlM8ALWThDPnPu3EL7HPD2VDaZTggzcCCmbvc70qqPcC9mt60ogcrTiA3HEjwTK8ymKeuJMc4q6dVz200XnYUtLR9GYjPXvFOVr6W1zUK1WbPToaWJJuKnxBLnd0ftDEbMmj4loHYyhZyMjM91zQS4p7z8eKa9h0JrbacekcirexG0z4n3xz0QOWSvFj3jLhWXUIU21iIAwJtI3RbWa90I7rzAIqI3UElUJG7tLtUXzw4KQNETvXzqWaujEMenYlNIzLGxgB3AuJEQ101}, and that $Q$ is a homogeneous of degree $d-m$ in the last inequality. Taking the $L^p$-norm of \eqref{8ThswELzXU3X7Ebd1KdZ7v1rN3GiirRXGKWK099ovBM0FDJCvkopYNQ2aN94Z7k0UnUKamE3OjU8DFYFFokbSI2J9V9gVlM8ALWThDPnPu3EL7HPD2VDaZTggzcCCmbvc70qqPcC9mt60ogcrTiA3HEjwTK8ymKeuJMc4q6dVz200XnYUtLR9GYjPXvFOVr6W1zUK1WbPToaWJJuKnxBLnd0ftDEbMmj4loHYyhZyMjM91zQS4p7z8eKa9h0JrbacekcirexG0z4n3xz0QOWSvFj3jLhWXUIU21iIAwJtI3RbWa90I7rzAIqI3UElUJG7tLtUXzw4KQNETvXzqWaujEMenYlNIzLGxgB3AuJEQ69} and combining with \eqref{8ThswELzXU3X7Ebd1KdZ7v1rN3GiirRXGKWK099ovBM0FDJCvkopYNQ2aN94Z7k0UnUKamE3OjU8DFYFFokbSI2J9V9gVlM8ALWThDPnPu3EL7HPD2VDaZTggzcCCmbvc70qqPcC9mt60ogcrTiA3HEjwTK8ymKeuJMc4q6dVz200XnYUtLR9GYjPXvFOVr6W1zUK1WbPToaWJJuKnxBLnd0ftDEbMmj4loHYyhZyMjM91zQS4p7z8eKa9h0JrbacekcirexG0z4n3xz0QOWSvFj3jLhWXUIU21iIAwJtI3RbWa90I7rzAIqI3UElUJG7tLtUXzw4KQNETvXzqWaujEMenYlNIzLGxgB3AuJEQ102}, we obtain     \begin{align}     \begin{split}       \zxvczxbcvdfghasdfrtsdafasdfasdfdsfgsdgh L(0) u - Q \zxvczxbcvdfghasdfrtsdafasdfasdfdsfgsdgh_{L^p(B_r)}       &\les \sum_{|\nu|=m} |a_{\nu}(0) - a_{\nu}(x)| \zxvczxbcvdfghasdfrtsdafasdfasdfdsfgsdgh \partial^{\nu} u \zxvczxbcvdfghasdfrtsdafasdfasdfdsfgsdgh_{L^p(B_r)}              + \sum_{|\nu|<m} |a_{\nu}(x)| \zxvczxbcvdfghasdfrtsdafasdfasdfdsfgsdgh \partial^{\nu} u \zxvczxbcvdfghasdfrtsdafasdfasdfdsfgsdgh_{L^p(B_r)}              + \zxvczxbcvdfghasdfrtsdafasdfasdfdsfgsdgh f-Q \zxvczxbcvdfghasdfrtsdafasdfasdfdsfgsdgh_{L^p(B_{2r})}       \\&       \les \big( c_0 + \MM   + \zxvczxbcvdfghasdfrtsdafasdfasdfdsfgsdgh Q \zxvczxbcvdfghasdfrtsdafasdfasdfdsfgsdgh_{L^p(B_1)} \big)               r^{d+\alpha+n/p}       \comma       r \leq \frac{1}{2}       ,     \end{split}     \label{8ThswELzXU3X7Ebd1KdZ7v1rN3GiirRXGKWK099ovBM0FDJCvkopYNQ2aN94Z7k0UnUKamE3OjU8DFYFFokbSI2J9V9gVlM8ALWThDPnPu3EL7HPD2VDaZTggzcCCmbvc70qqPcC9mt60ogcrTiA3HEjwTK8ymKeuJMc4q6dVz200XnYUtLR9GYjPXvFOVr6W1zUK1WbPToaWJJuKnxBLnd0ftDEbMmj4loHYyhZyMjM91zQS4p7z8eKa9h0JrbacekcirexG0z4n3xz0QOWSvFj3jLhWXUIU21iIAwJtI3RbWa90I7rzAIqI3UElUJG7tLtUXzw4KQNETvXzqWaujEMenYlNIzLGxgB3AuJEQ105}     \end{align} where we used \eqref{8ThswELzXU3X7Ebd1KdZ7v1rN3GiirRXGKWK099ovBM0FDJCvkopYNQ2aN94Z7k0UnUKamE3OjU8DFYFFokbSI2J9V9gVlM8ALWThDPnPu3EL7HPD2VDaZTggzcCCmbvc70qqPcC9mt60ogcrTiA3HEjwTK8ymKeuJMc4q6dVz200XnYUtLR9GYjPXvFOVr6W1zUK1WbPToaWJJuKnxBLnd0ftDEbMmj4loHYyhZyMjM91zQS4p7z8eKa9h0JrbacekcirexG0z4n3xz0QOWSvFj3jLhWXUIU21iIAwJtI3RbWa90I7rzAIqI3UElUJG7tLtUXzw4KQNETvXzqWaujEMenYlNIzLGxgB3AuJEQ10} and \eqref{8ThswELzXU3X7Ebd1KdZ7v1rN3GiirRXGKWK099ovBM0FDJCvkopYNQ2aN94Z7k0UnUKamE3OjU8DFYFFokbSI2J9V9gVlM8ALWThDPnPu3EL7HPD2VDaZTggzcCCmbvc70qqPcC9mt60ogcrTiA3HEjwTK8ymKeuJMc4q6dVz200XnYUtLR9GYjPXvFOVr6W1zUK1WbPToaWJJuKnxBLnd0ftDEbMmj4loHYyhZyMjM91zQS4p7z8eKa9h0JrbacekcirexG0z4n3xz0QOWSvFj3jLhWXUIU21iIAwJtI3RbWa90I7rzAIqI3UElUJG7tLtUXzw4KQNETvXzqWaujEMenYlNIzLGxgB3AuJEQ11} in the last inequality.  Letting   \begin{equation}    w(x) = \zxvczxbcvdfghasdfrtsdafasdfasdfdsfgsdgf_{|y|\leq1/2} \Gamma(x-y) \big( L(0) u(y) - Q(y)\big) \,dy     ,       \label{8ThswELzXU3X7Ebd1KdZ7v1rN3GiirRXGKWK099ovBM0FDJCvkopYNQ2aN94Z7k0UnUKamE3OjU8DFYFFokbSI2J9V9gVlM8ALWThDPnPu3EL7HPD2VDaZTggzcCCmbvc70qqPcC9mt60ogcrTiA3HEjwTK8ymKeuJMc4q6dVz200XnYUtLR9GYjPXvFOVr6W1zUK1WbPToaWJJuKnxBLnd0ftDEbMmj4loHYyhZyMjM91zQS4p7z8eKa9h0JrbacekcirexG0z4n3xz0QOWSvFj3jLhWXUIU21iIAwJtI3RbWa90I7rzAIqI3UElUJG7tLtUXzw4KQNETvXzqWaujEMenYlNIzLGxgB3AuJEQ63}   \end{equation} it is clear that     \begin{equation}       L(0)w       = L(0)u-Q       .     \label{8ThswELzXU3X7Ebd1KdZ7v1rN3GiirRXGKWK099ovBM0FDJCvkopYNQ2aN94Z7k0UnUKamE3OjU8DFYFFokbSI2J9V9gVlM8ALWThDPnPu3EL7HPD2VDaZTggzcCCmbvc70qqPcC9mt60ogcrTiA3HEjwTK8ymKeuJMc4q6dVz200XnYUtLR9GYjPXvFOVr6W1zUK1WbPToaWJJuKnxBLnd0ftDEbMmj4loHYyhZyMjM91zQS4p7z8eKa9h0JrbacekcirexG0z4n3xz0QOWSvFj3jLhWXUIU21iIAwJtI3RbWa90I7rzAIqI3UElUJG7tLtUXzw4KQNETvXzqWaujEMenYlNIzLGxgB3AuJEQ106}     \end{equation} Since $\alpha \in (0,1)$, we may apply Lemma~\ref{L10} for \eqref{8ThswELzXU3X7Ebd1KdZ7v1rN3GiirRXGKWK099ovBM0FDJCvkopYNQ2aN94Z7k0UnUKamE3OjU8DFYFFokbSI2J9V9gVlM8ALWThDPnPu3EL7HPD2VDaZTggzcCCmbvc70qqPcC9mt60ogcrTiA3HEjwTK8ymKeuJMc4q6dVz200XnYUtLR9GYjPXvFOVr6W1zUK1WbPToaWJJuKnxBLnd0ftDEbMmj4loHYyhZyMjM91zQS4p7z8eKa9h0JrbacekcirexG0z4n3xz0QOWSvFj3jLhWXUIU21iIAwJtI3RbWa90I7rzAIqI3UElUJG7tLtUXzw4KQNETvXzqWaujEMenYlNIzLGxgB3AuJEQ106} and use \eqref{8ThswELzXU3X7Ebd1KdZ7v1rN3GiirRXGKWK099ovBM0FDJCvkopYNQ2aN94Z7k0UnUKamE3OjU8DFYFFokbSI2J9V9gVlM8ALWThDPnPu3EL7HPD2VDaZTggzcCCmbvc70qqPcC9mt60ogcrTiA3HEjwTK8ymKeuJMc4q6dVz200XnYUtLR9GYjPXvFOVr6W1zUK1WbPToaWJJuKnxBLnd0ftDEbMmj4loHYyhZyMjM91zQS4p7z8eKa9h0JrbacekcirexG0z4n3xz0QOWSvFj3jLhWXUIU21iIAwJtI3RbWa90I7rzAIqI3UElUJG7tLtUXzw4KQNETvXzqWaujEMenYlNIzLGxgB3AuJEQ105} to get     \begin{equation}       \zxvczxbcvdfghasdfrtsdafasdfasdfdsfgsdgh w - P_w \zxvczxbcvdfghasdfrtsdafasdfasdfdsfgsdgh_{L^p(B_r)}       \les \big( c_0 + \MM   + \zxvczxbcvdfghasdfrtsdafasdfasdfdsfgsdgh Q \zxvczxbcvdfghasdfrtsdafasdfasdfdsfgsdgh_{L^p(B_1)} \big) r^{d+\alpha+n/p}       ,     \llabel{ iIW n1C Wr U1BH BMvJ 3ywmAd qN D LY8 lbx XMx 0D Dvco 3RL9 Qz5eqy wV Y qEN nO8 MH0 PY zeVN i3yb 2msNYY Wz G 2DC PoG 1Vb Bx e9oZ GcTU 3AZuEK bk p 6rN eTX 0DS Mc zd91 nbSV DKEkVa zI q NKU Qap NBP 5B 32Ey prwP FLvuPi wR P l1G TdQ BZE Aw 3d90 v8P5 CPAnX4 Yo 2 q7s yr5 BW8 Hc T7tM ioha BW9U4q rb u mEQ 6Xz MKR 2B REFX k3ZO MVMYSw 9S F 5ek q0m yNK Gn H0qi vlRA 18CbEz id O iuy ZZ6 kRo oJ kLQ0 Ewmz sKlld6 Kr K JmR xls 12K G2 bv8v LxfJ wrIcU6 Hx p q6p Fy7 Oim mo dXYt Kt0V VH22OC Aj f deT BAP vPl oK QzLE OQlq dpzxJ6 JI z Ujn TqY sQ4 BD QPW6 784x 8ThswELzXU3X7Ebd1KdZ7v1rN3GiirRXGKWK099ovBM0FDJCvkopYNQ2aN94Z7k0UnUKamE3OjU8DFYFFokbSI2J9V9gVlM8ALWThDPnPu3EL7HPD2VDaZTggzcCCmbvc70qqPcC9mt60ogcrTiA3HEjwTK8ymKeuJMc4q6dVz200XnYUtLR9GYjPXvFOVr6W1zUK1WbPToaWJJuKnxBLnd0ftDEbMmj4loHYyhZyMjM91zQS4p7z8eKa9h0JrbacekcirexG0z4n3xz0QOWSvFj3jLhWXUIU21iIAwJtI3RbWa90I7rzAIqI3UElUJG7tLtUXzw4KQNETvXzqWaujEMenYlNIzLGxgB3AuJEQ107}     \end{equation} where $P_w$ is defined in~\eqref{8ThswELzXU3X7Ebd1KdZ7v1rN3GiirRXGKWK099ovBM0FDJCvkopYNQ2aN94Z7k0UnUKamE3OjU8DFYFFokbSI2J9V9gVlM8ALWThDPnPu3EL7HPD2VDaZTggzcCCmbvc70qqPcC9mt60ogcrTiA3HEjwTK8ymKeuJMc4q6dVz200XnYUtLR9GYjPXvFOVr6W1zUK1WbPToaWJJuKnxBLnd0ftDEbMmj4loHYyhZyMjM91zQS4p7z8eKa9h0JrbacekcirexG0z4n3xz0QOWSvFj3jLhWXUIU21iIAwJtI3RbWa90I7rzAIqI3UElUJG7tLtUXzw4KQNETvXzqWaujEMenYlNIzLGxgB3AuJEQ93}. Moreover, by \eqref{8ThswELzXU3X7Ebd1KdZ7v1rN3GiirRXGKWK099ovBM0FDJCvkopYNQ2aN94Z7k0UnUKamE3OjU8DFYFFokbSI2J9V9gVlM8ALWThDPnPu3EL7HPD2VDaZTggzcCCmbvc70qqPcC9mt60ogcrTiA3HEjwTK8ymKeuJMc4q6dVz200XnYUtLR9GYjPXvFOVr6W1zUK1WbPToaWJJuKnxBLnd0ftDEbMmj4loHYyhZyMjM91zQS4p7z8eKa9h0JrbacekcirexG0z4n3xz0QOWSvFj3jLhWXUIU21iIAwJtI3RbWa90I7rzAIqI3UElUJG7tLtUXzw4KQNETvXzqWaujEMenYlNIzLGxgB3AuJEQ63},     \begin{equation}       \zxvczxbcvdfghasdfrtsdafasdfasdfdsfgsdgh w \zxvczxbcvdfghasdfrtsdafasdfasdfdsfgsdgh_{L^p(B_{1/2})}       \les \zxvczxbcvdfghasdfrtsdafasdfasdfdsfgsdgh u \zxvczxbcvdfghasdfrtsdafasdfasdfdsfgsdgh_{W^{m,p}(B_{1/2})}              + \zxvczxbcvdfghasdfrtsdafasdfasdfdsfgsdgh Q \zxvczxbcvdfghasdfrtsdafasdfasdfdsfgsdgh_{L^p(B_1)}       .     \llabel{q NKU Qap NBP 5B 32Ey prwP FLvuPi wR P l1G TdQ BZE Aw 3d90 v8P5 CPAnX4 Yo 2 q7s yr5 BW8 Hc T7tM ioha BW9U4q rb u mEQ 6Xz MKR 2B REFX k3ZO MVMYSw 9S F 5ek q0m yNK Gn H0qi vlRA 18CbEz id O iuy ZZ6 kRo oJ kLQ0 Ewmz sKlld6 Kr K JmR xls 12K G2 bv8v LxfJ wrIcU6 Hx p q6p Fy7 Oim mo dXYt Kt0V VH22OC Aj f deT BAP vPl oK QzLE OQlq dpzxJ6 JI z Ujn TqY sQ4 BD QPW6 784x NUfsk0 aM 7 8qz MuL 9Mr Ac uVVK Y55n M7WqnB 2R C pGZ vHh WUN g9 3F2e RT8U umC62V H3 Z dJX LMS cca 1m xoOO 6oOL OVzfpO BO X 5Ev KuL z5s EW 8a9y otqk cKbDJN Us l pYM JpJ jOW Uy 2U4Y 8ThswELzXU3X7Ebd1KdZ7v1rN3GiirRXGKWK099ovBM0FDJCvkopYNQ2aN94Z7k0UnUKamE3OjU8DFYFFokbSI2J9V9gVlM8ALWThDPnPu3EL7HPD2VDaZTggzcCCmbvc70qqPcC9mt60ogcrTiA3HEjwTK8ymKeuJMc4q6dVz200XnYUtLR9GYjPXvFOVr6W1zUK1WbPToaWJJuKnxBLnd0ftDEbMmj4loHYyhZyMjM91zQS4p7z8eKa9h0JrbacekcirexG0z4n3xz0QOWSvFj3jLhWXUIU21iIAwJtI3RbWa90I7rzAIqI3UElUJG7tLtUXzw4KQNETvXzqWaujEMenYlNIzLGxgB3AuJEQ109}     \end{equation} Thus  $v = u-w$ satisfies $L(0)v = Q$ and      \begin{equation}       \zxvczxbcvdfghasdfrtsdafasdfasdfdsfgsdgh v \zxvczxbcvdfghasdfrtsdafasdfasdfdsfgsdgh_{L^p(B_{1/2})}        \les \zxvczxbcvdfghasdfrtsdafasdfasdfdsfgsdgh u \zxvczxbcvdfghasdfrtsdafasdfasdfdsfgsdgh_{L^p(B_{1/2})}              + \zxvczxbcvdfghasdfrtsdafasdfasdfdsfgsdgh w \zxvczxbcvdfghasdfrtsdafasdfasdfdsfgsdgh_{L^p(B_{1/2})}       \les \zxvczxbcvdfghasdfrtsdafasdfasdfdsfgsdgh u \zxvczxbcvdfghasdfrtsdafasdfasdfdsfgsdgh_{W^{m,p}(B_{1/2})}             + \zxvczxbcvdfghasdfrtsdafasdfasdfdsfgsdgh Q \zxvczxbcvdfghasdfrtsdafasdfasdfdsfgsdgh_{L^p(B_1)}       \les  c_0 + \MM   + \zxvczxbcvdfghasdfrtsdafasdfasdfdsfgsdgh Q \zxvczxbcvdfghasdfrtsdafasdfasdfdsfgsdgh_{L^p(B_1)}       ,     \llabel{z id O iuy ZZ6 kRo oJ kLQ0 Ewmz sKlld6 Kr K JmR xls 12K G2 bv8v LxfJ wrIcU6 Hx p q6p Fy7 Oim mo dXYt Kt0V VH22OC Aj f deT BAP vPl oK QzLE OQlq dpzxJ6 JI z Ujn TqY sQ4 BD QPW6 784x NUfsk0 aM 7 8qz MuL 9Mr Ac uVVK Y55n M7WqnB 2R C pGZ vHh WUN g9 3F2e RT8U umC62V H3 Z dJX LMS cca 1m xoOO 6oOL OVzfpO BO X 5Ev KuL z5s EW 8a9y otqk cKbDJN Us l pYM JpJ jOW Uy 2U4Y VKH6 kVC1Vx 1u v ykO yDs zo5 bz d36q WH1k J7Jtkg V1 J xqr Fnq mcU yZ JTp9 oFIc FAk0IT A9 3 SrL axO 9oU Z3 jG6f BRL1 iZ7ZE6 zj 8 G3M Hu8 6Ay jt 3flY cmTk jiTSYv CF t JLq cJP tN7 E3 8ThswELzXU3X7Ebd1KdZ7v1rN3GiirRXGKWK099ovBM0FDJCvkopYNQ2aN94Z7k0UnUKamE3OjU8DFYFFokbSI2J9V9gVlM8ALWThDPnPu3EL7HPD2VDaZTggzcCCmbvc70qqPcC9mt60ogcrTiA3HEjwTK8ymKeuJMc4q6dVz200XnYUtLR9GYjPXvFOVr6W1zUK1WbPToaWJJuKnxBLnd0ftDEbMmj4loHYyhZyMjM91zQS4p7z8eKa9h0JrbacekcirexG0z4n3xz0QOWSvFj3jLhWXUIU21iIAwJtI3RbWa90I7rzAIqI3UElUJG7tLtUXzw4KQNETvXzqWaujEMenYlNIzLGxgB3AuJEQ110}     \end{equation} where we used \eqref{8ThswELzXU3X7Ebd1KdZ7v1rN3GiirRXGKWK099ovBM0FDJCvkopYNQ2aN94Z7k0UnUKamE3OjU8DFYFFokbSI2J9V9gVlM8ALWThDPnPu3EL7HPD2VDaZTggzcCCmbvc70qqPcC9mt60ogcrTiA3HEjwTK8ymKeuJMc4q6dVz200XnYUtLR9GYjPXvFOVr6W1zUK1WbPToaWJJuKnxBLnd0ftDEbMmj4loHYyhZyMjM91zQS4p7z8eKa9h0JrbacekcirexG0z4n3xz0QOWSvFj3jLhWXUIU21iIAwJtI3RbWa90I7rzAIqI3UElUJG7tLtUXzw4KQNETvXzqWaujEMenYlNIzLGxgB3AuJEQ102} in the last inequality. Letting $P_v$ be the Taylor polynomial of $v$ with the degree $d$, we have $L(0) P_v=Q$ and for $|x|=r\leq1/4$, and we may bound     \begin{align}     \begin{split}       | v(x) - P_v(x) |       &\les \zxvczxbcvdfghasdfrtsdafasdfasdfdsfgsdgh \partial^{d+1} v \zxvczxbcvdfghasdfrtsdafasdfasdfdsfgsdgh_{L^{\infty} (B_r)} r^{d+1}        \les \zxvczxbcvdfghasdfrtsdafasdfasdfdsfgsdgh \partial^{d+1} v \zxvczxbcvdfghasdfrtsdafasdfasdfdsfgsdgh_{L^{\infty} (B_{1/4})} r^{d+1}       \\&       \les \big( \zxvczxbcvdfghasdfrtsdafasdfasdfdsfgsdgh Q \zxvczxbcvdfghasdfrtsdafasdfasdfdsfgsdgh_{L^p(B_{1/2})}                       + \zxvczxbcvdfghasdfrtsdafasdfasdfdsfgsdgh v \zxvczxbcvdfghasdfrtsdafasdfasdfdsfgsdgh_{L^p(B_{1/2})}              \big)               r^{d+1}       \les \big( c_0 + \MM   + \zxvczxbcvdfghasdfrtsdafasdfasdfdsfgsdgh Q \zxvczxbcvdfghasdfrtsdafasdfasdfdsfgsdgh_{L^p(B_1)} \big)               r^{d+1}       ,     \end{split}     \llabel{NUfsk0 aM 7 8qz MuL 9Mr Ac uVVK Y55n M7WqnB 2R C pGZ vHh WUN g9 3F2e RT8U umC62V H3 Z dJX LMS cca 1m xoOO 6oOL OVzfpO BO X 5Ev KuL z5s EW 8a9y otqk cKbDJN Us l pYM JpJ jOW Uy 2U4Y VKH6 kVC1Vx 1u v ykO yDs zo5 bz d36q WH1k J7Jtkg V1 J xqr Fnq mcU yZ JTp9 oFIc FAk0IT A9 3 SrL axO 9oU Z3 jG6f BRL1 iZ7ZE6 zj 8 G3M Hu8 6Ay jt 3flY cmTk jiTSYv CF t JLq cJP tN7 E3 POqG OKe0 3K3WV0 ep W XDQ C97 YSb AD ZUNp 81GF fCPbj3 iq E t0E NXy pLv fo Iz6z oFoF 9lkIun Xj Y yYL 52U bRB jx kQUS U9mm XtzIHO Cz 1 KH4 9ez 6Pz qW F223 C0Iz 3CsvuT R9 s VtQ CcM 1e8ThswELzXU3X7Ebd1KdZ7v1rN3GiirRXGKWK099ovBM0FDJCvkopYNQ2aN94Z7k0UnUKamE3OjU8DFYFFokbSI2J9V9gVlM8ALWThDPnPu3EL7HPD2VDaZTggzcCCmbvc70qqPcC9mt60ogcrTiA3HEjwTK8ymKeuJMc4q6dVz200XnYUtLR9GYjPXvFOVr6W1zUK1WbPToaWJJuKnxBLnd0ftDEbMmj4loHYyhZyMjM91zQS4p7z8eKa9h0JrbacekcirexG0z4n3xz0QOWSvFj3jLhWXUIU21iIAwJtI3RbWa90I7rzAIqI3UElUJG7tLtUXzw4KQNETvXzqWaujEMenYlNIzLGxgB3AuJEQ111}     \end{align} and thus     \begin{equation}       \zxvczxbcvdfghasdfrtsdafasdfasdfdsfgsdgh v - P_v \zxvczxbcvdfghasdfrtsdafasdfasdfdsfgsdgh_{L^q(B_r)}        \les \big( c_0 + \MM   + \zxvczxbcvdfghasdfrtsdafasdfasdfdsfgsdgh Q \zxvczxbcvdfghasdfrtsdafasdfasdfdsfgsdgh_{L^p(B_1)} \big)               r^{d+1+n/q}       .     \llabel{VKH6 kVC1Vx 1u v ykO yDs zo5 bz d36q WH1k J7Jtkg V1 J xqr Fnq mcU yZ JTp9 oFIc FAk0IT A9 3 SrL axO 9oU Z3 jG6f BRL1 iZ7ZE6 zj 8 G3M Hu8 6Ay jt 3flY cmTk jiTSYv CF t JLq cJP tN7 E3 POqG OKe0 3K3WV0 ep W XDQ C97 YSb AD ZUNp 81GF fCPbj3 iq E t0E NXy pLv fo Iz6z oFoF 9lkIun Xj Y yYL 52U bRB jx kQUS U9mm XtzIHO Cz 1 KH4 9ez 6Pz qW F223 C0Iz 3CsvuT R9 s VtQ CcM 1eo pD Py2l EEzL U0USJt Jb 9 zgy Gyf iQ4 fo Cx26 k4jL E0ula6 aS I rZQ HER 5HV CE BL55 WCtB 2LCmve TD z Vcp 7UR gI7 Qu FbFw 9VTx JwGrzs VW M 9sM JeJ Nd2 VG GFsi WuqC 3YxXoJ GK w Io7 18ThswELzXU3X7Ebd1KdZ7v1rN3GiirRXGKWK099ovBM0FDJCvkopYNQ2aN94Z7k0UnUKamE3OjU8DFYFFokbSI2J9V9gVlM8ALWThDPnPu3EL7HPD2VDaZTggzcCCmbvc70qqPcC9mt60ogcrTiA3HEjwTK8ymKeuJMc4q6dVz200XnYUtLR9GYjPXvFOVr6W1zUK1WbPToaWJJuKnxBLnd0ftDEbMmj4loHYyhZyMjM91zQS4p7z8eKa9h0JrbacekcirexG0z4n3xz0QOWSvFj3jLhWXUIU21iIAwJtI3RbWa90I7rzAIqI3UElUJG7tLtUXzw4KQNETvXzqWaujEMenYlNIzLGxgB3AuJEQ112}     \end{equation} Letting $P = P_v+P_w$, we have $u-P = (v-P_v) + (w - P_w)$, and we obtain
    \begin{equation}       \zxvczxbcvdfghasdfrtsdafasdfasdfdsfgsdgh u - P \zxvczxbcvdfghasdfrtsdafasdfasdfdsfgsdgh_{L^q(B_r)}       \les \big( c_0 + \MM   + \zxvczxbcvdfghasdfrtsdafasdfasdfdsfgsdgh Q \zxvczxbcvdfghasdfrtsdafasdfasdfdsfgsdgh_{L^p(B_1)} \big)               r^{d+\alpha+n/q}       \comma       r \leq \frac{1}{4}       .     \label{8ThswELzXU3X7Ebd1KdZ7v1rN3GiirRXGKWK099ovBM0FDJCvkopYNQ2aN94Z7k0UnUKamE3OjU8DFYFFokbSI2J9V9gVlM8ALWThDPnPu3EL7HPD2VDaZTggzcCCmbvc70qqPcC9mt60ogcrTiA3HEjwTK8ymKeuJMc4q6dVz200XnYUtLR9GYjPXvFOVr6W1zUK1WbPToaWJJuKnxBLnd0ftDEbMmj4loHYyhZyMjM91zQS4p7z8eKa9h0JrbacekcirexG0z4n3xz0QOWSvFj3jLhWXUIU21iIAwJtI3RbWa90I7rzAIqI3UElUJG7tLtUXzw4KQNETvXzqWaujEMenYlNIzLGxgB3AuJEQ113}     \end{equation} Combining \eqref{8ThswELzXU3X7Ebd1KdZ7v1rN3GiirRXGKWK099ovBM0FDJCvkopYNQ2aN94Z7k0UnUKamE3OjU8DFYFFokbSI2J9V9gVlM8ALWThDPnPu3EL7HPD2VDaZTggzcCCmbvc70qqPcC9mt60ogcrTiA3HEjwTK8ymKeuJMc4q6dVz200XnYUtLR9GYjPXvFOVr6W1zUK1WbPToaWJJuKnxBLnd0ftDEbMmj4loHYyhZyMjM91zQS4p7z8eKa9h0JrbacekcirexG0z4n3xz0QOWSvFj3jLhWXUIU21iIAwJtI3RbWa90I7rzAIqI3UElUJG7tLtUXzw4KQNETvXzqWaujEMenYlNIzLGxgB3AuJEQ113} with \eqref{8ThswELzXU3X7Ebd1KdZ7v1rN3GiirRXGKWK099ovBM0FDJCvkopYNQ2aN94Z7k0UnUKamE3OjU8DFYFFokbSI2J9V9gVlM8ALWThDPnPu3EL7HPD2VDaZTggzcCCmbvc70qqPcC9mt60ogcrTiA3HEjwTK8ymKeuJMc4q6dVz200XnYUtLR9GYjPXvFOVr6W1zUK1WbPToaWJJuKnxBLnd0ftDEbMmj4loHYyhZyMjM91zQS4p7z8eKa9h0JrbacekcirexG0z4n3xz0QOWSvFj3jLhWXUIU21iIAwJtI3RbWa90I7rzAIqI3UElUJG7tLtUXzw4KQNETvXzqWaujEMenYlNIzLGxgB3AuJEQ14}, we conclude that $P$ is a homogeneous polynomial of degree~$d$. \par We now prove that $u$ and $P$ satisfy \eqref{8ThswELzXU3X7Ebd1KdZ7v1rN3GiirRXGKWK099ovBM0FDJCvkopYNQ2aN94Z7k0UnUKamE3OjU8DFYFFokbSI2J9V9gVlM8ALWThDPnPu3EL7HPD2VDaZTggzcCCmbvc70qqPcC9mt60ogcrTiA3HEjwTK8ymKeuJMc4q6dVz200XnYUtLR9GYjPXvFOVr6W1zUK1WbPToaWJJuKnxBLnd0ftDEbMmj4loHYyhZyMjM91zQS4p7z8eKa9h0JrbacekcirexG0z4n3xz0QOWSvFj3jLhWXUIU21iIAwJtI3RbWa90I7rzAIqI3UElUJG7tLtUXzw4KQNETvXzqWaujEMenYlNIzLGxgB3AuJEQ16} and~\eqref{8ThswELzXU3X7Ebd1KdZ7v1rN3GiirRXGKWK099ovBM0FDJCvkopYNQ2aN94Z7k0UnUKamE3OjU8DFYFFokbSI2J9V9gVlM8ALWThDPnPu3EL7HPD2VDaZTggzcCCmbvc70qqPcC9mt60ogcrTiA3HEjwTK8ymKeuJMc4q6dVz200XnYUtLR9GYjPXvFOVr6W1zUK1WbPToaWJJuKnxBLnd0ftDEbMmj4loHYyhZyMjM91zQS4p7z8eKa9h0JrbacekcirexG0z4n3xz0QOWSvFj3jLhWXUIU21iIAwJtI3RbWa90I7rzAIqI3UElUJG7tLtUXzw4KQNETvXzqWaujEMenYlNIzLGxgB3AuJEQ17}. Without loss of generality, we assume that      \begin{equation}       \sum_{|\nu|=m} |a_{\nu}(x) - a_{\nu}(0)|        + \sum_{|\nu|<m} |a_{\nu}(x)|        \leq \eps       ,     \label{8ThswELzXU3X7Ebd1KdZ7v1rN3GiirRXGKWK099ovBM0FDJCvkopYNQ2aN94Z7k0UnUKamE3OjU8DFYFFokbSI2J9V9gVlM8ALWThDPnPu3EL7HPD2VDaZTggzcCCmbvc70qqPcC9mt60ogcrTiA3HEjwTK8ymKeuJMc4q6dVz200XnYUtLR9GYjPXvFOVr6W1zUK1WbPToaWJJuKnxBLnd0ftDEbMmj4loHYyhZyMjM91zQS4p7z8eKa9h0JrbacekcirexG0z4n3xz0QOWSvFj3jLhWXUIU21iIAwJtI3RbWa90I7rzAIqI3UElUJG7tLtUXzw4KQNETvXzqWaujEMenYlNIzLGxgB3AuJEQ114}     \end{equation} for some small positive $\eps$ depending on $p$, $d$, and $\alpha$ only, since the general case follows immediately by a dilation $x \to Rx$ for some $R \in (0,1]$ sufficiently small. Let      \begin{equation}         \tilde\cc          = \sup_{r\leq1} \frac{\zxvczxbcvdfghasdfrtsdafasdfasdfdsfgsdgh u - P \zxvczxbcvdfghasdfrtsdafasdfasdfdsfgsdgh_{L^p(B_r)}}{r^{d+\alpha+n/p}}         ,     \label{8ThswELzXU3X7Ebd1KdZ7v1rN3GiirRXGKWK099ovBM0FDJCvkopYNQ2aN94Z7k0UnUKamE3OjU8DFYFFokbSI2J9V9gVlM8ALWThDPnPu3EL7HPD2VDaZTggzcCCmbvc70qqPcC9mt60ogcrTiA3HEjwTK8ymKeuJMc4q6dVz200XnYUtLR9GYjPXvFOVr6W1zUK1WbPToaWJJuKnxBLnd0ftDEbMmj4loHYyhZyMjM91zQS4p7z8eKa9h0JrbacekcirexG0z4n3xz0QOWSvFj3jLhWXUIU21iIAwJtI3RbWa90I7rzAIqI3UElUJG7tLtUXzw4KQNETvXzqWaujEMenYlNIzLGxgB3AuJEQ115}     \end{equation} which is finite due to~\eqref{8ThswELzXU3X7Ebd1KdZ7v1rN3GiirRXGKWK099ovBM0FDJCvkopYNQ2aN94Z7k0UnUKamE3OjU8DFYFFokbSI2J9V9gVlM8ALWThDPnPu3EL7HPD2VDaZTggzcCCmbvc70qqPcC9mt60ogcrTiA3HEjwTK8ymKeuJMc4q6dVz200XnYUtLR9GYjPXvFOVr6W1zUK1WbPToaWJJuKnxBLnd0ftDEbMmj4loHYyhZyMjM91zQS4p7z8eKa9h0JrbacekcirexG0z4n3xz0QOWSvFj3jLhWXUIU21iIAwJtI3RbWa90I7rzAIqI3UElUJG7tLtUXzw4KQNETvXzqWaujEMenYlNIzLGxgB3AuJEQ80}.  Using $L(0) P = Q$, we rewrite $Lu=f$ as      \begin{equation}       L \ppsi       = f - LP       = (f-Q) + (L(0)-L)P       .     \label{8ThswELzXU3X7Ebd1KdZ7v1rN3GiirRXGKWK099ovBM0FDJCvkopYNQ2aN94Z7k0UnUKamE3OjU8DFYFFokbSI2J9V9gVlM8ALWThDPnPu3EL7HPD2VDaZTggzcCCmbvc70qqPcC9mt60ogcrTiA3HEjwTK8ymKeuJMc4q6dVz200XnYUtLR9GYjPXvFOVr6W1zUK1WbPToaWJJuKnxBLnd0ftDEbMmj4loHYyhZyMjM91zQS4p7z8eKa9h0JrbacekcirexG0z4n3xz0QOWSvFj3jLhWXUIU21iIAwJtI3RbWa90I7rzAIqI3UElUJG7tLtUXzw4KQNETvXzqWaujEMenYlNIzLGxgB3AuJEQ116}     \end{equation} Applying the interior elliptic estimate to \eqref{8ThswELzXU3X7Ebd1KdZ7v1rN3GiirRXGKWK099ovBM0FDJCvkopYNQ2aN94Z7k0UnUKamE3OjU8DFYFFokbSI2J9V9gVlM8ALWThDPnPu3EL7HPD2VDaZTggzcCCmbvc70qqPcC9mt60ogcrTiA3HEjwTK8ymKeuJMc4q6dVz200XnYUtLR9GYjPXvFOVr6W1zUK1WbPToaWJJuKnxBLnd0ftDEbMmj4loHYyhZyMjM91zQS4p7z8eKa9h0JrbacekcirexG0z4n3xz0QOWSvFj3jLhWXUIU21iIAwJtI3RbWa90I7rzAIqI3UElUJG7tLtUXzw4KQNETvXzqWaujEMenYlNIzLGxgB3AuJEQ116} yields     \begin{align}     \begin{split}       \sum_{|\nu|\leq m} r^{|\nu|} \zxvczxbcvdfghasdfrtsdafasdfasdfdsfgsdgh \partial^{\nu} \ppsi \zxvczxbcvdfghasdfrtsdafasdfasdfdsfgsdgh_{L^p(B_r)}       &\les \zxvczxbcvdfghasdfrtsdafasdfasdfdsfgsdgh u-P \zxvczxbcvdfghasdfrtsdafasdfasdfdsfgsdgh_{L^p(B_{2r})}              + r^m \zxvczxbcvdfghasdfrtsdafasdfasdfdsfgsdgh f - Q \zxvczxbcvdfghasdfrtsdafasdfasdfdsfgsdgh_{L^p(B_{2r})}             + r^m \zxvczxbcvdfghasdfrtsdafasdfasdfdsfgsdgh ( L(0) - L ) P \zxvczxbcvdfghasdfrtsdafasdfasdfdsfgsdgh_{L^p(B_{2r})}       \\&       \les \big( \tilde\cc + \MM   + \zxvczxbcvdfghasdfrtsdafasdfasdfdsfgsdgh P \zxvczxbcvdfghasdfrtsdafasdfasdfdsfgsdgh_{L^p(B_1)} \big) r^{d+\alpha+n/p}       \comma       r \leq \frac{1}{2}       .     \end{split}     \label{8ThswELzXU3X7Ebd1KdZ7v1rN3GiirRXGKWK099ovBM0FDJCvkopYNQ2aN94Z7k0UnUKamE3OjU8DFYFFokbSI2J9V9gVlM8ALWThDPnPu3EL7HPD2VDaZTggzcCCmbvc70qqPcC9mt60ogcrTiA3HEjwTK8ymKeuJMc4q6dVz200XnYUtLR9GYjPXvFOVr6W1zUK1WbPToaWJJuKnxBLnd0ftDEbMmj4loHYyhZyMjM91zQS4p7z8eKa9h0JrbacekcirexG0z4n3xz0QOWSvFj3jLhWXUIU21iIAwJtI3RbWa90I7rzAIqI3UElUJG7tLtUXzw4KQNETvXzqWaujEMenYlNIzLGxgB3AuJEQ117}     \end{align} We now rewrite the equation \eqref{8ThswELzXU3X7Ebd1KdZ7v1rN3GiirRXGKWK099ovBM0FDJCvkopYNQ2aN94Z7k0UnUKamE3OjU8DFYFFokbSI2J9V9gVlM8ALWThDPnPu3EL7HPD2VDaZTggzcCCmbvc70qqPcC9mt60ogcrTiA3HEjwTK8ymKeuJMc4q6dVz200XnYUtLR9GYjPXvFOVr6W1zUK1WbPToaWJJuKnxBLnd0ftDEbMmj4loHYyhZyMjM91zQS4p7z8eKa9h0JrbacekcirexG0z4n3xz0QOWSvFj3jLhWXUIU21iIAwJtI3RbWa90I7rzAIqI3UElUJG7tLtUXzw4KQNETvXzqWaujEMenYlNIzLGxgB3AuJEQ116} as      \begin{equation}       L(0) \ppsi       = \sum_{|\nu|=m} (a_{\nu} (0) - a_{\nu}) \partial^{\nu} \ppsi          - \sum_{|\nu|<m} a_{\nu} \partial^{\nu} \ppsi           + (f-Q)          + (L(0) - L) P           .      \label{8ThswELzXU3X7Ebd1KdZ7v1rN3GiirRXGKWK099ovBM0FDJCvkopYNQ2aN94Z7k0UnUKamE3OjU8DFYFFokbSI2J9V9gVlM8ALWThDPnPu3EL7HPD2VDaZTggzcCCmbvc70qqPcC9mt60ogcrTiA3HEjwTK8ymKeuJMc4q6dVz200XnYUtLR9GYjPXvFOVr6W1zUK1WbPToaWJJuKnxBLnd0ftDEbMmj4loHYyhZyMjM91zQS4p7z8eKa9h0JrbacekcirexG0z4n3xz0QOWSvFj3jLhWXUIU21iIAwJtI3RbWa90I7rzAIqI3UElUJG7tLtUXzw4KQNETvXzqWaujEMenYlNIzLGxgB3AuJEQ118}     \end{equation} Denoting the right hand side of \eqref{8ThswELzXU3X7Ebd1KdZ7v1rN3GiirRXGKWK099ovBM0FDJCvkopYNQ2aN94Z7k0UnUKamE3OjU8DFYFFokbSI2J9V9gVlM8ALWThDPnPu3EL7HPD2VDaZTggzcCCmbvc70qqPcC9mt60ogcrTiA3HEjwTK8ymKeuJMc4q6dVz200XnYUtLR9GYjPXvFOVr6W1zUK1WbPToaWJJuKnxBLnd0ftDEbMmj4loHYyhZyMjM91zQS4p7z8eKa9h0JrbacekcirexG0z4n3xz0QOWSvFj3jLhWXUIU21iIAwJtI3RbWa90I7rzAIqI3UElUJG7tLtUXzw4KQNETvXzqWaujEMenYlNIzLGxgB3AuJEQ118} by $F$, we have     \begin{align}     \begin{split}       \zxvczxbcvdfghasdfrtsdafasdfasdfdsfgsdgh F \zxvczxbcvdfghasdfrtsdafasdfasdfdsfgsdgh_{L^p(B_r)}        &\les \sum_{|\nu|=m}                 \zxvczxbcvdfghasdfrtsdafasdfasdfdsfgsdgh a_{\nu}(0) - a_{\nu}(x)\zxvczxbcvdfghasdfrtsdafasdfasdfdsfgsdgh_{L^{\infty}(B_r)}                 \zxvczxbcvdfghasdfrtsdafasdfasdfdsfgsdgh \partial^{\nu} \ppsi \zxvczxbcvdfghasdfrtsdafasdfasdfdsfgsdgh_{L^p(B_r)}                + \sum_{|\nu|<m}                 \zxvczxbcvdfghasdfrtsdafasdfasdfdsfgsdgh a_{\nu}(x) \zxvczxbcvdfghasdfrtsdafasdfasdfdsfgsdgh_{L^{\infty}(B_r)}                 \zxvczxbcvdfghasdfrtsdafasdfasdfdsfgsdgh \partial^{\nu} \ppsi \zxvczxbcvdfghasdfrtsdafasdfasdfdsfgsdgh_{L^p(B_r)}       \\& \indeq                        + \zxvczxbcvdfghasdfrtsdafasdfasdfdsfgsdgh f-Q \zxvczxbcvdfghasdfrtsdafasdfasdfdsfgsdgh_{L^p(B_r)}                       +\sum_{|\nu|=m}                   \zxvczxbcvdfghasdfrtsdafasdfasdfdsfgsdgh a_{\nu}(0) - a_{\nu}(x)\zxvczxbcvdfghasdfrtsdafasdfasdfdsfgsdgh_{L^{\infty}(B_r)}                   \zxvczxbcvdfghasdfrtsdafasdfasdfdsfgsdgh  \partial^{\nu}P\zxvczxbcvdfghasdfrtsdafasdfasdfdsfgsdgh_{L^p(B_r)}       \\&\indeq                + \sum_{|\nu|<m}                  \zxvczxbcvdfghasdfrtsdafasdfasdfdsfgsdgh a_{\nu}(x) \zxvczxbcvdfghasdfrtsdafasdfasdfdsfgsdgh_{L^{\infty}(B_r)}                  \zxvczxbcvdfghasdfrtsdafasdfasdfdsfgsdgh  \partial^{\nu}P\zxvczxbcvdfghasdfrtsdafasdfasdfdsfgsdgh_{L^p(B_r)}                 .      \end{split}     \llabel{POqG OKe0 3K3WV0 ep W XDQ C97 YSb AD ZUNp 81GF fCPbj3 iq E t0E NXy pLv fo Iz6z oFoF 9lkIun Xj Y yYL 52U bRB jx kQUS U9mm XtzIHO Cz 1 KH4 9ez 6Pz qW F223 C0Iz 3CsvuT R9 s VtQ CcM 1eo pD Py2l EEzL U0USJt Jb 9 zgy Gyf iQ4 fo Cx26 k4jL E0ula6 aS I rZQ HER 5HV CE BL55 WCtB 2LCmve TD z Vcp 7UR gI7 Qu FbFw 9VTx JwGrzs VW M 9sM JeJ Nd2 VG GFsi WuqC 3YxXoJ GK w Io7 1fg sGm 0P YFBz X8eX 7pf9GJ b1 o XUs 1q0 6KP Ls MucN ytQb L0Z0Qq m1 l SPj 9MT etk L6 KfsC 6Zob Yhc2qu Xy 9 GPm ZYj 1Go ei feJ3 pRAf n6Ypy6 jN s 4Y5 nSE pqN 4m Rmam AGfY HhSaBr Ls D 8ThswELzXU3X7Ebd1KdZ7v1rN3GiirRXGKWK099ovBM0FDJCvkopYNQ2aN94Z7k0UnUKamE3OjU8DFYFFokbSI2J9V9gVlM8ALWThDPnPu3EL7HPD2VDaZTggzcCCmbvc70qqPcC9mt60ogcrTiA3HEjwTK8ymKeuJMc4q6dVz200XnYUtLR9GYjPXvFOVr6W1zUK1WbPToaWJJuKnxBLnd0ftDEbMmj4loHYyhZyMjM91zQS4p7z8eKa9h0JrbacekcirexG0z4n3xz0QOWSvFj3jLhWXUIU21iIAwJtI3RbWa90I7rzAIqI3UElUJG7tLtUXzw4KQNETvXzqWaujEMenYlNIzLGxgB3AuJEQ119}      \end{align} Now we employ \eqref{8ThswELzXU3X7Ebd1KdZ7v1rN3GiirRXGKWK099ovBM0FDJCvkopYNQ2aN94Z7k0UnUKamE3OjU8DFYFFokbSI2J9V9gVlM8ALWThDPnPu3EL7HPD2VDaZTggzcCCmbvc70qqPcC9mt60ogcrTiA3HEjwTK8ymKeuJMc4q6dVz200XnYUtLR9GYjPXvFOVr6W1zUK1WbPToaWJJuKnxBLnd0ftDEbMmj4loHYyhZyMjM91zQS4p7z8eKa9h0JrbacekcirexG0z4n3xz0QOWSvFj3jLhWXUIU21iIAwJtI3RbWa90I7rzAIqI3UElUJG7tLtUXzw4KQNETvXzqWaujEMenYlNIzLGxgB3AuJEQ114} and \eqref{8ThswELzXU3X7Ebd1KdZ7v1rN3GiirRXGKWK099ovBM0FDJCvkopYNQ2aN94Z7k0UnUKamE3OjU8DFYFFokbSI2J9V9gVlM8ALWThDPnPu3EL7HPD2VDaZTggzcCCmbvc70qqPcC9mt60ogcrTiA3HEjwTK8ymKeuJMc4q6dVz200XnYUtLR9GYjPXvFOVr6W1zUK1WbPToaWJJuKnxBLnd0ftDEbMmj4loHYyhZyMjM91zQS4p7z8eKa9h0JrbacekcirexG0z4n3xz0QOWSvFj3jLhWXUIU21iIAwJtI3RbWa90I7rzAIqI3UElUJG7tLtUXzw4KQNETvXzqWaujEMenYlNIzLGxgB3AuJEQ11} to obtain     \begin{align}     \begin{split}       \zxvczxbcvdfghasdfrtsdafasdfasdfdsfgsdgh F \zxvczxbcvdfghasdfrtsdafasdfasdfdsfgsdgh_{L^p(B_r)}        &\les \eps \sum_{|\nu| \leq m}\zxvczxbcvdfghasdfrtsdafasdfasdfdsfgsdgh \partial^{\nu} \ppsi \zxvczxbcvdfghasdfrtsdafasdfasdfdsfgsdgh_{L^p(B_r)}             + \zxvczxbcvdfghasdfrtsdafasdfasdfdsfgsdgh f-Q\zxvczxbcvdfghasdfrtsdafasdfasdfdsfgsdgh_{L^p(B_r)}             + \eps r^{\alpha} \zxvczxbcvdfghasdfrtsdafasdfasdfdsfgsdgh \partial^mP\zxvczxbcvdfghasdfrtsdafasdfasdfdsfgsdgh_{L^p(B_r)}             + \eps \sum_{|\nu| < m}\zxvczxbcvdfghasdfrtsdafasdfasdfdsfgsdgh  \partial^{\nu} P \zxvczxbcvdfghasdfrtsdafasdfasdfdsfgsdgh_{L^p(B_r)}       \\&       \les \big( \eps \tilde\cc                      + \MM                        + \zxvczxbcvdfghasdfrtsdafasdfasdfdsfgsdgh P \zxvczxbcvdfghasdfrtsdafasdfasdfdsfgsdgh_{L^p(B_1)}              \big) r^{d-m+\alpha+n/p}              ,     \end{split}    \llabel{o pD Py2l EEzL U0USJt Jb 9 zgy Gyf iQ4 fo Cx26 k4jL E0ula6 aS I rZQ HER 5HV CE BL55 WCtB 2LCmve TD z Vcp 7UR gI7 Qu FbFw 9VTx JwGrzs VW M 9sM JeJ Nd2 VG GFsi WuqC 3YxXoJ GK w Io7 1fg sGm 0P YFBz X8eX 7pf9GJ b1 o XUs 1q0 6KP Ls MucN ytQb L0Z0Qq m1 l SPj 9MT etk L6 KfsC 6Zob Yhc2qu Xy 9 GPm ZYj 1Go ei feJ3 pRAf n6Ypy6 jN s 4Y5 nSE pqN 4m Rmam AGfY HhSaBr Ls D THC SEl UyR Mh 66XU 7hNz pZVC5V nV 7 VjL 7kv WKf 7P 5hj6 t1vu gkLGdN X8 b gOX HWm 6W4 YE mxFG 4WaN EbGKsv 0p 4 OG0 Nrd uTe Za xNXq V4Bp mOdXIq 9a b PeD PbU Z4N Xt ohbY egCf xBNttE 8ThswELzXU3X7Ebd1KdZ7v1rN3GiirRXGKWK099ovBM0FDJCvkopYNQ2aN94Z7k0UnUKamE3OjU8DFYFFokbSI2J9V9gVlM8ALWThDPnPu3EL7HPD2VDaZTggzcCCmbvc70qqPcC9mt60ogcrTiA3HEjwTK8ymKeuJMc4q6dVz200XnYUtLR9GYjPXvFOVr6W1zUK1WbPToaWJJuKnxBLnd0ftDEbMmj4loHYyhZyMjM91zQS4p7z8eKa9h0JrbacekcirexG0z4n3xz0QOWSvFj3jLhWXUIU21iIAwJtI3RbWa90I7rzAIqI3UElUJG7tLtUXzw4KQNETvXzqWaujEMenYlNIzLGxgB3AuJEQ120}     \end{align} where we used \eqref{8ThswELzXU3X7Ebd1KdZ7v1rN3GiirRXGKWK099ovBM0FDJCvkopYNQ2aN94Z7k0UnUKamE3OjU8DFYFFokbSI2J9V9gVlM8ALWThDPnPu3EL7HPD2VDaZTggzcCCmbvc70qqPcC9mt60ogcrTiA3HEjwTK8ymKeuJMc4q6dVz200XnYUtLR9GYjPXvFOVr6W1zUK1WbPToaWJJuKnxBLnd0ftDEbMmj4loHYyhZyMjM91zQS4p7z8eKa9h0JrbacekcirexG0z4n3xz0QOWSvFj3jLhWXUIU21iIAwJtI3RbWa90I7rzAIqI3UElUJG7tLtUXzw4KQNETvXzqWaujEMenYlNIzLGxgB3AuJEQ117} and that $P$ is homogeneous polynomial of degree $d$ in the last inequality. Applying Lemma~\ref{L01}, there exists $v$ in $W^{m,p}(B_{1/2})$ such that $L(0)v = F$ in $B_{1/2}$ and      \begin{equation}       \zxvczxbcvdfghasdfrtsdafasdfasdfdsfgsdgh v \zxvczxbcvdfghasdfrtsdafasdfasdfdsfgsdgh_{L^p(B_r)}       \les \big( \eps \tilde\cc                      + \MM                        + \zxvczxbcvdfghasdfrtsdafasdfasdfdsfgsdgh P \zxvczxbcvdfghasdfrtsdafasdfasdfdsfgsdgh_{L^p(B_1)}              \big) r^{d+\alpha+n/p}              .     \label{8ThswELzXU3X7Ebd1KdZ7v1rN3GiirRXGKWK099ovBM0FDJCvkopYNQ2aN94Z7k0UnUKamE3OjU8DFYFFokbSI2J9V9gVlM8ALWThDPnPu3EL7HPD2VDaZTggzcCCmbvc70qqPcC9mt60ogcrTiA3HEjwTK8ymKeuJMc4q6dVz200XnYUtLR9GYjPXvFOVr6W1zUK1WbPToaWJJuKnxBLnd0ftDEbMmj4loHYyhZyMjM91zQS4p7z8eKa9h0JrbacekcirexG0z4n3xz0QOWSvFj3jLhWXUIU21iIAwJtI3RbWa90I7rzAIqI3UElUJG7tLtUXzw4KQNETvXzqWaujEMenYlNIzLGxgB3AuJEQ121}     \end{equation} By \eqref{8ThswELzXU3X7Ebd1KdZ7v1rN3GiirRXGKWK099ovBM0FDJCvkopYNQ2aN94Z7k0UnUKamE3OjU8DFYFFokbSI2J9V9gVlM8ALWThDPnPu3EL7HPD2VDaZTggzcCCmbvc70qqPcC9mt60ogcrTiA3HEjwTK8ymKeuJMc4q6dVz200XnYUtLR9GYjPXvFOVr6W1zUK1WbPToaWJJuKnxBLnd0ftDEbMmj4loHYyhZyMjM91zQS4p7z8eKa9h0JrbacekcirexG0z4n3xz0QOWSvFj3jLhWXUIU21iIAwJtI3RbWa90I7rzAIqI3UElUJG7tLtUXzw4KQNETvXzqWaujEMenYlNIzLGxgB3AuJEQ115} and \eqref{8ThswELzXU3X7Ebd1KdZ7v1rN3GiirRXGKWK099ovBM0FDJCvkopYNQ2aN94Z7k0UnUKamE3OjU8DFYFFokbSI2J9V9gVlM8ALWThDPnPu3EL7HPD2VDaZTggzcCCmbvc70qqPcC9mt60ogcrTiA3HEjwTK8ymKeuJMc4q6dVz200XnYUtLR9GYjPXvFOVr6W1zUK1WbPToaWJJuKnxBLnd0ftDEbMmj4loHYyhZyMjM91zQS4p7z8eKa9h0JrbacekcirexG0z4n3xz0QOWSvFj3jLhWXUIU21iIAwJtI3RbWa90I7rzAIqI3UElUJG7tLtUXzw4KQNETvXzqWaujEMenYlNIzLGxgB3AuJEQ121}, the Taylor series of $\ppsi - v$ begins  with a homogeneous polynomial of degree at least~$d+1$.  By the interior elliptic estimate, for any $r\leq 1/2$,     \begin{equation}       \zxvczxbcvdfghasdfrtsdafasdfasdfdsfgsdgh \ppsi - v \zxvczxbcvdfghasdfrtsdafasdfasdfdsfgsdgh_{L^p(B_r)}       \les  \zxvczxbcvdfghasdfrtsdafasdfasdfdsfgsdgh \ppsi - v \zxvczxbcvdfghasdfrtsdafasdfasdfdsfgsdgh_{L^p(B_{1/2})} r^{d+\alpha+n/p}       .     \llabel{fg sGm 0P YFBz X8eX 7pf9GJ b1 o XUs 1q0 6KP Ls MucN ytQb L0Z0Qq m1 l SPj 9MT etk L6 KfsC 6Zob Yhc2qu Xy 9 GPm ZYj 1Go ei feJ3 pRAf n6Ypy6 jN s 4Y5 nSE pqN 4m Rmam AGfY HhSaBr Ls D THC SEl UyR Mh 66XU 7hNz pZVC5V nV 7 VjL 7kv WKf 7P 5hj6 t1vu gkLGdN X8 b gOX HWm 6W4 YE mxFG 4WaN EbGKsv 0p 4 OG0 Nrd uTe Za xNXq V4Bp mOdXIq 9a b PeD PbU Z4N Xt ohbY egCf xBNttE wc D YSD 637 jJ2 ms 6Ta1 J2xZ PtKnPw AX A tJA Rc8 n5d 93 TZi7 q6Wo nEDLwW Sz e Sue YFX 8cM hm Y6is 15pX aOYBbV fS C haL kBR Ks6 UO qG4j DVab fbdtny fi D BFI 7uh B39 FJ 6mYr CUUT f28ThswELzXU3X7Ebd1KdZ7v1rN3GiirRXGKWK099ovBM0FDJCvkopYNQ2aN94Z7k0UnUKamE3OjU8DFYFFokbSI2J9V9gVlM8ALWThDPnPu3EL7HPD2VDaZTggzcCCmbvc70qqPcC9mt60ogcrTiA3HEjwTK8ymKeuJMc4q6dVz200XnYUtLR9GYjPXvFOVr6W1zUK1WbPToaWJJuKnxBLnd0ftDEbMmj4loHYyhZyMjM91zQS4p7z8eKa9h0JrbacekcirexG0z4n3xz0QOWSvFj3jLhWXUIU21iIAwJtI3RbWa90I7rzAIqI3UElUJG7tLtUXzw4KQNETvXzqWaujEMenYlNIzLGxgB3AuJEQ122}     \end{equation} Thus,     \begin{equation}       \zxvczxbcvdfghasdfrtsdafasdfasdfdsfgsdgh u-P \zxvczxbcvdfghasdfrtsdafasdfasdfdsfgsdgh_{L^p(B_r)}       \les \big( \eps \tilde\cc                       + \MM                        + \zxvczxbcvdfghasdfrtsdafasdfasdfdsfgsdgh u-P \zxvczxbcvdfghasdfrtsdafasdfasdfdsfgsdgh_{L^p(B_1)}                      + \zxvczxbcvdfghasdfrtsdafasdfasdfdsfgsdgh P \zxvczxbcvdfghasdfrtsdafasdfasdfdsfgsdgh_{L^p(B_1)}                \big) r^{d+\alpha+n/p}        \comma        r \leq \frac{1}{2}        .     \label{8ThswELzXU3X7Ebd1KdZ7v1rN3GiirRXGKWK099ovBM0FDJCvkopYNQ2aN94Z7k0UnUKamE3OjU8DFYFFokbSI2J9V9gVlM8ALWThDPnPu3EL7HPD2VDaZTggzcCCmbvc70qqPcC9mt60ogcrTiA3HEjwTK8ymKeuJMc4q6dVz200XnYUtLR9GYjPXvFOVr6W1zUK1WbPToaWJJuKnxBLnd0ftDEbMmj4loHYyhZyMjM91zQS4p7z8eKa9h0JrbacekcirexG0z4n3xz0QOWSvFj3jLhWXUIU21iIAwJtI3RbWa90I7rzAIqI3UElUJG7tLtUXzw4KQNETvXzqWaujEMenYlNIzLGxgB3AuJEQ123}     \end{equation} Note that \eqref{8ThswELzXU3X7Ebd1KdZ7v1rN3GiirRXGKWK099ovBM0FDJCvkopYNQ2aN94Z7k0UnUKamE3OjU8DFYFFokbSI2J9V9gVlM8ALWThDPnPu3EL7HPD2VDaZTggzcCCmbvc70qqPcC9mt60ogcrTiA3HEjwTK8ymKeuJMc4q6dVz200XnYUtLR9GYjPXvFOVr6W1zUK1WbPToaWJJuKnxBLnd0ftDEbMmj4loHYyhZyMjM91zQS4p7z8eKa9h0JrbacekcirexG0z4n3xz0QOWSvFj3jLhWXUIU21iIAwJtI3RbWa90I7rzAIqI3UElUJG7tLtUXzw4KQNETvXzqWaujEMenYlNIzLGxgB3AuJEQ123} also holds for $1/2\leq r \leq 1$. Taking the supremum over $(0,1]$, we get      \begin{equation}       \tilde\cc        \les \eps \tilde\cc               + \MM                + \zxvczxbcvdfghasdfrtsdafasdfasdfdsfgsdgh u-P \zxvczxbcvdfghasdfrtsdafasdfasdfdsfgsdgh_{L^p(B_1)}              + \zxvczxbcvdfghasdfrtsdafasdfasdfdsfgsdgh P \zxvczxbcvdfghasdfrtsdafasdfasdfdsfgsdgh_{L^p(B_1)}               .     \label{8ThswELzXU3X7Ebd1KdZ7v1rN3GiirRXGKWK099ovBM0FDJCvkopYNQ2aN94Z7k0UnUKamE3OjU8DFYFFokbSI2J9V9gVlM8ALWThDPnPu3EL7HPD2VDaZTggzcCCmbvc70qqPcC9mt60ogcrTiA3HEjwTK8ymKeuJMc4q6dVz200XnYUtLR9GYjPXvFOVr6W1zUK1WbPToaWJJuKnxBLnd0ftDEbMmj4loHYyhZyMjM91zQS4p7z8eKa9h0JrbacekcirexG0z4n3xz0QOWSvFj3jLhWXUIU21iIAwJtI3RbWa90I7rzAIqI3UElUJG7tLtUXzw4KQNETvXzqWaujEMenYlNIzLGxgB3AuJEQ124}     \end{equation} Denote by $C$ the implicit constant in \eqref{8ThswELzXU3X7Ebd1KdZ7v1rN3GiirRXGKWK099ovBM0FDJCvkopYNQ2aN94Z7k0UnUKamE3OjU8DFYFFokbSI2J9V9gVlM8ALWThDPnPu3EL7HPD2VDaZTggzcCCmbvc70qqPcC9mt60ogcrTiA3HEjwTK8ymKeuJMc4q6dVz200XnYUtLR9GYjPXvFOVr6W1zUK1WbPToaWJJuKnxBLnd0ftDEbMmj4loHYyhZyMjM91zQS4p7z8eKa9h0JrbacekcirexG0z4n3xz0QOWSvFj3jLhWXUIU21iIAwJtI3RbWa90I7rzAIqI3UElUJG7tLtUXzw4KQNETvXzqWaujEMenYlNIzLGxgB3AuJEQ124} and choose $\eps$ such that $C\eps=1/2$ to get     \begin{equation}       \tilde\cc        \leq C               (\MM                 + \zxvczxbcvdfghasdfrtsdafasdfasdfdsfgsdgh u-P \zxvczxbcvdfghasdfrtsdafasdfasdfdsfgsdgh_{L^p(B_1)}              + \zxvczxbcvdfghasdfrtsdafasdfasdfdsfgsdgh P \zxvczxbcvdfghasdfrtsdafasdfasdfdsfgsdgh_{L^p(B_1)} 
             ),      \llabel{THC SEl UyR Mh 66XU 7hNz pZVC5V nV 7 VjL 7kv WKf 7P 5hj6 t1vu gkLGdN X8 b gOX HWm 6W4 YE mxFG 4WaN EbGKsv 0p 4 OG0 Nrd uTe Za xNXq V4Bp mOdXIq 9a b PeD PbU Z4N Xt ohbY egCf xBNttE wc D YSD 637 jJ2 ms 6Ta1 J2xZ PtKnPw AX A tJA Rc8 n5d 93 TZi7 q6Wo nEDLwW Sz e Sue YFX 8cM hm Y6is 15pX aOYBbV fS C haL kBR Ks6 UO qG4j DVab fbdtny fi D BFI 7uh B39 FJ 6mYr CUUT f2X38J 43 K yZg 87i gFR 5R z1t3 jH9x lOg1h7 P7 W w8w jMJ qH3 l5 J5wU 8eH0 OogRCv L7 f JJg 1ug RfM XI GSuE Efbh 3hdNY3 x1 9 7jR qeP cdu sb fkuJ hEpw MvNBZV zL u qxJ 9b1 BTf Yk RJLj Oo8ThswELzXU3X7Ebd1KdZ7v1rN3GiirRXGKWK099ovBM0FDJCvkopYNQ2aN94Z7k0UnUKamE3OjU8DFYFFokbSI2J9V9gVlM8ALWThDPnPu3EL7HPD2VDaZTggzcCCmbvc70qqPcC9mt60ogcrTiA3HEjwTK8ymKeuJMc4q6dVz200XnYUtLR9GYjPXvFOVr6W1zUK1WbPToaWJJuKnxBLnd0ftDEbMmj4loHYyhZyMjM91zQS4p7z8eKa9h0JrbacekcirexG0z4n3xz0QOWSvFj3jLhWXUIU21iIAwJtI3RbWa90I7rzAIqI3UElUJG7tLtUXzw4KQNETvXzqWaujEMenYlNIzLGxgB3AuJEQ125}     \end{equation}  i.e.,     \begin{equation}       \zxvczxbcvdfghasdfrtsdafasdfasdfdsfgsdgh u-P \zxvczxbcvdfghasdfrtsdafasdfasdfdsfgsdgh_{L^p(B_r)}       \leq C              (\MM                 + \zxvczxbcvdfghasdfrtsdafasdfasdfdsfgsdgh u-P \zxvczxbcvdfghasdfrtsdafasdfasdfdsfgsdgh_{L^p(B_1)}              + \zxvczxbcvdfghasdfrtsdafasdfasdfdsfgsdgh P \zxvczxbcvdfghasdfrtsdafasdfasdfdsfgsdgh_{L^p(B_1)}               )              r^{d+\alpha +n/p}              \comma r \leq 1              .     \llabel{wc D YSD 637 jJ2 ms 6Ta1 J2xZ PtKnPw AX A tJA Rc8 n5d 93 TZi7 q6Wo nEDLwW Sz e Sue YFX 8cM hm Y6is 15pX aOYBbV fS C haL kBR Ks6 UO qG4j DVab fbdtny fi D BFI 7uh B39 FJ 6mYr CUUT f2X38J 43 K yZg 87i gFR 5R z1t3 jH9x lOg1h7 P7 W w8w jMJ qH3 l5 J5wU 8eH0 OogRCv L7 f JJg 1ug RfM XI GSuE Efbh 3hdNY3 x1 9 7jR qeP cdu sb fkuJ hEpw MvNBZV zL u qxJ 9b1 BTf Yk RJLj Oo1a EPIXvZ Aj v Xne fhK GsJ Ga wqjt U7r6 MPoydE H2 6 203 mGi JhF nT NCDB YlnP oKO6Pu XU 3 uu9 mSg 41v ma kk0E WUpS UtGBtD e6 d Kdx ZNT FuT i1 fMcM hq7P Ovf0hg Hl 8 fqv I3R K39 fn 9M8ThswELzXU3X7Ebd1KdZ7v1rN3GiirRXGKWK099ovBM0FDJCvkopYNQ2aN94Z7k0UnUKamE3OjU8DFYFFokbSI2J9V9gVlM8ALWThDPnPu3EL7HPD2VDaZTggzcCCmbvc70qqPcC9mt60ogcrTiA3HEjwTK8ymKeuJMc4q6dVz200XnYUtLR9GYjPXvFOVr6W1zUK1WbPToaWJJuKnxBLnd0ftDEbMmj4loHYyhZyMjM91zQS4p7z8eKa9h0JrbacekcirexG0z4n3xz0QOWSvFj3jLhWXUIU21iIAwJtI3RbWa90I7rzAIqI3UElUJG7tLtUXzw4KQNETvXzqWaujEMenYlNIzLGxgB3AuJEQ126}     \end{equation} Then by Sobolev's embedding we have the estimate     \begin{equation}     \zxvczxbcvdfghasdfrtsdafasdfasdfdsfgsdgh u -P \zxvczxbcvdfghasdfrtsdafasdfasdfdsfgsdgh_{L^q(B_r)}      \les              (\MM                 + \zxvczxbcvdfghasdfrtsdafasdfasdfdsfgsdgh u-P \zxvczxbcvdfghasdfrtsdafasdfasdfdsfgsdgh_{L^p(B_1)}              + \zxvczxbcvdfghasdfrtsdafasdfasdfdsfgsdgh P \zxvczxbcvdfghasdfrtsdafasdfasdfdsfgsdgh_{L^p(B_1)}               )              r^{d+\alpha+n/q}      .     \label{8ThswELzXU3X7Ebd1KdZ7v1rN3GiirRXGKWK099ovBM0FDJCvkopYNQ2aN94Z7k0UnUKamE3OjU8DFYFFokbSI2J9V9gVlM8ALWThDPnPu3EL7HPD2VDaZTggzcCCmbvc70qqPcC9mt60ogcrTiA3HEjwTK8ymKeuJMc4q6dVz200XnYUtLR9GYjPXvFOVr6W1zUK1WbPToaWJJuKnxBLnd0ftDEbMmj4loHYyhZyMjM91zQS4p7z8eKa9h0JrbacekcirexG0z4n3xz0QOWSvFj3jLhWXUIU21iIAwJtI3RbWa90I7rzAIqI3UElUJG7tLtUXzw4KQNETvXzqWaujEMenYlNIzLGxgB3AuJEQ127}     \end{equation} We now bound $P$ as     \begin{align}     \begin{split}       \zxvczxbcvdfghasdfrtsdafasdfasdfdsfgsdgh P \zxvczxbcvdfghasdfrtsdafasdfasdfdsfgsdgh_{L^p(B_r)}       &\leq \zxvczxbcvdfghasdfrtsdafasdfasdfdsfgsdgh u \zxvczxbcvdfghasdfrtsdafasdfasdfdsfgsdgh_{L^p(B_r)} + \zxvczxbcvdfghasdfrtsdafasdfasdfdsfgsdgh u - P \zxvczxbcvdfghasdfrtsdafasdfasdfdsfgsdgh_{L^p(B_r)}       \\&       \les               (\MM                 + \zxvczxbcvdfghasdfrtsdafasdfasdfdsfgsdgh u \zxvczxbcvdfghasdfrtsdafasdfasdfdsfgsdgh_{L^p(B_1)}              + \zxvczxbcvdfghasdfrtsdafasdfasdfdsfgsdgh Q \zxvczxbcvdfghasdfrtsdafasdfasdfdsfgsdgh_{L^p(B_1)}               ) r^{d+n/p}              +                (\MM                 + \zxvczxbcvdfghasdfrtsdafasdfasdfdsfgsdgh u \zxvczxbcvdfghasdfrtsdafasdfasdfdsfgsdgh_{L^p(B_1)}               + \zxvczxbcvdfghasdfrtsdafasdfasdfdsfgsdgh P \zxvczxbcvdfghasdfrtsdafasdfasdfdsfgsdgh_{L^p(B_1)}               ) r^{d+\alpha+n/p}         \\&        \les (\MM                 + \zxvczxbcvdfghasdfrtsdafasdfasdfdsfgsdgh u \zxvczxbcvdfghasdfrtsdafasdfasdfdsfgsdgh_{L^p(B_1)}              + \zxvczxbcvdfghasdfrtsdafasdfasdfdsfgsdgh Q \zxvczxbcvdfghasdfrtsdafasdfasdfdsfgsdgh_{L^p(B_1)}               )r^{d+n/p}              + \zxvczxbcvdfghasdfrtsdafasdfasdfdsfgsdgh P \zxvczxbcvdfghasdfrtsdafasdfasdfdsfgsdgh_{L^p(B_1)} |x|^{d+\alpha+n/p}              .     \llabel{X38J 43 K yZg 87i gFR 5R z1t3 jH9x lOg1h7 P7 W w8w jMJ qH3 l5 J5wU 8eH0 OogRCv L7 f JJg 1ug RfM XI GSuE Efbh 3hdNY3 x1 9 7jR qeP cdu sb fkuJ hEpw MvNBZV zL u qxJ 9b1 BTf Yk RJLj Oo1a EPIXvZ Aj v Xne fhK GsJ Ga wqjt U7r6 MPoydE H2 6 203 mGi JhF nT NCDB YlnP oKO6Pu XU 3 uu9 mSg 41v ma kk0E WUpS UtGBtD e6 d Kdx ZNT FuT i1 fMcM hq7P Ovf0hg Hl 8 fqv I3R K39 fn 9MaC Zgow 6e1iXj KC 5 lHO lpG pkK Xd Dxtz 0HxE fSMjXY L8 F vh7 dmJ kE8 QA KDo1 FqML HOZ2iL 9i I m3L Kva YiN K9 sb48 NxwY NR0nx2 t5 b WCk x2a 31k a8 fUIa RGzr 7oigRX 5s m 9PQ 7Sr 5St 8ThswELzXU3X7Ebd1KdZ7v1rN3GiirRXGKWK099ovBM0FDJCvkopYNQ2aN94Z7k0UnUKamE3OjU8DFYFFokbSI2J9V9gVlM8ALWThDPnPu3EL7HPD2VDaZTggzcCCmbvc70qqPcC9mt60ogcrTiA3HEjwTK8ymKeuJMc4q6dVz200XnYUtLR9GYjPXvFOVr6W1zUK1WbPToaWJJuKnxBLnd0ftDEbMmj4loHYyhZyMjM91zQS4p7z8eKa9h0JrbacekcirexG0z4n3xz0QOWSvFj3jLhWXUIU21iIAwJtI3RbWa90I7rzAIqI3UElUJG7tLtUXzw4KQNETvXzqWaujEMenYlNIzLGxgB3AuJEQ128}     \end{split}     \end{align} By the homogeneity of $P$, we have     \begin{align}     \begin{split}       \zxvczxbcvdfghasdfrtsdafasdfasdfdsfgsdgh P \zxvczxbcvdfghasdfrtsdafasdfasdfdsfgsdgh_{L^p(B_r)}       &\les  (\MM                 + \zxvczxbcvdfghasdfrtsdafasdfasdfdsfgsdgh u \zxvczxbcvdfghasdfrtsdafasdfasdfdsfgsdgh_{L^p(B_1)}              + \zxvczxbcvdfghasdfrtsdafasdfasdfdsfgsdgh Q \zxvczxbcvdfghasdfrtsdafasdfasdfdsfgsdgh_{L^p(B_1)}               ) r^{d+n/p}              + \zxvczxbcvdfghasdfrtsdafasdfasdfdsfgsdgh P \zxvczxbcvdfghasdfrtsdafasdfasdfdsfgsdgh_{L^p(B_r)}               r^{\alpha}               .     \end{split}     \label{8ThswELzXU3X7Ebd1KdZ7v1rN3GiirRXGKWK099ovBM0FDJCvkopYNQ2aN94Z7k0UnUKamE3OjU8DFYFFokbSI2J9V9gVlM8ALWThDPnPu3EL7HPD2VDaZTggzcCCmbvc70qqPcC9mt60ogcrTiA3HEjwTK8ymKeuJMc4q6dVz200XnYUtLR9GYjPXvFOVr6W1zUK1WbPToaWJJuKnxBLnd0ftDEbMmj4loHYyhZyMjM91zQS4p7z8eKa9h0JrbacekcirexG0z4n3xz0QOWSvFj3jLhWXUIU21iIAwJtI3RbWa90I7rzAIqI3UElUJG7tLtUXzw4KQNETvXzqWaujEMenYlNIzLGxgB3AuJEQ129}     \end{align} We denote by $C'$ the implicit constant in \eqref{8ThswELzXU3X7Ebd1KdZ7v1rN3GiirRXGKWK099ovBM0FDJCvkopYNQ2aN94Z7k0UnUKamE3OjU8DFYFFokbSI2J9V9gVlM8ALWThDPnPu3EL7HPD2VDaZTggzcCCmbvc70qqPcC9mt60ogcrTiA3HEjwTK8ymKeuJMc4q6dVz200XnYUtLR9GYjPXvFOVr6W1zUK1WbPToaWJJuKnxBLnd0ftDEbMmj4loHYyhZyMjM91zQS4p7z8eKa9h0JrbacekcirexG0z4n3xz0QOWSvFj3jLhWXUIU21iIAwJtI3RbWa90I7rzAIqI3UElUJG7tLtUXzw4KQNETvXzqWaujEMenYlNIzLGxgB3AuJEQ129} and choose $r$ sufficiently small such that $C'r^{\alpha} \leq 1/2$ to get      \begin{equation}       \zxvczxbcvdfghasdfrtsdafasdfasdfdsfgsdgh P \zxvczxbcvdfghasdfrtsdafasdfasdfdsfgsdgh_{L^p(B_r)}       \leq C' (\MM                     + \zxvczxbcvdfghasdfrtsdafasdfasdfdsfgsdgh u \zxvczxbcvdfghasdfrtsdafasdfasdfdsfgsdgh_{L^p(B_1)}                  + \zxvczxbcvdfghasdfrtsdafasdfasdfdsfgsdgh Q \zxvczxbcvdfghasdfrtsdafasdfasdfdsfgsdgh_{L^p(B_1)}                   ) r^{d+n/p}              + \frac{1}{2} \zxvczxbcvdfghasdfrtsdafasdfasdfdsfgsdgh P \zxvczxbcvdfghasdfrtsdafasdfasdfdsfgsdgh_{L^p(B_r)}              .     \llabel{1a EPIXvZ Aj v Xne fhK GsJ Ga wqjt U7r6 MPoydE H2 6 203 mGi JhF nT NCDB YlnP oKO6Pu XU 3 uu9 mSg 41v ma kk0E WUpS UtGBtD e6 d Kdx ZNT FuT i1 fMcM hq7P Ovf0hg Hl 8 fqv I3R K39 fn 9MaC Zgow 6e1iXj KC 5 lHO lpG pkK Xd Dxtz 0HxE fSMjXY L8 F vh7 dmJ kE8 QA KDo1 FqML HOZ2iL 9i I m3L Kva YiN K9 sb48 NxwY NR0nx2 t5 b WCk x2a 31k a8 fUIa RGzr 7oigRX 5s m 9PQ 7Sr 5St ZE Ymp8 VIWS hdzgDI 9v R F5J 81x 33n Ne fjBT VvGP vGsxQh Al G Fbe 1bQ i6J ap OJJa ceGq 1vvb8r F2 F 3M6 8eD lzG tX tVm5 y14v mwIXa2 OG Y hxU sXJ 0qg l5 ZGAt HPZd oDWrSb BS u NKi 6KW8ThswELzXU3X7Ebd1KdZ7v1rN3GiirRXGKWK099ovBM0FDJCvkopYNQ2aN94Z7k0UnUKamE3OjU8DFYFFokbSI2J9V9gVlM8ALWThDPnPu3EL7HPD2VDaZTggzcCCmbvc70qqPcC9mt60ogcrTiA3HEjwTK8ymKeuJMc4q6dVz200XnYUtLR9GYjPXvFOVr6W1zUK1WbPToaWJJuKnxBLnd0ftDEbMmj4loHYyhZyMjM91zQS4p7z8eKa9h0JrbacekcirexG0z4n3xz0QOWSvFj3jLhWXUIU21iIAwJtI3RbWa90I7rzAIqI3UElUJG7tLtUXzw4KQNETvXzqWaujEMenYlNIzLGxgB3AuJEQ130}     \end{equation} Thus,     \begin{equation}      \zxvczxbcvdfghasdfrtsdafasdfasdfdsfgsdgh P \zxvczxbcvdfghasdfrtsdafasdfasdfdsfgsdgh_{L^p(B_r)}     \les (\MM               + \zxvczxbcvdfghasdfrtsdafasdfasdfdsfgsdgh u \zxvczxbcvdfghasdfrtsdafasdfasdfdsfgsdgh_{L^p(B_1)}            + \zxvczxbcvdfghasdfrtsdafasdfasdfdsfgsdgh Q \zxvczxbcvdfghasdfrtsdafasdfasdfdsfgsdgh_{L^p(B_1)}            ) r^{d+n/p}            ,     \label{8ThswELzXU3X7Ebd1KdZ7v1rN3GiirRXGKWK099ovBM0FDJCvkopYNQ2aN94Z7k0UnUKamE3OjU8DFYFFokbSI2J9V9gVlM8ALWThDPnPu3EL7HPD2VDaZTggzcCCmbvc70qqPcC9mt60ogcrTiA3HEjwTK8ymKeuJMc4q6dVz200XnYUtLR9GYjPXvFOVr6W1zUK1WbPToaWJJuKnxBLnd0ftDEbMmj4loHYyhZyMjM91zQS4p7z8eKa9h0JrbacekcirexG0z4n3xz0QOWSvFj3jLhWXUIU21iIAwJtI3RbWa90I7rzAIqI3UElUJG7tLtUXzw4KQNETvXzqWaujEMenYlNIzLGxgB3AuJEQ131}     \end{equation} which is~\eqref{8ThswELzXU3X7Ebd1KdZ7v1rN3GiirRXGKWK099ovBM0FDJCvkopYNQ2aN94Z7k0UnUKamE3OjU8DFYFFokbSI2J9V9gVlM8ALWThDPnPu3EL7HPD2VDaZTggzcCCmbvc70qqPcC9mt60ogcrTiA3HEjwTK8ymKeuJMc4q6dVz200XnYUtLR9GYjPXvFOVr6W1zUK1WbPToaWJJuKnxBLnd0ftDEbMmj4loHYyhZyMjM91zQS4p7z8eKa9h0JrbacekcirexG0z4n3xz0QOWSvFj3jLhWXUIU21iIAwJtI3RbWa90I7rzAIqI3UElUJG7tLtUXzw4KQNETvXzqWaujEMenYlNIzLGxgB3AuJEQ16}.  Moreover, by \eqref{8ThswELzXU3X7Ebd1KdZ7v1rN3GiirRXGKWK099ovBM0FDJCvkopYNQ2aN94Z7k0UnUKamE3OjU8DFYFFokbSI2J9V9gVlM8ALWThDPnPu3EL7HPD2VDaZTggzcCCmbvc70qqPcC9mt60ogcrTiA3HEjwTK8ymKeuJMc4q6dVz200XnYUtLR9GYjPXvFOVr6W1zUK1WbPToaWJJuKnxBLnd0ftDEbMmj4loHYyhZyMjM91zQS4p7z8eKa9h0JrbacekcirexG0z4n3xz0QOWSvFj3jLhWXUIU21iIAwJtI3RbWa90I7rzAIqI3UElUJG7tLtUXzw4KQNETvXzqWaujEMenYlNIzLGxgB3AuJEQ127} and \eqref{8ThswELzXU3X7Ebd1KdZ7v1rN3GiirRXGKWK099ovBM0FDJCvkopYNQ2aN94Z7k0UnUKamE3OjU8DFYFFokbSI2J9V9gVlM8ALWThDPnPu3EL7HPD2VDaZTggzcCCmbvc70qqPcC9mt60ogcrTiA3HEjwTK8ymKeuJMc4q6dVz200XnYUtLR9GYjPXvFOVr6W1zUK1WbPToaWJJuKnxBLnd0ftDEbMmj4loHYyhZyMjM91zQS4p7z8eKa9h0JrbacekcirexG0z4n3xz0QOWSvFj3jLhWXUIU21iIAwJtI3RbWa90I7rzAIqI3UElUJG7tLtUXzw4KQNETvXzqWaujEMenYlNIzLGxgB3AuJEQ131}, we obtain      \begin{align}     \begin{split}       \zxvczxbcvdfghasdfrtsdafasdfasdfdsfgsdgh u - P \zxvczxbcvdfghasdfrtsdafasdfasdfdsfgsdgh_{L^q(B_r)}       &\les (M                 + \zxvczxbcvdfghasdfrtsdafasdfasdfdsfgsdgh u-P \zxvczxbcvdfghasdfrtsdafasdfasdfdsfgsdgh_{L^p(B_1)}                + \zxvczxbcvdfghasdfrtsdafasdfasdfdsfgsdgh u \zxvczxbcvdfghasdfrtsdafasdfasdfdsfgsdgh_{L^p(B_1)}                + \zxvczxbcvdfghasdfrtsdafasdfasdfdsfgsdgh Q \zxvczxbcvdfghasdfrtsdafasdfasdfdsfgsdgh_{L^p(B_1)}               ) r^{d+n/q}        \\&        \les  (M                 + \zxvczxbcvdfghasdfrtsdafasdfasdfdsfgsdgh P \zxvczxbcvdfghasdfrtsdafasdfasdfdsfgsdgh_{L^p(B_1)}                + \zxvczxbcvdfghasdfrtsdafasdfasdfdsfgsdgh u \zxvczxbcvdfghasdfrtsdafasdfasdfdsfgsdgh_{L^p(B_1)}                + \zxvczxbcvdfghasdfrtsdafasdfasdfdsfgsdgh Q \zxvczxbcvdfghasdfrtsdafasdfasdfdsfgsdgh_{L^p(B_1)}               ) r^{d+n/q}        \\&        \les (M                 + \zxvczxbcvdfghasdfrtsdafasdfasdfdsfgsdgh u \zxvczxbcvdfghasdfrtsdafasdfasdfdsfgsdgh_{L^p(B_1)}                + \zxvczxbcvdfghasdfrtsdafasdfasdfdsfgsdgh Q \zxvczxbcvdfghasdfrtsdafasdfasdfdsfgsdgh_{L^p(B_1)}               ) r^{d+n/q}               ,     \end{split}    \llabel{aC Zgow 6e1iXj KC 5 lHO lpG pkK Xd Dxtz 0HxE fSMjXY L8 F vh7 dmJ kE8 QA KDo1 FqML HOZ2iL 9i I m3L Kva YiN K9 sb48 NxwY NR0nx2 t5 b WCk x2a 31k a8 fUIa RGzr 7oigRX 5s m 9PQ 7Sr 5St ZE Ymp8 VIWS hdzgDI 9v R F5J 81x 33n Ne fjBT VvGP vGsxQh Al G Fbe 1bQ i6J ap OJJa ceGq 1vvb8r F2 F 3M6 8eD lzG tX tVm5 y14v mwIXa2 OG Y hxU sXJ 0qg l5 ZGAt HPZd oDWrSb BS u NKi 6KW gr3 9s 9tc7 WM4A ws1PzI 5c C O7Z 8y9 lMT LA dwhz Mxz9 hjlWHj bJ 5 CqM jht y9l Mn 4rc7 6Amk KJimvH 9r O tbc tCK rsi B0 4cFV Dl1g cvfWh6 5n x y9Z S4W Pyo QB yr3v fBkj TZKtEZ 7r U fd8ThswELzXU3X7Ebd1KdZ7v1rN3GiirRXGKWK099ovBM0FDJCvkopYNQ2aN94Z7k0UnUKamE3OjU8DFYFFokbSI2J9V9gVlM8ALWThDPnPu3EL7HPD2VDaZTggzcCCmbvc70qqPcC9mt60ogcrTiA3HEjwTK8ymKeuJMc4q6dVz200XnYUtLR9GYjPXvFOVr6W1zUK1WbPToaWJJuKnxBLnd0ftDEbMmj4loHYyhZyMjM91zQS4p7z8eKa9h0JrbacekcirexG0z4n3xz0QOWSvFj3jLhWXUIU21iIAwJtI3RbWa90I7rzAIqI3UElUJG7tLtUXzw4KQNETvXzqWaujEMenYlNIzLGxgB3AuJEQ132}     \end{align} where we used \eqref{8ThswELzXU3X7Ebd1KdZ7v1rN3GiirRXGKWK099ovBM0FDJCvkopYNQ2aN94Z7k0UnUKamE3OjU8DFYFFokbSI2J9V9gVlM8ALWThDPnPu3EL7HPD2VDaZTggzcCCmbvc70qqPcC9mt60ogcrTiA3HEjwTK8ymKeuJMc4q6dVz200XnYUtLR9GYjPXvFOVr6W1zUK1WbPToaWJJuKnxBLnd0ftDEbMmj4loHYyhZyMjM91zQS4p7z8eKa9h0JrbacekcirexG0z4n3xz0QOWSvFj3jLhWXUIU21iIAwJtI3RbWa90I7rzAIqI3UElUJG7tLtUXzw4KQNETvXzqWaujEMenYlNIzLGxgB3AuJEQ131} in the last inequality. Therefore, \eqref{8ThswELzXU3X7Ebd1KdZ7v1rN3GiirRXGKWK099ovBM0FDJCvkopYNQ2aN94Z7k0UnUKamE3OjU8DFYFFokbSI2J9V9gVlM8ALWThDPnPu3EL7HPD2VDaZTggzcCCmbvc70qqPcC9mt60ogcrTiA3HEjwTK8ymKeuJMc4q6dVz200XnYUtLR9GYjPXvFOVr6W1zUK1WbPToaWJJuKnxBLnd0ftDEbMmj4loHYyhZyMjM91zQS4p7z8eKa9h0JrbacekcirexG0z4n3xz0QOWSvFj3jLhWXUIU21iIAwJtI3RbWa90I7rzAIqI3UElUJG7tLtUXzw4KQNETvXzqWaujEMenYlNIzLGxgB3AuJEQ16} and \eqref{8ThswELzXU3X7Ebd1KdZ7v1rN3GiirRXGKWK099ovBM0FDJCvkopYNQ2aN94Z7k0UnUKamE3OjU8DFYFFokbSI2J9V9gVlM8ALWThDPnPu3EL7HPD2VDaZTggzcCCmbvc70qqPcC9mt60ogcrTiA3HEjwTK8ymKeuJMc4q6dVz200XnYUtLR9GYjPXvFOVr6W1zUK1WbPToaWJJuKnxBLnd0ftDEbMmj4loHYyhZyMjM91zQS4p7z8eKa9h0JrbacekcirexG0z4n3xz0QOWSvFj3jLhWXUIU21iIAwJtI3RbWa90I7rzAIqI3UElUJG7tLtUXzw4KQNETvXzqWaujEMenYlNIzLGxgB3AuJEQ17} follow. \end{proof} \par The next lemma is a consequence of Theorem~\ref{T01}.  \par \cole \begin{Lemma} \label{L05} Under the same conditions as in Theorem~\ref{T01}, assume that $f \in C^{d-m+1, \alpha}_{L^p}(0)$ and $a_{\nu} \in C^{1-m+|\nu|, \alpha}$ for~$|\nu| = 0,1$. Then there exists $P_d \in \dPP_d$ such that   \begin{equation}   \sup_{r\leq1} \frac{\zxvczxbcvdfghasdfrtsdafasdfasdfdsfgsdgh u - P_d \zxvczxbcvdfghasdfrtsdafasdfasdfdsfgsdgh_{L^p(B_r)}}{r^{d+\alpha+n/p}}     < \infty    .    \label{8ThswELzXU3X7Ebd1KdZ7v1rN3GiirRXGKWK099ovBM0FDJCvkopYNQ2aN94Z7k0UnUKamE3OjU8DFYFFokbSI2J9V9gVlM8ALWThDPnPu3EL7HPD2VDaZTggzcCCmbvc70qqPcC9mt60ogcrTiA3HEjwTK8ymKeuJMc4q6dVz200XnYUtLR9GYjPXvFOVr6W1zUK1WbPToaWJJuKnxBLnd0ftDEbMmj4loHYyhZyMjM91zQS4p7z8eKa9h0JrbacekcirexG0z4n3xz0QOWSvFj3jLhWXUIU21iIAwJtI3RbWa90I7rzAIqI3UElUJG7tLtUXzw4KQNETvXzqWaujEMenYlNIzLGxgB3AuJEQ133}   \end{equation} Moreover, $u \in C^{d+1, \alpha}_{L^q}(0)$ for any $q \in [ 1, np/(n-mp))$, and    \begin{equation}       \lbrack u \rbrack_{C^{d+1,\alpha}_{L^q}(0)}       \les \zxvczxbcvdfghasdfrtsdafasdfasdfdsfgsdgh u \zxvczxbcvdfghasdfrtsdafasdfasdfdsfgsdgh_{L^p(B_1)}             + \lbrack f \rbrack_{C^{d-m}_{L^p}(0)}             + \lbrack f \rbrack_{C^{d-m+1,\alpha}_{L^p}(0)}   ,   \label{8ThswELzXU3X7Ebd1KdZ7v1rN3GiirRXGKWK099ovBM0FDJCvkopYNQ2aN94Z7k0UnUKamE3OjU8DFYFFokbSI2J9V9gVlM8ALWThDPnPu3EL7HPD2VDaZTggzcCCmbvc70qqPcC9mt60ogcrTiA3HEjwTK8ymKeuJMc4q6dVz200XnYUtLR9GYjPXvFOVr6W1zUK1WbPToaWJJuKnxBLnd0ftDEbMmj4loHYyhZyMjM91zQS4p7z8eKa9h0JrbacekcirexG0z4n3xz0QOWSvFj3jLhWXUIU21iIAwJtI3RbWa90I7rzAIqI3UElUJG7tLtUXzw4KQNETvXzqWaujEMenYlNIzLGxgB3AuJEQ134}   \end{equation} where the implicit constant in \eqref{8ThswELzXU3X7Ebd1KdZ7v1rN3GiirRXGKWK099ovBM0FDJCvkopYNQ2aN94Z7k0UnUKamE3OjU8DFYFFokbSI2J9V9gVlM8ALWThDPnPu3EL7HPD2VDaZTggzcCCmbvc70qqPcC9mt60ogcrTiA3HEjwTK8ymKeuJMc4q6dVz200XnYUtLR9GYjPXvFOVr6W1zUK1WbPToaWJJuKnxBLnd0ftDEbMmj4loHYyhZyMjM91zQS4p7z8eKa9h0JrbacekcirexG0z4n3xz0QOWSvFj3jLhWXUIU21iIAwJtI3RbWa90I7rzAIqI3UElUJG7tLtUXzw4KQNETvXzqWaujEMenYlNIzLGxgB3AuJEQ134} depends on $p$, $q$, $d$, and~$a_{\nu}$. \end{Lemma} \colb \par \begin{proof}[Proof of Lemma~\ref{L05}] Let $P_d$ and $Q_{d-m}$ be $P$ and $Q$ in the proof of Theorem~\ref{T01} (the subindices are added to emphasize the orders of $P$ and $Q$). Then \eqref{8ThswELzXU3X7Ebd1KdZ7v1rN3GiirRXGKWK099ovBM0FDJCvkopYNQ2aN94Z7k0UnUKamE3OjU8DFYFFokbSI2J9V9gVlM8ALWThDPnPu3EL7HPD2VDaZTggzcCCmbvc70qqPcC9mt60ogcrTiA3HEjwTK8ymKeuJMc4q6dVz200XnYUtLR9GYjPXvFOVr6W1zUK1WbPToaWJJuKnxBLnd0ftDEbMmj4loHYyhZyMjM91zQS4p7z8eKa9h0JrbacekcirexG0z4n3xz0QOWSvFj3jLhWXUIU21iIAwJtI3RbWa90I7rzAIqI3UElUJG7tLtUXzw4KQNETvXzqWaujEMenYlNIzLGxgB3AuJEQ133} follows immediately from~\eqref{8ThswELzXU3X7Ebd1KdZ7v1rN3GiirRXGKWK099ovBM0FDJCvkopYNQ2aN94Z7k0UnUKamE3OjU8DFYFFokbSI2J9V9gVlM8ALWThDPnPu3EL7HPD2VDaZTggzcCCmbvc70qqPcC9mt60ogcrTiA3HEjwTK8ymKeuJMc4q6dVz200XnYUtLR9GYjPXvFOVr6W1zUK1WbPToaWJJuKnxBLnd0ftDEbMmj4loHYyhZyMjM91zQS4p7z8eKa9h0JrbacekcirexG0z4n3xz0QOWSvFj3jLhWXUIU21iIAwJtI3RbWa90I7rzAIqI3UElUJG7tLtUXzw4KQNETvXzqWaujEMenYlNIzLGxgB3AuJEQ117}. Recalling that $Q_{d-m} = \sum_{|\nu|=m} a_{\nu}(0) \partial^{\nu}P_d$, we rewrite $Lu = f$ as   \begin{equation}    L(u-P_d)    =  (f - Q_{d-m})    +\sum_{|\nu|=m} (a_{\nu}(0) - a_{\nu} ) \partial^{\nu}P_d    - \sum_{|\nu|<m} a_{\nu} \partial^{\nu} P_d    .   \label{8ThswELzXU3X7Ebd1KdZ7v1rN3GiirRXGKWK099ovBM0FDJCvkopYNQ2aN94Z7k0UnUKamE3OjU8DFYFFokbSI2J9V9gVlM8ALWThDPnPu3EL7HPD2VDaZTggzcCCmbvc70qqPcC9mt60ogcrTiA3HEjwTK8ymKeuJMc4q6dVz200XnYUtLR9GYjPXvFOVr6W1zUK1WbPToaWJJuKnxBLnd0ftDEbMmj4loHYyhZyMjM91zQS4p7z8eKa9h0JrbacekcirexG0z4n3xz0QOWSvFj3jLhWXUIU21iIAwJtI3RbWa90I7rzAIqI3UElUJG7tLtUXzw4KQNETvXzqWaujEMenYlNIzLGxgB3AuJEQ135}   \end{equation} Denoting the right hand side of \eqref{8ThswELzXU3X7Ebd1KdZ7v1rN3GiirRXGKWK099ovBM0FDJCvkopYNQ2aN94Z7k0UnUKamE3OjU8DFYFFokbSI2J9V9gVlM8ALWThDPnPu3EL7HPD2VDaZTggzcCCmbvc70qqPcC9mt60ogcrTiA3HEjwTK8ymKeuJMc4q6dVz200XnYUtLR9GYjPXvFOVr6W1zUK1WbPToaWJJuKnxBLnd0ftDEbMmj4loHYyhZyMjM91zQS4p7z8eKa9h0JrbacekcirexG0z4n3xz0QOWSvFj3jLhWXUIU21iIAwJtI3RbWa90I7rzAIqI3UElUJG7tLtUXzw4KQNETvXzqWaujEMenYlNIzLGxgB3AuJEQ135} by $\FF$, our goal is to show    \begin{equation}    \zxvczxbcvdfghasdfrtsdafasdfasdfdsfgsdgh \FF - \tilde Q \zxvczxbcvdfghasdfrtsdafasdfasdfdsfgsdgh_{L^p(B_r)}     \les r^{d-m+1+\alpha + n/p}    ,   \llabel{ZE Ymp8 VIWS hdzgDI 9v R F5J 81x 33n Ne fjBT VvGP vGsxQh Al G Fbe 1bQ i6J ap OJJa ceGq 1vvb8r F2 F 3M6 8eD lzG tX tVm5 y14v mwIXa2 OG Y hxU sXJ 0qg l5 ZGAt HPZd oDWrSb BS u NKi 6KW gr3 9s 9tc7 WM4A ws1PzI 5c C O7Z 8y9 lMT LA dwhz Mxz9 hjlWHj bJ 5 CqM jht y9l Mn 4rc7 6Amk KJimvH 9r O tbc tCK rsi B0 4cFV Dl1g cvfWh6 5n x y9Z S4W Pyo QB yr3v fBkj TZKtEZ 7r U fdM icd yCV qn D036 HJWM tYfL9f yX x O7m IcF E1O uL QsAQ NfWv 6kV8Im 7Q 6 GsX NCV 0YP oC jnWn 6L25 qUMTe7 1v a hnH DAo XAb Tc zhPc fjrj W5M5G0 nz N M5T nlJ WOP Lh M6U2 ZFxw pg4Nej P88ThswELzXU3X7Ebd1KdZ7v1rN3GiirRXGKWK099ovBM0FDJCvkopYNQ2aN94Z7k0UnUKamE3OjU8DFYFFokbSI2J9V9gVlM8ALWThDPnPu3EL7HPD2VDaZTggzcCCmbvc70qqPcC9mt60ogcrTiA3HEjwTK8ymKeuJMc4q6dVz200XnYUtLR9GYjPXvFOVr6W1zUK1WbPToaWJJuKnxBLnd0ftDEbMmj4loHYyhZyMjM91zQS4p7z8eKa9h0JrbacekcirexG0z4n3xz0QOWSvFj3jLhWXUIU21iIAwJtI3RbWa90I7rzAIqI3UElUJG7tLtUXzw4KQNETvXzqWaujEMenYlNIzLGxgB3AuJEQ136}   \end{equation} for some $\tilde Q \in \dPP_{d-m+1}$. Since $f \in C_{L^p}^{d-m+1,\alpha}$, there exists $Q_{d-m+1} \in \dPP_{d+1-m}$ such that   \begin{equation}    \zxvczxbcvdfghasdfrtsdafasdfasdfdsfgsdgh f - Q_{d-m} - Q_{d-m+1} \zxvczxbcvdfghasdfrtsdafasdfasdfdsfgsdgh_{L^p(B_r)}     \leq \lbrack f \rbrack_{C^{d-m+1,\alpha}}(0)             r^{d-m+1+\alpha+n/p}             .    \label{8ThswELzXU3X7Ebd1KdZ7v1rN3GiirRXGKWK099ovBM0FDJCvkopYNQ2aN94Z7k0UnUKamE3OjU8DFYFFokbSI2J9V9gVlM8ALWThDPnPu3EL7HPD2VDaZTggzcCCmbvc70qqPcC9mt60ogcrTiA3HEjwTK8ymKeuJMc4q6dVz200XnYUtLR9GYjPXvFOVr6W1zUK1WbPToaWJJuKnxBLnd0ftDEbMmj4loHYyhZyMjM91zQS4p7z8eKa9h0JrbacekcirexG0z4n3xz0QOWSvFj3jLhWXUIU21iIAwJtI3RbWa90I7rzAIqI3UElUJG7tLtUXzw4KQNETvXzqWaujEMenYlNIzLGxgB3AuJEQ137}   \end{equation}
Therefore,    \begin{equation}    \FF     = (f - Q_{d-m} - Q_{d-m+1})    + Q_{d-m+1}     +\sum_{|\nu|=m} (a_{\nu}(0) - a_{\nu} ) \partial^{\nu}P_d    - \sum_{|\nu|<m} a_{\nu} \partial^{\nu} P_d    .    \llabel{ gr3 9s 9tc7 WM4A ws1PzI 5c C O7Z 8y9 lMT LA dwhz Mxz9 hjlWHj bJ 5 CqM jht y9l Mn 4rc7 6Amk KJimvH 9r O tbc tCK rsi B0 4cFV Dl1g cvfWh6 5n x y9Z S4W Pyo QB yr3v fBkj TZKtEZ 7r U fdM icd yCV qn D036 HJWM tYfL9f yX x O7m IcF E1O uL QsAQ NfWv 6kV8Im 7Q 6 GsX NCV 0YP oC jnWn 6L25 qUMTe7 1v a hnH DAo XAb Tc zhPc fjrj W5M5G0 nz N M5T nlJ WOP Lh M6U2 ZFxw pg4Nej P8 U Q09 JX9 n7S kE WixE Rwgy Fvttzp 4A s v5F Tnn MzL Vh FUn5 6tFY CxZ1Bz Q3 E TfD lCa d7V fo MwPm ngrD HPfZV0 aY k Ojr ZUw 799 et oYuB MIC4 ovEY8D OL N URV Q5l ti1 iS NZAd wWr6 Q8oP8ThswELzXU3X7Ebd1KdZ7v1rN3GiirRXGKWK099ovBM0FDJCvkopYNQ2aN94Z7k0UnUKamE3OjU8DFYFFokbSI2J9V9gVlM8ALWThDPnPu3EL7HPD2VDaZTggzcCCmbvc70qqPcC9mt60ogcrTiA3HEjwTK8ymKeuJMc4q6dVz200XnYUtLR9GYjPXvFOVr6W1zUK1WbPToaWJJuKnxBLnd0ftDEbMmj4loHYyhZyMjM91zQS4p7z8eKa9h0JrbacekcirexG0z4n3xz0QOWSvFj3jLhWXUIU21iIAwJtI3RbWa90I7rzAIqI3UElUJG7tLtUXzw4KQNETvXzqWaujEMenYlNIzLGxgB3AuJEQ138}   \end{equation} Let $Q = \sum_{|\nu|=m} (a_{\nu}(0) - a_{\nu} ) \partial^{\nu}P_d - \sum_{|\nu|<m} a_{\nu} \partial^{\nu} P_d $. The assumptions on coefficients $a_{\nu}$ allow us to approximate $Q$ by a homogeneous polynomial of order~$d-m+1$. Indeed, we have   \begin{equation}       \zxvczxbcvdfghasdfrtsdafasdfasdfdsfgsdgh a_{\nu} - a_{\nu}(0) - p_{\nu}^{(1)} \zxvczxbcvdfghasdfrtsdafasdfasdfdsfgsdgh_{L^p(B_r)}       \leq \lbrack a_{\nu} \rbrack_{C^{1,\alpha}_{L^p}} (0)              r^{1+\alpha+n/p}       \comma        |\nu| = m       ,     \label{8ThswELzXU3X7Ebd1KdZ7v1rN3GiirRXGKWK099ovBM0FDJCvkopYNQ2aN94Z7k0UnUKamE3OjU8DFYFFokbSI2J9V9gVlM8ALWThDPnPu3EL7HPD2VDaZTggzcCCmbvc70qqPcC9mt60ogcrTiA3HEjwTK8ymKeuJMc4q6dVz200XnYUtLR9GYjPXvFOVr6W1zUK1WbPToaWJJuKnxBLnd0ftDEbMmj4loHYyhZyMjM91zQS4p7z8eKa9h0JrbacekcirexG0z4n3xz0QOWSvFj3jLhWXUIU21iIAwJtI3RbWa90I7rzAIqI3UElUJG7tLtUXzw4KQNETvXzqWaujEMenYlNIzLGxgB3AuJEQ139}   \end{equation} and      \begin{equation}       \zxvczxbcvdfghasdfrtsdafasdfasdfdsfgsdgh a_{\nu} - p_{\nu}^{(0)} \zxvczxbcvdfghasdfrtsdafasdfasdfdsfgsdgh_{L^p(B_r)}       \leq \lbrack a_{\nu} \rbrack_{C^{0,\alpha}_{L^p}}(0)       \comma        |\nu|=m-1       ,     \label{8ThswELzXU3X7Ebd1KdZ7v1rN3GiirRXGKWK099ovBM0FDJCvkopYNQ2aN94Z7k0UnUKamE3OjU8DFYFFokbSI2J9V9gVlM8ALWThDPnPu3EL7HPD2VDaZTggzcCCmbvc70qqPcC9mt60ogcrTiA3HEjwTK8ymKeuJMc4q6dVz200XnYUtLR9GYjPXvFOVr6W1zUK1WbPToaWJJuKnxBLnd0ftDEbMmj4loHYyhZyMjM91zQS4p7z8eKa9h0JrbacekcirexG0z4n3xz0QOWSvFj3jLhWXUIU21iIAwJtI3RbWa90I7rzAIqI3UElUJG7tLtUXzw4KQNETvXzqWaujEMenYlNIzLGxgB3AuJEQ140}     \end{equation} for some $p_{\nu}^{(1)} \in \dPP_1$ and $p_{\nu}^{(0)} \in \dPP_0$. Then $     \tilde P     =         - \sum_{|\nu|=m} p_{\nu}^{(1)} \partial^{\nu} P_d        - \sum_{|\nu|=m-1} p_{\nu}^{(0)} \partial^{\nu} P_d $ is a homogeneous polynomial of order $d-m+1$, so is $\tilde P + Q_{d-m+1}$. Combining \eqref{8ThswELzXU3X7Ebd1KdZ7v1rN3GiirRXGKWK099ovBM0FDJCvkopYNQ2aN94Z7k0UnUKamE3OjU8DFYFFokbSI2J9V9gVlM8ALWThDPnPu3EL7HPD2VDaZTggzcCCmbvc70qqPcC9mt60ogcrTiA3HEjwTK8ymKeuJMc4q6dVz200XnYUtLR9GYjPXvFOVr6W1zUK1WbPToaWJJuKnxBLnd0ftDEbMmj4loHYyhZyMjM91zQS4p7z8eKa9h0JrbacekcirexG0z4n3xz0QOWSvFj3jLhWXUIU21iIAwJtI3RbWa90I7rzAIqI3UElUJG7tLtUXzw4KQNETvXzqWaujEMenYlNIzLGxgB3AuJEQ137},  \eqref{8ThswELzXU3X7Ebd1KdZ7v1rN3GiirRXGKWK099ovBM0FDJCvkopYNQ2aN94Z7k0UnUKamE3OjU8DFYFFokbSI2J9V9gVlM8ALWThDPnPu3EL7HPD2VDaZTggzcCCmbvc70qqPcC9mt60ogcrTiA3HEjwTK8ymKeuJMc4q6dVz200XnYUtLR9GYjPXvFOVr6W1zUK1WbPToaWJJuKnxBLnd0ftDEbMmj4loHYyhZyMjM91zQS4p7z8eKa9h0JrbacekcirexG0z4n3xz0QOWSvFj3jLhWXUIU21iIAwJtI3RbWa90I7rzAIqI3UElUJG7tLtUXzw4KQNETvXzqWaujEMenYlNIzLGxgB3AuJEQ139}, and  \eqref{8ThswELzXU3X7Ebd1KdZ7v1rN3GiirRXGKWK099ovBM0FDJCvkopYNQ2aN94Z7k0UnUKamE3OjU8DFYFFokbSI2J9V9gVlM8ALWThDPnPu3EL7HPD2VDaZTggzcCCmbvc70qqPcC9mt60ogcrTiA3HEjwTK8ymKeuJMc4q6dVz200XnYUtLR9GYjPXvFOVr6W1zUK1WbPToaWJJuKnxBLnd0ftDEbMmj4loHYyhZyMjM91zQS4p7z8eKa9h0JrbacekcirexG0z4n3xz0QOWSvFj3jLhWXUIU21iIAwJtI3RbWa90I7rzAIqI3UElUJG7tLtUXzw4KQNETvXzqWaujEMenYlNIzLGxgB3AuJEQ140}, we have  \begin{equation} \zxvczxbcvdfghasdfrtsdafasdfasdfdsfgsdgh \tilde F - \tilde P - Q_{d-m+1} \zxvczxbcvdfghasdfrtsdafasdfasdfdsfgsdgh_{L^p(B_r)} \leq ( \lbrack f \rbrack_{C^{d-m+1,\alpha}}(0) 	+  \lbrack a_{\nu} \rbrack_{C^{1,\alpha}_{L^p}} (0) \zxvczxbcvdfghasdfrtsdafasdfasdfdsfgsdgh P_d \zxvczxbcvdfghasdfrtsdafasdfasdfdsfgsdgh_{L^p(B_1)} 	 +  \lbrack a_{\nu} \rbrack_{C^{0,\alpha}_{L^p}} (0) 	 ) r^{d-m+1+\alpha + n/p} 	 . \label{8ThswELzXU3X7Ebd1KdZ7v1rN3GiirRXGKWK099ovBM0FDJCvkopYNQ2aN94Z7k0UnUKamE3OjU8DFYFFokbSI2J9V9gVlM8ALWThDPnPu3EL7HPD2VDaZTggzcCCmbvc70qqPcC9mt60ogcrTiA3HEjwTK8ymKeuJMc4q6dVz200XnYUtLR9GYjPXvFOVr6W1zUK1WbPToaWJJuKnxBLnd0ftDEbMmj4loHYyhZyMjM91zQS4p7z8eKa9h0JrbacekcirexG0z4n3xz0QOWSvFj3jLhWXUIU21iIAwJtI3RbWa90I7rzAIqI3UElUJG7tLtUXzw4KQNETvXzqWaujEMenYlNIzLGxgB3AuJEQ141} \end{equation} Conditions \eqref{8ThswELzXU3X7Ebd1KdZ7v1rN3GiirRXGKWK099ovBM0FDJCvkopYNQ2aN94Z7k0UnUKamE3OjU8DFYFFokbSI2J9V9gVlM8ALWThDPnPu3EL7HPD2VDaZTggzcCCmbvc70qqPcC9mt60ogcrTiA3HEjwTK8ymKeuJMc4q6dVz200XnYUtLR9GYjPXvFOVr6W1zUK1WbPToaWJJuKnxBLnd0ftDEbMmj4loHYyhZyMjM91zQS4p7z8eKa9h0JrbacekcirexG0z4n3xz0QOWSvFj3jLhWXUIU21iIAwJtI3RbWa90I7rzAIqI3UElUJG7tLtUXzw4KQNETvXzqWaujEMenYlNIzLGxgB3AuJEQ133} with \eqref{8ThswELzXU3X7Ebd1KdZ7v1rN3GiirRXGKWK099ovBM0FDJCvkopYNQ2aN94Z7k0UnUKamE3OjU8DFYFFokbSI2J9V9gVlM8ALWThDPnPu3EL7HPD2VDaZTggzcCCmbvc70qqPcC9mt60ogcrTiA3HEjwTK8ymKeuJMc4q6dVz200XnYUtLR9GYjPXvFOVr6W1zUK1WbPToaWJJuKnxBLnd0ftDEbMmj4loHYyhZyMjM91zQS4p7z8eKa9h0JrbacekcirexG0z4n3xz0QOWSvFj3jLhWXUIU21iIAwJtI3RbWa90I7rzAIqI3UElUJG7tLtUXzw4KQNETvXzqWaujEMenYlNIzLGxgB3AuJEQ141} allow us to apply Theorem~\ref{T01} with $d$ replaced by~$d+1$. Then there exists a homogeneous polynomial $P_{d+1}$ of degree $d+1$ satisfying     \begin{equation}       \sum_{|\nu|=m} a_{\nu} (0) \partial^{\nu} P_{d+1}       = \tilde P + Q_{d-m+1}       = \tilde Q       .     \label{8ThswELzXU3X7Ebd1KdZ7v1rN3GiirRXGKWK099ovBM0FDJCvkopYNQ2aN94Z7k0UnUKamE3OjU8DFYFFokbSI2J9V9gVlM8ALWThDPnPu3EL7HPD2VDaZTggzcCCmbvc70qqPcC9mt60ogcrTiA3HEjwTK8ymKeuJMc4q6dVz200XnYUtLR9GYjPXvFOVr6W1zUK1WbPToaWJJuKnxBLnd0ftDEbMmj4loHYyhZyMjM91zQS4p7z8eKa9h0JrbacekcirexG0z4n3xz0QOWSvFj3jLhWXUIU21iIAwJtI3RbWa90I7rzAIqI3UElUJG7tLtUXzw4KQNETvXzqWaujEMenYlNIzLGxgB3AuJEQ142}     \end{equation} Moreover, we have the bound     \begin{equation}       \zxvczxbcvdfghasdfrtsdafasdfasdfdsfgsdgh u \zxvczxbcvdfghasdfrtsdafasdfasdfdsfgsdgh_{L^p(B_r)}       \les ( \lbrack f \rbrack_{C^{d-m+1,\alpha}_{L^p}(0)}                  + \zxvczxbcvdfghasdfrtsdafasdfasdfdsfgsdgh u \zxvczxbcvdfghasdfrtsdafasdfasdfdsfgsdgh_{L^p(B_1)}                 + \zxvczxbcvdfghasdfrtsdafasdfasdfdsfgsdgh \tilde Q \zxvczxbcvdfghasdfrtsdafasdfasdfdsfgsdgh_{L^p(B_{1/2}(0))}              ) r^{d+1+\alpha+n/p}        ,     \label{8ThswELzXU3X7Ebd1KdZ7v1rN3GiirRXGKWK099ovBM0FDJCvkopYNQ2aN94Z7k0UnUKamE3OjU8DFYFFokbSI2J9V9gVlM8ALWThDPnPu3EL7HPD2VDaZTggzcCCmbvc70qqPcC9mt60ogcrTiA3HEjwTK8ymKeuJMc4q6dVz200XnYUtLR9GYjPXvFOVr6W1zUK1WbPToaWJJuKnxBLnd0ftDEbMmj4loHYyhZyMjM91zQS4p7z8eKa9h0JrbacekcirexG0z4n3xz0QOWSvFj3jLhWXUIU21iIAwJtI3RbWa90I7rzAIqI3UElUJG7tLtUXzw4KQNETvXzqWaujEMenYlNIzLGxgB3AuJEQ143}     \end{equation} and, for any $q \in [ 1, np/(n-mp))$,     \begin{equation}      \zxvczxbcvdfghasdfrtsdafasdfasdfdsfgsdgh u - P_d - P_{d+1} \zxvczxbcvdfghasdfrtsdafasdfasdfdsfgsdgh_{L^q(B_r)}     \les ( \lbrack f \rbrack_{C^{d-m+1,\alpha}_{L^p}(0)} 	                 + \zxvczxbcvdfghasdfrtsdafasdfasdfdsfgsdgh u \zxvczxbcvdfghasdfrtsdafasdfasdfdsfgsdgh_{L^p(B_1)}                 + \zxvczxbcvdfghasdfrtsdafasdfasdfdsfgsdgh \tilde Q \zxvczxbcvdfghasdfrtsdafasdfasdfdsfgsdgh_{L^p(B_{1/2}(0))}              ) r^{d+1+\alpha+n/q}     ,     \label{8ThswELzXU3X7Ebd1KdZ7v1rN3GiirRXGKWK099ovBM0FDJCvkopYNQ2aN94Z7k0UnUKamE3OjU8DFYFFokbSI2J9V9gVlM8ALWThDPnPu3EL7HPD2VDaZTggzcCCmbvc70qqPcC9mt60ogcrTiA3HEjwTK8ymKeuJMc4q6dVz200XnYUtLR9GYjPXvFOVr6W1zUK1WbPToaWJJuKnxBLnd0ftDEbMmj4loHYyhZyMjM91zQS4p7z8eKa9h0JrbacekcirexG0z4n3xz0QOWSvFj3jLhWXUIU21iIAwJtI3RbWa90I7rzAIqI3UElUJG7tLtUXzw4KQNETvXzqWaujEMenYlNIzLGxgB3AuJEQ144}     \end{equation} for all~$r\leq1/4$. By definition, $\zxvczxbcvdfghasdfrtsdafasdfasdfdsfgsdgh \tilde Q \zxvczxbcvdfghasdfrtsdafasdfasdfdsfgsdgh_{L^p(B_r)} \les \lbrack a_{\nu} \rbrack_{C_{L^p}^{1-m+|\nu|}(0)} + \lbrack f \rbrack_{C^{d-m+1,\alpha}_{L^p}(0)} $, and then \eqref{8ThswELzXU3X7Ebd1KdZ7v1rN3GiirRXGKWK099ovBM0FDJCvkopYNQ2aN94Z7k0UnUKamE3OjU8DFYFFokbSI2J9V9gVlM8ALWThDPnPu3EL7HPD2VDaZTggzcCCmbvc70qqPcC9mt60ogcrTiA3HEjwTK8ymKeuJMc4q6dVz200XnYUtLR9GYjPXvFOVr6W1zUK1WbPToaWJJuKnxBLnd0ftDEbMmj4loHYyhZyMjM91zQS4p7z8eKa9h0JrbacekcirexG0z4n3xz0QOWSvFj3jLhWXUIU21iIAwJtI3RbWa90I7rzAIqI3UElUJG7tLtUXzw4KQNETvXzqWaujEMenYlNIzLGxgB3AuJEQ134} follows. \end{proof} \par \begin{Remark} {\rm Lemma~\ref{L05} demonstrates how we obtain the by one order higher H\"older norm for solution $u$ given the one order higher norm for $f$ and~$a_{\nu}$. To prove Theorem~\ref{T02}, we apply the argument in Lemma~\ref{L05} multiple times as long as $f$ and $a_{\nu}$ are controlled. } \end{Remark} \par \begin{proof}[Proof of Theorem~\ref{T02}] Theorem~\ref{T01} and Lemma~\ref{L05} are special cases of Theorem~\ref{T02} with $k=0$ and $k=1$ respectively. We complete the proof of Theorem~\ref{T02} by an induction argument on~$k$.  Repeating the argument in Lemma~\ref{L05} $k$ times, from \eqref{8ThswELzXU3X7Ebd1KdZ7v1rN3GiirRXGKWK099ovBM0FDJCvkopYNQ2aN94Z7k0UnUKamE3OjU8DFYFFokbSI2J9V9gVlM8ALWThDPnPu3EL7HPD2VDaZTggzcCCmbvc70qqPcC9mt60ogcrTiA3HEjwTK8ymKeuJMc4q6dVz200XnYUtLR9GYjPXvFOVr6W1zUK1WbPToaWJJuKnxBLnd0ftDEbMmj4loHYyhZyMjM91zQS4p7z8eKa9h0JrbacekcirexG0z4n3xz0QOWSvFj3jLhWXUIU21iIAwJtI3RbWa90I7rzAIqI3UElUJG7tLtUXzw4KQNETvXzqWaujEMenYlNIzLGxgB3AuJEQ143} we have   \begin{equation}   \sup_{r\leq1} \frac{ \zxvczxbcvdfghasdfrtsdafasdfasdfdsfgsdgh u \zxvczxbcvdfghasdfrtsdafasdfasdfdsfgsdgh_{L^p(B_r)}}  				{r^{d+k+\alpha + n/p}}   < \infty   .   \llabel{M icd yCV qn D036 HJWM tYfL9f yX x O7m IcF E1O uL QsAQ NfWv 6kV8Im 7Q 6 GsX NCV 0YP oC jnWn 6L25 qUMTe7 1v a hnH DAo XAb Tc zhPc fjrj W5M5G0 nz N M5T nlJ WOP Lh M6U2 ZFxw pg4Nej P8 U Q09 JX9 n7S kE WixE Rwgy Fvttzp 4A s v5F Tnn MzL Vh FUn5 6tFY CxZ1Bz Q3 E TfD lCa d7V fo MwPm ngrD HPfZV0 aY k Ojr ZUw 799 et oYuB MIC4 ovEY8D OL N URV Q5l ti1 iS NZAd wWr6 Q8oPFf ae 5 lAR 9gD RSi HO eJOW wxLv 20GoMt 2H z 7Yc aly PZx eR uFM0 7gaV 9UIz7S 43 k 5Tr ZiD Mt7 pE NCYi uHL7 gac7Gq yN 6 Z1u x56 YZh 2d yJVx 9MeU OMWBQf l0 E mIc 5Zr yfy 3i rahC y9Pi8ThswELzXU3X7Ebd1KdZ7v1rN3GiirRXGKWK099ovBM0FDJCvkopYNQ2aN94Z7k0UnUKamE3OjU8DFYFFokbSI2J9V9gVlM8ALWThDPnPu3EL7HPD2VDaZTggzcCCmbvc70qqPcC9mt60ogcrTiA3HEjwTK8ymKeuJMc4q6dVz200XnYUtLR9GYjPXvFOVr6W1zUK1WbPToaWJJuKnxBLnd0ftDEbMmj4loHYyhZyMjM91zQS4p7z8eKa9h0JrbacekcirexG0z4n3xz0QOWSvFj3jLhWXUIU21iIAwJtI3RbWa90I7rzAIqI3UElUJG7tLtUXzw4KQNETvXzqWaujEMenYlNIzLGxgB3AuJEQ145}   \end{equation} On the other hand, \eqref{8ThswELzXU3X7Ebd1KdZ7v1rN3GiirRXGKWK099ovBM0FDJCvkopYNQ2aN94Z7k0UnUKamE3OjU8DFYFFokbSI2J9V9gVlM8ALWThDPnPu3EL7HPD2VDaZTggzcCCmbvc70qqPcC9mt60ogcrTiA3HEjwTK8ymKeuJMc4q6dVz200XnYUtLR9GYjPXvFOVr6W1zUK1WbPToaWJJuKnxBLnd0ftDEbMmj4loHYyhZyMjM91zQS4p7z8eKa9h0JrbacekcirexG0z4n3xz0QOWSvFj3jLhWXUIU21iIAwJtI3RbWa90I7rzAIqI3UElUJG7tLtUXzw4KQNETvXzqWaujEMenYlNIzLGxgB3AuJEQ142} and \eqref{8ThswELzXU3X7Ebd1KdZ7v1rN3GiirRXGKWK099ovBM0FDJCvkopYNQ2aN94Z7k0UnUKamE3OjU8DFYFFokbSI2J9V9gVlM8ALWThDPnPu3EL7HPD2VDaZTggzcCCmbvc70qqPcC9mt60ogcrTiA3HEjwTK8ymKeuJMc4q6dVz200XnYUtLR9GYjPXvFOVr6W1zUK1WbPToaWJJuKnxBLnd0ftDEbMmj4loHYyhZyMjM91zQS4p7z8eKa9h0JrbacekcirexG0z4n3xz0QOWSvFj3jLhWXUIU21iIAwJtI3RbWa90I7rzAIqI3UElUJG7tLtUXzw4KQNETvXzqWaujEMenYlNIzLGxgB3AuJEQ144} implies there exists $P_{d+k} \in \PP_{d+k}$ and $\tilde Q_{d-m+k} \in \PP_{d-m+k}$ such that $\tilde Q_{d-m+k} = \sum_{|\nu| = m} a_{\nu} (0) \partial^{\nu} P_{d+k}$, and    \begin{equation} L(u-P_{d+k}) = (f - Q_{d+k-m}) + \sum_{|\nu|=m} (a_{\nu}(0) - a_{\nu}) \partial^{\nu} P_{d+k} - \sum_{|\nu|<m} a_{\nu} \partial^{\nu} P_{d+k} .    \llabel{ U Q09 JX9 n7S kE WixE Rwgy Fvttzp 4A s v5F Tnn MzL Vh FUn5 6tFY CxZ1Bz Q3 E TfD lCa d7V fo MwPm ngrD HPfZV0 aY k Ojr ZUw 799 et oYuB MIC4 ovEY8D OL N URV Q5l ti1 iS NZAd wWr6 Q8oPFf ae 5 lAR 9gD RSi HO eJOW wxLv 20GoMt 2H z 7Yc aly PZx eR uFM0 7gaV 9UIz7S 43 k 5Tr ZiD Mt7 pE NCYi uHL7 gac7Gq yN 6 Z1u x56 YZh 2d yJVx 9MeU OMWBQf l0 E mIc 5Zr yfy 3i rahC y9Pi MJ7ofo Op d enn sLi xZx Jt CjC9 M71v O0fxiR 51 m FIB QRo 1oW Iq 3gDP stD2 ntfoX7 YU o S5k GuV IGM cf HZe3 7ZoG A1dDmk XO 2 KYR LpJ jII om M6Nu u8O0 jO5Nab Ub R nZn 15k hG9 4S 21V48ThswELzXU3X7Ebd1KdZ7v1rN3GiirRXGKWK099ovBM0FDJCvkopYNQ2aN94Z7k0UnUKamE3OjU8DFYFFokbSI2J9V9gVlM8ALWThDPnPu3EL7HPD2VDaZTggzcCCmbvc70qqPcC9mt60ogcrTiA3HEjwTK8ymKeuJMc4q6dVz200XnYUtLR9GYjPXvFOVr6W1zUK1WbPToaWJJuKnxBLnd0ftDEbMmj4loHYyhZyMjM91zQS4p7z8eKa9h0JrbacekcirexG0z4n3xz0QOWSvFj3jLhWXUIU21iIAwJtI3RbWa90I7rzAIqI3UElUJG7tLtUXzw4KQNETvXzqWaujEMenYlNIzLGxgB3AuJEQ146}   \end{equation}  Therefore, we proceed similarly as Lemma~\ref{L05} and get the conclusion for $k+1$, which completes the induction argument. \end{proof} \par \startnewsection{Schauder estimates for parabolic equations}{sec04} The method used in the previous section may be extended to parabolic equations, as shown here. First, we introduce the corresponding terminology. For any $(x_0,t_0) \in \mathbb R^n \times \mathbb R$ and any $r>0$, we denote the parabolic ball centered at $(x_0, t_0)$ with radius $r$ by     \begin{equation}       Q_r(x_0,t_0)       = \{ (x,t) \in \mathbb R^n \times \mathbb R :       |x-x_0| < r,       -r^m < t - t_0 < 0       \},     \llabel{Ff ae 5 lAR 9gD RSi HO eJOW wxLv 20GoMt 2H z 7Yc aly PZx eR uFM0 7gaV 9UIz7S 43 k 5Tr ZiD Mt7 pE NCYi uHL7 gac7Gq yN 6 Z1u x56 YZh 2d yJVx 9MeU OMWBQf l0 E mIc 5Zr yfy 3i rahC y9Pi MJ7ofo Op d enn sLi xZx Jt CjC9 M71v O0fxiR 51 m FIB QRo 1oW Iq 3gDP stD2 ntfoX7 YU o S5k GuV IGM cf HZe3 7ZoG A1dDmk XO 2 KYR LpJ jII om M6Nu u8O0 jO5Nab Ub R nZn 15k hG9 4S 21V4 Ip45 7ooaiP u2 j hIz osW FDu O5 HdGr djvv tTLBjo vL L iCo 6L5 Lwa Pm vD6Z pal6 9Ljn11 re T 2CP mvj rL3 xH mDYK uv5T npC1fM oU R RTo Loi lk0 FE ghak m5M9 cOIPdQ lG D LnX erC ykJ C18ThswELzXU3X7Ebd1KdZ7v1rN3GiirRXGKWK099ovBM0FDJCvkopYNQ2aN94Z7k0UnUKamE3OjU8DFYFFokbSI2J9V9gVlM8ALWThDPnPu3EL7HPD2VDaZTggzcCCmbvc70qqPcC9mt60ogcrTiA3HEjwTK8ymKeuJMc4q6dVz200XnYUtLR9GYjPXvFOVr6W1zUK1WbPToaWJJuKnxBLnd0ftDEbMmj4loHYyhZyMjM91zQS4p7z8eKa9h0JrbacekcirexG0z4n3xz0QOWSvFj3jLhWXUIU21iIAwJtI3RbWa90I7rzAIqI3UElUJG7tLtUXzw4KQNETvXzqWaujEMenYlNIzLGxgB3AuJEQ152}     \end{equation} and write $Q_1$ for the parabolic ball centered at $(0,0)$ with radius~$1$. With $m\in 2\mathbb{N}$ fixed, consider a general order parabolic equation \begin{equation} \label{8ThswELzXU3X7Ebd1KdZ7v1rN3GiirRXGKWK099ovBM0FDJCvkopYNQ2aN94Z7k0UnUKamE3OjU8DFYFFokbSI2J9V9gVlM8ALWThDPnPu3EL7HPD2VDaZTggzcCCmbvc70qqPcC9mt60ogcrTiA3HEjwTK8ymKeuJMc4q6dVz200XnYUtLR9GYjPXvFOVr6W1zUK1WbPToaWJJuKnxBLnd0ftDEbMmj4loHYyhZyMjM91zQS4p7z8eKa9h0JrbacekcirexG0z4n3xz0QOWSvFj3jLhWXUIU21iIAwJtI3RbWa90I7rzAIqI3UElUJG7tLtUXzw4KQNETvXzqWaujEMenYlNIzLGxgB3AuJEQP01} Lu = f , \end{equation} where    \begin{equation}       L u        \equiv u_t                  - \sum_{|\nu|\leq m} a_{\nu} \partial^{\nu} u      \label{8ThswELzXU3X7Ebd1KdZ7v1rN3GiirRXGKWK099ovBM0FDJCvkopYNQ2aN94Z7k0UnUKamE3OjU8DFYFFokbSI2J9V9gVlM8ALWThDPnPu3EL7HPD2VDaZTggzcCCmbvc70qqPcC9mt60ogcrTiA3HEjwTK8ymKeuJMc4q6dVz200XnYUtLR9GYjPXvFOVr6W1zUK1WbPToaWJJuKnxBLnd0ftDEbMmj4loHYyhZyMjM91zQS4p7z8eKa9h0JrbacekcirexG0z4n3xz0QOWSvFj3jLhWXUIU21iIAwJtI3RbWa90I7rzAIqI3UElUJG7tLtUXzw4KQNETvXzqWaujEMenYlNIzLGxgB3AuJEQ147}     \end{equation} is an $m$-th order parabolic operator in $Q_1$, which satisfies the parabolicity condition     \begin{equation}       (-1)^{m/2-1} \sum_{|\nu|=m} a_{\nu} \xi^{\nu}        \geq       \frac{1}{K}       \comma       \xi \in \mathbb S^{n-1}       \comma       (x,t) \in Q_1       ,      \label{8ThswELzXU3X7Ebd1KdZ7v1rN3GiirRXGKWK099ovBM0FDJCvkopYNQ2aN94Z7k0UnUKamE3OjU8DFYFFokbSI2J9V9gVlM8ALWThDPnPu3EL7HPD2VDaZTggzcCCmbvc70qqPcC9mt60ogcrTiA3HEjwTK8ymKeuJMc4q6dVz200XnYUtLR9GYjPXvFOVr6W1zUK1WbPToaWJJuKnxBLnd0ftDEbMmj4loHYyhZyMjM91zQS4p7z8eKa9h0JrbacekcirexG0z4n3xz0QOWSvFj3jLhWXUIU21iIAwJtI3RbWa90I7rzAIqI3UElUJG7tLtUXzw4KQNETvXzqWaujEMenYlNIzLGxgB3AuJEQ148}     \end{equation} the boundedness      \begin{equation}       \sum_{|\nu|\leq m} |a_{\nu}(x,t)|        \leq K       \comma       (x,t) \in Q_1      ,     \llabel{ MJ7ofo Op d enn sLi xZx Jt CjC9 M71v O0fxiR 51 m FIB QRo 1oW Iq 3gDP stD2 ntfoX7 YU o S5k GuV IGM cf HZe3 7ZoG A1dDmk XO 2 KYR LpJ jII om M6Nu u8O0 jO5Nab Ub R nZn 15k hG9 4S 21V4 Ip45 7ooaiP u2 j hIz osW FDu O5 HdGr djvv tTLBjo vL L iCo 6L5 Lwa Pm vD6Z pal6 9Ljn11 re T 2CP mvj rL3 xH mDYK uv5T npC1fM oU R RTo Loi lk0 FE ghak m5M9 cOIPdQ lG D LnX erC ykJ C1 0FHh vvnY aTGuqU rf T QPv wEq iHO vO hD6A nXuv GlzVAv pz d Ok3 6ym yUo Fb AcAA BItO es52Vq d0 Y c7U 2gB t0W fF VQZh rJHr lBLdCx 8I o dWp AlD S8C HB rNLz xWp6 ypjuwW mg X toy 1vP b8ThswELzXU3X7Ebd1KdZ7v1rN3GiirRXGKWK099ovBM0FDJCvkopYNQ2aN94Z7k0UnUKamE3OjU8DFYFFokbSI2J9V9gVlM8ALWThDPnPu3EL7HPD2VDaZTggzcCCmbvc70qqPcC9mt60ogcrTiA3HEjwTK8ymKeuJMc4q6dVz200XnYUtLR9GYjPXvFOVr6W1zUK1WbPToaWJJuKnxBLnd0ftDEbMmj4loHYyhZyMjM91zQS4p7z8eKa9h0JrbacekcirexG0z4n3xz0QOWSvFj3jLhWXUIU21iIAwJtI3RbWa90I7rzAIqI3UElUJG7tLtUXzw4KQNETvXzqWaujEMenYlNIzLGxgB3AuJEQ149}     \end{equation} and the H\"older continuity of the leading coefficients     \begin{equation}       \sum_{|\nu|=m} |a_{\nu}(x,t) - a_{\nu}(0,0) |        \leq K|(x,t)|^{\alpha}       \comma        (x,t) \in Q_1       ,     \label{8ThswELzXU3X7Ebd1KdZ7v1rN3GiirRXGKWK099ovBM0FDJCvkopYNQ2aN94Z7k0UnUKamE3OjU8DFYFFokbSI2J9V9gVlM8ALWThDPnPu3EL7HPD2VDaZTggzcCCmbvc70qqPcC9mt60ogcrTiA3HEjwTK8ymKeuJMc4q6dVz200XnYUtLR9GYjPXvFOVr6W1zUK1WbPToaWJJuKnxBLnd0ftDEbMmj4loHYyhZyMjM91zQS4p7z8eKa9h0JrbacekcirexG0z4n3xz0QOWSvFj3jLhWXUIU21iIAwJtI3RbWa90I7rzAIqI3UElUJG7tLtUXzw4KQNETvXzqWaujEMenYlNIzLGxgB3AuJEQ150}     \end{equation} for some positive constants $K$, and $\alpha \in (0,1)$, where $|(x,t)|$ is the parabolic norm of $(x,t)$, i.e.,  $       |(x,t)|        = ( |x|^m + |t|)^{1/m}       . $ We allow all constants to depend on the space dimension $n$, the order of the differential operator $m$, as well as on $K$ and $\alpha$ without mention.   \par Denote by $W^{m,1}_p(Q_1)$ the Sobolev space of functions whose $x$-derivatives up to $m$-th order and the $t$-derivative of order $1$ belong to~$L^p(Q_1)$.  Similarly to the elliptic case,  let $\mathcal{Q}_d$ the set of polynomials in $n$ variables of degree at most $d$ and $\dQQ_d$ be the set of homogeneous polynomials of degree exactly $d$, with an addition of the zero polynomial. The homogeneity and the polynomial degree in this section are in parabolic sense, which are defined as follows. A function $f(x,t)$ is homogeneous of degree $d$ if for any $\lambda >0$ and $(x,t) \in \mathbb R^n \times \mathbb R \backslash \{ (0,0) \}$, we have     \begin{equation}     f ( \lambda x , \lambda^m t )      = \lambda^d f(x,t)     .     \llabel{ Ip45 7ooaiP u2 j hIz osW FDu O5 HdGr djvv tTLBjo vL L iCo 6L5 Lwa Pm vD6Z pal6 9Ljn11 re T 2CP mvj rL3 xH mDYK uv5T npC1fM oU R RTo Loi lk0 FE ghak m5M9 cOIPdQ lG D LnX erC ykJ C1 0FHh vvnY aTGuqU rf T QPv wEq iHO vO hD6A nXuv GlzVAv pz d Ok3 6ym yUo Fb AcAA BItO es52Vq d0 Y c7U 2gB t0W fF VQZh rJHr lBLdCx 8I o dWp AlD S8C HB rNLz xWp6 ypjuwW mg X toy 1vP bra uH yMNb kUrZ D6Ee2f zI D tkZ Eti Lmg re 1woD juLB BSdasY Vc F Uhy ViC xB1 5y Ltql qoUh gL3bZN YV k orz wa3 650 qW hF22 epiX cAjA4Z V4 b cXx uB3 NQN p0 GxW2 Vs1z jtqe2p LE B iS3 8ThswELzXU3X7Ebd1KdZ7v1rN3GiirRXGKWK099ovBM0FDJCvkopYNQ2aN94Z7k0UnUKamE3OjU8DFYFFokbSI2J9V9gVlM8ALWThDPnPu3EL7HPD2VDaZTggzcCCmbvc70qqPcC9mt60ogcrTiA3HEjwTK8ymKeuJMc4q6dVz200XnYUtLR9GYjPXvFOVr6W1zUK1WbPToaWJJuKnxBLnd0ftDEbMmj4loHYyhZyMjM91zQS4p7z8eKa9h0JrbacekcirexG0z4n3xz0QOWSvFj3jLhWXUIU21iIAwJtI3RbWa90I7rzAIqI3UElUJG7tLtUXzw4KQNETvXzqWaujEMenYlNIzLGxgB3AuJEQ151}     \end{equation} A polynomial $P(x,t)$ is of  degree at most $d$ if it can be decomposed into a sum of  homogeneous polynomials, whose  degree is at most~$d$. \par \begin{Definition}{\rm Let $u \in L^p(Q_1)$ for~$p \in [1, \infty]$. For an integer $d \geq 1$ and $\alpha \in (0,1)$,  we say that $u \in C^{d+\alpha, (d+\alpha)/m}_{L^p}(0)$ if there exists $P \in \mathcal{Q}_{d}$ satisfying   \begin{equation}       \sup_{r\leq1} \frac{\zxvczxbcvdfghasdfrtsdafasdfasdfdsfgsdgh u-P \zxvczxbcvdfghasdfrtsdafasdfasdfdsfgsdgh_{L^p(Q_r)}}                                          {r^{d+\alpha+(m+n)/p}}       < \infty       .     \llabel{ 0FHh vvnY aTGuqU rf T QPv wEq iHO vO hD6A nXuv GlzVAv pz d Ok3 6ym yUo Fb AcAA BItO es52Vq d0 Y c7U 2gB t0W fF VQZh rJHr lBLdCx 8I o dWp AlD S8C HB rNLz xWp6 ypjuwW mg X toy 1vP bra uH yMNb kUrZ D6Ee2f zI D tkZ Eti Lmg re 1woD juLB BSdasY Vc F Uhy ViC xB1 5y Ltql qoUh gL3bZN YV k orz wa3 650 qW hF22 epiX cAjA4Z V4 b cXx uB3 NQN p0 GxW2 Vs1z jtqe2p LE B iS3 0E0 NKH gY N50v XaK6 pNpwdB X2 Y v7V 0Ud dTc Pi dRNN CLG4 7Fc3PL Bx K 3Be x1X zyX cj 0Z6a Jk0H KuQnwd Dh P Q1Q rwA 05v 9c 3pnz ttzt x2IirW CZ B oS5 xlO KCi D3 WFh4 dvCL QANAQJ Gg y8ThswELzXU3X7Ebd1KdZ7v1rN3GiirRXGKWK099ovBM0FDJCvkopYNQ2aN94Z7k0UnUKamE3OjU8DFYFFokbSI2J9V9gVlM8ALWThDPnPu3EL7HPD2VDaZTggzcCCmbvc70qqPcC9mt60ogcrTiA3HEjwTK8ymKeuJMc4q6dVz200XnYUtLR9GYjPXvFOVr6W1zUK1WbPToaWJJuKnxBLnd0ftDEbMmj4loHYyhZyMjM91zQS4p7z8eKa9h0JrbacekcirexG0z4n3xz0QOWSvFj3jLhWXUIU21iIAwJtI3RbWa90I7rzAIqI3UElUJG7tLtUXzw4KQNETvXzqWaujEMenYlNIzLGxgB3AuJEQ160}   \end{equation} We also introduce the corresponding semi-norms as   \begin{equation}       \lbrack u \rbrack_{C^{i,i/m}_{L^p}(0)}        = \sum_{|\nu|+ml=i} | \partial^{\nu}_x \partial^l_t P(0) |       ,     \llabel{ra uH yMNb kUrZ D6Ee2f zI D tkZ Eti Lmg re 1woD juLB BSdasY Vc F Uhy ViC xB1 5y Ltql qoUh gL3bZN YV k orz wa3 650 qW hF22 epiX cAjA4Z V4 b cXx uB3 NQN p0 GxW2 Vs1z jtqe2p LE B iS3 0E0 NKH gY N50v XaK6 pNpwdB X2 Y v7V 0Ud dTc Pi dRNN CLG4 7Fc3PL Bx K 3Be x1X zyX cj 0Z6a Jk0H KuQnwd Dh P Q1Q rwA 05v 9c 3pnz ttzt x2IirW CZ B oS5 xlO KCi D3 WFh4 dvCL QANAQJ Gg y vOD NTD FKj Mc 0RJP m4HU SQkLnT Q4 Y 6CC MvN jAR Zb lir7 RFsI NzHiJl cg f xSC Hts ZOG 1V uOzk 5G1C LtmRYI eD 3 5BB uxZ JdY LO CwS9 lokS NasDLj 5h 8 yni u7h u3c di zYh1 PdwE l3m8Xt8ThswELzXU3X7Ebd1KdZ7v1rN3GiirRXGKWK099ovBM0FDJCvkopYNQ2aN94Z7k0UnUKamE3OjU8DFYFFokbSI2J9V9gVlM8ALWThDPnPu3EL7HPD2VDaZTggzcCCmbvc70qqPcC9mt60ogcrTiA3HEjwTK8ymKeuJMc4q6dVz200XnYUtLR9GYjPXvFOVr6W1zUK1WbPToaWJJuKnxBLnd0ftDEbMmj4loHYyhZyMjM91zQS4p7z8eKa9h0JrbacekcirexG0z4n3xz0QOWSvFj3jLhWXUIU21iIAwJtI3RbWa90I7rzAIqI3UElUJG7tLtUXzw4KQNETvXzqWaujEMenYlNIzLGxgB3AuJEQ161}   \end{equation} for $i=0,1,\cdots,d$, and     \begin{equation}       \lbrack u \rbrack_{C^{d+\alpha, (d+\alpha)/m}_{L^p}(0)}        = \sup_{0<r\leq1} \frac{\zxvczxbcvdfghasdfrtsdafasdfasdfdsfgsdgh u - P \zxvczxbcvdfghasdfrtsdafasdfasdfdsfgsdgh_{L^p(Q_r)}}                                          {r^{d+\alpha+(m+n)/p}}       .     \llabel{0E0 NKH gY N50v XaK6 pNpwdB X2 Y v7V 0Ud dTc Pi dRNN CLG4 7Fc3PL Bx K 3Be x1X zyX cj 0Z6a Jk0H KuQnwd Dh P Q1Q rwA 05v 9c 3pnz ttzt x2IirW CZ B oS5 xlO KCi D3 WFh4 dvCL QANAQJ Gg y vOD NTD FKj Mc 0RJP m4HU SQkLnT Q4 Y 6CC MvN jAR Zb lir7 RFsI NzHiJl cg f xSC Hts ZOG 1V uOzk 5G1C LtmRYI eD 3 5BB uxZ JdY LO CwS9 lokS NasDLj 5h 8 yni u7h u3c di zYh1 PdwE l3m8Xt yX Q RCA bwe aLi N8 qA9N 6DRE wy6gZe xs A 4fG EKH KQP PP KMbk sY1j M4h3Jj gS U One p1w RqN GA grL4 c18W v4kchD gR x 7Gj jIB zcK QV f7gA TrZx Oy6FF7 y9 3 iuu AQt 9TK Rx S5GO TFGx 48ThswELzXU3X7Ebd1KdZ7v1rN3GiirRXGKWK099ovBM0FDJCvkopYNQ2aN94Z7k0UnUKamE3OjU8DFYFFokbSI2J9V9gVlM8ALWThDPnPu3EL7HPD2VDaZTggzcCCmbvc70qqPcC9mt60ogcrTiA3HEjwTK8ymKeuJMc4q6dVz200XnYUtLR9GYjPXvFOVr6W1zUK1WbPToaWJJuKnxBLnd0ftDEbMmj4loHYyhZyMjM91zQS4p7z8eKa9h0JrbacekcirexG0z4n3xz0QOWSvFj3jLhWXUIU21iIAwJtI3RbWa90I7rzAIqI3UElUJG7tLtUXzw4KQNETvXzqWaujEMenYlNIzLGxgB3AuJEQ162}     \end{equation} } \end{Definition} \par The following is a parabolic analog of Theorem~\ref{T02}, asserting an $L^{p}$ Schauder estimate for the parabolic equation~\eqref{8ThswELzXU3X7Ebd1KdZ7v1rN3GiirRXGKWK099ovBM0FDJCvkopYNQ2aN94Z7k0UnUKamE3OjU8DFYFFokbSI2J9V9gVlM8ALWThDPnPu3EL7HPD2VDaZTggzcCCmbvc70qqPcC9mt60ogcrTiA3HEjwTK8ymKeuJMc4q6dVz200XnYUtLR9GYjPXvFOVr6W1zUK1WbPToaWJJuKnxBLnd0ftDEbMmj4loHYyhZyMjM91zQS4p7z8eKa9h0JrbacekcirexG0z4n3xz0QOWSvFj3jLhWXUIU21iIAwJtI3RbWa90I7rzAIqI3UElUJG7tLtUXzw4KQNETvXzqWaujEMenYlNIzLGxgB3AuJEQP01}. \par \cole \begin{Theorem} \label{T04} Let $u \in W^{m,1}_p(Q_1)$, where   \begin{equation}   \label{8ThswELzXU3X7Ebd1KdZ7v1rN3GiirRXGKWK099ovBM0FDJCvkopYNQ2aN94Z7k0UnUKamE3OjU8DFYFFokbSI2J9V9gVlM8ALWThDPnPu3EL7HPD2VDaZTggzcCCmbvc70qqPcC9mt60ogcrTiA3HEjwTK8ymKeuJMc4q6dVz200XnYUtLR9GYjPXvFOVr6W1zUK1WbPToaWJJuKnxBLnd0ftDEbMmj4loHYyhZyMjM91zQS4p7z8eKa9h0JrbacekcirexG0z4n3xz0QOWSvFj3jLhWXUIU21iIAwJtI3RbWa90I7rzAIqI3UElUJG7tLtUXzw4KQNETvXzqWaujEMenYlNIzLGxgB3AuJEQP03}   1 < p < 1 + \frac{n}{m}   ,   \end{equation} 
be a solution of $Lu=f$ in $Q_1$, where $f \in L^p(Q_1)$ and $L$ a parabolic operator \eqref{8ThswELzXU3X7Ebd1KdZ7v1rN3GiirRXGKWK099ovBM0FDJCvkopYNQ2aN94Z7k0UnUKamE3OjU8DFYFFokbSI2J9V9gVlM8ALWThDPnPu3EL7HPD2VDaZTggzcCCmbvc70qqPcC9mt60ogcrTiA3HEjwTK8ymKeuJMc4q6dVz200XnYUtLR9GYjPXvFOVr6W1zUK1WbPToaWJJuKnxBLnd0ftDEbMmj4loHYyhZyMjM91zQS4p7z8eKa9h0JrbacekcirexG0z4n3xz0QOWSvFj3jLhWXUIU21iIAwJtI3RbWa90I7rzAIqI3UElUJG7tLtUXzw4KQNETvXzqWaujEMenYlNIzLGxgB3AuJEQ147} satisfying~\eqref{8ThswELzXU3X7Ebd1KdZ7v1rN3GiirRXGKWK099ovBM0FDJCvkopYNQ2aN94Z7k0UnUKamE3OjU8DFYFFokbSI2J9V9gVlM8ALWThDPnPu3EL7HPD2VDaZTggzcCCmbvc70qqPcC9mt60ogcrTiA3HEjwTK8ymKeuJMc4q6dVz200XnYUtLR9GYjPXvFOVr6W1zUK1WbPToaWJJuKnxBLnd0ftDEbMmj4loHYyhZyMjM91zQS4p7z8eKa9h0JrbacekcirexG0z4n3xz0QOWSvFj3jLhWXUIU21iIAwJtI3RbWa90I7rzAIqI3UElUJG7tLtUXzw4KQNETvXzqWaujEMenYlNIzLGxgB3AuJEQ148}--\eqref{8ThswELzXU3X7Ebd1KdZ7v1rN3GiirRXGKWK099ovBM0FDJCvkopYNQ2aN94Z7k0UnUKamE3OjU8DFYFFokbSI2J9V9gVlM8ALWThDPnPu3EL7HPD2VDaZTggzcCCmbvc70qqPcC9mt60ogcrTiA3HEjwTK8ymKeuJMc4q6dVz200XnYUtLR9GYjPXvFOVr6W1zUK1WbPToaWJJuKnxBLnd0ftDEbMmj4loHYyhZyMjM91zQS4p7z8eKa9h0JrbacekcirexG0z4n3xz0QOWSvFj3jLhWXUIU21iIAwJtI3RbWa90I7rzAIqI3UElUJG7tLtUXzw4KQNETvXzqWaujEMenYlNIzLGxgB3AuJEQ150}. Suppose that      \begin{equation}     c     = \sup_{r \leq 1}\frac{\zxvczxbcvdfghasdfrtsdafasdfasdfdsfgsdgh u \zxvczxbcvdfghasdfrtsdafasdfasdfdsfgsdgh_{L^p(Q_r)}}                                      {r^{d+(m+n)/p}}     < \infty     \llabel{ vOD NTD FKj Mc 0RJP m4HU SQkLnT Q4 Y 6CC MvN jAR Zb lir7 RFsI NzHiJl cg f xSC Hts ZOG 1V uOzk 5G1C LtmRYI eD 3 5BB uxZ JdY LO CwS9 lokS NasDLj 5h 8 yni u7h u3c di zYh1 PdwE l3m8Xt yX Q RCA bwe aLi N8 qA9N 6DRE wy6gZe xs A 4fG EKH KQP PP KMbk sY1j M4h3Jj gS U One p1w RqN GA grL4 c18W v4kchD gR x 7Gj jIB zcK QV f7gA TrZx Oy6FF7 y9 3 iuu AQt 9TK Rx S5GO TFGx 4Xx1U3 R4 s 7U1 mpa bpD Hg kicx aCjk hnobr0 p4 c ody xTC kVj 8t W4iP 2OhT RF6kU2 k2 o oZJ Fsq Y4B FS NI3u W2fj OMFf7x Jv e ilb UVT ArC Tv qWLi vbRp g2wpAJ On l RUE PKh j9h dG M0Mi g8ThswELzXU3X7Ebd1KdZ7v1rN3GiirRXGKWK099ovBM0FDJCvkopYNQ2aN94Z7k0UnUKamE3OjU8DFYFFokbSI2J9V9gVlM8ALWThDPnPu3EL7HPD2VDaZTggzcCCmbvc70qqPcC9mt60ogcrTiA3HEjwTK8ymKeuJMc4q6dVz200XnYUtLR9GYjPXvFOVr6W1zUK1WbPToaWJJuKnxBLnd0ftDEbMmj4loHYyhZyMjM91zQS4p7z8eKa9h0JrbacekcirexG0z4n3xz0QOWSvFj3jLhWXUIU21iIAwJtI3RbWa90I7rzAIqI3UElUJG7tLtUXzw4KQNETvXzqWaujEMenYlNIzLGxgB3AuJEQ163}     \end{equation} and      \begin{equation}       c_f       = \sup_{r \leq 1}\frac{\zxvczxbcvdfghasdfrtsdafasdfasdfdsfgsdgh f \zxvczxbcvdfghasdfrtsdafasdfasdfdsfgsdgh_{L^p(Q_r)}}                                        {r^{d-m+(m+n)/p}}       < \infty,     \llabel{ yX Q RCA bwe aLi N8 qA9N 6DRE wy6gZe xs A 4fG EKH KQP PP KMbk sY1j M4h3Jj gS U One p1w RqN GA grL4 c18W v4kchD gR x 7Gj jIB zcK QV f7gA TrZx Oy6FF7 y9 3 iuu AQt 9TK Rx S5GO TFGx 4Xx1U3 R4 s 7U1 mpa bpD Hg kicx aCjk hnobr0 p4 c ody xTC kVj 8t W4iP 2OhT RF6kU2 k2 o oZJ Fsq Y4B FS NI3u W2fj OMFf7x Jv e ilb UVT ArC Tv qWLi vbRp g2wpAJ On l RUE PKh j9h dG M0Mi gcqQ wkyunB Jr T LDc Pgn OSC HO sSgQ sR35 MB7Bgk Pk 6 nJh 01P Cxd Ds w514 O648 VD8iJ5 4F W 6rs 6Sy qGz MK fXop oe4e o52UNB 4Q 8 f8N Uz8 u2n GO AXHW gKtG AtGGJs bm z 2qj vSv GBu 5e 48ThswELzXU3X7Ebd1KdZ7v1rN3GiirRXGKWK099ovBM0FDJCvkopYNQ2aN94Z7k0UnUKamE3OjU8DFYFFokbSI2J9V9gVlM8ALWThDPnPu3EL7HPD2VDaZTggzcCCmbvc70qqPcC9mt60ogcrTiA3HEjwTK8ymKeuJMc4q6dVz200XnYUtLR9GYjPXvFOVr6W1zUK1WbPToaWJJuKnxBLnd0ftDEbMmj4loHYyhZyMjM91zQS4p7z8eKa9h0JrbacekcirexG0z4n3xz0QOWSvFj3jLhWXUIU21iIAwJtI3RbWa90I7rzAIqI3UElUJG7tLtUXzw4KQNETvXzqWaujEMenYlNIzLGxgB3AuJEQ164}     \end{equation} for some integer~$d \geq m$. Assume additionally that $f \in C^{d-m+k, \alpha}_{L^p}(0)$ and $a_{\nu} \in C^{k-m+|\nu|+\alpha}_{L^p}(0)$ for some $\alpha \in (0,1)$ and $k \in \mathbb N$, for any $|\nu| \geq \max \{ m-k,0 \}$. Then for all $q \in [1, (m+n)p/(m+n-mp))$, we have $u \in C^{d+k+\alpha}_{L^{q}}(0)$, and      \begin{equation}       \sum_{i=d}^{d+k} \lbrack u \rbrack_{C^i_{L^{q}}(0)}       + \lbrack u \rbrack_{C^{d+k,\alpha}_{L^{q}}(0)}       \les \zxvczxbcvdfghasdfrtsdafasdfasdfdsfgsdgh u \zxvczxbcvdfghasdfrtsdafasdfasdfdsfgsdgh_{L^p(Q_1)}             + \sum_{i=d-m}^{d-m+k} \lbrack f \rbrack_{C^i_{L^p}(0)}             + \lbrack f \rbrack_{C^{d-m+k,\alpha}_{L^p}(0)}             +1        ,     \label{8ThswELzXU3X7Ebd1KdZ7v1rN3GiirRXGKWK099ovBM0FDJCvkopYNQ2aN94Z7k0UnUKamE3OjU8DFYFFokbSI2J9V9gVlM8ALWThDPnPu3EL7HPD2VDaZTggzcCCmbvc70qqPcC9mt60ogcrTiA3HEjwTK8ymKeuJMc4q6dVz200XnYUtLR9GYjPXvFOVr6W1zUK1WbPToaWJJuKnxBLnd0ftDEbMmj4loHYyhZyMjM91zQS4p7z8eKa9h0JrbacekcirexG0z4n3xz0QOWSvFj3jLhWXUIU21iIAwJtI3RbWa90I7rzAIqI3UElUJG7tLtUXzw4KQNETvXzqWaujEMenYlNIzLGxgB3AuJEQ165}     \end{equation} where the implicit constant in \eqref{8ThswELzXU3X7Ebd1KdZ7v1rN3GiirRXGKWK099ovBM0FDJCvkopYNQ2aN94Z7k0UnUKamE3OjU8DFYFFokbSI2J9V9gVlM8ALWThDPnPu3EL7HPD2VDaZTggzcCCmbvc70qqPcC9mt60ogcrTiA3HEjwTK8ymKeuJMc4q6dVz200XnYUtLR9GYjPXvFOVr6W1zUK1WbPToaWJJuKnxBLnd0ftDEbMmj4loHYyhZyMjM91zQS4p7z8eKa9h0JrbacekcirexG0z4n3xz0QOWSvFj3jLhWXUIU21iIAwJtI3RbWa90I7rzAIqI3UElUJG7tLtUXzw4KQNETvXzqWaujEMenYlNIzLGxgB3AuJEQ165} depends only on $p$, $d$, $k$, $c$, $c_f$, and $\lbrack a_{\nu} \rbrack_{C^i_{L^p}(0)}$ for $d-m \leq i \leq d-m+k $ and $\lbrack a_{\nu} \rbrack_{C_{L^p}^{l-m+|\nu|,\alpha}(0)}$ for $|\nu| \geq m-k$. \end{Theorem}  \colb \par Next, we estimate the high $L^p$-H\"older type norms around the points where $u$ vanishes of order $d$ in the $L^p$ sense.  As in the elliptic case, the result holds for the $L^q$-norm where $q\in [1, (m+n)p/(m+n-mp))$. \par \cole \begin{Theorem} \label{T03} Let $u \in W^{m,1}_p(Q_1)$, with $p$ as in \eqref{8ThswELzXU3X7Ebd1KdZ7v1rN3GiirRXGKWK099ovBM0FDJCvkopYNQ2aN94Z7k0UnUKamE3OjU8DFYFFokbSI2J9V9gVlM8ALWThDPnPu3EL7HPD2VDaZTggzcCCmbvc70qqPcC9mt60ogcrTiA3HEjwTK8ymKeuJMc4q6dVz200XnYUtLR9GYjPXvFOVr6W1zUK1WbPToaWJJuKnxBLnd0ftDEbMmj4loHYyhZyMjM91zQS4p7z8eKa9h0JrbacekcirexG0z4n3xz0QOWSvFj3jLhWXUIU21iIAwJtI3RbWa90I7rzAIqI3UElUJG7tLtUXzw4KQNETvXzqWaujEMenYlNIzLGxgB3AuJEQP03}, be a solution of $Lu = f$, where $L$ is the parabolic operator~\eqref{8ThswELzXU3X7Ebd1KdZ7v1rN3GiirRXGKWK099ovBM0FDJCvkopYNQ2aN94Z7k0UnUKamE3OjU8DFYFFokbSI2J9V9gVlM8ALWThDPnPu3EL7HPD2VDaZTggzcCCmbvc70qqPcC9mt60ogcrTiA3HEjwTK8ymKeuJMc4q6dVz200XnYUtLR9GYjPXvFOVr6W1zUK1WbPToaWJJuKnxBLnd0ftDEbMmj4loHYyhZyMjM91zQS4p7z8eKa9h0JrbacekcirexG0z4n3xz0QOWSvFj3jLhWXUIU21iIAwJtI3RbWa90I7rzAIqI3UElUJG7tLtUXzw4KQNETvXzqWaujEMenYlNIzLGxgB3AuJEQ147} satisfying \eqref{8ThswELzXU3X7Ebd1KdZ7v1rN3GiirRXGKWK099ovBM0FDJCvkopYNQ2aN94Z7k0UnUKamE3OjU8DFYFFokbSI2J9V9gVlM8ALWThDPnPu3EL7HPD2VDaZTggzcCCmbvc70qqPcC9mt60ogcrTiA3HEjwTK8ymKeuJMc4q6dVz200XnYUtLR9GYjPXvFOVr6W1zUK1WbPToaWJJuKnxBLnd0ftDEbMmj4loHYyhZyMjM91zQS4p7z8eKa9h0JrbacekcirexG0z4n3xz0QOWSvFj3jLhWXUIU21iIAwJtI3RbWa90I7rzAIqI3UElUJG7tLtUXzw4KQNETvXzqWaujEMenYlNIzLGxgB3AuJEQ148}--\eqref{8ThswELzXU3X7Ebd1KdZ7v1rN3GiirRXGKWK099ovBM0FDJCvkopYNQ2aN94Z7k0UnUKamE3OjU8DFYFFokbSI2J9V9gVlM8ALWThDPnPu3EL7HPD2VDaZTggzcCCmbvc70qqPcC9mt60ogcrTiA3HEjwTK8ymKeuJMc4q6dVz200XnYUtLR9GYjPXvFOVr6W1zUK1WbPToaWJJuKnxBLnd0ftDEbMmj4loHYyhZyMjM91zQS4p7z8eKa9h0JrbacekcirexG0z4n3xz0QOWSvFj3jLhWXUIU21iIAwJtI3RbWa90I7rzAIqI3UElUJG7tLtUXzw4KQNETvXzqWaujEMenYlNIzLGxgB3AuJEQ150} and~$f \in L^p(Q_1)$. Suppose that     \begin{equation}       \zxvczxbcvdfghasdfrtsdafasdfasdfdsfgsdgh f - Q \zxvczxbcvdfghasdfrtsdafasdfasdfdsfgsdgh_{L^p(Q_r)}       \leq \MM r^{d-m+\alpha+(m+n)/p}       \comma       r \in (0,1]       ,     \llabel{Xx1U3 R4 s 7U1 mpa bpD Hg kicx aCjk hnobr0 p4 c ody xTC kVj 8t W4iP 2OhT RF6kU2 k2 o oZJ Fsq Y4B FS NI3u W2fj OMFf7x Jv e ilb UVT ArC Tv qWLi vbRp g2wpAJ On l RUE PKh j9h dG M0Mi gcqQ wkyunB Jr T LDc Pgn OSC HO sSgQ sR35 MB7Bgk Pk 6 nJh 01P Cxd Ds w514 O648 VD8iJ5 4F W 6rs 6Sy qGz MK fXop oe4e o52UNB 4Q 8 f8N Uz8 u2n GO AXHW gKtG AtGGJs bm z 2qj vSv GBu 5e 4JgL Aqrm gMmS08 ZF s xQm 28M 3z4 Ho 1xxj j8Uk bMbm8M 0c L PL5 TS2 kIQ jZ Kb9Q Ux2U i5Aflw 1S L DGI uWU dCP jy wVVM 2ct8 cmgOBS 7d Q ViX R8F bta 1m tEFj TO0k owcK2d 6M Z iW8 PrK PI18ThswELzXU3X7Ebd1KdZ7v1rN3GiirRXGKWK099ovBM0FDJCvkopYNQ2aN94Z7k0UnUKamE3OjU8DFYFFokbSI2J9V9gVlM8ALWThDPnPu3EL7HPD2VDaZTggzcCCmbvc70qqPcC9mt60ogcrTiA3HEjwTK8ymKeuJMc4q6dVz200XnYUtLR9GYjPXvFOVr6W1zUK1WbPToaWJJuKnxBLnd0ftDEbMmj4loHYyhZyMjM91zQS4p7z8eKa9h0JrbacekcirexG0z4n3xz0QOWSvFj3jLhWXUIU21iIAwJtI3RbWa90I7rzAIqI3UElUJG7tLtUXzw4KQNETvXzqWaujEMenYlNIzLGxgB3AuJEQ154}     \end{equation} where $d\in \{m,m+1,m+2,\ldots\}$, for some $Q \in \dQQ_{d-m}$ and~$M > 0$. If     \begin{equation}       c_0       = \sup_{r\leq1} \frac{\zxvczxbcvdfghasdfrtsdafasdfasdfdsfgsdgh u \zxvczxbcvdfghasdfrtsdafasdfasdfdsfgsdgh_{L^p(Q_r)}}                                          {r^{d-1+\eta+(m+n)/p}}       < \infty      ,     \label{8ThswELzXU3X7Ebd1KdZ7v1rN3GiirRXGKWK099ovBM0FDJCvkopYNQ2aN94Z7k0UnUKamE3OjU8DFYFFokbSI2J9V9gVlM8ALWThDPnPu3EL7HPD2VDaZTggzcCCmbvc70qqPcC9mt60ogcrTiA3HEjwTK8ymKeuJMc4q6dVz200XnYUtLR9GYjPXvFOVr6W1zUK1WbPToaWJJuKnxBLnd0ftDEbMmj4loHYyhZyMjM91zQS4p7z8eKa9h0JrbacekcirexG0z4n3xz0QOWSvFj3jLhWXUIU21iIAwJtI3RbWa90I7rzAIqI3UElUJG7tLtUXzw4KQNETvXzqWaujEMenYlNIzLGxgB3AuJEQ155}     \end{equation} where $\eta \in (0,1]$,  there is $P \in \dQQ_d$ solving      \begin{equation}       P_t       - \sum_{|\nu|=m} a_{\nu}(0,0) \partial^{\nu} P       = Q     \llabel{cqQ wkyunB Jr T LDc Pgn OSC HO sSgQ sR35 MB7Bgk Pk 6 nJh 01P Cxd Ds w514 O648 VD8iJ5 4F W 6rs 6Sy qGz MK fXop oe4e o52UNB 4Q 8 f8N Uz8 u2n GO AXHW gKtG AtGGJs bm z 2qj vSv GBu 5e 4JgL Aqrm gMmS08 ZF s xQm 28M 3z4 Ho 1xxj j8Uk bMbm8M 0c L PL5 TS2 kIQ jZ Kb9Q Ux2U i5Aflw 1S L DGI uWU dCP jy wVVM 2ct8 cmgOBS 7d Q ViX R8F bta 1m tEFj TO0k owcK2d 6M Z iW8 PrK PI1 sX WJNB cREV Y4H5QQ GH b plP bwd Txp OI 5OQZ AKyi ix7Qey YI 9 1Ea 16r KXK L2 ifQX QPdP NL6EJi Hc K rBs 2qG tQb aq edOj Lixj GiNWr1 Pb Y SZe Sxx Fin aK 9Eki CHV2 a13f7G 3G 3 oDK K08ThswELzXU3X7Ebd1KdZ7v1rN3GiirRXGKWK099ovBM0FDJCvkopYNQ2aN94Z7k0UnUKamE3OjU8DFYFFokbSI2J9V9gVlM8ALWThDPnPu3EL7HPD2VDaZTggzcCCmbvc70qqPcC9mt60ogcrTiA3HEjwTK8ymKeuJMc4q6dVz200XnYUtLR9GYjPXvFOVr6W1zUK1WbPToaWJJuKnxBLnd0ftDEbMmj4loHYyhZyMjM91zQS4p7z8eKa9h0JrbacekcirexG0z4n3xz0QOWSvFj3jLhWXUIU21iIAwJtI3RbWa90I7rzAIqI3UElUJG7tLtUXzw4KQNETvXzqWaujEMenYlNIzLGxgB3AuJEQ156}     \end{equation} in $\mathbb R^n \times \mathbb R$ such that      \begin{equation}      \zxvczxbcvdfghasdfrtsdafasdfasdfdsfgsdgh P \zxvczxbcvdfghasdfrtsdafasdfasdfdsfgsdgh_{L^p(Q_r)}      \les ( M              +c_0              + \zxvczxbcvdfghasdfrtsdafasdfasdfdsfgsdgh Q \zxvczxbcvdfghasdfrtsdafasdfasdfdsfgsdgh_{L^p(Q_1)}             ) r^{d+(m+n)/p}     \llabel{JgL Aqrm gMmS08 ZF s xQm 28M 3z4 Ho 1xxj j8Uk bMbm8M 0c L PL5 TS2 kIQ jZ Kb9Q Ux2U i5Aflw 1S L DGI uWU dCP jy wVVM 2ct8 cmgOBS 7d Q ViX R8F bta 1m tEFj TO0k owcK2d 6M Z iW8 PrK PI1 sX WJNB cREV Y4H5QQ GH b plP bwd Txp OI 5OQZ AKyi ix7Qey YI 9 1Ea 16r KXK L2 ifQX QPdP NL6EJi Hc K rBs 2qG tQb aq edOj Lixj GiNWr1 Pb Y SZe Sxx Fin aK 9Eki CHV2 a13f7G 3G 3 oDK K0i bKV y4 53E2 nFQS 8Hnqg0 E3 2 ADd dEV nmJ 7H Bc1t 2K2i hCzZuy 9k p sHn 8Ko uAR kv sHKP y8Yo dOOqBi hF 1 Z3C vUF hmj gB muZq 7ggW Lg5dQB 1k p Fxk k35 GFo dk 00YD 13qI qqbLwy QC c y8ThswELzXU3X7Ebd1KdZ7v1rN3GiirRXGKWK099ovBM0FDJCvkopYNQ2aN94Z7k0UnUKamE3OjU8DFYFFokbSI2J9V9gVlM8ALWThDPnPu3EL7HPD2VDaZTggzcCCmbvc70qqPcC9mt60ogcrTiA3HEjwTK8ymKeuJMc4q6dVz200XnYUtLR9GYjPXvFOVr6W1zUK1WbPToaWJJuKnxBLnd0ftDEbMmj4loHYyhZyMjM91zQS4p7z8eKa9h0JrbacekcirexG0z4n3xz0QOWSvFj3jLhWXUIU21iIAwJtI3RbWa90I7rzAIqI3UElUJG7tLtUXzw4KQNETvXzqWaujEMenYlNIzLGxgB3AuJEQ157}     \end{equation} and      \begin{equation}       \sum_{|\nu|\leq m} r^{|\nu|} \zxvczxbcvdfghasdfrtsdafasdfasdfdsfgsdgh \partial^{\nu} (u-P) \zxvczxbcvdfghasdfrtsdafasdfasdfdsfgsdgh_{L^q(Q_r)}       \les (\MM                  + c_0               + \zxvczxbcvdfghasdfrtsdafasdfasdfdsfgsdgh Q \zxvczxbcvdfghasdfrtsdafasdfasdfdsfgsdgh_{L^p(Q_1)}               ) r^{d+\alpha+(m+n)/q}       \comma r \in (0, 1/2]     \llabel{ sX WJNB cREV Y4H5QQ GH b plP bwd Txp OI 5OQZ AKyi ix7Qey YI 9 1Ea 16r KXK L2 ifQX QPdP NL6EJi Hc K rBs 2qG tQb aq edOj Lixj GiNWr1 Pb Y SZe Sxx Fin aK 9Eki CHV2 a13f7G 3G 3 oDK K0i bKV y4 53E2 nFQS 8Hnqg0 E3 2 ADd dEV nmJ 7H Bc1t 2K2i hCzZuy 9k p sHn 8Ko uAR kv sHKP y8Yo dOOqBi hF 1 Z3C vUF hmj gB muZq 7ggW Lg5dQB 1k p Fxk k35 GFo dk 00YD 13qI qqbLwy QC c yZR wHA fp7 9o imtC c5CV 8cEuwU w7 k 8Q7 nCq WkM gY rtVR IySM tZUGCH XV 9 mr9 GHZ ol0 VE eIjQ vwgw 17pDhX JS F UcY bqU gnG V8 IFWb S1GX az0ZTt 81 w 7En IhF F72 v2 PkWO Xlkr w6IPu5 68ThswELzXU3X7Ebd1KdZ7v1rN3GiirRXGKWK099ovBM0FDJCvkopYNQ2aN94Z7k0UnUKamE3OjU8DFYFFokbSI2J9V9gVlM8ALWThDPnPu3EL7HPD2VDaZTggzcCCmbvc70qqPcC9mt60ogcrTiA3HEjwTK8ymKeuJMc4q6dVz200XnYUtLR9GYjPXvFOVr6W1zUK1WbPToaWJJuKnxBLnd0ftDEbMmj4loHYyhZyMjM91zQS4p7z8eKa9h0JrbacekcirexG0z4n3xz0QOWSvFj3jLhWXUIU21iIAwJtI3RbWa90I7rzAIqI3UElUJG7tLtUXzw4KQNETvXzqWaujEMenYlNIzLGxgB3AuJEQ158}     \end{equation} for $q \in [1, (m+n)p/(m+n - mp))$. Also,     \begin{equation}       \zxvczxbcvdfghasdfrtsdafasdfasdfdsfgsdgh u \zxvczxbcvdfghasdfrtsdafasdfasdfdsfgsdgh_{L^p(Q_r)}       \les ( c_0                + \MM                   + \zxvczxbcvdfghasdfrtsdafasdfasdfdsfgsdgh Q \zxvczxbcvdfghasdfrtsdafasdfasdfdsfgsdgh_{L^p (Q_1)}              ) r^{d+(m+n)/p}              ,     \llabel{i bKV y4 53E2 nFQS 8Hnqg0 E3 2 ADd dEV nmJ 7H Bc1t 2K2i hCzZuy 9k p sHn 8Ko uAR kv sHKP y8Yo dOOqBi hF 1 Z3C vUF hmj gB muZq 7ggW Lg5dQB 1k p Fxk k35 GFo dk 00YD 13qI qqbLwy QC c yZR wHA fp7 9o imtC c5CV 8cEuwU w7 k 8Q7 nCq WkM gY rtVR IySM tZUGCH XV 9 mr9 GHZ ol0 VE eIjQ vwgw 17pDhX JS F UcY bqU gnG V8 IFWb S1GX az0ZTt 81 w 7En IhF F72 v2 PkWO Xlkr w6IPu5 67 9 vcW 1f6 z99 lM 2LI1 Y6Na axfl18 gT 0 gDp tVl CN4 jf GSbC ro5D v78Cxa uk Y iUI WWy YDR w8 z7Kj Px7C hC7zJv b1 b 0rF d7n Mxk 09 1wHv y4u5 vLLsJ8 Nm A kWt xuf 4P5 Nw P23b 06sF NQ68ThswELzXU3X7Ebd1KdZ7v1rN3GiirRXGKWK099ovBM0FDJCvkopYNQ2aN94Z7k0UnUKamE3OjU8DFYFFokbSI2J9V9gVlM8ALWThDPnPu3EL7HPD2VDaZTggzcCCmbvc70qqPcC9mt60ogcrTiA3HEjwTK8ymKeuJMc4q6dVz200XnYUtLR9GYjPXvFOVr6W1zUK1WbPToaWJJuKnxBLnd0ftDEbMmj4loHYyhZyMjM91zQS4p7z8eKa9h0JrbacekcirexG0z4n3xz0QOWSvFj3jLhWXUIU21iIAwJtI3RbWa90I7rzAIqI3UElUJG7tLtUXzw4KQNETvXzqWaujEMenYlNIzLGxgB3AuJEQ159}     \end{equation}     where all the implicit constants depend only on $p$, $q$, and~$d$.  \end{Theorem} \colb  \par Throughout this section, we use the notation \begin{equation}    \llabel{ZR wHA fp7 9o imtC c5CV 8cEuwU w7 k 8Q7 nCq WkM gY rtVR IySM tZUGCH XV 9 mr9 GHZ ol0 VE eIjQ vwgw 17pDhX JS F UcY bqU gnG V8 IFWb S1GX az0ZTt 81 w 7En IhF F72 v2 PkWO Xlkr w6IPu5 67 9 vcW 1f6 z99 lM 2LI1 Y6Na axfl18 gT 0 gDp tVl CN4 jf GSbC ro5D v78Cxa uk Y iUI WWy YDR w8 z7Kj Px7C hC7zJv b1 b 0rF d7n Mxk 09 1wHv y4u5 vLLsJ8 Nm A kWt xuf 4P5 Nw P23b 06sF NQ6xgD hu R GbK 7j2 O4g y4 p4BL top3 h2kfyI 9w O 4Aa EWb 36Y yH YiI1 S3CO J7aN1r 0s Q OrC AC4 vL7 yr CGkI RlNu GbOuuk 1a w LDK 2zl Ka4 0h yJnD V4iF xsqO00 1r q CeO AO2 es7 DR aCpU G548ThswELzXU3X7Ebd1KdZ7v1rN3GiirRXGKWK099ovBM0FDJCvkopYNQ2aN94Z7k0UnUKamE3OjU8DFYFFokbSI2J9V9gVlM8ALWThDPnPu3EL7HPD2VDaZTggzcCCmbvc70qqPcC9mt60ogcrTiA3HEjwTK8ymKeuJMc4q6dVz200XnYUtLR9GYjPXvFOVr6W1zUK1WbPToaWJJuKnxBLnd0ftDEbMmj4loHYyhZyMjM91zQS4p7z8eKa9h0JrbacekcirexG0z4n3xz0QOWSvFj3jLhWXUIU21iIAwJtI3RbWa90I7rzAIqI3UElUJG7tLtUXzw4KQNETvXzqWaujEMenYlNIzLGxgB3AuJEQ177} L(0)  = \partial_t  - \sum_{|\nu|=m} a_{\nu}(0,0) \partial^{\nu} , \end{equation} and $\Gamma$ is the fundamental solution of~$L(0)$, which satisfies \begin{equation} \label{8ThswELzXU3X7Ebd1KdZ7v1rN3GiirRXGKWK099ovBM0FDJCvkopYNQ2aN94Z7k0UnUKamE3OjU8DFYFFokbSI2J9V9gVlM8ALWThDPnPu3EL7HPD2VDaZTggzcCCmbvc70qqPcC9mt60ogcrTiA3HEjwTK8ymKeuJMc4q6dVz200XnYUtLR9GYjPXvFOVr6W1zUK1WbPToaWJJuKnxBLnd0ftDEbMmj4loHYyhZyMjM91zQS4p7z8eKa9h0JrbacekcirexG0z4n3xz0QOWSvFj3jLhWXUIU21iIAwJtI3RbWa90I7rzAIqI3UElUJG7tLtUXzw4KQNETvXzqWaujEMenYlNIzLGxgB3AuJEQP055}       | \partial^{\beta}_x \partial^l_t \Gamma (x,t) |        \les \frac{1}{|(x,t)|^{n+|\beta|+ml}}    \comma \beta\in \mathbb{N}_0^{n}    \commaone l\in\mathbb{N}_0       . \end{equation} \par \subsection{Interior $W^{m,1}_q$ existence estimate} We now state an analog of Lemma~\ref{L03}. The proof is similar to the elliptic case except that the fundamental solution $\Gamma$ now satisfies \eqref{8ThswELzXU3X7Ebd1KdZ7v1rN3GiirRXGKWK099ovBM0FDJCvkopYNQ2aN94Z7k0UnUKamE3OjU8DFYFFokbSI2J9V9gVlM8ALWThDPnPu3EL7HPD2VDaZTggzcCCmbvc70qqPcC9mt60ogcrTiA3HEjwTK8ymKeuJMc4q6dVz200XnYUtLR9GYjPXvFOVr6W1zUK1WbPToaWJJuKnxBLnd0ftDEbMmj4loHYyhZyMjM91zQS4p7z8eKa9h0JrbacekcirexG0z4n3xz0QOWSvFj3jLhWXUIU21iIAwJtI3RbWa90I7rzAIqI3UElUJG7tLtUXzw4KQNETvXzqWaujEMenYlNIzLGxgB3AuJEQP055}. \par \cole  \begin{Lemma} \label{L07} Assume that $L = \partial_t -  \sum_{|\nu|\leq m} a_{\nu}(x) \partial^{\nu}$, defined in $Q_1$, satisfies the conditions~\eqref{8ThswELzXU3X7Ebd1KdZ7v1rN3GiirRXGKWK099ovBM0FDJCvkopYNQ2aN94Z7k0UnUKamE3OjU8DFYFFokbSI2J9V9gVlM8ALWThDPnPu3EL7HPD2VDaZTggzcCCmbvc70qqPcC9mt60ogcrTiA3HEjwTK8ymKeuJMc4q6dVz200XnYUtLR9GYjPXvFOVr6W1zUK1WbPToaWJJuKnxBLnd0ftDEbMmj4loHYyhZyMjM91zQS4p7z8eKa9h0JrbacekcirexG0z4n3xz0QOWSvFj3jLhWXUIU21iIAwJtI3RbWa90I7rzAIqI3UElUJG7tLtUXzw4KQNETvXzqWaujEMenYlNIzLGxgB3AuJEQ148}$-$\eqref{8ThswELzXU3X7Ebd1KdZ7v1rN3GiirRXGKWK099ovBM0FDJCvkopYNQ2aN94Z7k0UnUKamE3OjU8DFYFFokbSI2J9V9gVlM8ALWThDPnPu3EL7HPD2VDaZTggzcCCmbvc70qqPcC9mt60ogcrTiA3HEjwTK8ymKeuJMc4q6dVz200XnYUtLR9GYjPXvFOVr6W1zUK1WbPToaWJJuKnxBLnd0ftDEbMmj4loHYyhZyMjM91zQS4p7z8eKa9h0JrbacekcirexG0z4n3xz0QOWSvFj3jLhWXUIU21iIAwJtI3RbWa90I7rzAIqI3UElUJG7tLtUXzw4KQNETvXzqWaujEMenYlNIzLGxgB3AuJEQ150}. Suppose that $f \in L^p(Q_1)$ is such that there exist $\MM>0$ and $\gamma \in (0,1)$ satisfying     \begin{equation}       \zxvczxbcvdfghasdfrtsdafasdfasdfdsfgsdgh f \zxvczxbcvdfghasdfrtsdafasdfasdfdsfgsdgh_{L^p(Q_r)}       \leq \MM r^{d-m+\gamma + (m+n)/p}       \comma r \in (0, 1]       .     \llabel{7 9 vcW 1f6 z99 lM 2LI1 Y6Na axfl18 gT 0 gDp tVl CN4 jf GSbC ro5D v78Cxa uk Y iUI WWy YDR w8 z7Kj Px7C hC7zJv b1 b 0rF d7n Mxk 09 1wHv y4u5 vLLsJ8 Nm A kWt xuf 4P5 Nw P23b 06sF NQ6xgD hu R GbK 7j2 O4g y4 p4BL top3 h2kfyI 9w O 4Aa EWb 36Y yH YiI1 S3CO J7aN1r 0s Q OrC AC4 vL7 yr CGkI RlNu GbOuuk 1a w LDK 2zl Ka4 0h yJnD V4iF xsqO00 1r q CeO AO2 es7 DR aCpU G54F 2i97xS Qr c bPZ 6K8 Kud n9 e6SY o396 Fr8LUx yX O jdF sMr l54 Eh T8vr xxF2 phKPbs zr l pMA ubE RMG QA aCBu 2Lqw Gasprf IZ O iKV Vbu Vae 6a bauf y9Kc Fk6cBl Z5 r KUj htW E1C nt 9Rm8ThswELzXU3X7Ebd1KdZ7v1rN3GiirRXGKWK099ovBM0FDJCvkopYNQ2aN94Z7k0UnUKamE3OjU8DFYFFokbSI2J9V9gVlM8ALWThDPnPu3EL7HPD2VDaZTggzcCCmbvc70qqPcC9mt60ogcrTiA3HEjwTK8ymKeuJMc4q6dVz200XnYUtLR9GYjPXvFOVr6W1zUK1WbPToaWJJuKnxBLnd0ftDEbMmj4loHYyhZyMjM91zQS4p7z8eKa9h0JrbacekcirexG0z4n3xz0QOWSvFj3jLhWXUIU21iIAwJtI3RbWa90I7rzAIqI3UElUJG7tLtUXzw4KQNETvXzqWaujEMenYlNIzLGxgB3AuJEQ174}     \end{equation}  Then there exist $R>0$, depending on $\alpha$, and $u \in W^{m,1}_q(Q_1)$ solving     $       Lu = f     $ in $Q_R$ such that     \begin{equation}       \sum_{|\nu|\leq m} r^{|\nu|} \zxvczxbcvdfghasdfrtsdafasdfasdfdsfgsdgh \partial^{\nu} u \zxvczxbcvdfghasdfrtsdafasdfasdfdsfgsdgh_{L^q(Q_r)}       \les \MM r^{d+\gamma +(m+n)/q}       \comma r \leq R       ,     \label{8ThswELzXU3X7Ebd1KdZ7v1rN3GiirRXGKWK099ovBM0FDJCvkopYNQ2aN94Z7k0UnUKamE3OjU8DFYFFokbSI2J9V9gVlM8ALWThDPnPu3EL7HPD2VDaZTggzcCCmbvc70qqPcC9mt60ogcrTiA3HEjwTK8ymKeuJMc4q6dVz200XnYUtLR9GYjPXvFOVr6W1zUK1WbPToaWJJuKnxBLnd0ftDEbMmj4loHYyhZyMjM91zQS4p7z8eKa9h0JrbacekcirexG0z4n3xz0QOWSvFj3jLhWXUIU21iIAwJtI3RbWa90I7rzAIqI3UElUJG7tLtUXzw4KQNETvXzqWaujEMenYlNIzLGxgB3AuJEQ175}     \end{equation} for $q \in [1, (m+n)p/(m+n-mp))$. \end{Lemma} \colb  \par To prove Lemma~\ref{L07}, we first need an estimate for the integral  $\zxvczxbcvdfghasdfrtsdafasdfasdfdsfgsdgf (\fractext{|f(y,s)|}{|(y,s)|^b}) \, dy \, ds$ over~$Q_r$ and $Q_1 \backslash Q_r$.  \par \cole \begin{Lemma} \label{L06} Assume that $f\in L^p(Q_1)$, where $p\in [1,\infty]$, satisfies   \begin{equation}    \zxvczxbcvdfghasdfrtsdafasdfasdfdsfgsdgh f\zxvczxbcvdfghasdfrtsdafasdfasdfdsfgsdgh_{L^{p}(Q_r)}    \leq M r^{a}    \comma r\in(0,1]    ,    \llabel{xgD hu R GbK 7j2 O4g y4 p4BL top3 h2kfyI 9w O 4Aa EWb 36Y yH YiI1 S3CO J7aN1r 0s Q OrC AC4 vL7 yr CGkI RlNu GbOuuk 1a w LDK 2zl Ka4 0h yJnD V4iF xsqO00 1r q CeO AO2 es7 DR aCpU G54F 2i97xS Qr c bPZ 6K8 Kud n9 e6SY o396 Fr8LUx yX O jdF sMr l54 Eh T8vr xxF2 phKPbs zr l pMA ubE RMG QA aCBu 2Lqw Gasprf IZ O iKV Vbu Vae 6a bauf y9Kc Fk6cBl Z5 r KUj htW E1C nt 9Rmd whJR ySGVSO VT v 9FY 4uz yAH Sp 6yT9 s6R6 oOi3aq Zl L 7bI vWZ 18c Fa iwpt C1nd Fyp4oK xD f Qz2 813 6a8 zX wsGl Ysh9 Gp3Tal nr R UKt tBK eFr 45 43qU 2hh3 WbYw09 g2 W LIX zvQ zMk j8ThswELzXU3X7Ebd1KdZ7v1rN3GiirRXGKWK099ovBM0FDJCvkopYNQ2aN94Z7k0UnUKamE3OjU8DFYFFokbSI2J9V9gVlM8ALWThDPnPu3EL7HPD2VDaZTggzcCCmbvc70qqPcC9mt60ogcrTiA3HEjwTK8ymKeuJMc4q6dVz200XnYUtLR9GYjPXvFOVr6W1zUK1WbPToaWJJuKnxBLnd0ftDEbMmj4loHYyhZyMjM91zQS4p7z8eKa9h0JrbacekcirexG0z4n3xz0QOWSvFj3jLhWXUIU21iIAwJtI3RbWa90I7rzAIqI3UElUJG7tLtUXzw4KQNETvXzqWaujEMenYlNIzLGxgB3AuJEQ166}   \end{equation} where~$a \in \mathbb{R}$.\\ (i) If for some $b\in\mathbb{R}$, we have   \begin{equation}
   b < a + \frac{m+n}{p'}    ,    \label{8ThswELzXU3X7Ebd1KdZ7v1rN3GiirRXGKWK099ovBM0FDJCvkopYNQ2aN94Z7k0UnUKamE3OjU8DFYFFokbSI2J9V9gVlM8ALWThDPnPu3EL7HPD2VDaZTggzcCCmbvc70qqPcC9mt60ogcrTiA3HEjwTK8ymKeuJMc4q6dVz200XnYUtLR9GYjPXvFOVr6W1zUK1WbPToaWJJuKnxBLnd0ftDEbMmj4loHYyhZyMjM91zQS4p7z8eKa9h0JrbacekcirexG0z4n3xz0QOWSvFj3jLhWXUIU21iIAwJtI3RbWa90I7rzAIqI3UElUJG7tLtUXzw4KQNETvXzqWaujEMenYlNIzLGxgB3AuJEQ167}   \end{equation} then   \begin{equation}    \zxvczxbcvdfghasdfrtsdafasdfasdfdsfgsdgf_{Q_r}    \frac{|f(y,s)|\,dy \, ds}{|(y,s)|^{b}}    \les    M r^{a-b+(m+n)/p'}    \comma r\in (0,1]    ,    \llabel{F 2i97xS Qr c bPZ 6K8 Kud n9 e6SY o396 Fr8LUx yX O jdF sMr l54 Eh T8vr xxF2 phKPbs zr l pMA ubE RMG QA aCBu 2Lqw Gasprf IZ O iKV Vbu Vae 6a bauf y9Kc Fk6cBl Z5 r KUj htW E1C nt 9Rmd whJR ySGVSO VT v 9FY 4uz yAH Sp 6yT9 s6R6 oOi3aq Zl L 7bI vWZ 18c Fa iwpt C1nd Fyp4oK xD f Qz2 813 6a8 zX wsGl Ysh9 Gp3Tal nr R UKt tBK eFr 45 43qU 2hh3 WbYw09 g2 W LIX zvQ zMk j5 f0xL seH9 dscinG wu P JLP 1gE N5W qY sSoW Peqj MimTyb Hj j cbn 0NO 5hz P9 W40r 2w77 TAoz70 N1 a u09 boc DSx Gc 3tvK LXaC 1dKgw9 H3 o 2kE oul In9 TS PyL2 HXO7 tSZse0 1Z 9 Hds lDq 8ThswELzXU3X7Ebd1KdZ7v1rN3GiirRXGKWK099ovBM0FDJCvkopYNQ2aN94Z7k0UnUKamE3OjU8DFYFFokbSI2J9V9gVlM8ALWThDPnPu3EL7HPD2VDaZTggzcCCmbvc70qqPcC9mt60ogcrTiA3HEjwTK8ymKeuJMc4q6dVz200XnYUtLR9GYjPXvFOVr6W1zUK1WbPToaWJJuKnxBLnd0ftDEbMmj4loHYyhZyMjM91zQS4p7z8eKa9h0JrbacekcirexG0z4n3xz0QOWSvFj3jLhWXUIU21iIAwJtI3RbWa90I7rzAIqI3UElUJG7tLtUXzw4KQNETvXzqWaujEMenYlNIzLGxgB3AuJEQ168}   \end{equation} where the implicit constant depends on $a$, $b$, and~$p$.\\ (ii) If   \begin{equation}       b       > a + \frac{m+n}{p'}    ,    \llabel{d whJR ySGVSO VT v 9FY 4uz yAH Sp 6yT9 s6R6 oOi3aq Zl L 7bI vWZ 18c Fa iwpt C1nd Fyp4oK xD f Qz2 813 6a8 zX wsGl Ysh9 Gp3Tal nr R UKt tBK eFr 45 43qU 2hh3 WbYw09 g2 W LIX zvQ zMk j5 f0xL seH9 dscinG wu P JLP 1gE N5W qY sSoW Peqj MimTyb Hj j cbn 0NO 5hz P9 W40r 2w77 TAoz70 N1 a u09 boc DSx Gc 3tvK LXaC 1dKgw9 H3 o 2kE oul In9 TS PyL2 HXO7 tSZse0 1Z 9 Hds lDq 0tm SO AVqt A1FQ zEMKSb ak z nw8 39w nH1 Dp CjGI k5X3 B6S6UI 7H I gAa f9E V33 Bk kuo3 FyEi 8Ty2AB PY z SWj Pj5 tYZ ET Yzg6 Ix5t ATPMdl Gk e 67X b7F ktE sz yFyc mVhG JZ29aP gz k Yj48ThswELzXU3X7Ebd1KdZ7v1rN3GiirRXGKWK099ovBM0FDJCvkopYNQ2aN94Z7k0UnUKamE3OjU8DFYFFokbSI2J9V9gVlM8ALWThDPnPu3EL7HPD2VDaZTggzcCCmbvc70qqPcC9mt60ogcrTiA3HEjwTK8ymKeuJMc4q6dVz200XnYUtLR9GYjPXvFOVr6W1zUK1WbPToaWJJuKnxBLnd0ftDEbMmj4loHYyhZyMjM91zQS4p7z8eKa9h0JrbacekcirexG0z4n3xz0QOWSvFj3jLhWXUIU21iIAwJtI3RbWa90I7rzAIqI3UElUJG7tLtUXzw4KQNETvXzqWaujEMenYlNIzLGxgB3AuJEQ169}   \end{equation} we have   \begin{equation}    \zxvczxbcvdfghasdfrtsdafasdfasdfdsfgsdgf_{Q_1\backslash Q_r}    \frac{|f(y,s)|\,dy \, ds}{|(y,s)|^{b}}    \les    M r^{a-b+(m+n)/p'}    \comma r\in (0,1]    ,    \llabel{5 f0xL seH9 dscinG wu P JLP 1gE N5W qY sSoW Peqj MimTyb Hj j cbn 0NO 5hz P9 W40r 2w77 TAoz70 N1 a u09 boc DSx Gc 3tvK LXaC 1dKgw9 H3 o 2kE oul In9 TS PyL2 HXO7 tSZse0 1Z 9 Hds lDq 0tm SO AVqt A1FQ zEMKSb ak z nw8 39w nH1 Dp CjGI k5X3 B6S6UI 7H I gAa f9E V33 Bk kuo3 FyEi 8Ty2AB PY z SWj Pj5 tYZ ET Yzg6 Ix5t ATPMdl Gk e 67X b7F ktE sz yFyc mVhG JZ29aP gz k Yj4 cEr HCd P7 XFHU O9zo y4AZai SR O pIn 0tp 7kZ zU VHQt m3ip 3xEd41 By 7 2ux IiY 8BC Lb OYGo LDwp juza6i Pa k Zdh aD3 xSX yj pdOw oqQq Jl6RFg lO t X67 nm7 s1l ZJ mGUr dIdX Q7jps7 rc 8ThswELzXU3X7Ebd1KdZ7v1rN3GiirRXGKWK099ovBM0FDJCvkopYNQ2aN94Z7k0UnUKamE3OjU8DFYFFokbSI2J9V9gVlM8ALWThDPnPu3EL7HPD2VDaZTggzcCCmbvc70qqPcC9mt60ogcrTiA3HEjwTK8ymKeuJMc4q6dVz200XnYUtLR9GYjPXvFOVr6W1zUK1WbPToaWJJuKnxBLnd0ftDEbMmj4loHYyhZyMjM91zQS4p7z8eKa9h0JrbacekcirexG0z4n3xz0QOWSvFj3jLhWXUIU21iIAwJtI3RbWa90I7rzAIqI3UElUJG7tLtUXzw4KQNETvXzqWaujEMenYlNIzLGxgB3AuJEQ170}   \end{equation} where the implicit constant depends on $a$, $b$, and~$p$. \par \end{Lemma} \colb \begin{proof}[Proof of Lemma~\ref{L06}] We prove the statement when $p \in (1, \infty)$; it is easy to check that the same holds if $p = 1$ or~$p = \infty$. \par (i) We write   \begin{align}    \begin{split}    \zxvczxbcvdfghasdfrtsdafasdfasdfdsfgsdgf_{Q_r}    \frac{|f(y,s)|\,dy\,ds}{|(y,s)|^{b}}    &    =    \sum_{m=0}^{\infty}    \zxvczxbcvdfghasdfrtsdafasdfasdfdsfgsdgf_{Q_{2^{-m} r}\backslash Q_{2^{-m-1}r}}    \frac{|f(y,s)|\,dy\, ds}{|(y,s)|^{b}}    \\&    \les    \sum_{m=0}^{\infty}    \zxvczxbcvdfghasdfrtsdafasdfasdfdsfgsdgh f\zxvczxbcvdfghasdfrtsdafasdfasdfdsfgsdgh_{L^p(Q_{2^{-m}r})}    \zxvczxbcvdfghasdfrtsdafasdfasdfdsfgsdgh |(y,s)|^{-b}\zxvczxbcvdfghasdfrtsdafasdfasdfdsfgsdgh_{L^{p'}(Q_{2^{-m} r}\backslash Q_{2^{-m-1}r})}    \\&    \les    \sum_{m=0}^{\infty}    M (2^{-m}r)^{a-b+(m+n)/p'}    \les    M r^{a-b+(m+n)/p'}    ,    \end{split}    \label{8ThswELzXU3X7Ebd1KdZ7v1rN3GiirRXGKWK099ovBM0FDJCvkopYNQ2aN94Z7k0UnUKamE3OjU8DFYFFokbSI2J9V9gVlM8ALWThDPnPu3EL7HPD2VDaZTggzcCCmbvc70qqPcC9mt60ogcrTiA3HEjwTK8ymKeuJMc4q6dVz200XnYUtLR9GYjPXvFOVr6W1zUK1WbPToaWJJuKnxBLnd0ftDEbMmj4loHYyhZyMjM91zQS4p7z8eKa9h0JrbacekcirexG0z4n3xz0QOWSvFj3jLhWXUIU21iIAwJtI3RbWa90I7rzAIqI3UElUJG7tLtUXzw4KQNETvXzqWaujEMenYlNIzLGxgB3AuJEQ171}   \end{align} under the condition~\eqref{8ThswELzXU3X7Ebd1KdZ7v1rN3GiirRXGKWK099ovBM0FDJCvkopYNQ2aN94Z7k0UnUKamE3OjU8DFYFFokbSI2J9V9gVlM8ALWThDPnPu3EL7HPD2VDaZTggzcCCmbvc70qqPcC9mt60ogcrTiA3HEjwTK8ymKeuJMc4q6dVz200XnYUtLR9GYjPXvFOVr6W1zUK1WbPToaWJJuKnxBLnd0ftDEbMmj4loHYyhZyMjM91zQS4p7z8eKa9h0JrbacekcirexG0z4n3xz0QOWSvFj3jLhWXUIU21iIAwJtI3RbWa90I7rzAIqI3UElUJG7tLtUXzw4KQNETvXzqWaujEMenYlNIzLGxgB3AuJEQ167}. \par (ii) With $m\in \mathbb{N}_0$ such that, we have   \begin{align}    \begin{split}    \zxvczxbcvdfghasdfrtsdafasdfasdfdsfgsdgf_{Q_1\backslash Q_r}    \frac{|f(y,s)|\,dy \, ds}{|(y,s)|^{b}}    &    \leq    \sum_{m=0}^{m_0}    \zxvczxbcvdfghasdfrtsdafasdfasdfdsfgsdgf_{Q_{2^{-m}}\backslash Q_{2^{-m-1}}}    \frac{|f(y,s)|\,dy \, ds}{|(y,s)|^{b}}    \les    \sum_{m=0}^{m_0}    \zxvczxbcvdfghasdfrtsdafasdfasdfdsfgsdgh f\zxvczxbcvdfghasdfrtsdafasdfasdfdsfgsdgh_{L^p(Q_{2^{-m}})}    \zxvczxbcvdfghasdfrtsdafasdfasdfdsfgsdgh |(y,s)|^{-b}\zxvczxbcvdfghasdfrtsdafasdfasdfdsfgsdgh_{L^{p'}(Q_{2^{-m} }\backslash Q_{2^{-m-1}})}    \\&    \les    \sum_{m=0}^{m_0}    (2^{-m})^{a-b+(m+n)/p'}    \les    (2^{-m_0})^{a-b+(m+n)/p'}    \les    r^{a-b+(m+n)/p'}    ,    \end{split}    \llabel{0tm SO AVqt A1FQ zEMKSb ak z nw8 39w nH1 Dp CjGI k5X3 B6S6UI 7H I gAa f9E V33 Bk kuo3 FyEi 8Ty2AB PY z SWj Pj5 tYZ ET Yzg6 Ix5t ATPMdl Gk e 67X b7F ktE sz yFyc mVhG JZ29aP gz k Yj4 cEr HCd P7 XFHU O9zo y4AZai SR O pIn 0tp 7kZ zU VHQt m3ip 3xEd41 By 7 2ux IiY 8BC Lb OYGo LDwp juza6i Pa k Zdh aD3 xSX yj pdOw oqQq Jl6RFg lO t X67 nm7 s1l ZJ mGUr dIdX Q7jps7 rc d ACY ZMs BKA Nx tkqf Nhkt sbBf2O BN Z 5pf oqS Xtd 3c HFLN tLgR oHrnNl wR n ylZ NWV NfH vO B1nU Ayjt xTWW4o Cq P Rtu Vua nMk Lv qbxp Ni0x YnOkcd FB d rw1 Nu7 cKy bL jCF7 P4dx j0Sbz8ThswELzXU3X7Ebd1KdZ7v1rN3GiirRXGKWK099ovBM0FDJCvkopYNQ2aN94Z7k0UnUKamE3OjU8DFYFFokbSI2J9V9gVlM8ALWThDPnPu3EL7HPD2VDaZTggzcCCmbvc70qqPcC9mt60ogcrTiA3HEjwTK8ymKeuJMc4q6dVz200XnYUtLR9GYjPXvFOVr6W1zUK1WbPToaWJJuKnxBLnd0ftDEbMmj4loHYyhZyMjM91zQS4p7z8eKa9h0JrbacekcirexG0z4n3xz0QOWSvFj3jLhWXUIU21iIAwJtI3RbWa90I7rzAIqI3UElUJG7tLtUXzw4KQNETvXzqWaujEMenYlNIzLGxgB3AuJEQ172}   \end{align} and the proof is concluded. \end{proof} \par \begin{proof}[Proof of Lemma~\ref{L07}] We only argue for the case where $a_{\nu} \equiv a_{\nu}(0,0)$ for $|\nu|=m$ and $a_{\nu} \equiv 0$ for~$|\nu|<m$. The other cases follow by the contraction mapping argument and rescaling analogously to Lemma~\ref{L01}. Without loss of generality, we may assume $f(x,t) = 0$ for~$|(x,t)|>1$.  For the integral     \begin{equation}       w(x,t)        = \zxvczxbcvdfghasdfrtsdafasdfasdfdsfgsdgf_{|(y,s)| < 1} \Gamma(x-y,t-s) f(y,s) \, dy \, ds     \llabel{ cEr HCd P7 XFHU O9zo y4AZai SR O pIn 0tp 7kZ zU VHQt m3ip 3xEd41 By 7 2ux IiY 8BC Lb OYGo LDwp juza6i Pa k Zdh aD3 xSX yj pdOw oqQq Jl6RFg lO t X67 nm7 s1l ZJ mGUr dIdX Q7jps7 rc d ACY ZMs BKA Nx tkqf Nhkt sbBf2O BN Z 5pf oqS Xtd 3c HFLN tLgR oHrnNl wR n ylZ NWV NfH vO B1nU Ayjt xTWW4o Cq P Rtu Vua nMk Lv qbxp Ni0x YnOkcd FB d rw1 Nu7 cKy bL jCF7 P4dx j0Sbz9 fa V CWk VFo s9t 2a QIPK ORuE jEMtbS Hs Y eG5 Z7u MWW Aw RnR8 FwFC zXVVxn FU f yKL Nk4 eOI ly n3Cl I5HP 8XP6S4 KF f Il6 2Vl bXg ca uth8 61pU WUx2aQ TW g rZw cAx 52T kq oZXV g0QG 8ThswELzXU3X7Ebd1KdZ7v1rN3GiirRXGKWK099ovBM0FDJCvkopYNQ2aN94Z7k0UnUKamE3OjU8DFYFFokbSI2J9V9gVlM8ALWThDPnPu3EL7HPD2VDaZTggzcCCmbvc70qqPcC9mt60ogcrTiA3HEjwTK8ymKeuJMc4q6dVz200XnYUtLR9GYjPXvFOVr6W1zUK1WbPToaWJJuKnxBLnd0ftDEbMmj4loHYyhZyMjM91zQS4p7z8eKa9h0JrbacekcirexG0z4n3xz0QOWSvFj3jLhWXUIU21iIAwJtI3RbWa90I7rzAIqI3UElUJG7tLtUXzw4KQNETvXzqWaujEMenYlNIzLGxgB3AuJEQ176}     \end{equation} we have     $       L(0)w = f      $ in $Q_1$ and     \begin{equation}       \zxvczxbcvdfghasdfrtsdafasdfasdfdsfgsdgh w \zxvczxbcvdfghasdfrtsdafasdfasdfdsfgsdgh_{W^{m,1}_p(Q_1)}       \les \zxvczxbcvdfghasdfrtsdafasdfasdfdsfgsdgh f \zxvczxbcvdfghasdfrtsdafasdfasdfdsfgsdgh_{L^p(Q_1)}       \les \MM       .     \llabel{d ACY ZMs BKA Nx tkqf Nhkt sbBf2O BN Z 5pf oqS Xtd 3c HFLN tLgR oHrnNl wR n ylZ NWV NfH vO B1nU Ayjt xTWW4o Cq P Rtu Vua nMk Lv qbxp Ni0x YnOkcd FB d rw1 Nu7 cKy bL jCF7 P4dx j0Sbz9 fa V CWk VFo s9t 2a QIPK ORuE jEMtbS Hs Y eG5 Z7u MWW Aw RnR8 FwFC zXVVxn FU f yKL Nk4 eOI ly n3Cl I5HP 8XP6S4 KF f Il6 2Vl bXg ca uth8 61pU WUx2aQ TW g rZw cAx 52T kq oZXV g0QG rBrrpe iw u WyJ td9 ooD 8t UzAd LSnI tarmhP AW B mnm nsb xLI qX 4RQS TyoF DIikpe IL h WZZ 8ic JGa 91 HxRb 97kn Whp9sA Vz P o85 60p RN2 PS MGMM FK5X W52OnW Iy o Yng xWn o86 8S Kbbu 8ThswELzXU3X7Ebd1KdZ7v1rN3GiirRXGKWK099ovBM0FDJCvkopYNQ2aN94Z7k0UnUKamE3OjU8DFYFFokbSI2J9V9gVlM8ALWThDPnPu3EL7HPD2VDaZTggzcCCmbvc70qqPcC9mt60ogcrTiA3HEjwTK8ymKeuJMc4q6dVz200XnYUtLR9GYjPXvFOVr6W1zUK1WbPToaWJJuKnxBLnd0ftDEbMmj4loHYyhZyMjM91zQS4p7z8eKa9h0JrbacekcirexG0z4n3xz0QOWSvFj3jLhWXUIU21iIAwJtI3RbWa90I7rzAIqI3UElUJG7tLtUXzw4KQNETvXzqWaujEMenYlNIzLGxgB3AuJEQ178}     \end{equation} Since the polynomial \begin{equation} P_w(x,t)  = \sum_{0 \leq |\beta| + ml \leq d}  \frac{x^{\beta} t^l}{\beta! l!} \zxvczxbcvdfghasdfrtsdafasdfasdfdsfgsdgf_{|(y,s)|<1}  \partial^{\beta}_x \partial^l_t \Gamma (-y,-s) f(y,s) \, dy \, ds    \llabel{9 fa V CWk VFo s9t 2a QIPK ORuE jEMtbS Hs Y eG5 Z7u MWW Aw RnR8 FwFC zXVVxn FU f yKL Nk4 eOI ly n3Cl I5HP 8XP6S4 KF f Il6 2Vl bXg ca uth8 61pU WUx2aQ TW g rZw cAx 52T kq oZXV g0QG rBrrpe iw u WyJ td9 ooD 8t UzAd LSnI tarmhP AW B mnm nsb xLI qX 4RQS TyoF DIikpe IL h WZZ 8ic JGa 91 HxRb 97kn Whp9sA Vz P o85 60p RN2 PS MGMM FK5X W52OnW Iy o Yng xWn o86 8S Kbbu 1Iq1 SyPkHJ VC v seV GWr hUd ew Xw6C SY1b e3hD9P Kh a 1y0 SRw yxi AG zdCM VMmi JaemmP 8x r bJX bKL DYE 1F pXUK ADtF 9ewhNe fd 2 XRu tTl 1HY JV p5cA hM1J fK7UIc pk d TbE ndM 6FW HA 8ThswELzXU3X7Ebd1KdZ7v1rN3GiirRXGKWK099ovBM0FDJCvkopYNQ2aN94Z7k0UnUKamE3OjU8DFYFFokbSI2J9V9gVlM8ALWThDPnPu3EL7HPD2VDaZTggzcCCmbvc70qqPcC9mt60ogcrTiA3HEjwTK8ymKeuJMc4q6dVz200XnYUtLR9GYjPXvFOVr6W1zUK1WbPToaWJJuKnxBLnd0ftDEbMmj4loHYyhZyMjM91zQS4p7z8eKa9h0JrbacekcirexG0z4n3xz0QOWSvFj3jLhWXUIU21iIAwJtI3RbWa90I7rzAIqI3UElUJG7tLtUXzw4KQNETvXzqWaujEMenYlNIzLGxgB3AuJEQ108} \end{equation} satisfies  $L(0)P_w = 0$
the difference $u = w - P_w$  solves~$L(0)u = f$. We now show that $u$ satisfies~\eqref{8ThswELzXU3X7Ebd1KdZ7v1rN3GiirRXGKWK099ovBM0FDJCvkopYNQ2aN94Z7k0UnUKamE3OjU8DFYFFokbSI2J9V9gVlM8ALWThDPnPu3EL7HPD2VDaZTggzcCCmbvc70qqPcC9mt60ogcrTiA3HEjwTK8ymKeuJMc4q6dVz200XnYUtLR9GYjPXvFOVr6W1zUK1WbPToaWJJuKnxBLnd0ftDEbMmj4loHYyhZyMjM91zQS4p7z8eKa9h0JrbacekcirexG0z4n3xz0QOWSvFj3jLhWXUIU21iIAwJtI3RbWa90I7rzAIqI3UElUJG7tLtUXzw4KQNETvXzqWaujEMenYlNIzLGxgB3AuJEQ175}. For any positive $r \leq 1/2$ and $q \in [1, (m+n)p/(m+n-mp))$, we have \begin{align} \begin{split} \zxvczxbcvdfghasdfrtsdafasdfasdfdsfgsdgh u \zxvczxbcvdfghasdfrtsdafasdfasdfdsfgsdgh_{L^q(Q_r)} &\leq \biggl \zxvczxbcvdfghasdfrtsdafasdfasdfdsfgsdgh \zxvczxbcvdfghasdfrtsdafasdfasdfdsfgsdgf_{|(y,s)| \leq 2r} \Gamma(x-y,t-s) f(y,s) \,dy \,ds \biggr \zxvczxbcvdfghasdfrtsdafasdfasdfdsfgsdgh_{L^q(Q_r)} \\& \indeq + \biggl \zxvczxbcvdfghasdfrtsdafasdfasdfdsfgsdgh \zxvczxbcvdfghasdfrtsdafasdfasdfdsfgsdgf_{|(y,s)| \leq 2r}         \sum_{0 \leq |\beta| + ml \leq d}  \partial^{\beta}_x \partial^l_t \Gamma (-y,-s) \frac{x^{\beta} t^l}{\beta! l!}  f(y,s) \, dy \, ds       \biggr \zxvczxbcvdfghasdfrtsdafasdfasdfdsfgsdgh_{L^q(Q_r)}       \\& \indeq  + \biggl \zxvczxbcvdfghasdfrtsdafasdfasdfdsfgsdgh \zxvczxbcvdfghasdfrtsdafasdfasdfdsfgsdgf_{2r\leq|(y,s)| < 1}           \biggl( \Gamma(x-y,t-s)             -\sum_{0 \leq |\beta| + ml \leq d}           \partial^{\beta}_x \partial^l_t \Gamma (-y,-s) \frac{x^{\beta} t^l}{\beta! l!}           \biggr) f(y,s) \,dy \,ds          \biggr\zxvczxbcvdfghasdfrtsdafasdfasdfdsfgsdgh_{L^q(Q_r)} \\& = I_1 + I_2 + I_3. \end{split}    \llabel{rBrrpe iw u WyJ td9 ooD 8t UzAd LSnI tarmhP AW B mnm nsb xLI qX 4RQS TyoF DIikpe IL h WZZ 8ic JGa 91 HxRb 97kn Whp9sA Vz P o85 60p RN2 PS MGMM FK5X W52OnW Iy o Yng xWn o86 8S Kbbu 1Iq1 SyPkHJ VC v seV GWr hUd ew Xw6C SY1b e3hD9P Kh a 1y0 SRw yxi AG zdCM VMmi JaemmP 8x r bJX bKL DYE 1F pXUK ADtF 9ewhNe fd 2 XRu tTl 1HY JV p5cA hM1J fK7UIc pk d TbE ndM 6FW HA 72Pg LHzX lUo39o W9 0 BuD eJS lnV Rv z8VD V48t Id4Dtg FO O a47 LEH 8Qw nR GNBM 0RRU LluASz jx x wGI BHm Vyy Ld kGww 5eEg HFvsFU nz l 0vg OaQ DCV Ez 64r8 UvVH TtDykr Eu F aS3 5p5 yn8ThswELzXU3X7Ebd1KdZ7v1rN3GiirRXGKWK099ovBM0FDJCvkopYNQ2aN94Z7k0UnUKamE3OjU8DFYFFokbSI2J9V9gVlM8ALWThDPnPu3EL7HPD2VDaZTggzcCCmbvc70qqPcC9mt60ogcrTiA3HEjwTK8ymKeuJMc4q6dVz200XnYUtLR9GYjPXvFOVr6W1zUK1WbPToaWJJuKnxBLnd0ftDEbMmj4loHYyhZyMjM91zQS4p7z8eKa9h0JrbacekcirexG0z4n3xz0QOWSvFj3jLhWXUIU21iIAwJtI3RbWa90I7rzAIqI3UElUJG7tLtUXzw4KQNETvXzqWaujEMenYlNIzLGxgB3AuJEQ153} \end{align} Using Young's inequality and \eqref{8ThswELzXU3X7Ebd1KdZ7v1rN3GiirRXGKWK099ovBM0FDJCvkopYNQ2aN94Z7k0UnUKamE3OjU8DFYFFokbSI2J9V9gVlM8ALWThDPnPu3EL7HPD2VDaZTggzcCCmbvc70qqPcC9mt60ogcrTiA3HEjwTK8ymKeuJMc4q6dVz200XnYUtLR9GYjPXvFOVr6W1zUK1WbPToaWJJuKnxBLnd0ftDEbMmj4loHYyhZyMjM91zQS4p7z8eKa9h0JrbacekcirexG0z4n3xz0QOWSvFj3jLhWXUIU21iIAwJtI3RbWa90I7rzAIqI3UElUJG7tLtUXzw4KQNETvXzqWaujEMenYlNIzLGxgB3AuJEQP055}, we have     \begin{align}     \begin{split}       I_1        &\leq         \biggl\zxvczxbcvdfghasdfrtsdafasdfasdfdsfgsdgh \zxvczxbcvdfghasdfrtsdafasdfasdfdsfgsdgf  \Gamma(x-y,t-s)  f (y,s)\chi_{Q_{2r}}(y,s) \,dy \,ds \biggr\zxvczxbcvdfghasdfrtsdafasdfasdfdsfgsdgh_{L^q(Q_r)}       \leq \zxvczxbcvdfghasdfrtsdafasdfasdfdsfgsdgh \Gamma \zxvczxbcvdfghasdfrtsdafasdfasdfdsfgsdgh_{L^1(Q_{3r})}               \zxvczxbcvdfghasdfrtsdafasdfasdfdsfgsdgh f \zxvczxbcvdfghasdfrtsdafasdfasdfdsfgsdgh_{L^q(Q_{2r})}        \\&       \les \MM r^{d-m+\gamma+(m+n)/q}              \zxvczxbcvdfghasdfrtsdafasdfasdfdsfgsdgf_{|(x,t)|\leq3r} \frac{1}{|(x,t)|^{n}} \,dx \,dt        \les \MM              r^{d+\gamma+(m+n)/q}              .     \end{split}     \llabel{1Iq1 SyPkHJ VC v seV GWr hUd ew Xw6C SY1b e3hD9P Kh a 1y0 SRw yxi AG zdCM VMmi JaemmP 8x r bJX bKL DYE 1F pXUK ADtF 9ewhNe fd 2 XRu tTl 1HY JV p5cA hM1J fK7UIc pk d TbE ndM 6FW HA 72Pg LHzX lUo39o W9 0 BuD eJS lnV Rv z8VD V48t Id4Dtg FO O a47 LEH 8Qw nR GNBM 0RRU LluASz jx x wGI BHm Vyy Ld kGww 5eEg HFvsFU nz l 0vg OaQ DCV Ez 64r8 UvVH TtDykr Eu F aS3 5p5 yn6 QZ UcX3 mfET Exz1kv qE p OVV EFP IVp zQ lMOI Z2yT TxIUOm 0f W L1W oxC tlX Ws 9HU4 EF0I Z1WDv3 TP 4 2LN 7Tr SuR 8u Mv1t Lepv ZoeoKL xf 9 zMJ 6PU In1 S8 I4KY 13wJ TACh5X l8 O 5g0 Z8ThswELzXU3X7Ebd1KdZ7v1rN3GiirRXGKWK099ovBM0FDJCvkopYNQ2aN94Z7k0UnUKamE3OjU8DFYFFokbSI2J9V9gVlM8ALWThDPnPu3EL7HPD2VDaZTggzcCCmbvc70qqPcC9mt60ogcrTiA3HEjwTK8ymKeuJMc4q6dVz200XnYUtLR9GYjPXvFOVr6W1zUK1WbPToaWJJuKnxBLnd0ftDEbMmj4loHYyhZyMjM91zQS4p7z8eKa9h0JrbacekcirexG0z4n3xz0QOWSvFj3jLhWXUIU21iIAwJtI3RbWa90I7rzAIqI3UElUJG7tLtUXzw4KQNETvXzqWaujEMenYlNIzLGxgB3AuJEQ187}     \end{align} To estimate $I_2$,  we first claim that    \begin{equation}    \biggl| \zxvczxbcvdfghasdfrtsdafasdfasdfdsfgsdgf_{|(y,s)|\leq2r}  \partial_x^{\beta} \partial_t^l \Gamma(-y, -s) \frac{x^{\beta} t^l}{\beta!l!}f(y,s) \,dy \,ds    \biggr|  \les \MM r^{d+\gamma} ,    \llabel{72Pg LHzX lUo39o W9 0 BuD eJS lnV Rv z8VD V48t Id4Dtg FO O a47 LEH 8Qw nR GNBM 0RRU LluASz jx x wGI BHm Vyy Ld kGww 5eEg HFvsFU nz l 0vg OaQ DCV Ez 64r8 UvVH TtDykr Eu F aS3 5p5 yn6 QZ UcX3 mfET Exz1kv qE p OVV EFP IVp zQ lMOI Z2yT TxIUOm 0f W L1W oxC tlX Ws 9HU4 EF0I Z1WDv3 TP 4 2LN 7Tr SuR 8u Mv1t Lepv ZoeoKL xf 9 zMJ 6PU In1 S8 I4KY 13wJ TACh5X l8 O 5g0 ZGw Ddt u6 8wvr vnDC oqYjJ3 nF K WMA K8V OeG o4 DKxn EOyB wgmttc ES 8 dmT oAD 0YB Fl yGRB pBbo 8tQYBw bS X 2lc YnU 0fh At myR3 CKcU AQzzET Ng b ghH T64 KdO fL qFWu k07t DkzfQ1 dg B 8ThswELzXU3X7Ebd1KdZ7v1rN3GiirRXGKWK099ovBM0FDJCvkopYNQ2aN94Z7k0UnUKamE3OjU8DFYFFokbSI2J9V9gVlM8ALWThDPnPu3EL7HPD2VDaZTggzcCCmbvc70qqPcC9mt60ogcrTiA3HEjwTK8ymKeuJMc4q6dVz200XnYUtLR9GYjPXvFOVr6W1zUK1WbPToaWJJuKnxBLnd0ftDEbMmj4loHYyhZyMjM91zQS4p7z8eKa9h0JrbacekcirexG0z4n3xz0QOWSvFj3jLhWXUIU21iIAwJtI3RbWa90I7rzAIqI3UElUJG7tLtUXzw4KQNETvXzqWaujEMenYlNIzLGxgB3AuJEQ173} \end{equation}  for all~$|(x,t)| \leq r$.  Indeed, by \eqref{8ThswELzXU3X7Ebd1KdZ7v1rN3GiirRXGKWK099ovBM0FDJCvkopYNQ2aN94Z7k0UnUKamE3OjU8DFYFFokbSI2J9V9gVlM8ALWThDPnPu3EL7HPD2VDaZTggzcCCmbvc70qqPcC9mt60ogcrTiA3HEjwTK8ymKeuJMc4q6dVz200XnYUtLR9GYjPXvFOVr6W1zUK1WbPToaWJJuKnxBLnd0ftDEbMmj4loHYyhZyMjM91zQS4p7z8eKa9h0JrbacekcirexG0z4n3xz0QOWSvFj3jLhWXUIU21iIAwJtI3RbWa90I7rzAIqI3UElUJG7tLtUXzw4KQNETvXzqWaujEMenYlNIzLGxgB3AuJEQP055}, we have     \begin{align}     \begin{split}       &\biggl| \zxvczxbcvdfghasdfrtsdafasdfasdfdsfgsdgf_{|(y,s)|\leq2r}  \partial_x^{\beta} \partial_t^l \Gamma(-y, -s) \frac{x^{\beta} t^l}{\beta!l!}f(y,s) \,dy \,ds  \biggr|        \les |(x,t)|^{|\beta|+ml}              \zxvczxbcvdfghasdfrtsdafasdfasdfdsfgsdgf_{|(y,s)|\leq2r} \frac{|f(y,s)|}{|(y,s)|^{n+|\beta|+ml}} \,dy \,ds       \\& \indeq \indeq \indeq \indeq        \les r^{|\beta| + ml}               \zxvczxbcvdfghasdfrtsdafasdfasdfdsfgsdgf_{|(y,s)| \leq 2r} \frac{|f(y,s)|}{|(y,s)|^{n+|\beta|+ml}} \,dy \,ds       \les M r^{|\beta|+ml} r^{d+\gamma-|\beta|-ml}       \les r^{d+\gamma}       ,     \end{split}     \llabel{6 QZ UcX3 mfET Exz1kv qE p OVV EFP IVp zQ lMOI Z2yT TxIUOm 0f W L1W oxC tlX Ws 9HU4 EF0I Z1WDv3 TP 4 2LN 7Tr SuR 8u Mv1t Lepv ZoeoKL xf 9 zMJ 6PU In1 S8 I4KY 13wJ TACh5X l8 O 5g0 ZGw Ddt u6 8wvr vnDC oqYjJ3 nF K WMA K8V OeG o4 DKxn EOyB wgmttc ES 8 dmT oAD 0YB Fl yGRB pBbo 8tQYBw bS X 2lc YnU 0fh At myR3 CKcU AQzzET Ng b ghH T64 KdO fL qFWu k07t DkzfQ1 dg B cw0 LSY lr7 9U 81QP qrdf H1tb8k Kn D l52 FhC j7T Xi P7GF C7HJ KfXgrP 4K O Og1 8BM 001 mJ PTpu bQr6 1JQu6o Gr 4 baj 60k zdX oD gAOX 2DBk LymrtN 6T 7 us2 Cp6 eZm 1a VJTY 8vYP OzMnsA 8ThswELzXU3X7Ebd1KdZ7v1rN3GiirRXGKWK099ovBM0FDJCvkopYNQ2aN94Z7k0UnUKamE3OjU8DFYFFokbSI2J9V9gVlM8ALWThDPnPu3EL7HPD2VDaZTggzcCCmbvc70qqPcC9mt60ogcrTiA3HEjwTK8ymKeuJMc4q6dVz200XnYUtLR9GYjPXvFOVr6W1zUK1WbPToaWJJuKnxBLnd0ftDEbMmj4loHYyhZyMjM91zQS4p7z8eKa9h0JrbacekcirexG0z4n3xz0QOWSvFj3jLhWXUIU21iIAwJtI3RbWa90I7rzAIqI3UElUJG7tLtUXzw4KQNETvXzqWaujEMenYlNIzLGxgB3AuJEQ188}     \end{align} where we used Lemma~\ref{L06} (i) with $a = d-m+\gamma + (m+n)/p$ and $b= n+|\beta|+ml$ in the third inequality. Therefore,     \begin{align}     \begin{split}       I_2        &=         \biggl( \zxvczxbcvdfghasdfrtsdafasdfasdfdsfgsdgf_{|(x,t)|\leq r}           \biggl|  \zxvczxbcvdfghasdfrtsdafasdfasdfdsfgsdgf_{|(y,s)| \leq 2r}         \sum_{0 \leq |\beta| + ml \leq d}  \partial^{\beta}_x \partial^l_t \Gamma (-y,-s) \frac{x^{\beta} t^l}{\beta! l!}  f(y,s) \, dy \, ds                   \biggr|^q \,dx \,dt           \biggr)^{1/q}      \\&      \les \biggl( \zxvczxbcvdfghasdfrtsdafasdfasdfdsfgsdgf_{|(x,t)|\leq r} \MM^q r^{(d+\gamma)q} \,dx \,dt \biggr)^{1/q}      \les \MM r^{d+\gamma+(m+n)/q}      .     \end{split}     \llabel{Gw Ddt u6 8wvr vnDC oqYjJ3 nF K WMA K8V OeG o4 DKxn EOyB wgmttc ES 8 dmT oAD 0YB Fl yGRB pBbo 8tQYBw bS X 2lc YnU 0fh At myR3 CKcU AQzzET Ng b ghH T64 KdO fL qFWu k07t DkzfQ1 dg B cw0 LSY lr7 9U 81QP qrdf H1tb8k Kn D l52 FhC j7T Xi P7GF C7HJ KfXgrP 4K O Og1 8BM 001 mJ PTpu bQr6 1JQu6o Gr 4 baj 60k zdX oD gAOX 2DBk LymrtN 6T 7 us2 Cp6 eZm 1a VJTY 8vYP OzMnsA qs 3 RL6 xHu mXN AB 5eXn ZRHa iECOaa MB w Ab1 5iF WGu cZ lU8J niDN KiPGWz q4 1 iBj 1kq bak ZF SvXq vSiR bLTriS y8 Q YOa mQU ZhO rG HYHW guPB zlAhua o5 9 RKU trF 5Kb js KseT PXhU qR8ThswELzXU3X7Ebd1KdZ7v1rN3GiirRXGKWK099ovBM0FDJCvkopYNQ2aN94Z7k0UnUKamE3OjU8DFYFFokbSI2J9V9gVlM8ALWThDPnPu3EL7HPD2VDaZTggzcCCmbvc70qqPcC9mt60ogcrTiA3HEjwTK8ymKeuJMc4q6dVz200XnYUtLR9GYjPXvFOVr6W1zUK1WbPToaWJJuKnxBLnd0ftDEbMmj4loHYyhZyMjM91zQS4p7z8eKa9h0JrbacekcirexG0z4n3xz0QOWSvFj3jLhWXUIU21iIAwJtI3RbWa90I7rzAIqI3UElUJG7tLtUXzw4KQNETvXzqWaujEMenYlNIzLGxgB3AuJEQ190}     \end{align} We now estimate~$I_3$. Similarly as for $I_2$, it suffices to show that for all $|x|\leq r$, we have     \begin{equation}       I(x)       \les \MM r^{d+\gamma}       ,        \llabel{cw0 LSY lr7 9U 81QP qrdf H1tb8k Kn D l52 FhC j7T Xi P7GF C7HJ KfXgrP 4K O Og1 8BM 001 mJ PTpu bQr6 1JQu6o Gr 4 baj 60k zdX oD gAOX 2DBk LymrtN 6T 7 us2 Cp6 eZm 1a VJTY 8vYP OzMnsA qs 3 RL6 xHu mXN AB 5eXn ZRHa iECOaa MB w Ab1 5iF WGu cZ lU8J niDN KiPGWz q4 1 iBj 1kq bak ZF SvXq vSiR bLTriS y8 Q YOa mQU ZhO rG HYHW guPB zlAhua o5 9 RKU trF 5Kb js KseT PXhU qRgnNA LV t aw4 YJB tK9 fN 7bN9 IEwK LTYGtn Cc c 2nf Mcx 7Vo Bt 1IC5 teMH X4g3JK 4J s deo Dl1 Xgb m9 xWDg Z31P chRS1R 8W 1 hap 5Rh 6Jj yT NXSC Uscx K4275D 72 g pRW xcf AbZ Y7 Apto 5S8ThswELzXU3X7Ebd1KdZ7v1rN3GiirRXGKWK099ovBM0FDJCvkopYNQ2aN94Z7k0UnUKamE3OjU8DFYFFokbSI2J9V9gVlM8ALWThDPnPu3EL7HPD2VDaZTggzcCCmbvc70qqPcC9mt60ogcrTiA3HEjwTK8ymKeuJMc4q6dVz200XnYUtLR9GYjPXvFOVr6W1zUK1WbPToaWJJuKnxBLnd0ftDEbMmj4loHYyhZyMjM91zQS4p7z8eKa9h0JrbacekcirexG0z4n3xz0QOWSvFj3jLhWXUIU21iIAwJtI3RbWa90I7rzAIqI3UElUJG7tLtUXzw4KQNETvXzqWaujEMenYlNIzLGxgB3AuJEQ191}     \end{equation} where      \begin{equation}       I(x)        = \biggl| \zxvczxbcvdfghasdfrtsdafasdfasdfdsfgsdgf_{2r\leq|(y,s)| < 1}           \biggl( \Gamma(x-y,t-s)             -\sum_{0 \leq |\beta| + ml \leq d}  \partial^{\beta}_x \partial^l_t \Gamma (-y,-s) \frac{x^{\beta} t^l}{\beta! l!}           \biggr) f(y,s) \,dy \,ds       \biggr|       .     \llabel{qs 3 RL6 xHu mXN AB 5eXn ZRHa iECOaa MB w Ab1 5iF WGu cZ lU8J niDN KiPGWz q4 1 iBj 1kq bak ZF SvXq vSiR bLTriS y8 Q YOa mQU ZhO rG HYHW guPB zlAhua o5 9 RKU trF 5Kb js KseT PXhU qRgnNA LV t aw4 YJB tK9 fN 7bN9 IEwK LTYGtn Cc c 2nf Mcx 7Vo Bt 1IC5 teMH X4g3JK 4J s deo Dl1 Xgb m9 xWDg Z31P chRS1R 8W 1 hap 5Rh 6Jj yT NXSC Uscx K4275D 72 g pRW xcf AbZ Y7 Apto 5SpT zO1dPA Vy Z JiW Clu OjO tE wxUB 7cTt EDqcAb YG d ZQZ fsQ 1At Hy xnPL 5K7D 91u03s 8K 2 0ro fZ9 w7T jx yG7q bCAh ssUZQu PK 7 xUe K7F 4HK fr CEPJ rgWH DZQpvR kO 8 Xve aSB OXS ee XV8ThswELzXU3X7Ebd1KdZ7v1rN3GiirRXGKWK099ovBM0FDJCvkopYNQ2aN94Z7k0UnUKamE3OjU8DFYFFokbSI2J9V9gVlM8ALWThDPnPu3EL7HPD2VDaZTggzcCCmbvc70qqPcC9mt60ogcrTiA3HEjwTK8ymKeuJMc4q6dVz200XnYUtLR9GYjPXvFOVr6W1zUK1WbPToaWJJuKnxBLnd0ftDEbMmj4loHYyhZyMjM91zQS4p7z8eKa9h0JrbacekcirexG0z4n3xz0QOWSvFj3jLhWXUIU21iIAwJtI3RbWa90I7rzAIqI3UElUJG7tLtUXzw4KQNETvXzqWaujEMenYlNIzLGxgB3AuJEQ192}     \end{equation} By Taylor's theorem, there exists $\xi$ and $K$ in $(0,1)$ such that      \begin{align}     \begin{split}       &\biggl|           \Gamma(x-y,t-s)             -\sum_{0 \leq |\beta| + ml \leq d}  \partial^{\beta}_x \partial^l_t \Gamma (-y,-s) \frac{x^{\beta} t^l}{\beta! l!}           \biggr|       \\& \indeq       \les \sum_{i=0}^d \sum_{|\beta|+l=i, |\beta|+ml \geq d+1}               |\partial^{\beta}_x \partial^l_t \Gamma (-y,-s) |              \frac{|x|^{\beta}|t|^l}{\beta ! l!}              + \sum_{|\beta|+l=d+1}               |\partial^{\beta}_x \partial^l_t  \Gamma (\xi x - y , K t - s )|              \frac{|x|^{\beta}|t|^l}{\beta ! l !}       \\& \indeq        \les \sum_{i=0}^d \sum_{|\beta|+l=i, |\beta|+ml \geq d+1}               \frac{1}{|(y,s)|^{n+|\beta|+ml}}              \frac{|x|^{\beta}|t|^l}{\beta ! l!}              + \sum_{|\beta|+l=d+1}               \frac{1}{|(\xi x - y , K t - s )|^{n+|\beta|+ml}}              \frac{|x|^{\beta}|t|^l}{\beta ! l !}       .     \end{split}     \llabel{gnNA LV t aw4 YJB tK9 fN 7bN9 IEwK LTYGtn Cc c 2nf Mcx 7Vo Bt 1IC5 teMH X4g3JK 4J s deo Dl1 Xgb m9 xWDg Z31P chRS1R 8W 1 hap 5Rh 6Jj yT NXSC Uscx K4275D 72 g pRW xcf AbZ Y7 Apto 5SpT zO1dPA Vy Z JiW Clu OjO tE wxUB 7cTt EDqcAb YG d ZQZ fsQ 1At Hy xnPL 5K7D 91u03s 8K 2 0ro fZ9 w7T jx yG7q bCAh ssUZQu PK 7 xUe K7F 4HK fr CEPJ rgWH DZQpvR kO 8 Xve aSB OXS ee XV5j kgzL UTmMbo ma J fxu 8gA rnd zS IB0Y QSXv cZW8vo CO o OHy rEu GnS 2f nGEj jaLz ZIocQe gw H fSF KjW 2Lb KS nIcG 9Wnq Zya6qA YM S h2M mEA sw1 8n sJFY Anbr xZT45Z wB s BvK 9gS Ugy 8ThswELzXU3X7Ebd1KdZ7v1rN3GiirRXGKWK099ovBM0FDJCvkopYNQ2aN94Z7k0UnUKamE3OjU8DFYFFokbSI2J9V9gVlM8ALWThDPnPu3EL7HPD2VDaZTggzcCCmbvc70qqPcC9mt60ogcrTiA3HEjwTK8ymKeuJMc4q6dVz200XnYUtLR9GYjPXvFOVr6W1zUK1WbPToaWJJuKnxBLnd0ftDEbMmj4loHYyhZyMjM91zQS4p7z8eKa9h0JrbacekcirexG0z4n3xz0QOWSvFj3jLhWXUIU21iIAwJtI3RbWa90I7rzAIqI3UElUJG7tLtUXzw4KQNETvXzqWaujEMenYlNIzLGxgB3AuJEQ193}     \end{align} Using $|(y,s)| \leq 2 |(\xi x - y , K t - s )|$, we obtain     \begin{align}     \begin{split}      &       \biggl|          \Gamma(x-y,t-s)             -\sum_{0 \leq |\beta| + ml \leq d}  \partial^{\beta}_x \partial^l_t \Gamma (-y,-s) \frac{x^{\beta} t^l}{\beta! l!}        \biggr|       \\&\indeq       \les \sum_{j = d+1}^{m(d+1)}               \biggl( \frac{1}{|(y,s)|^{n+j}}                        + \frac{1}{|(\xi x - y , K t - s ) |^{n+j}}               \biggr) |(x,t)|^j       \les \sum_{j = d+1}^{m(d+1)}               \frac{|(x,t)|^j}{|(y,s)|^{n+j}}              .    \llabel{pT zO1dPA Vy Z JiW Clu OjO tE wxUB 7cTt EDqcAb YG d ZQZ fsQ 1At Hy xnPL 5K7D 91u03s 8K 2 0ro fZ9 w7T jx yG7q bCAh ssUZQu PK 7 xUe K7F 4HK fr CEPJ rgWH DZQpvR kO 8 Xve aSB OXS ee XV5j kgzL UTmMbo ma J fxu 8gA rnd zS IB0Y QSXv cZW8vo CO o OHy rEu GnS 2f nGEj jaLz ZIocQe gw H fSF KjW 2Lb KS nIcG 9Wnq Zya6qA YM S h2M mEA sw1 8n sJFY Anbr xZT45Z wB s BvK 9gS Ugy Bk 3dHq dvYU LhWgGK aM f Fk7 8mP 20m eV aQp2 NWIb 6hVBSe SV w nEq bq6 ucn X8 JLkI RJbJ EbwEYw nv L BgM 94G plc lu 2s3U m15E YAjs1G Ln h zG8 vmh ghs Qc EDE1 KnaH wtuxOg UD L BE5 9FL8ThswELzXU3X7Ebd1KdZ7v1rN3GiirRXGKWK099ovBM0FDJCvkopYNQ2aN94Z7k0UnUKamE3OjU8DFYFFokbSI2J9V9gVlM8ALWThDPnPu3EL7HPD2VDaZTggzcCCmbvc70qqPcC9mt60ogcrTiA3HEjwTK8ymKeuJMc4q6dVz200XnYUtLR9GYjPXvFOVr6W1zUK1WbPToaWJJuKnxBLnd0ftDEbMmj4loHYyhZyMjM91zQS4p7z8eKa9h0JrbacekcirexG0z4n3xz0QOWSvFj3jLhWXUIU21iIAwJtI3RbWa90I7rzAIqI3UElUJG7tLtUXzw4KQNETvXzqWaujEMenYlNIzLGxgB3AuJEQ194}    \end{split}     \end{align} Therefore,     \begin{equation}       I(x)       \les \sum_{j=d+1}^{m(d+1)}               |(x,t)|^j              \zxvczxbcvdfghasdfrtsdafasdfasdfdsfgsdgf_{2r \leq |(y,s)|<1} \frac{|f(y,s)|}{|(y,s)|^{n+j}} \,dy \,ds       \les \sum_{j=d+1}^{m(d+1)}               r^j              \sum_{i=1}^{N-1}              \zxvczxbcvdfghasdfrtsdafasdfasdfdsfgsdgf_{2^i r < |(y,s)|<2^{i+1}r} \frac{|f(y,s)|}{|(y,s)|^{n+j}} \,dy \,ds       ,      \llabel{5j kgzL UTmMbo ma J fxu 8gA rnd zS IB0Y QSXv cZW8vo CO o OHy rEu GnS 2f nGEj jaLz ZIocQe gw H fSF KjW 2Lb KS nIcG 9Wnq Zya6qA YM S h2M mEA sw1 8n sJFY Anbr xZT45Z wB s BvK 9gS Ugy Bk 3dHq dvYU LhWgGK aM f Fk7 8mP 20m eV aQp2 NWIb 6hVBSe SV w nEq bq6 ucn X8 JLkI RJbJ EbwEYw nv L BgM 94G plc lu 2s3U m15E YAjs1G Ln h zG8 vmh ghs Qc EDE1 KnaH wtuxOg UD L BE5 9FL xIp vu KfJE UTQS EaZ6hu BC a KXr lni r1X mL KH3h VPrq ixmTkR zh 0 OGp Obo N6K LC E0Ga Udta nZ9Lvt 1K Z eN5 GQc LQL L0 P9GX uakH m6kqk7 qm X UVH 2bU Hga v0 Wp6Q 8JyI TzlpqW 0Y k 1f8ThswELzXU3X7Ebd1KdZ7v1rN3GiirRXGKWK099ovBM0FDJCvkopYNQ2aN94Z7k0UnUKamE3OjU8DFYFFokbSI2J9V9gVlM8ALWThDPnPu3EL7HPD2VDaZTggzcCCmbvc70qqPcC9mt60ogcrTiA3HEjwTK8ymKeuJMc4q6dVz200XnYUtLR9GYjPXvFOVr6W1zUK1WbPToaWJJuKnxBLnd0ftDEbMmj4loHYyhZyMjM91zQS4p7z8eKa9h0JrbacekcirexG0z4n3xz0QOWSvFj3jLhWXUIU21iIAwJtI3RbWa90I7rzAIqI3UElUJG7tLtUXzw4KQNETvXzqWaujEMenYlNIzLGxgB3AuJEQ195}     \end{equation} where $N$ is an integer satisfying $2^{N-1}r \leq 1 \leq 2^Nr$. Applying Lemma~\ref{L06} (ii), we get     \begin{align}     \begin{split}
      I(x)       &       \les \MM r^{d+\gamma}       .            \end{split}     \llabel{Bk 3dHq dvYU LhWgGK aM f Fk7 8mP 20m eV aQp2 NWIb 6hVBSe SV w nEq bq6 ucn X8 JLkI RJbJ EbwEYw nv L BgM 94G plc lu 2s3U m15E YAjs1G Ln h zG8 vmh ghs Qc EDE1 KnaH wtuxOg UD L BE5 9FL xIp vu KfJE UTQS EaZ6hu BC a KXr lni r1X mL KH3h VPrq ixmTkR zh 0 OGp Obo N6K LC E0Ga Udta nZ9Lvt 1K Z eN5 GQc LQL L0 P9GX uakH m6kqk7 qm X UVH 2bU Hga v0 Wp6Q 8JyI TzlpqW 0Y k 1fX 8gj Gci bR arme Si8l w03Win NX w 1gv vcD eDP Sa bsVw Zu4h aO1V2D qw k JoR Shj MBg ry glA9 3DBd S0mYAc El 5 aEd pII DT5 mb SVuX o8Nl Y24WCA 6d f CVF 6Al a6i Ns 7GCh OvFA hbxw9Q 718ThswELzXU3X7Ebd1KdZ7v1rN3GiirRXGKWK099ovBM0FDJCvkopYNQ2aN94Z7k0UnUKamE3OjU8DFYFFokbSI2J9V9gVlM8ALWThDPnPu3EL7HPD2VDaZTggzcCCmbvc70qqPcC9mt60ogcrTiA3HEjwTK8ymKeuJMc4q6dVz200XnYUtLR9GYjPXvFOVr6W1zUK1WbPToaWJJuKnxBLnd0ftDEbMmj4loHYyhZyMjM91zQS4p7z8eKa9h0JrbacekcirexG0z4n3xz0QOWSvFj3jLhWXUIU21iIAwJtI3RbWa90I7rzAIqI3UElUJG7tLtUXzw4KQNETvXzqWaujEMenYlNIzLGxgB3AuJEQ196}     \end{align} Similarly to $I_2$, we obtain that $I_3 \les \MM r^{d+\gamma+(m+n)/q}$. Therefore, $       \zxvczxbcvdfghasdfrtsdafasdfasdfdsfgsdgh u \zxvczxbcvdfghasdfrtsdafasdfasdfdsfgsdgh_{L^q(Q_r)}        \les \MM r^{d+\gamma+(m+n)/q} $, for $r\in(0,1/2]$, and \eqref{8ThswELzXU3X7Ebd1KdZ7v1rN3GiirRXGKWK099ovBM0FDJCvkopYNQ2aN94Z7k0UnUKamE3OjU8DFYFFokbSI2J9V9gVlM8ALWThDPnPu3EL7HPD2VDaZTggzcCCmbvc70qqPcC9mt60ogcrTiA3HEjwTK8ymKeuJMc4q6dVz200XnYUtLR9GYjPXvFOVr6W1zUK1WbPToaWJJuKnxBLnd0ftDEbMmj4loHYyhZyMjM91zQS4p7z8eKa9h0JrbacekcirexG0z4n3xz0QOWSvFj3jLhWXUIU21iIAwJtI3RbWa90I7rzAIqI3UElUJG7tLtUXzw4KQNETvXzqWaujEMenYlNIzLGxgB3AuJEQ175} follows. \end{proof} \par As in the elliptic case, we point out that the proof of Lemma~\ref{L07} also gives the next statement. \par \cole \begin{Lemma} \label{L08} Suppose that $f \in L^p(Q_1)$ satisfies     \begin{equation}       \zxvczxbcvdfghasdfrtsdafasdfasdfdsfgsdgh f \zxvczxbcvdfghasdfrtsdafasdfasdfdsfgsdgh_{L^p(Q_r)}       \leq \MM r^{d-m+\alpha + (m+n)/p}       \comma r \leq 1       ,    \llabel{ xIp vu KfJE UTQS EaZ6hu BC a KXr lni r1X mL KH3h VPrq ixmTkR zh 0 OGp Obo N6K LC E0Ga Udta nZ9Lvt 1K Z eN5 GQc LQL L0 P9GX uakH m6kqk7 qm X UVH 2bU Hga v0 Wp6Q 8JyI TzlpqW 0Y k 1fX 8gj Gci bR arme Si8l w03Win NX w 1gv vcD eDP Sa bsVw Zu4h aO1V2D qw k JoR Shj MBg ry glA9 3DBd S0mYAc El 5 aEd pII DT5 mb SVuX o8Nl Y24WCA 6d f CVF 6Al a6i Ns 7GCh OvFA hbxw9Q 71 Z RC8 yRi 1zZ dM rpt7 3dou ogkAkG GE 4 87V ii4 Ofw Je sXUR dzVL HU0zms 8W 2 Ztz iY5 mw9 aB ZIwk 5WNm vNM2Hd jn e wMR 8qp 2Vv up cV4P cjOG eu35u5 cQ X NTy kfT ZXA JH UnSs 4zxf Hwf18ThswELzXU3X7Ebd1KdZ7v1rN3GiirRXGKWK099ovBM0FDJCvkopYNQ2aN94Z7k0UnUKamE3OjU8DFYFFokbSI2J9V9gVlM8ALWThDPnPu3EL7HPD2VDaZTggzcCCmbvc70qqPcC9mt60ogcrTiA3HEjwTK8ymKeuJMc4q6dVz200XnYUtLR9GYjPXvFOVr6W1zUK1WbPToaWJJuKnxBLnd0ftDEbMmj4loHYyhZyMjM91zQS4p7z8eKa9h0JrbacekcirexG0z4n3xz0QOWSvFj3jLhWXUIU21iIAwJtI3RbWa90I7rzAIqI3UElUJG7tLtUXzw4KQNETvXzqWaujEMenYlNIzLGxgB3AuJEQ198}     \end{equation} for some integer $d \geq m $ and $ \alpha \in (0,1)$. For any solution $u \in W^{m,1}_p(Q_1)$ of \eqref{8ThswELzXU3X7Ebd1KdZ7v1rN3GiirRXGKWK099ovBM0FDJCvkopYNQ2aN94Z7k0UnUKamE3OjU8DFYFFokbSI2J9V9gVlM8ALWThDPnPu3EL7HPD2VDaZTggzcCCmbvc70qqPcC9mt60ogcrTiA3HEjwTK8ymKeuJMc4q6dVz200XnYUtLR9GYjPXvFOVr6W1zUK1WbPToaWJJuKnxBLnd0ftDEbMmj4loHYyhZyMjM91zQS4p7z8eKa9h0JrbacekcirexG0z4n3xz0QOWSvFj3jLhWXUIU21iIAwJtI3RbWa90I7rzAIqI3UElUJG7tLtUXzw4KQNETvXzqWaujEMenYlNIzLGxgB3AuJEQP01}, there exists a polynomial $P_d$ of  degree at most $d$ such that      \begin{equation}     L(0) P_d = 0     ,    \llabel{X 8gj Gci bR arme Si8l w03Win NX w 1gv vcD eDP Sa bsVw Zu4h aO1V2D qw k JoR Shj MBg ry glA9 3DBd S0mYAc El 5 aEd pII DT5 mb SVuX o8Nl Y24WCA 6d f CVF 6Al a6i Ns 7GCh OvFA hbxw9Q 71 Z RC8 yRi 1zZ dM rpt7 3dou ogkAkG GE 4 87V ii4 Ofw Je sXUR dzVL HU0zms 8W 2 Ztz iY5 mw9 aB ZIwk 5WNm vNM2Hd jn e wMR 8qp 2Vv up cV4P cjOG eu35u5 cQ X NTy kfT ZXA JH UnSs 4zxf Hwf10r it J Yox Rto 5OM FP hakR gzDY Pm02mG 18 v mfV 11N n87 zS X59D E0cN 99uEUz 2r T h1F P8x jrm q2 Z7ut pdRJ 2DdYkj y9 J Yko c38 Kdu Z9 vydO wkO0 djhXSx Sv H wJo XE7 9f8 qh iBr8 KYTx8ThswELzXU3X7Ebd1KdZ7v1rN3GiirRXGKWK099ovBM0FDJCvkopYNQ2aN94Z7k0UnUKamE3OjU8DFYFFokbSI2J9V9gVlM8ALWThDPnPu3EL7HPD2VDaZTggzcCCmbvc70qqPcC9mt60ogcrTiA3HEjwTK8ymKeuJMc4q6dVz200XnYUtLR9GYjPXvFOVr6W1zUK1WbPToaWJJuKnxBLnd0ftDEbMmj4loHYyhZyMjM91zQS4p7z8eKa9h0JrbacekcirexG0z4n3xz0QOWSvFj3jLhWXUIU21iIAwJtI3RbWa90I7rzAIqI3UElUJG7tLtUXzw4KQNETvXzqWaujEMenYlNIzLGxgB3AuJEQ199}     \end{equation} and for any $q$ such that $1\leq q < (m+n)p/(m+n-mp)$, we have     \begin{equation}     \zxvczxbcvdfghasdfrtsdafasdfasdfdsfgsdgh u - P_d \zxvczxbcvdfghasdfrtsdafasdfasdfdsfgsdgh_{L^q(Q_r)}     \les (\MM + \zxvczxbcvdfghasdfrtsdafasdfasdfdsfgsdgh u \zxvczxbcvdfghasdfrtsdafasdfasdfdsfgsdgh_{L^p(Q_1)})            r^{d+\alpha+(m+n)/q}      \comma       (x,t) \in Q_{1/2}      ,    \llabel{ Z RC8 yRi 1zZ dM rpt7 3dou ogkAkG GE 4 87V ii4 Ofw Je sXUR dzVL HU0zms 8W 2 Ztz iY5 mw9 aB ZIwk 5WNm vNM2Hd jn e wMR 8qp 2Vv up cV4P cjOG eu35u5 cQ X NTy kfT ZXA JH UnSs 4zxf Hwf10r it J Yox Rto 5OM FP hakR gzDY Pm02mG 18 v mfV 11N n87 zS X59D E0cN 99uEUz 2r T h1F P8x jrm q2 Z7ut pdRJ 2DdYkj y9 J Yko c38 Kdu Z9 vydO wkO0 djhXSx Sv H wJo XE7 9f8 qh iBr8 KYTx OfcYYF sM y j0H vK3 ayU wt 4nA5 H76b wUqyJQ od O u8U Gjb t6v lc xYZt 6AUx wpYr18 uO v 62v jnw FrC rf Z4nl vJuh 2SpVLO vp O lZn PTG 07V Re ixBm XBxO BzpFW5 iB I O7R Vmo GnJ u8 Axol8ThswELzXU3X7Ebd1KdZ7v1rN3GiirRXGKWK099ovBM0FDJCvkopYNQ2aN94Z7k0UnUKamE3OjU8DFYFFokbSI2J9V9gVlM8ALWThDPnPu3EL7HPD2VDaZTggzcCCmbvc70qqPcC9mt60ogcrTiA3HEjwTK8ymKeuJMc4q6dVz200XnYUtLR9GYjPXvFOVr6W1zUK1WbPToaWJJuKnxBLnd0ftDEbMmj4loHYyhZyMjM91zQS4p7z8eKa9h0JrbacekcirexG0z4n3xz0QOWSvFj3jLhWXUIU21iIAwJtI3RbWa90I7rzAIqI3UElUJG7tLtUXzw4KQNETvXzqWaujEMenYlNIzLGxgB3AuJEQ200}     \end{equation} where the implicit constants depending only on $p$, $q$, and~$d$. \end{Lemma}  \colb \par \subsection{The bootstrap argument} The following bootstrapping argument is needed in the proof of Theorem~\ref{T03}.  Since the argument is similar to Lemma~\ref{L04}, we omit the proof. \par \cole \begin{Lemma} \label{L09} Under the assumptions of Theorem~\ref{T03}, suppose that there exists $k \in \mathbb{N}_0$ such that $(k+1) \alpha +  \eta < 1$. Then if $u$ vanishes of order $d-1+ \eta+k\alpha+(m+n)/p$, i.e.,     \begin{equation}       c_k        = \sup_{r \leq 1}\frac{\zxvczxbcvdfghasdfrtsdafasdfasdfdsfgsdgh u \zxvczxbcvdfghasdfrtsdafasdfasdfdsfgsdgh_{L^p(Q_r)}}{r^{d-1+ \eta+k\alpha+(m+n)/p}}       < \infty       ,     \llabel{0r it J Yox Rto 5OM FP hakR gzDY Pm02mG 18 v mfV 11N n87 zS X59D E0cN 99uEUz 2r T h1F P8x jrm q2 Z7ut pdRJ 2DdYkj y9 J Yko c38 Kdu Z9 vydO wkO0 djhXSx Sv H wJo XE7 9f8 qh iBr8 KYTx OfcYYF sM y j0H vK3 ayU wt 4nA5 H76b wUqyJQ od O u8U Gjb t6v lc xYZt 6AUx wpYr18 uO v 62v jnw FrC rf Z4nl vJuh 2SpVLO vp O lZn PTG 07V Re ixBm XBxO BzpFW5 iB I O7R Vmo GnJ u8 Axol YAxl JUrYKV Kk p aIk VCu PiD O8 IHPU ndze LPTILB P5 B qYy DLZ DZa db jcJA T644 Vp6byb 1g 4 dE7 Ydz keO YL hCRe Ommx F9zsu0 rp 8 Ajz d2v Heo 7L 5zVn L8IQ WnYATK KV 1 f14 s2J geC b38ThswELzXU3X7Ebd1KdZ7v1rN3GiirRXGKWK099ovBM0FDJCvkopYNQ2aN94Z7k0UnUKamE3OjU8DFYFFokbSI2J9V9gVlM8ALWThDPnPu3EL7HPD2VDaZTggzcCCmbvc70qqPcC9mt60ogcrTiA3HEjwTK8ymKeuJMc4q6dVz200XnYUtLR9GYjPXvFOVr6W1zUK1WbPToaWJJuKnxBLnd0ftDEbMmj4loHYyhZyMjM91zQS4p7z8eKa9h0JrbacekcirexG0z4n3xz0QOWSvFj3jLhWXUIU21iIAwJtI3RbWa90I7rzAIqI3UElUJG7tLtUXzw4KQNETvXzqWaujEMenYlNIzLGxgB3AuJEQ201}     \end{equation} then $u$ vanishes of $L^p$-order $d-1+\eta+(k+1)\alpha+(m+n)/p$, and moreover,      \begin{equation}      c_{k+1}      = \sup_{r \leq 1}\frac{\zxvczxbcvdfghasdfrtsdafasdfasdfdsfgsdgh u \zxvczxbcvdfghasdfrtsdafasdfasdfdsfgsdgh_{L^p(Q_r)}}{r^{d-1+ \eta+(k+1)\alpha+(m+n)/p}}      \les c_k + M   + \zxvczxbcvdfghasdfrtsdafasdfasdfdsfgsdgh Q \zxvczxbcvdfghasdfrtsdafasdfasdfdsfgsdgh_{L^p(Q_1)}      < \infty     \label{8ThswELzXU3X7Ebd1KdZ7v1rN3GiirRXGKWK099ovBM0FDJCvkopYNQ2aN94Z7k0UnUKamE3OjU8DFYFFokbSI2J9V9gVlM8ALWThDPnPu3EL7HPD2VDaZTggzcCCmbvc70qqPcC9mt60ogcrTiA3HEjwTK8ymKeuJMc4q6dVz200XnYUtLR9GYjPXvFOVr6W1zUK1WbPToaWJJuKnxBLnd0ftDEbMmj4loHYyhZyMjM91zQS4p7z8eKa9h0JrbacekcirexG0z4n3xz0QOWSvFj3jLhWXUIU21iIAwJtI3RbWa90I7rzAIqI3UElUJG7tLtUXzw4KQNETvXzqWaujEMenYlNIzLGxgB3AuJEQ202}     \end{equation} also holds, where $M$ and $Q$ are defined in Theorem~\ref{T03}, and the implicit constant in \eqref{8ThswELzXU3X7Ebd1KdZ7v1rN3GiirRXGKWK099ovBM0FDJCvkopYNQ2aN94Z7k0UnUKamE3OjU8DFYFFokbSI2J9V9gVlM8ALWThDPnPu3EL7HPD2VDaZTggzcCCmbvc70qqPcC9mt60ogcrTiA3HEjwTK8ymKeuJMc4q6dVz200XnYUtLR9GYjPXvFOVr6W1zUK1WbPToaWJJuKnxBLnd0ftDEbMmj4loHYyhZyMjM91zQS4p7z8eKa9h0JrbacekcirexG0z4n3xz0QOWSvFj3jLhWXUIU21iIAwJtI3RbWa90I7rzAIqI3UElUJG7tLtUXzw4KQNETvXzqWaujEMenYlNIzLGxgB3AuJEQ202} depends on $d$, $p$, and~$q$. \end{Lemma} \colb \par We point out that, as in the elliptic case, we have the inequalities \begin{equation} \label{8ThswELzXU3X7Ebd1KdZ7v1rN3GiirRXGKWK099ovBM0FDJCvkopYNQ2aN94Z7k0UnUKamE3OjU8DFYFFokbSI2J9V9gVlM8ALWThDPnPu3EL7HPD2VDaZTggzcCCmbvc70qqPcC9mt60ogcrTiA3HEjwTK8ymKeuJMc4q6dVz200XnYUtLR9GYjPXvFOVr6W1zUK1WbPToaWJJuKnxBLnd0ftDEbMmj4loHYyhZyMjM91zQS4p7z8eKa9h0JrbacekcirexG0z4n3xz0QOWSvFj3jLhWXUIU21iIAwJtI3RbWa90I7rzAIqI3UElUJG7tLtUXzw4KQNETvXzqWaujEMenYlNIzLGxgB3AuJEQP04} \sum_{|\nu|\leq m} r^{|\nu|} \zxvczxbcvdfghasdfrtsdafasdfasdfdsfgsdgh \partial^{\nu} u \zxvczxbcvdfghasdfrtsdafasdfasdfdsfgsdgh_{L^p(Q_r)}       \les \bigl( c_0 + \MM + \zxvczxbcvdfghasdfrtsdafasdfasdfdsfgsdgh Q \zxvczxbcvdfghasdfrtsdafasdfasdfdsfgsdgh_{L^p(Q_1)} \bigr)                r^{d-1+ \eta+k\alpha + (m+n)/p} \end{equation} and \begin{equation} \label{8ThswELzXU3X7Ebd1KdZ7v1rN3GiirRXGKWK099ovBM0FDJCvkopYNQ2aN94Z7k0UnUKamE3OjU8DFYFFokbSI2J9V9gVlM8ALWThDPnPu3EL7HPD2VDaZTggzcCCmbvc70qqPcC9mt60ogcrTiA3HEjwTK8ymKeuJMc4q6dVz200XnYUtLR9GYjPXvFOVr6W1zUK1WbPToaWJJuKnxBLnd0ftDEbMmj4loHYyhZyMjM91zQS4p7z8eKa9h0JrbacekcirexG0z4n3xz0QOWSvFj3jLhWXUIU21iIAwJtI3RbWa90I7rzAIqI3UElUJG7tLtUXzw4KQNETvXzqWaujEMenYlNIzLGxgB3AuJEQP05} \zxvczxbcvdfghasdfrtsdafasdfasdfdsfgsdgh L(0) u - Q \zxvczxbcvdfghasdfrtsdafasdfasdfdsfgsdgh_{L^p(Q_r)} \les \bigl( c_0 + M + \zxvczxbcvdfghasdfrtsdafasdfasdfdsfgsdgh Q \zxvczxbcvdfghasdfrtsdafasdfasdfdsfgsdgh_{L^p(Q_1)} \bigr)  r^{d-m+\eta+(k+1)\alpha - 1+ (m+n)/p} . \end{equation} \par \subsection{The Schauder estimates for parabolic equations.} \begin{proof}[Proof of Theorem~\ref{T03}] \par Let $k_0 \in \mathbb{N}_0$ be such that $k_0 \alpha + \eta < 1 \leq (k_0+1) \alpha + \eta$.  As above, we can assume that the second inequality is strict. We first show that  \begin{equation} \llabel{ OfcYYF sM y j0H vK3 ayU wt 4nA5 H76b wUqyJQ od O u8U Gjb t6v lc xYZt 6AUx wpYr18 uO v 62v jnw FrC rf Z4nl vJuh 2SpVLO vp O lZn PTG 07V Re ixBm XBxO BzpFW5 iB I O7R Vmo GnJ u8 Axol YAxl JUrYKV Kk p aIk VCu PiD O8 IHPU ndze LPTILB P5 B qYy DLZ DZa db jcJA T644 Vp6byb 1g 4 dE7 Ydz keO YL hCRe Ommx F9zsu0 rp 8 Ajz d2v Heo 7L 5zVn L8IQ WnYATK KV 1 f14 s2J geC b3 v9UJ djNN VBINix 1q 5 oyr SBM 2Xt gr v8RQ MaXk a4AN9i Ni n zfH xGp A57 uA E4jM fg6S 6eNGKv JL 3 tyH 3qw dPr x2 jFXW 2Wih pSSxDr aA 7 PXg jK6 GGl Og 5PkR d2n5 3eEx4N yG h d8Z RkO N8ThswELzXU3X7Ebd1KdZ7v1rN3GiirRXGKWK099ovBM0FDJCvkopYNQ2aN94Z7k0UnUKamE3OjU8DFYFFokbSI2J9V9gVlM8ALWThDPnPu3EL7HPD2VDaZTggzcCCmbvc70qqPcC9mt60ogcrTiA3HEjwTK8ymKeuJMc4q6dVz200XnYUtLR9GYjPXvFOVr6W1zUK1WbPToaWJJuKnxBLnd0ftDEbMmj4loHYyhZyMjM91zQS4p7z8eKa9h0JrbacekcirexG0z4n3xz0QOWSvFj3jLhWXUIU21iIAwJtI3RbWa90I7rzAIqI3UElUJG7tLtUXzw4KQNETvXzqWaujEMenYlNIzLGxgB3AuJEQP06} \zxvczxbcvdfghasdfrtsdafasdfasdfdsfgsdgh u \zxvczxbcvdfghasdfrtsdafasdfasdfdsfgsdgh_{L^p(Q_r)} \leq (c_0 + M + \zxvczxbcvdfghasdfrtsdafasdfasdfdsfgsdgh Q \zxvczxbcvdfghasdfrtsdafasdfasdfdsfgsdgh_{L^p(Q_r)}) 	r^{d+(m+n)/p} 	. \end{equation} Applying Lemma~\ref{L09} $k_0$ times, we obtain  \begin{equation} \llabel{ YAxl JUrYKV Kk p aIk VCu PiD O8 IHPU ndze LPTILB P5 B qYy DLZ DZa db jcJA T644 Vp6byb 1g 4 dE7 Ydz keO YL hCRe Ommx F9zsu0 rp 8 Ajz d2v Heo 7L 5zVn L8IQ WnYATK KV 1 f14 s2J geC b3 v9UJ djNN VBINix 1q 5 oyr SBM 2Xt gr v8RQ MaXk a4AN9i Ni n zfH xGp A57 uA E4jM fg6S 6eNGKv JL 3 tyH 3qw dPr x2 jFXW 2Wih pSSxDr aA 7 PXg jK6 GGl Og 5PkR d2n5 3eEx4N yG h d8Z RkO NMQ qL q4sE RG0C ssQkdZ Ua O vWr pla BOW rS wSG1 SM8I z9qkpd v0 C RMs GcZ LAz 4G k70e O7k6 df4uYn R6 T 5Du KOT say 0D awWQ vn2U OOPNqQ T7 H 4Hf iKY Jcl Rq M2g9 lcQZ cvCNBP 2B b tjv 8ThswELzXU3X7Ebd1KdZ7v1rN3GiirRXGKWK099ovBM0FDJCvkopYNQ2aN94Z7k0UnUKamE3OjU8DFYFFokbSI2J9V9gVlM8ALWThDPnPu3EL7HPD2VDaZTggzcCCmbvc70qqPcC9mt60ogcrTiA3HEjwTK8ymKeuJMc4q6dVz200XnYUtLR9GYjPXvFOVr6W1zUK1WbPToaWJJuKnxBLnd0ftDEbMmj4loHYyhZyMjM91zQS4p7z8eKa9h0JrbacekcirexG0z4n3xz0QOWSvFj3jLhWXUIU21iIAwJtI3RbWa90I7rzAIqI3UElUJG7tLtUXzw4KQNETvXzqWaujEMenYlNIzLGxgB3AuJEQP07} c_{k_0}  = \sup_{r\leq1} \frac{\zxvczxbcvdfghasdfrtsdafasdfasdfdsfgsdgh u \zxvczxbcvdfghasdfrtsdafasdfasdfdsfgsdgh_{L^p(Q_r)}}{r^{d-1+\eta+k_0\alpha+(m+n)/p}} \les c_0 + M + \zxvczxbcvdfghasdfrtsdafasdfasdfdsfgsdgh Q \zxvczxbcvdfghasdfrtsdafasdfasdfdsfgsdgh_{L^p(Q_1)} , \end{equation} where $c_0$ is defined in~\eqref{8ThswELzXU3X7Ebd1KdZ7v1rN3GiirRXGKWK099ovBM0FDJCvkopYNQ2aN94Z7k0UnUKamE3OjU8DFYFFokbSI2J9V9gVlM8ALWThDPnPu3EL7HPD2VDaZTggzcCCmbvc70qqPcC9mt60ogcrTiA3HEjwTK8ymKeuJMc4q6dVz200XnYUtLR9GYjPXvFOVr6W1zUK1WbPToaWJJuKnxBLnd0ftDEbMmj4loHYyhZyMjM91zQS4p7z8eKa9h0JrbacekcirexG0z4n3xz0QOWSvFj3jLhWXUIU21iIAwJtI3RbWa90I7rzAIqI3UElUJG7tLtUXzw4KQNETvXzqWaujEMenYlNIzLGxgB3AuJEQ155}. As in \eqref{8ThswELzXU3X7Ebd1KdZ7v1rN3GiirRXGKWK099ovBM0FDJCvkopYNQ2aN94Z7k0UnUKamE3OjU8DFYFFokbSI2J9V9gVlM8ALWThDPnPu3EL7HPD2VDaZTggzcCCmbvc70qqPcC9mt60ogcrTiA3HEjwTK8ymKeuJMc4q6dVz200XnYUtLR9GYjPXvFOVr6W1zUK1WbPToaWJJuKnxBLnd0ftDEbMmj4loHYyhZyMjM91zQS4p7z8eKa9h0JrbacekcirexG0z4n3xz0QOWSvFj3jLhWXUIU21iIAwJtI3RbWa90I7rzAIqI3UElUJG7tLtUXzw4KQNETvXzqWaujEMenYlNIzLGxgB3AuJEQP04}, we have \begin{equation} \label{8ThswELzXU3X7Ebd1KdZ7v1rN3GiirRXGKWK099ovBM0FDJCvkopYNQ2aN94Z7k0UnUKamE3OjU8DFYFFokbSI2J9V9gVlM8ALWThDPnPu3EL7HPD2VDaZTggzcCCmbvc70qqPcC9mt60ogcrTiA3HEjwTK8ymKeuJMc4q6dVz200XnYUtLR9GYjPXvFOVr6W1zUK1WbPToaWJJuKnxBLnd0ftDEbMmj4loHYyhZyMjM91zQS4p7z8eKa9h0JrbacekcirexG0z4n3xz0QOWSvFj3jLhWXUIU21iIAwJtI3RbWa90I7rzAIqI3UElUJG7tLtUXzw4KQNETvXzqWaujEMenYlNIzLGxgB3AuJEQP08} \sum_{|\nu|\leq m} r^{|\nu|} \zxvczxbcvdfghasdfrtsdafasdfasdfdsfgsdgh \partial^{\nu} u \zxvczxbcvdfghasdfrtsdafasdfasdfdsfgsdgh_{L^p(Q_r)} \les \bigl( c_0 + \MM + \zxvczxbcvdfghasdfrtsdafasdfasdfdsfgsdgh Q \zxvczxbcvdfghasdfrtsdafasdfasdfdsfgsdgh_{L^p(Q_1)} \bigr) r^{d-1+ \eta+k_0\alpha + (m+n)/p} \comma r \in (0, 1/2] , \end{equation} while as in \eqref{8ThswELzXU3X7Ebd1KdZ7v1rN3GiirRXGKWK099ovBM0FDJCvkopYNQ2aN94Z7k0UnUKamE3OjU8DFYFFokbSI2J9V9gVlM8ALWThDPnPu3EL7HPD2VDaZTggzcCCmbvc70qqPcC9mt60ogcrTiA3HEjwTK8ymKeuJMc4q6dVz200XnYUtLR9GYjPXvFOVr6W1zUK1WbPToaWJJuKnxBLnd0ftDEbMmj4loHYyhZyMjM91zQS4p7z8eKa9h0JrbacekcirexG0z4n3xz0QOWSvFj3jLhWXUIU21iIAwJtI3RbWa90I7rzAIqI3UElUJG7tLtUXzw4KQNETvXzqWaujEMenYlNIzLGxgB3AuJEQP05}, we may bound \begin{equation} \label{8ThswELzXU3X7Ebd1KdZ7v1rN3GiirRXGKWK099ovBM0FDJCvkopYNQ2aN94Z7k0UnUKamE3OjU8DFYFFokbSI2J9V9gVlM8ALWThDPnPu3EL7HPD2VDaZTggzcCCmbvc70qqPcC9mt60ogcrTiA3HEjwTK8ymKeuJMc4q6dVz200XnYUtLR9GYjPXvFOVr6W1zUK1WbPToaWJJuKnxBLnd0ftDEbMmj4loHYyhZyMjM91zQS4p7z8eKa9h0JrbacekcirexG0z4n3xz0QOWSvFj3jLhWXUIU21iIAwJtI3RbWa90I7rzAIqI3UElUJG7tLtUXzw4KQNETvXzqWaujEMenYlNIzLGxgB3AuJEQP09} \zxvczxbcvdfghasdfrtsdafasdfasdfdsfgsdgh L(0) u - Q \zxvczxbcvdfghasdfrtsdafasdfasdfdsfgsdgh_{L^p(Q_r)} \les \bigl( c_0 + M + \zxvczxbcvdfghasdfrtsdafasdfasdfdsfgsdgh Q \zxvczxbcvdfghasdfrtsdafasdfasdfdsfgsdgh_{L^p(Q_1)} \bigr)  r^{d-m+\eta+(k_0+1)\alpha -1 + (m+n)/p} \comma r \in (0, 1/2] . \end{equation} Defining \begin{equation} \llabel{ v9UJ djNN VBINix 1q 5 oyr SBM 2Xt gr v8RQ MaXk a4AN9i Ni n zfH xGp A57 uA E4jM fg6S 6eNGKv JL 3 tyH 3qw dPr x2 jFXW 2Wih pSSxDr aA 7 PXg jK6 GGl Og 5PkR d2n5 3eEx4N yG h d8Z RkO NMQ qL q4sE RG0C ssQkdZ Ua O vWr pla BOW rS wSG1 SM8I z9qkpd v0 C RMs GcZ LAz 4G k70e O7k6 df4uYn R6 T 5Du KOT say 0D awWQ vn2U OOPNqQ T7 H 4Hf iKY Jcl Rq M2g9 lcQZ cvCNBP 2B b tjv VYj ojr rh 78tW R886 ANdxeA SV P hK3 uPr QRs 6O SW1B wWM0 yNG9iB RI 7 opG CXk hZp Eo 2JNt kyYO pCY9HL 3o 7 Zu0 J9F Tz6 tZ GLn8 HAes o9umpy uc s 4l3 CA6 DCQ 0m 0llF Pbc8 z5Ad2l GN w8ThswELzXU3X7Ebd1KdZ7v1rN3GiirRXGKWK099ovBM0FDJCvkopYNQ2aN94Z7k0UnUKamE3OjU8DFYFFokbSI2J9V9gVlM8ALWThDPnPu3EL7HPD2VDaZTggzcCCmbvc70qqPcC9mt60ogcrTiA3HEjwTK8ymKeuJMc4q6dVz200XnYUtLR9GYjPXvFOVr6W1zUK1WbPToaWJJuKnxBLnd0ftDEbMmj4loHYyhZyMjM91zQS4p7z8eKa9h0JrbacekcirexG0z4n3xz0QOWSvFj3jLhWXUIU21iIAwJtI3RbWa90I7rzAIqI3UElUJG7tLtUXzw4KQNETvXzqWaujEMenYlNIzLGxgB3AuJEQP10} w(x) = \zxvczxbcvdfghasdfrtsdafasdfasdfdsfgsdgf_{|(y,s)| \leq 1/2} \Gamma (x-y, t-s) (L(0) u(y,s) - Q(y,s)) \, dy \, ds , \end{equation} we have $L(0) w = L(0)u - Q$. So we apply Lemma~\ref{L08} and use \eqref{8ThswELzXU3X7Ebd1KdZ7v1rN3GiirRXGKWK099ovBM0FDJCvkopYNQ2aN94Z7k0UnUKamE3OjU8DFYFFokbSI2J9V9gVlM8ALWThDPnPu3EL7HPD2VDaZTggzcCCmbvc70qqPcC9mt60ogcrTiA3HEjwTK8ymKeuJMc4q6dVz200XnYUtLR9GYjPXvFOVr6W1zUK1WbPToaWJJuKnxBLnd0ftDEbMmj4loHYyhZyMjM91zQS4p7z8eKa9h0JrbacekcirexG0z4n3xz0QOWSvFj3jLhWXUIU21iIAwJtI3RbWa90I7rzAIqI3UElUJG7tLtUXzw4KQNETvXzqWaujEMenYlNIzLGxgB3AuJEQP09} to get \begin{equation} \label{8ThswELzXU3X7Ebd1KdZ7v1rN3GiirRXGKWK099ovBM0FDJCvkopYNQ2aN94Z7k0UnUKamE3OjU8DFYFFokbSI2J9V9gVlM8ALWThDPnPu3EL7HPD2VDaZTggzcCCmbvc70qqPcC9mt60ogcrTiA3HEjwTK8ymKeuJMc4q6dVz200XnYUtLR9GYjPXvFOVr6W1zUK1WbPToaWJJuKnxBLnd0ftDEbMmj4loHYyhZyMjM91zQS4p7z8eKa9h0JrbacekcirexG0z4n3xz0QOWSvFj3jLhWXUIU21iIAwJtI3RbWa90I7rzAIqI3UElUJG7tLtUXzw4KQNETvXzqWaujEMenYlNIzLGxgB3AuJEQP11} \zxvczxbcvdfghasdfrtsdafasdfasdfdsfgsdgh w - P_w \zxvczxbcvdfghasdfrtsdafasdfasdfdsfgsdgh_{L^p(Q_r)} \les (c_0 + M + \zxvczxbcvdfghasdfrtsdafasdfasdfdsfgsdgh Q \zxvczxbcvdfghasdfrtsdafasdfasdfdsfgsdgh_{L^p(Q_1)})r^{d+ \eta + (k_0+1)\alpha-1+ (m+n)/p} , \end{equation} where $P_w$ is the $d$-degree Taylor polynomial of~$w$.  Thus $v = u - w$ satisfies $L(0)v =Q$ and  \begin{equation} \llabel{MQ qL q4sE RG0C ssQkdZ Ua O vWr pla BOW rS wSG1 SM8I z9qkpd v0 C RMs GcZ LAz 4G k70e O7k6 df4uYn R6 T 5Du KOT say 0D awWQ vn2U OOPNqQ T7 H 4Hf iKY Jcl Rq M2g9 lcQZ cvCNBP 2B b tjv VYj ojr rh 78tW R886 ANdxeA SV P hK3 uPr QRs 6O SW1B wWM0 yNG9iB RI 7 opG CXk hZp Eo 2JNt kyYO pCY9HL 3o 7 Zu0 J9F Tz6 tZ GLn8 HAes o9umpy uc s 4l3 CA6 DCQ 0m 0llF Pbc8 z5Ad2l GN w SgA XeN HTN pw dS6e 3ila 2tlbXN 7c 1 itX aDZ Fak df Jkz7 TzaO 4kbVhn YH f Tda 9C3 WCb tw MXHW xoCC c4Ws2C UH B sNL FEf jS4 SG I4I4 hqHh 2nCaQ4 nM p nzY oYE 5fD sX hCHJ zTQO cbKmvE8ThswELzXU3X7Ebd1KdZ7v1rN3GiirRXGKWK099ovBM0FDJCvkopYNQ2aN94Z7k0UnUKamE3OjU8DFYFFokbSI2J9V9gVlM8ALWThDPnPu3EL7HPD2VDaZTggzcCCmbvc70qqPcC9mt60ogcrTiA3HEjwTK8ymKeuJMc4q6dVz200XnYUtLR9GYjPXvFOVr6W1zUK1WbPToaWJJuKnxBLnd0ftDEbMmj4loHYyhZyMjM91zQS4p7z8eKa9h0JrbacekcirexG0z4n3xz0QOWSvFj3jLhWXUIU21iIAwJtI3RbWa90I7rzAIqI3UElUJG7tLtUXzw4KQNETvXzqWaujEMenYlNIzLGxgB3AuJEQP12} \zxvczxbcvdfghasdfrtsdafasdfasdfdsfgsdgh v \zxvczxbcvdfghasdfrtsdafasdfasdfdsfgsdgh_{L^p(Q_{1/2})} \les \zxvczxbcvdfghasdfrtsdafasdfasdfdsfgsdgh u \zxvczxbcvdfghasdfrtsdafasdfasdfdsfgsdgh_{L^p(B_{1/2})}  + \zxvczxbcvdfghasdfrtsdafasdfasdfdsfgsdgh w \zxvczxbcvdfghasdfrtsdafasdfasdfdsfgsdgh_{L^p(B_{1/2})} \les \zxvczxbcvdfghasdfrtsdafasdfasdfdsfgsdgh u \zxvczxbcvdfghasdfrtsdafasdfasdfdsfgsdgh_{W^{m,1}_p(B_{1/2})} + \zxvczxbcvdfghasdfrtsdafasdfasdfdsfgsdgh Q \zxvczxbcvdfghasdfrtsdafasdfasdfdsfgsdgh_{L^p(Q_1)} \les c_0 + M + \zxvczxbcvdfghasdfrtsdafasdfasdfdsfgsdgh Q \zxvczxbcvdfghasdfrtsdafasdfasdfdsfgsdgh_{L^p(Q_1)} , \end{equation} where we used \eqref{8ThswELzXU3X7Ebd1KdZ7v1rN3GiirRXGKWK099ovBM0FDJCvkopYNQ2aN94Z7k0UnUKamE3OjU8DFYFFokbSI2J9V9gVlM8ALWThDPnPu3EL7HPD2VDaZTggzcCCmbvc70qqPcC9mt60ogcrTiA3HEjwTK8ymKeuJMc4q6dVz200XnYUtLR9GYjPXvFOVr6W1zUK1WbPToaWJJuKnxBLnd0ftDEbMmj4loHYyhZyMjM91zQS4p7z8eKa9h0JrbacekcirexG0z4n3xz0QOWSvFj3jLhWXUIU21iIAwJtI3RbWa90I7rzAIqI3UElUJG7tLtUXzw4KQNETvXzqWaujEMenYlNIzLGxgB3AuJEQP08} in the last inequality. We now let $P_v$ be the Taylor polynomial of $v$ with the degree~$d$. Then $L(0) P_v = Q$ and that  \begin{equation} \label{8ThswELzXU3X7Ebd1KdZ7v1rN3GiirRXGKWK099ovBM0FDJCvkopYNQ2aN94Z7k0UnUKamE3OjU8DFYFFokbSI2J9V9gVlM8ALWThDPnPu3EL7HPD2VDaZTggzcCCmbvc70qqPcC9mt60ogcrTiA3HEjwTK8ymKeuJMc4q6dVz200XnYUtLR9GYjPXvFOVr6W1zUK1WbPToaWJJuKnxBLnd0ftDEbMmj4loHYyhZyMjM91zQS4p7z8eKa9h0JrbacekcirexG0z4n3xz0QOWSvFj3jLhWXUIU21iIAwJtI3RbWa90I7rzAIqI3UElUJG7tLtUXzw4KQNETvXzqWaujEMenYlNIzLGxgB3AuJEQP13} \zxvczxbcvdfghasdfrtsdafasdfasdfdsfgsdgh v - P_v \zxvczxbcvdfghasdfrtsdafasdfasdfdsfgsdgh_{L^p(Q_r)} \les (c_0 + M + \zxvczxbcvdfghasdfrtsdafasdfasdfdsfgsdgh Q \zxvczxbcvdfghasdfrtsdafasdfasdfdsfgsdgh_{L^p(Q_1)}
) r^{d+1+(m+n)/p} . \end{equation} Therefore, we let $\bar P = P_v + P_w$ and combine \eqref{8ThswELzXU3X7Ebd1KdZ7v1rN3GiirRXGKWK099ovBM0FDJCvkopYNQ2aN94Z7k0UnUKamE3OjU8DFYFFokbSI2J9V9gVlM8ALWThDPnPu3EL7HPD2VDaZTggzcCCmbvc70qqPcC9mt60ogcrTiA3HEjwTK8ymKeuJMc4q6dVz200XnYUtLR9GYjPXvFOVr6W1zUK1WbPToaWJJuKnxBLnd0ftDEbMmj4loHYyhZyMjM91zQS4p7z8eKa9h0JrbacekcirexG0z4n3xz0QOWSvFj3jLhWXUIU21iIAwJtI3RbWa90I7rzAIqI3UElUJG7tLtUXzw4KQNETvXzqWaujEMenYlNIzLGxgB3AuJEQP11} with \eqref{8ThswELzXU3X7Ebd1KdZ7v1rN3GiirRXGKWK099ovBM0FDJCvkopYNQ2aN94Z7k0UnUKamE3OjU8DFYFFokbSI2J9V9gVlM8ALWThDPnPu3EL7HPD2VDaZTggzcCCmbvc70qqPcC9mt60ogcrTiA3HEjwTK8ymKeuJMc4q6dVz200XnYUtLR9GYjPXvFOVr6W1zUK1WbPToaWJJuKnxBLnd0ftDEbMmj4loHYyhZyMjM91zQS4p7z8eKa9h0JrbacekcirexG0z4n3xz0QOWSvFj3jLhWXUIU21iIAwJtI3RbWa90I7rzAIqI3UElUJG7tLtUXzw4KQNETvXzqWaujEMenYlNIzLGxgB3AuJEQP13} to obtain \begin{align} \begin{split} \llabel{VYj ojr rh 78tW R886 ANdxeA SV P hK3 uPr QRs 6O SW1B wWM0 yNG9iB RI 7 opG CXk hZp Eo 2JNt kyYO pCY9HL 3o 7 Zu0 J9F Tz6 tZ GLn8 HAes o9umpy uc s 4l3 CA6 DCQ 0m 0llF Pbc8 z5Ad2l GN w SgA XeN HTN pw dS6e 3ila 2tlbXN 7c 1 itX aDZ Fak df Jkz7 TzaO 4kbVhn YH f Tda 9C3 WCb tw MXHW xoCC c4Ws2C UH B sNL FEf jS4 SG I4I4 hqHh 2nCaQ4 nM p nzY oYE 5fD sX hCHJ zTQO cbKmvE pl W Und VUo rrq iJ zRqT dIWS QBL96D FU d 64k 5gv Qh0 dj rGlw 795x V6KzhT l5 Y FtC rpy bHH 86 h3qn Lyzy ycGoqm Cb f h9h prB CQp Fe CxhU Z2oJ F3aKgQ H8 R yIm F9t Eks gP FMMJ TAIy z8ThswELzXU3X7Ebd1KdZ7v1rN3GiirRXGKWK099ovBM0FDJCvkopYNQ2aN94Z7k0UnUKamE3OjU8DFYFFokbSI2J9V9gVlM8ALWThDPnPu3EL7HPD2VDaZTggzcCCmbvc70qqPcC9mt60ogcrTiA3HEjwTK8ymKeuJMc4q6dVz200XnYUtLR9GYjPXvFOVr6W1zUK1WbPToaWJJuKnxBLnd0ftDEbMmj4loHYyhZyMjM91zQS4p7z8eKa9h0JrbacekcirexG0z4n3xz0QOWSvFj3jLhWXUIU21iIAwJtI3RbWa90I7rzAIqI3UElUJG7tLtUXzw4KQNETvXzqWaujEMenYlNIzLGxgB3AuJEQP14} \zxvczxbcvdfghasdfrtsdafasdfasdfdsfgsdgh u - \bar P \zxvczxbcvdfghasdfrtsdafasdfasdfdsfgsdgh_{L^p(Q_r)}  &\les (c_0 + M + \zxvczxbcvdfghasdfrtsdafasdfasdfdsfgsdgh Q \zxvczxbcvdfghasdfrtsdafasdfasdfdsfgsdgh_{L^p(Q_1)}) r^{d + \min(\eta + (k_0+1)\alpha-1, 1) + (m+n)/p} \\& \les (c_0 + M + \zxvczxbcvdfghasdfrtsdafasdfasdfdsfgsdgh Q \zxvczxbcvdfghasdfrtsdafasdfasdfdsfgsdgh_{L^p(Q_1)}) r^{d + \eta + (k_0+1)\alpha-1 + (m+n)/p} . \end{split} \end{align} Let $\tilde L = \partial_t - \sum_{|\nu|=m} a_{\nu} \partial^{\nu}$.  Denoting $\phi =\sum_{|\nu|<m} a_{\nu} \partial^{\nu} u + (f-Q) $, we get by \eqref{8ThswELzXU3X7Ebd1KdZ7v1rN3GiirRXGKWK099ovBM0FDJCvkopYNQ2aN94Z7k0UnUKamE3OjU8DFYFFokbSI2J9V9gVlM8ALWThDPnPu3EL7HPD2VDaZTggzcCCmbvc70qqPcC9mt60ogcrTiA3HEjwTK8ymKeuJMc4q6dVz200XnYUtLR9GYjPXvFOVr6W1zUK1WbPToaWJJuKnxBLnd0ftDEbMmj4loHYyhZyMjM91zQS4p7z8eKa9h0JrbacekcirexG0z4n3xz0QOWSvFj3jLhWXUIU21iIAwJtI3RbWa90I7rzAIqI3UElUJG7tLtUXzw4KQNETvXzqWaujEMenYlNIzLGxgB3AuJEQP08} that     \begin{align}     \begin{split}       \zxvczxbcvdfghasdfrtsdafasdfasdfdsfgsdgh \phi \zxvczxbcvdfghasdfrtsdafasdfasdfdsfgsdgh_{L^p(Q_r)}       &\les r^{1-m} \zxvczxbcvdfghasdfrtsdafasdfasdfdsfgsdgh u \zxvczxbcvdfghasdfrtsdafasdfasdfdsfgsdgh_{L^p(Q_{2r})}              + r \zxvczxbcvdfghasdfrtsdafasdfasdfdsfgsdgh f \zxvczxbcvdfghasdfrtsdafasdfasdfdsfgsdgh_{L^p(Q_r)}             + \zxvczxbcvdfghasdfrtsdafasdfasdfdsfgsdgh f - Q \zxvczxbcvdfghasdfrtsdafasdfasdfdsfgsdgh_{L^p(Q_{2r})}       \\&       \les (c_0 + \MM + \zxvczxbcvdfghasdfrtsdafasdfasdfdsfgsdgh Q \zxvczxbcvdfghasdfrtsdafasdfasdfdsfgsdgh_{L^p(Q_1)})              r^{d-m+\min( \eta + k_0 \alpha, \alpha) + (m+n)/p}        \comma        r \in (0, 1/2]        .     \end{split}     \label{8ThswELzXU3X7Ebd1KdZ7v1rN3GiirRXGKWK099ovBM0FDJCvkopYNQ2aN94Z7k0UnUKamE3OjU8DFYFFokbSI2J9V9gVlM8ALWThDPnPu3EL7HPD2VDaZTggzcCCmbvc70qqPcC9mt60ogcrTiA3HEjwTK8ymKeuJMc4q6dVz200XnYUtLR9GYjPXvFOVr6W1zUK1WbPToaWJJuKnxBLnd0ftDEbMmj4loHYyhZyMjM91zQS4p7z8eKa9h0JrbacekcirexG0z4n3xz0QOWSvFj3jLhWXUIU21iIAwJtI3RbWa90I7rzAIqI3UElUJG7tLtUXzw4KQNETvXzqWaujEMenYlNIzLGxgB3AuJEQ206}     \end{align} We rewrite the equation $Lu = f$ as     \begin{equation}       \tilde L (u-\bar P)       = (L(0) - \tilde L) \bar P + \phi       = \sum_{|\nu|=m} (a_{\nu} - a_{\nu}(0,0)) \partial^{\nu} \bar P + \phi       .     \label{8ThswELzXU3X7Ebd1KdZ7v1rN3GiirRXGKWK099ovBM0FDJCvkopYNQ2aN94Z7k0UnUKamE3OjU8DFYFFokbSI2J9V9gVlM8ALWThDPnPu3EL7HPD2VDaZTggzcCCmbvc70qqPcC9mt60ogcrTiA3HEjwTK8ymKeuJMc4q6dVz200XnYUtLR9GYjPXvFOVr6W1zUK1WbPToaWJJuKnxBLnd0ftDEbMmj4loHYyhZyMjM91zQS4p7z8eKa9h0JrbacekcirexG0z4n3xz0QOWSvFj3jLhWXUIU21iIAwJtI3RbWa90I7rzAIqI3UElUJG7tLtUXzw4KQNETvXzqWaujEMenYlNIzLGxgB3AuJEQ207}     \end{equation} Applying H\"older's inequality yields     \begin{equation}       \sum_{|\nu|=m} \zxvczxbcvdfghasdfrtsdafasdfasdfdsfgsdgh (a_{\nu} - a_{\nu}(0,0)) \partial^{\nu} \bar P \zxvczxbcvdfghasdfrtsdafasdfasdfdsfgsdgh_{L^p(Q_r)}       \les \sum_{|\nu|=m} \zxvczxbcvdfghasdfrtsdafasdfasdfdsfgsdgh a_{\nu} - a_{\nu}(0,0) \zxvczxbcvdfghasdfrtsdafasdfasdfdsfgsdgh_{L^{\infty}(Q_r)}                                        \zxvczxbcvdfghasdfrtsdafasdfasdfdsfgsdgh \partial^{\nu} \bar P \zxvczxbcvdfghasdfrtsdafasdfasdfdsfgsdgh_{L^p(Q_r)}       \les \zxvczxbcvdfghasdfrtsdafasdfasdfdsfgsdgh \bar P \zxvczxbcvdfghasdfrtsdafasdfasdfdsfgsdgh_{L^p(Q_1) }r^{d-m+\alpha+(m+n)/p}       ,     \label{8ThswELzXU3X7Ebd1KdZ7v1rN3GiirRXGKWK099ovBM0FDJCvkopYNQ2aN94Z7k0UnUKamE3OjU8DFYFFokbSI2J9V9gVlM8ALWThDPnPu3EL7HPD2VDaZTggzcCCmbvc70qqPcC9mt60ogcrTiA3HEjwTK8ymKeuJMc4q6dVz200XnYUtLR9GYjPXvFOVr6W1zUK1WbPToaWJJuKnxBLnd0ftDEbMmj4loHYyhZyMjM91zQS4p7z8eKa9h0JrbacekcirexG0z4n3xz0QOWSvFj3jLhWXUIU21iIAwJtI3RbWa90I7rzAIqI3UElUJG7tLtUXzw4KQNETvXzqWaujEMenYlNIzLGxgB3AuJEQ208}     \end{equation} for all~$r \leq 1$. By combining \eqref{8ThswELzXU3X7Ebd1KdZ7v1rN3GiirRXGKWK099ovBM0FDJCvkopYNQ2aN94Z7k0UnUKamE3OjU8DFYFFokbSI2J9V9gVlM8ALWThDPnPu3EL7HPD2VDaZTggzcCCmbvc70qqPcC9mt60ogcrTiA3HEjwTK8ymKeuJMc4q6dVz200XnYUtLR9GYjPXvFOVr6W1zUK1WbPToaWJJuKnxBLnd0ftDEbMmj4loHYyhZyMjM91zQS4p7z8eKa9h0JrbacekcirexG0z4n3xz0QOWSvFj3jLhWXUIU21iIAwJtI3RbWa90I7rzAIqI3UElUJG7tLtUXzw4KQNETvXzqWaujEMenYlNIzLGxgB3AuJEQ206} and~\eqref{8ThswELzXU3X7Ebd1KdZ7v1rN3GiirRXGKWK099ovBM0FDJCvkopYNQ2aN94Z7k0UnUKamE3OjU8DFYFFokbSI2J9V9gVlM8ALWThDPnPu3EL7HPD2VDaZTggzcCCmbvc70qqPcC9mt60ogcrTiA3HEjwTK8ymKeuJMc4q6dVz200XnYUtLR9GYjPXvFOVr6W1zUK1WbPToaWJJuKnxBLnd0ftDEbMmj4loHYyhZyMjM91zQS4p7z8eKa9h0JrbacekcirexG0z4n3xz0QOWSvFj3jLhWXUIU21iIAwJtI3RbWa90I7rzAIqI3UElUJG7tLtUXzw4KQNETvXzqWaujEMenYlNIzLGxgB3AuJEQ208}, we have the estimate for the right hand side of \eqref{8ThswELzXU3X7Ebd1KdZ7v1rN3GiirRXGKWK099ovBM0FDJCvkopYNQ2aN94Z7k0UnUKamE3OjU8DFYFFokbSI2J9V9gVlM8ALWThDPnPu3EL7HPD2VDaZTggzcCCmbvc70qqPcC9mt60ogcrTiA3HEjwTK8ymKeuJMc4q6dVz200XnYUtLR9GYjPXvFOVr6W1zUK1WbPToaWJJuKnxBLnd0ftDEbMmj4loHYyhZyMjM91zQS4p7z8eKa9h0JrbacekcirexG0z4n3xz0QOWSvFj3jLhWXUIU21iIAwJtI3RbWa90I7rzAIqI3UElUJG7tLtUXzw4KQNETvXzqWaujEMenYlNIzLGxgB3AuJEQ207} which reads \begin{equation}       \biggl \zxvczxbcvdfghasdfrtsdafasdfasdfdsfgsdgh \sum_{|\nu| = m} (a_{\nu} - a_{\nu}(0,0)) \partial^{\nu}\bar P + \phi \biggr \zxvczxbcvdfghasdfrtsdafasdfasdfdsfgsdgh_{L^p(Q_r)}       \les (c_0               + \MM                + \zxvczxbcvdfghasdfrtsdafasdfasdfdsfgsdgh Q \zxvczxbcvdfghasdfrtsdafasdfasdfdsfgsdgh_{L^p(Q_1)}              + \zxvczxbcvdfghasdfrtsdafasdfasdfdsfgsdgh \bar P \zxvczxbcvdfghasdfrtsdafasdfasdfdsfgsdgh_{L^p(Q_1)}              ) r^{d-m+\min( \eta + k_0 \alpha,\alpha) + (m+n)/p}       ,     \label{8ThswELzXU3X7Ebd1KdZ7v1rN3GiirRXGKWK099ovBM0FDJCvkopYNQ2aN94Z7k0UnUKamE3OjU8DFYFFokbSI2J9V9gVlM8ALWThDPnPu3EL7HPD2VDaZTggzcCCmbvc70qqPcC9mt60ogcrTiA3HEjwTK8ymKeuJMc4q6dVz200XnYUtLR9GYjPXvFOVr6W1zUK1WbPToaWJJuKnxBLnd0ftDEbMmj4loHYyhZyMjM91zQS4p7z8eKa9h0JrbacekcirexG0z4n3xz0QOWSvFj3jLhWXUIU21iIAwJtI3RbWa90I7rzAIqI3UElUJG7tLtUXzw4KQNETvXzqWaujEMenYlNIzLGxgB3AuJEQ209}     \end{equation} for all~$r \in (0, 1/2]$. We now rewrite the equation as      \begin{align}     \begin{split}       L(0) (u-\bar P)       &= (L(0) - \tilde L) ( u - \bar P)           + \sum_{|\nu| = m} (a_{\nu} - a_{\nu}(0,0)) \partial^{\nu}\bar P + \phi       \\&       = \sum_{|\nu|=m} (a_{\nu} - a_{\nu}(0,0)) \partial^{\nu} (u-\bar P)          +\sum_{|\nu| = m} (a_{\nu} - a_{\nu}(0,0)) \partial^{\nu}\bar P + \phi          .     \label{8ThswELzXU3X7Ebd1KdZ7v1rN3GiirRXGKWK099ovBM0FDJCvkopYNQ2aN94Z7k0UnUKamE3OjU8DFYFFokbSI2J9V9gVlM8ALWThDPnPu3EL7HPD2VDaZTggzcCCmbvc70qqPcC9mt60ogcrTiA3HEjwTK8ymKeuJMc4q6dVz200XnYUtLR9GYjPXvFOVr6W1zUK1WbPToaWJJuKnxBLnd0ftDEbMmj4loHYyhZyMjM91zQS4p7z8eKa9h0JrbacekcirexG0z4n3xz0QOWSvFj3jLhWXUIU21iIAwJtI3RbWa90I7rzAIqI3UElUJG7tLtUXzw4KQNETvXzqWaujEMenYlNIzLGxgB3AuJEQ211}     \end{split}     \end{align} The first term on the right hand side of \eqref{8ThswELzXU3X7Ebd1KdZ7v1rN3GiirRXGKWK099ovBM0FDJCvkopYNQ2aN94Z7k0UnUKamE3OjU8DFYFFokbSI2J9V9gVlM8ALWThDPnPu3EL7HPD2VDaZTggzcCCmbvc70qqPcC9mt60ogcrTiA3HEjwTK8ymKeuJMc4q6dVz200XnYUtLR9GYjPXvFOVr6W1zUK1WbPToaWJJuKnxBLnd0ftDEbMmj4loHYyhZyMjM91zQS4p7z8eKa9h0JrbacekcirexG0z4n3xz0QOWSvFj3jLhWXUIU21iIAwJtI3RbWa90I7rzAIqI3UElUJG7tLtUXzw4KQNETvXzqWaujEMenYlNIzLGxgB3AuJEQ211} is estimated as     \begin{align}     \begin{split}       &\sum_{|\nu|=m} \zxvczxbcvdfghasdfrtsdafasdfasdfdsfgsdgh (a_{\nu} - a_{\nu}(0,0)) \partial^{\nu} (u-\bar P) \zxvczxbcvdfghasdfrtsdafasdfasdfdsfgsdgh_{L^p(Q_r)}       \\& \indeq \indeq \indeq       \les \sum_{|\nu|=m} \zxvczxbcvdfghasdfrtsdafasdfasdfdsfgsdgh a_{\nu} - a_{\nu}(0,0)\zxvczxbcvdfghasdfrtsdafasdfasdfdsfgsdgh_{L^{\infty}(Q_r)}               \zxvczxbcvdfghasdfrtsdafasdfasdfdsfgsdgh \partial^{\nu} (u-\bar P) \zxvczxbcvdfghasdfrtsdafasdfasdfdsfgsdgh_{L^p(Q_r)}        \\& \indeq \indeq \indeq       \les (c_0              + \MM               + \zxvczxbcvdfghasdfrtsdafasdfasdfdsfgsdgh Q \zxvczxbcvdfghasdfrtsdafasdfasdfdsfgsdgh_{L^p(Q_1)}              + \zxvczxbcvdfghasdfrtsdafasdfasdfdsfgsdgh \bar P \zxvczxbcvdfghasdfrtsdafasdfasdfdsfgsdgh_{L^p(Q_1)}              ) r^{d-m-1+ \eta+\alpha+(m+n)/p}               .     \end{split}     \label{8ThswELzXU3X7Ebd1KdZ7v1rN3GiirRXGKWK099ovBM0FDJCvkopYNQ2aN94Z7k0UnUKamE3OjU8DFYFFokbSI2J9V9gVlM8ALWThDPnPu3EL7HPD2VDaZTggzcCCmbvc70qqPcC9mt60ogcrTiA3HEjwTK8ymKeuJMc4q6dVz200XnYUtLR9GYjPXvFOVr6W1zUK1WbPToaWJJuKnxBLnd0ftDEbMmj4loHYyhZyMjM91zQS4p7z8eKa9h0JrbacekcirexG0z4n3xz0QOWSvFj3jLhWXUIU21iIAwJtI3RbWa90I7rzAIqI3UElUJG7tLtUXzw4KQNETvXzqWaujEMenYlNIzLGxgB3AuJEQ212}     \end{align} Let $\tilde \phi$ be the right hand side of~\eqref{8ThswELzXU3X7Ebd1KdZ7v1rN3GiirRXGKWK099ovBM0FDJCvkopYNQ2aN94Z7k0UnUKamE3OjU8DFYFFokbSI2J9V9gVlM8ALWThDPnPu3EL7HPD2VDaZTggzcCCmbvc70qqPcC9mt60ogcrTiA3HEjwTK8ymKeuJMc4q6dVz200XnYUtLR9GYjPXvFOVr6W1zUK1WbPToaWJJuKnxBLnd0ftDEbMmj4loHYyhZyMjM91zQS4p7z8eKa9h0JrbacekcirexG0z4n3xz0QOWSvFj3jLhWXUIU21iIAwJtI3RbWa90I7rzAIqI3UElUJG7tLtUXzw4KQNETvXzqWaujEMenYlNIzLGxgB3AuJEQ211}. By combining \eqref{8ThswELzXU3X7Ebd1KdZ7v1rN3GiirRXGKWK099ovBM0FDJCvkopYNQ2aN94Z7k0UnUKamE3OjU8DFYFFokbSI2J9V9gVlM8ALWThDPnPu3EL7HPD2VDaZTggzcCCmbvc70qqPcC9mt60ogcrTiA3HEjwTK8ymKeuJMc4q6dVz200XnYUtLR9GYjPXvFOVr6W1zUK1WbPToaWJJuKnxBLnd0ftDEbMmj4loHYyhZyMjM91zQS4p7z8eKa9h0JrbacekcirexG0z4n3xz0QOWSvFj3jLhWXUIU21iIAwJtI3RbWa90I7rzAIqI3UElUJG7tLtUXzw4KQNETvXzqWaujEMenYlNIzLGxgB3AuJEQ209} and \eqref{8ThswELzXU3X7Ebd1KdZ7v1rN3GiirRXGKWK099ovBM0FDJCvkopYNQ2aN94Z7k0UnUKamE3OjU8DFYFFokbSI2J9V9gVlM8ALWThDPnPu3EL7HPD2VDaZTggzcCCmbvc70qqPcC9mt60ogcrTiA3HEjwTK8ymKeuJMc4q6dVz200XnYUtLR9GYjPXvFOVr6W1zUK1WbPToaWJJuKnxBLnd0ftDEbMmj4loHYyhZyMjM91zQS4p7z8eKa9h0JrbacekcirexG0z4n3xz0QOWSvFj3jLhWXUIU21iIAwJtI3RbWa90I7rzAIqI3UElUJG7tLtUXzw4KQNETvXzqWaujEMenYlNIzLGxgB3AuJEQ212}, we obtain \begin{equation} \zxvczxbcvdfghasdfrtsdafasdfasdfdsfgsdgh \tilde \phi \zxvczxbcvdfghasdfrtsdafasdfasdfdsfgsdgh_{L^p(Q_r)}       \les (c_0              + \MM               + \zxvczxbcvdfghasdfrtsdafasdfasdfdsfgsdgh Q \zxvczxbcvdfghasdfrtsdafasdfasdfdsfgsdgh_{L^p(Q_1)}              + \zxvczxbcvdfghasdfrtsdafasdfasdfdsfgsdgh \bar P \zxvczxbcvdfghasdfrtsdafasdfasdfdsfgsdgh_{L^p(Q_1)}              ) r^{d-m-1+ \eta+\alpha+(m+n)/p}     \llabel{ SgA XeN HTN pw dS6e 3ila 2tlbXN 7c 1 itX aDZ Fak df Jkz7 TzaO 4kbVhn YH f Tda 9C3 WCb tw MXHW xoCC c4Ws2C UH B sNL FEf jS4 SG I4I4 hqHh 2nCaQ4 nM p nzY oYE 5fD sX hCHJ zTQO cbKmvE pl W Und VUo rrq iJ zRqT dIWS QBL96D FU d 64k 5gv Qh0 dj rGlw 795x V6KzhT l5 Y FtC rpy bHH 86 h3qn Lyzy ycGoqm Cb f h9h prB CQp Fe CxhU Z2oJ F3aKgQ H8 R yIm F9t Eks gP FMMJ TAIy z3ohWj Hx M R86 KJO NKT c3 uyRN nSKH lhb11Q 9C w rf8 iiX qyY L4 zh9s 8NTE ve539G zL g vhD N7F eXo 5k AWAT 6Vrw htDQwy tu H Oa5 UIO Exb Mp V2AH puuC HWItfO ru x YfF qsa P8u fH F16C E8ThswELzXU3X7Ebd1KdZ7v1rN3GiirRXGKWK099ovBM0FDJCvkopYNQ2aN94Z7k0UnUKamE3OjU8DFYFFokbSI2J9V9gVlM8ALWThDPnPu3EL7HPD2VDaZTggzcCCmbvc70qqPcC9mt60ogcrTiA3HEjwTK8ymKeuJMc4q6dVz200XnYUtLR9GYjPXvFOVr6W1zUK1WbPToaWJJuKnxBLnd0ftDEbMmj4loHYyhZyMjM91zQS4p7z8eKa9h0JrbacekcirexG0z4n3xz0QOWSvFj3jLhWXUIU21iIAwJtI3RbWa90I7rzAIqI3UElUJG7tLtUXzw4KQNETvXzqWaujEMenYlNIzLGxgB3AuJEQ82} \end{equation} and $L(0)(u-\bar P) = \tilde \phi$. Thus there exists a polynomial $P_0$ of degree less than $d$ such that    \begin{align}   \begin{split}       &       \zxvczxbcvdfghasdfrtsdafasdfasdfdsfgsdgh u(x,t) - \bar P(x,t) - P_0(x,t) \zxvczxbcvdfghasdfrtsdafasdfasdfdsfgsdgh_{L^p(Q_r)}       \\&\indeq       \les ( c_0                + \MM                + \zxvczxbcvdfghasdfrtsdafasdfasdfdsfgsdgh Q \zxvczxbcvdfghasdfrtsdafasdfasdfdsfgsdgh_{L^p(Q_1)}               + \zxvczxbcvdfghasdfrtsdafasdfasdfdsfgsdgh u \zxvczxbcvdfghasdfrtsdafasdfasdfdsfgsdgh_{L^p(Q_1)}               + \zxvczxbcvdfghasdfrtsdafasdfasdfdsfgsdgh \bar P \zxvczxbcvdfghasdfrtsdafasdfasdfdsfgsdgh_{L^p(Q_1)}              ) r^{d-1+ \eta + (k_0+1)\alpha + (m+n)/q}     ,     \end{split}     \llabel{ pl W Und VUo rrq iJ zRqT dIWS QBL96D FU d 64k 5gv Qh0 dj rGlw 795x V6KzhT l5 Y FtC rpy bHH 86 h3qn Lyzy ycGoqm Cb f h9h prB CQp Fe CxhU Z2oJ F3aKgQ H8 R yIm F9t Eks gP FMMJ TAIy z3ohWj Hx M R86 KJO NKT c3 uyRN nSKH lhb11Q 9C w rf8 iiX qyY L4 zh9s 8NTE ve539G zL g vhD N7F eXo 5k AWAT 6Vrw htDQwy tu H Oa5 UIO Exb Mp V2AH puuC HWItfO ru x YfF qsa P8u fH F16C EBXK tj6ohs uv T 8BB PDN gGf KQ g6MB K2x9 jqRbHm jI U EKB Im0 bbK ac wqIX ijrF uq9906 Vy m 3Ve 1gB dMy 9i hnbA 3gBo 5aBKK5 gf J SmN eCW wOM t9 xutz wDkX IY7nNh Wd D ppZ UOq 2Ae 0a W8ThswELzXU3X7Ebd1KdZ7v1rN3GiirRXGKWK099ovBM0FDJCvkopYNQ2aN94Z7k0UnUKamE3OjU8DFYFFokbSI2J9V9gVlM8ALWThDPnPu3EL7HPD2VDaZTggzcCCmbvc70qqPcC9mt60ogcrTiA3HEjwTK8ymKeuJMc4q6dVz200XnYUtLR9GYjPXvFOVr6W1zUK1WbPToaWJJuKnxBLnd0ftDEbMmj4loHYyhZyMjM91zQS4p7z8eKa9h0JrbacekcirexG0z4n3xz0QOWSvFj3jLhWXUIU21iIAwJtI3RbWa90I7rzAIqI3UElUJG7tLtUXzw4KQNETvXzqWaujEMenYlNIzLGxgB3AuJEQ213}     \end{align} and then    \begin{align}   \begin{split}     &       \zxvczxbcvdfghasdfrtsdafasdfasdfdsfgsdgh u(x,t)  - P_0(x,t) \zxvczxbcvdfghasdfrtsdafasdfasdfdsfgsdgh_{L^p(Q_r)}     \\&\indeq       \les ( c_0                + \MM                + \zxvczxbcvdfghasdfrtsdafasdfasdfdsfgsdgh Q \zxvczxbcvdfghasdfrtsdafasdfasdfdsfgsdgh_{L^p(Q_1)}               + \zxvczxbcvdfghasdfrtsdafasdfasdfdsfgsdgh u \zxvczxbcvdfghasdfrtsdafasdfasdfdsfgsdgh_{L^p(Q_1)}               + \zxvczxbcvdfghasdfrtsdafasdfasdfdsfgsdgh \bar P \zxvczxbcvdfghasdfrtsdafasdfasdfdsfgsdgh_{L^p(Q_1)}              ) r^{d-1+ \eta + (k_0+1)\alpha + (m+n)/p}       \comma       (x,t) \in Q_R       .     \end{split}     \llabel{3ohWj Hx M R86 KJO NKT c3 uyRN nSKH lhb11Q 9C w rf8 iiX qyY L4 zh9s 8NTE ve539G zL g vhD N7F eXo 5k AWAT 6Vrw htDQwy tu H Oa5 UIO Exb Mp V2AH puuC HWItfO ru x YfF qsa P8u fH F16C EBXK tj6ohs uv T 8BB PDN gGf KQ g6MB K2x9 jqRbHm jI U EKB Im0 bbK ac wqIX ijrF uq9906 Vy m 3Ve 1gB dMy 9i hnbA 3gBo 5aBKK5 gf J SmN eCW wOM t9 xutz wDkX IY7nNh Wd D ppZ UOq 2Ae 0a W7A6 XoIc TSLNDZ yf 2 XjB cUw eQT Zt cuXI DYsD hdAu3V MB B BKW IcF NWQ dO u3Fb c6F8 VN77Da IH E 3MZ luL YvB mN Z2wE auXX DGpeKR nw o UVB 2oM VVe hW 0ejG gbgz Iw9FwQ hN Y rFI 4pT lqr8ThswELzXU3X7Ebd1KdZ7v1rN3GiirRXGKWK099ovBM0FDJCvkopYNQ2aN94Z7k0UnUKamE3OjU8DFYFFokbSI2J9V9gVlM8ALWThDPnPu3EL7HPD2VDaZTggzcCCmbvc70qqPcC9mt60ogcrTiA3HEjwTK8ymKeuJMc4q6dVz200XnYUtLR9GYjPXvFOVr6W1zUK1WbPToaWJJuKnxBLnd0ftDEbMmj4loHYyhZyMjM91zQS4p7z8eKa9h0JrbacekcirexG0z4n3xz0QOWSvFj3jLhWXUIU21iIAwJtI3RbWa90I7rzAIqI3UElUJG7tLtUXzw4KQNETvXzqWaujEMenYlNIzLGxgB3AuJEQ214}     \end{align} Arguing similarly as in Theorem~\ref{T01}, we have $P_0 \equiv 0$ by \eqref{8ThswELzXU3X7Ebd1KdZ7v1rN3GiirRXGKWK099ovBM0FDJCvkopYNQ2aN94Z7k0UnUKamE3OjU8DFYFFokbSI2J9V9gVlM8ALWThDPnPu3EL7HPD2VDaZTggzcCCmbvc70qqPcC9mt60ogcrTiA3HEjwTK8ymKeuJMc4q6dVz200XnYUtLR9GYjPXvFOVr6W1zUK1WbPToaWJJuKnxBLnd0ftDEbMmj4loHYyhZyMjM91zQS4p7z8eKa9h0JrbacekcirexG0z4n3xz0QOWSvFj3jLhWXUIU21iIAwJtI3RbWa90I7rzAIqI3UElUJG7tLtUXzw4KQNETvXzqWaujEMenYlNIzLGxgB3AuJEQ155} and therefore,     \begin{equation}       \sup_{r\leq1} \frac{\zxvczxbcvdfghasdfrtsdafasdfasdfdsfgsdgh u \zxvczxbcvdfghasdfrtsdafasdfasdfdsfgsdgh_{L^p(Q_r)}}{r^{d+(m+n)/p}}       < \infty       ,     \llabel{BXK tj6ohs uv T 8BB PDN gGf KQ g6MB K2x9 jqRbHm jI U EKB Im0 bbK ac wqIX ijrF uq9906 Vy m 3Ve 1gB dMy 9i hnbA 3gBo 5aBKK5 gf J SmN eCW wOM t9 xutz wDkX IY7nNh Wd D ppZ UOq 2Ae 0a W7A6 XoIc TSLNDZ yf 2 XjB cUw eQT Zt cuXI DYsD hdAu3V MB B BKW IcF NWQ dO u3Fb c6F8 VN77Da IH E 3MZ luL YvB mN Z2wE auXX DGpeKR nw o UVB 2oM VVe hW 0ejG gbgz Iw9FwQ hN Y rFI 4pT lqr Wn Xzz2 qBba lv3snl 2j a vzU Snc pwh cG J0Di 3Lr3 rs6F23 6o b LtD vN9 KqA pO uold 3sec xqgSQN ZN f w5t BGX Pdv W0 k6G4 Byh9 V3IicO nR 2 obf x3j rwt 37 u82f wxwj SmOQq0 pq 4 qfv rN8ThswELzXU3X7Ebd1KdZ7v1rN3GiirRXGKWK099ovBM0FDJCvkopYNQ2aN94Z7k0UnUKamE3OjU8DFYFFokbSI2J9V9gVlM8ALWThDPnPu3EL7HPD2VDaZTggzcCCmbvc70qqPcC9mt60ogcrTiA3HEjwTK8ymKeuJMc4q6dVz200XnYUtLR9GYjPXvFOVr6W1zUK1WbPToaWJJuKnxBLnd0ftDEbMmj4loHYyhZyMjM91zQS4p7z8eKa9h0JrbacekcirexG0z4n3xz0QOWSvFj3jLhWXUIU21iIAwJtI3RbWa90I7rzAIqI3UElUJG7tLtUXzw4KQNETvXzqWaujEMenYlNIzLGxgB3AuJEQ215}     \end{equation} since $1\leq (k_0+1)\alpha + \eta$. Then, instead of \eqref{8ThswELzXU3X7Ebd1KdZ7v1rN3GiirRXGKWK099ovBM0FDJCvkopYNQ2aN94Z7k0UnUKamE3OjU8DFYFFokbSI2J9V9gVlM8ALWThDPnPu3EL7HPD2VDaZTggzcCCmbvc70qqPcC9mt60ogcrTiA3HEjwTK8ymKeuJMc4q6dVz200XnYUtLR9GYjPXvFOVr6W1zUK1WbPToaWJJuKnxBLnd0ftDEbMmj4loHYyhZyMjM91zQS4p7z8eKa9h0JrbacekcirexG0z4n3xz0QOWSvFj3jLhWXUIU21iIAwJtI3RbWa90I7rzAIqI3UElUJG7tLtUXzw4KQNETvXzqWaujEMenYlNIzLGxgB3AuJEQ212}, we have   \begin{align}   \begin{split}
    &       \sum_{|\nu|=m} \zxvczxbcvdfghasdfrtsdafasdfasdfdsfgsdgh (a_{\nu} - a_{\nu}(0,0) ) \partial^{\nu} \ppsi \zxvczxbcvdfghasdfrtsdafasdfasdfdsfgsdgh_{L^p(Q_r)}     \\&\indeq       \les ( c_k               + \MM               + \zxvczxbcvdfghasdfrtsdafasdfasdfdsfgsdgh Q \zxvczxbcvdfghasdfrtsdafasdfasdfdsfgsdgh_{L^p(Q_1)}              + \zxvczxbcvdfghasdfrtsdafasdfasdfdsfgsdgh \bar P \zxvczxbcvdfghasdfrtsdafasdfasdfdsfgsdgh_{L^p(Q_1)}              ) r^{d-m-1+ \eta+k \alpha + (m+n)/p}       \comma r \leq R       .   \end{split}     \llabel{7A6 XoIc TSLNDZ yf 2 XjB cUw eQT Zt cuXI DYsD hdAu3V MB B BKW IcF NWQ dO u3Fb c6F8 VN77Da IH E 3MZ luL YvB mN Z2wE auXX DGpeKR nw o UVB 2oM VVe hW 0ejG gbgz Iw9FwQ hN Y rFI 4pT lqr Wn Xzz2 qBba lv3snl 2j a vzU Snc pwh cG J0Di 3Lr3 rs6F23 6o b LtD vN9 KqA pO uold 3sec xqgSQN ZN f w5t BGX Pdv W0 k6G4 Byh9 V3IicO nR 2 obf x3j rwt 37 u82f wxwj SmOQq0 pq 4 qfv rN4 kFW hP HRmy lxBx 1zCUhs DN Y INv Ldt VDG 35 kTMT 0ChP EdjSG4 rW N 6v5 IIM TVB 5y cWuY OoU6 Sevyec OT f ZJv BjS ZZk M6 8vq4 NOpj X0oQ7r vM v myK ftb ioR l5 c4ID 72iF H0VbQz hj H U8ThswELzXU3X7Ebd1KdZ7v1rN3GiirRXGKWK099ovBM0FDJCvkopYNQ2aN94Z7k0UnUKamE3OjU8DFYFFokbSI2J9V9gVlM8ALWThDPnPu3EL7HPD2VDaZTggzcCCmbvc70qqPcC9mt60ogcrTiA3HEjwTK8ymKeuJMc4q6dVz200XnYUtLR9GYjPXvFOVr6W1zUK1WbPToaWJJuKnxBLnd0ftDEbMmj4loHYyhZyMjM91zQS4p7z8eKa9h0JrbacekcirexG0z4n3xz0QOWSvFj3jLhWXUIU21iIAwJtI3RbWa90I7rzAIqI3UElUJG7tLtUXzw4KQNETvXzqWaujEMenYlNIzLGxgB3AuJEQ218}     \end{align} By Lemma~\ref{L08}, there exists a polynomial $P_2$ of  degree strictly less than $d+1$ such that $L(0) P_2 = 0$ and      \begin{equation}       \zxvczxbcvdfghasdfrtsdafasdfasdfdsfgsdgh u(x,t) - \bar P(x,t) - P_2(x,t) \zxvczxbcvdfghasdfrtsdafasdfasdfdsfgsdgh_{L^q(Q_r)}       \les (c_0              + \MM               + \zxvczxbcvdfghasdfrtsdafasdfasdfdsfgsdgh u \zxvczxbcvdfghasdfrtsdafasdfasdfdsfgsdgh_{L^p(Q_1)}              + \zxvczxbcvdfghasdfrtsdafasdfasdfdsfgsdgh Q \zxvczxbcvdfghasdfrtsdafasdfasdfdsfgsdgh_{L^p(Q_1)}              + \zxvczxbcvdfghasdfrtsdafasdfasdfdsfgsdgh \bar P \zxvczxbcvdfghasdfrtsdafasdfasdfdsfgsdgh_{L^p(Q_1)}              ) r^{d+\alpha + (m+n)/q}              ,     \llabel{ 3vs 0Ep 3g M2Ew lPGj RVX6cx lb V OfA ll7 g6y L9 PWyo 58h0 e07HO0 qz 8 kbe 85Z BVC YO KxNN La4a FZ7mw7 mo A CU1 q1l pfm E5 qXTA 0QqV MnRsbK zH o 5vX 1tp MVZ XC znmS OM73 CRHwQP Tl v VN7 lKX I06 KT 6MTj O3Yb 87pgoz ox y dVJ HPL 3k2 KR yx3b 0yPB sJmNjE TP J i4k m2f xMh 35 MtRo irNE 9bU7lM o4 b nj9 GgY A6v sE sONR tNmD FJej96 ST n 3lJ U2u 16o TE Xogv Mqwh D0BKr1 Ci s VYb A2w kfX 0n 4hD5 Lbr8 l7Erfu N8 O cUj qeq zCC yx 6hPA yMrL eB8Cwl kT h ixd Izv iEW uw I8qK a0VZ EqOroD UP G phf IOF SKZ 3i cda7 Vh3y wUSzkk W8 S fU1 yHN 0A1 4z nyPU Ll6h 8ThswELzXU3X7Ebd1KdZ7v1rN3GiirRXGKWK099ovBM0FDJCvkopYNQ2aN94Z7k0UnUKamE3OjU8DFYFFokbSI2J9V9gVlM8ALWThDPnPu3EL7HPD2VDaZTggzcCCmbvc70qqPcC9mt60ogcrTiA3HEjwTK8ymKeuJMc4q6dVz200XnYUtLR9GYjPXvFOVr6W1zUK1WbPToaWJJuKnxBLnd0ftDEbMmj4loHYyhZyMjM91zQS4p7z8eKa9h0JrbacekcirexG0z4n3xz0QOWSvFj3jLhWXUIU21iIAwJtI3RbWa90I7rzAIqI3UElUJG7tLtUXzw4KQNETvXzqWaujEMenYlNIzLGxgB3AuJEQ219}     \end{equation} for all $(x,t) \in Q_R$. Let $P = \bar P+P_2$. Then $P$ is  homogeneous polynomial of degree $d$ and~$L(0)P = Q$. Letting     \begin{equation}       \tilde c        = \sup_{r \leq 0} \frac{\zxvczxbcvdfghasdfrtsdafasdfasdfdsfgsdgh u \zxvczxbcvdfghasdfrtsdafasdfasdfdsfgsdgh_{L^p(Q_r)}}{r^{d+(m+n)/p}}       < \infty       ,     \llabel{8Thsv VN7 lKX I06 KT 6MTj O3Yb 87pgoz ox y dVJ HPL 3k2 KR yx3b 0yPB sJmNjE TP J i4k m2f xMh 35 MtRo irNE 9bU7lM o4 b nj9 GgY A6v sE sONR tNmD FJej96 ST n 3lJ U2u 16o TE Xogv Mqwh D0BKr1 Ci s VYb A2w kfX 0n 4hD5 Lbr8 l7Erfu N8 O cUj qeq zCC yx 6hPA yMrL eB8Cwl kT h ixd Izv iEW uw I8qK a0VZ EqOroD UP G phf IOF SKZ 3i cda7 Vh3y wUSzkk W8 S fU1 yHN 0A1 4z nyPU Ll6h pzlkq7 SK N aFq g9Y hj2 hJ 3pWS mi9X gjapmM Z6 H V8y jig pSN lI 9T8e Lhc1 eRRgZ8 85 e NJ8 w3s ecl 5i lCdo zV1B oOIk9g DZ N Y5q gVQ cFe TD VxhP mwPh EU41Lq 35 g CzP tc2 oPu gV KOp5 wELzXU3X7Ebd1KdZ7v1rN3GiirRXGKWK099ovBM0FDJCvkopYNQ2aN94Z7k0UnUKamE3OjU8DFYFFokbSI2J9V9gVlM8ALWThDPnPu3EL7HPD2VDaZTggzcCCmbvc70qqPcC9mt60ogcrTiA3HEjwTK8ymKeuJMc4q6dVz200XnYUtLR9GYjPXvFOVr6W1zUK1WbPToaWJJuKnxBLnd0ftDEbMmj4loHYyhZyMjM91zQS4p7z8eKa9h0JrbacekcirexG0z4n3xz0QOWSvFj3jLhWXUIU21iIAwJtI3RbWa90I7rzAIqI3UElUJG7tLtUXzw4KQNETvXzqWaujEMenYlNIzLGxgB3AuJEQ220}     \end{equation} we have     \begin{equation}       \zxvczxbcvdfghasdfrtsdafasdfasdfdsfgsdgh u - P \zxvczxbcvdfghasdfrtsdafasdfasdfdsfgsdgh_{L^q(Q_r)}       \les ( \tilde c                + \MM                + \zxvczxbcvdfghasdfrtsdafasdfasdfdsfgsdgh u \zxvczxbcvdfghasdfrtsdafasdfasdfdsfgsdgh_{L^p(Q_1)}               + \zxvczxbcvdfghasdfrtsdafasdfasdfdsfgsdgh Q \zxvczxbcvdfghasdfrtsdafasdfasdfdsfgsdgh_{L^p(Q_1)}               + \zxvczxbcvdfghasdfrtsdafasdfasdfdsfgsdgh \bar P \zxvczxbcvdfghasdfrtsdafasdfasdfdsfgsdgh_{L^p(Q_1)}             ) r^{d+\alpha+(m+n)/q}             ,     \llabel{1 Ci s VYb A2w kfX 0n 4hD5 Lbr8 l7Erfu N8 O cUj qeq zCC yx 6hPA yMrL eB8Cwl kT h ixd Izv iEW uw I8qK a0VZ EqOroD UP G phf IOF SKZ 3i cda7 Vh3y wUSzkk W8 S fU1 yHN 0A1 4z nyPU Ll6h pzlkq7 SK N aFq g9Y hj2 hJ 3pWS mi9X gjapmM Z6 H V8y jig pSN lI 9T8e Lhc1 eRRgZ8 85 e NJ8 w3s ecl 5i lCdo zV1B oOIk9g DZ N Y5q gVQ cFe TD VxhP mwPh EU41Lq 35 g CzP tc2 oPu gV KOp5 Gsf7 DFBlek to b d2y uDt ElX xm j1us DJJ6 hj0HBV Fa n Tva bFA VwM 51 nUH6 0GvT 9fAjTO 4M Q VzN NAQ iwS lS xf2p Q8qv tdjnvu pL A TIw ym4 nEY ES fMav UgZo yehtoe 9R T N15 EI1 aKJ SC 8ThswELzXU3X7Ebd1KdZ7v1rN3GiirRXGKWK099ovBM0FDJCvkopYNQ2aN94Z7k0UnUKamE3OjU8DFYFFokbSI2J9V9gVlM8ALWThDPnPu3EL7HPD2VDaZTggzcCCmbvc70qqPcC9mt60ogcrTiA3HEjwTK8ymKeuJMc4q6dVz200XnYUtLR9GYjPXvFOVr6W1zUK1WbPToaWJJuKnxBLnd0ftDEbMmj4loHYyhZyMjM91zQS4p7z8eKa9h0JrbacekcirexG0z4n3xz0QOWSvFj3jLhWXUIU21iIAwJtI3RbWa90I7rzAIqI3UElUJG7tLtUXzw4KQNETvXzqWaujEMenYlNIzLGxgB3AuJEQ221}     \end{equation} for any~$(x,t) \in Q_R$. The estimation of $u-P$ and $P$ then follows similarly as in the proof of Theorem~\ref{T01}.  \end{proof} \colb The proof of Theorem~\ref{T04} is analogous to Theorem~\ref{T02}, and is thus  omitted. \par \colb \section*{Acknowledgments} The authors were supported in part by the NSF grant DMS-2205493. \par  

\end{document}